\documentclass[DIV=14]{scrartcl}
\usepackage{preamble} 

\addtokomafont{disposition}{\rmfamily}

\title{\textsc{diffeomorphism groups of solid tori \\ and the rational pseudoisotopy stable range}}
\author{João Lobo Fernandes \and Samuel Muñoz-Echániz}
\date{}

\begin{document}
\raggedbottom


\maketitle

\begin{abstract}
    \noindent \small \textsc{abstract.} \ We compute the rational homotopy groups of the classifying space
\(\BDiffb(S^1 \times D^{d-1})\) of the topological group of diffeomorphisms of \(S^1 \times D^{d-1}\) fixing the boundary for \(d \geq 6\), in a range of degrees up until around \(d\). This extends results of Budney--Gabai, Bustamante--Randal-Williams, and Watanabe.

As consequences of this computation, we determine the rational pseudoisotopy stable range for compact spin manifolds with fundamental group \(\ZZ\) of dimension $d\geq 6$ to be \([0,d-5]\), and compute in this range the rational homotopy groups of \(\BDiff_\partial(S^1 \times N)\) for compact simply-connected spin \((d-1)\)-manifolds \(N\). Finally, by combining our results with work of Krannich--Randal-Williams and Kupers--Randal-Williams on \(\BDiffb(D^d)\), we compute the rational homotopy groups of the space \(\Emb_\partial(D^{d-2}, D^d)\) of long knots in codimension \(2\) for \(d \geq 6\), again in the same range.
\end{abstract}

\tableofcontents

\section{Introduction}

The homotopy type of the classifying space $\BDiffb(M)$ of the topological group of diffeomorphisms of a smooth manifold\footnote{By \textit{manifold}, we will always mean \emph{smooth} manifold.} $M$ fixing its boundary has become one of the main objects of interest in the interplay between geometric topology and homotopy theory. Classically, this has been studied via the \emph{surgery-pseudoisotopy approach}, a two-step program that, first, approximates $\BDiffb(M)$ by a \emph{larger} classifying space of \emph{block diffeomorphisms} $\BbDiffb(M),$ which can be understood via \emph{surgery theory}, and second, describes the failure of this approximation in a range around $\frac{d}{3}$, where $d$ is the dimension of $M$, in terms of Waldhausen's algebraic $K$-theory of $M$ (more on this in the next subsection). While this program has been very useful and successful in providing low degree information about $\BDiffb(M)$ (see e.g. \cite{FarrellHsiang, BurgheleaCPn, BurgheleaLashofTransfer}), it fails to address this homotopy type away from the aforementioned range of $\frac{d}{3}$.     
\vspace{2mm}

In recent years, and out of the vision of Michael Weiss (as seen in \cite{dalian}), a new program for studying $\BDiffb(M)$ has emerged, which, in principle, allows one to understand this homotopy type in an unlimited range of degrees. Among several successes of this program, we highlight the computations of Kupers--Randal-Williams \cite{evendisc}, Krannich \cite{krannichpseudo}, and Krannich--Randal-Williams \cite{krannich2021diffeomorphismsdiscssecondweiss} on the rational homotopy groups $\pi_*^\QQ(\BDiffb(D^d))$ in ranges far from the one mentioned above. In turn, hinging on these computations, one can obtain information about $\BDiffb(M)$ for any compact simply-connected spin manifold $M$ (see e.g. \cite[Cor. C]{KKGoodwillie}). We can thus view $\BDiffb(D^d)$ as a \emph{stepping stone} to understanding the diffeomorphism group of any compact simply-connected spin manifold.
\vspace{2mm}


In this work, we focus on diffeomorphism groups of compact spin manifolds with fundamental group $\ZZ$. For this goal, the \textit{solid torus} $S^1\times D^{d-1}$ is a natural analog to $D^d$. By an adaptation, due to Bustamante--Randal-Williams \cite{BRW}, of the program just mentioned, together with recent advances of the first-named author \cite{joaoTriads} to treat odd-dimensional manifolds and of the second-named author \cite{MElongknots} to simplify the approach via \textit{trace methods}, we carry out a computation of the rational homotopy groups of $\BDiffb(S^1\times D^{d-1})$ in a certain range, and derive several corollaries. In the following sections, we present our main results.

\subsection{Diffeomorphisms and concordances of solid tori}

The homotopy type of $\BDiffb(S^1\times D^{d-1})$ has previously been studied by Budney--Gabai \cite{budneyGabai}, Bustamante--Randal-Williams \cite{BRW}, and Watanabe \cite{watanabe}. Our first main result extends our current knowledge of its rational homotopy type in high-dimensions---that is, $d\geq 6$---in the following way. Throughout this paper, we will denote the rational homotopy groups of a space $X$ by $\pi_*^\QQ(X)$. 

\begin{mainteo}\label{main even}\label{main odd}
    For $n\geq 3$, $d\in \{2n,2n+1\}$ and $1\leq *\leq 2n-3$, there is a surjective homomorphism 
    \[\pi_*^\QQ(\BDiffb(S^1\times D^{d-1}))\longrightarrow \left\{\begin{array}{cl}
        0 & \text{if $d=2n$,}\\
        \K^{\QQ}_{*+1}(\ZZ[\ZZ]) & \text{if $d=2n+1$,}
    \end{array}
    \right.\]
    which is an isomorphism in degrees $1\leq *\leq 2n-4$, where $\K^\QQ_*(\ZZ[\ZZ])$ denotes the rationalised algebraic $K$-theory groups of the ring $\ZZ[\ZZ]$ of Laurent polynomials. Moreover, the kernel of this homomorphism is a countably infinite-dimensional $\QQ$-vector space in degree $*=2n-3$.
\end{mainteo}

\begin{rmk}[Relation to previous work] \cref{main even} recovers \cite[Thm.~B]{BRW} of Bustamante--Randal-Williams by setting $d = 2n$ and $* \leq n-3$. Thus, our result improves loc.cit. in two ways: first, it effectively doubles the range, and second, it includes the odd-dimensional case $d = 2n+1$. In a different direction, Budney--Gabai and Watanabe \cite{budneyGabai,watanabe} discovered that the $(d-3)$-rd rational homotopy group of $\BDiffb(S^1 \times D^{d-1})$ is infinitely generated for all\footnote{Although stated only for $d=4$, their proof shows that $\pi_{d-3}^{\QQ}\bigl(\BDiffb(S^1\times D^{d-1})\bigr)$ is infinitely generated for all $d\ge 4$.} $d \geq 4$. \cref{main even} recovers this result for even $d \geq 6$, and presents a new source of infinite generation for odd $d \geq 7$, namely in degree $d-4$.
\end{rmk}

\begin{rmk}
    The rationalised algebraic $K$-theory groups $\K^\QQ_*(\ZZ[\ZZ])$ featuring in \cref{main even} above can be completely computed by combining Borel's computation \cite{BorelKZ} of $\K^\QQ_*(\ZZ)$ with the Bass--Heller--Swan splitting \cite{BassHellerSwan}. Explicitly, there is an isomorphism 
    \[\K^\QQ_*(\ZZ[\ZZ])\cong \begin{cases}
        \QQ &\text{if $*=0,1$ or $*\equiv 1,2$ (mod $4$) with $*\geq 5$,}\\
        0 &\text{otherwise.}
    \end{cases}\]
\end{rmk}

\begin{rmk}
    Given a compact manifold $M$, its diffeomorphism group $\Diffb(M)$ is an infinite-dimensional Frechét manifold and hence, by work of Palais \cite{PalaisDiffBanach}, has the homotopy type of a countable CW complex. In particular, its (rational) homotopy groups are countably generated. Since all countably infinite-dimensional $\QQ$-vector spaces are abstractly isomorphic to $\QQ^\infty$, exhibiting an infinitely generated subspace of $\pi_*^\QQ(\Diffb(M))$ in effect determines this vector space up to abstract isomorphism. (The same applies to other automorphism spaces considered in this text.) From now on, by \emph{infinite-dimensional} or \emph{infinite generation} we shall always mean \emph{countably} infinite.
\end{rmk}

\begin{rmk}
    By smoothing theory, there is a similar description for the rational homotopy groups of the classifying space $\BTop_\partial(S^1\times D^{d-1})$ of the group of homemomorphisms of solid tori---see \cref{Homeomorphisms}.
\end{rmk}


 In the remainder of this introduction, we discuss applications of this computation to the rational homotopy type of diffeomorphism groups of more general manifolds, as well as to other variants of moduli spaces of manifolds. We begin by focusing on the topological group $\C(S^1\times D^{d-1})$ of \textit{concordance diffeomorphisms} (or \textit{pseudoisotopies}) of solid tori, where $\C(M)\coloneq \Diff_{M\times \{0\}}(M\times [0,1])$. This group serves as a bridge between diffeomorphism groups of manifolds of different dimensions, as it fits into a fibre sequence
\begin{equation}\label{rest fib seqe}
    \Diffb(M\times [0,1])\longrightarrow \C(M)\longrightarrow \Diffb(M).
\end{equation}
Moreover, its classifying space $\BC(M)$ plays an essential role in the surgery-pseudoisotopy approach mentioned above, as we will discuss at length in the next section. Our second main result is a computation of its rational homotopy groups in a range of degrees when $M=S^1\times D^{d-1}$.

\begin{mainteo}\label{main conc intro}
    For $n\geq 3$, $d\in\{2n,2n+1\}$ and $1\leq *\leq 2n-3$, there is a surjective homomorphism 
    \[\pi_*^\QQ(\BC(S^1\times D^{d-1}))\longrightarrow
        \K^{\QQ}_{*+1}(\ZZ[\ZZ])\]
    which is an isomorphism in degrees $1\leq *\leq 2n-4$. Moreover, the kernel of this homomorphism is infinite-dimensional in degree $*=2n-3$.    
\end{mainteo}


\begin{rmk}\label{restrictionRmkIntro} Our methods can also analyse the restriction map $r:\BC(S^1\times D^{d-1})\to \BDiffb(S^1\times D^{d-1})$:

\begin{enumerate}[label = (\roman*)]

    \item When $d=2n+1\geq 7$, it essentially follows from \cref{main even} (see \cref{main odd conc}) that $r$ is rationally $(2n-1)$-connected. In particular, since $\pi_{2n-2}^{\QQ}(\BDiffb(S^1\times D^{2n}))$ is infinite-dimensional by \cite{budneyGabai, watanabe}, the same holds for $\smash{\pi_{2n-2}^{\QQ}(\BC(S^1\times D^{2n}))}$.
    
    \item When $d=2n$, by \cref{main even,main conc intro}, the rational homotopy groups in degree $2n-3$ of domain and target of $r$ are both infinite-dimensional $\QQ$-vector spaces. We can additionally prove that the homomorphism $\smash{\pi_{2n-3}^{\QQ}(r)}$ is surjective with infinitely generated kernel (see \cref{surjectionInWatanabeDegree}). 
\end{enumerate} 

Both of these can be seen as an improvement of \cite[Thm. 1.3 and Cor. 1.5]{watanabe} and \cite[Thm. 1.3]{Botvinnik2023}.
\end{rmk}

\begin{rmk}
    Technically, \cref{main conc intro} is not directly deducible from \cref{main even}. In fact, the former will be a key step in showing the latter. We refer to \cref{strategy section} for an explanation of the strategy of proof. 
\end{rmk}

\subsection{Rational pseudoisotopy range for manifolds with infinite cyclic fundamental group}

One immediate observation from \cref{main conc intro} is that, in the range of the theorem, the rational homotopy groups of $\BC(S^1\times D^{d-1})$ are independent of the dimension $d\geq 6$. In a considerably lower range, this phenomenon is explained by an important result in geometric topology: \textit{Igusa's stability theorem} \cite{IgusaStableRange}. This result is concerned with the connectivity of the \textit{concordance stabilisation map}
\[s_M:\C(M)\longrightarrow \C(M\times [0,1])\]for a compact manifold $M$ (see e.g. \cite[Chapter II, \S 1]{IgusaStableRange}), and, in particular, establishes $\frac{\dim(M)}{3}$ as an asymptotic lower bound for the \textit{concordance} (or \textit{pseudoisotopy}) \textit{stable range} of $M$, denoted
\[
\phi(M)\coloneq \max\{ k\in \ZZ ~~|~~ s_{M\times [0,1]^m} \text{ is $k$-connected for all $m\geq 0$}\}.
\]
(Replacing $k$-connected above by \emph{rationally} $k$-connnected gives the \textit{rational concordance stable range} $\phi^{\QQ}(M)$.) The (rational) concordance stable range plays a central role in the \textit{surgery-pseudoisotopy approach} for the (rational) homotopy type of $\BDiffb(M)$. For example, it is precisely the range in which the failure of approximating the latter space by its block analogue (as mentioned above) is controlled by Waldhausen's algebraic $K$-theory of $M$, as a consequence of the work of Weiss--Williams \cite{WWI}. 

An important property of $\phi(M)$, established by Goodwillie--Krannich--Kupers \cite{KKGoodwillie}, is that it essentially only depends on the tangential $2$-type of $M$ (see p.25 in loc.cit.). As a consequence of this and the work of \cite{krannich2021diffeomorphismsdiscssecondweiss} on the rational concordance stable range of $D^d$ for $d\geq 6$, it was established in \cite[Cor. C]{KKGoodwillie} that the rational concordance stable range for compact simply-connected spin manifolds $M^d$ with $d\geq 6$ is $\phi^\QQ(M)=d-4$. Similarly, by a refined version of \cref{main conc intro}, we deduce an analogous result for compact spin manifolds with infinite cyclic fundamental group.

\begin{mainteo}\label{main range}
    Let $d\geq 6$, and let $M$ be a connected compact spin $d$-dimensional smooth manifold (possibly with boundary) with $\pi_1(M)\cong \ZZ$. Then, the concordance stabilisation map
    \[s_M:\C(M)\longrightarrow \C(M\times [0,1])\]is rationally $(d-5)$-connected\footnote{For $d \geq 6$, the group of path-components $\pi_0(\C(M))$ is an abelian group (hence rationalisable) by work of Hatcher--Wagoner \cite{HatcherWag} and Igusa \cite[Thm. 8.a.1]{IgusaCorrection} (see \eqref{HatcherWagoner}). Thus, the notion of \emph{rational connectivity} is well-defined in this case.}, but not rationally $(d-4)$-connected. That is, $\phi^\QQ(M)=d-5$. Moreover,
    \[\pi_{d-4}^\QQ(\C(M\times [0,1]),\C(M))\]is an infinite-dimensional $\QQ$-vector space.
\end{mainteo}

Exploiting, once again, the tangential invariance of the concordance stable range, we obtain the following upper bound to the rational concordance stable range, as suggested to us by Manuel Krannich.

\begin{maincor}\label{application to G}
    Let $d\geq 6$ and $M$ be a connected compact spin $d$-dimensional smooth manifold (possibly with boundary) such that $\H^1(\pi_1(M);\QQ)\neq 0$. Then, the rational vector space 
    \[\pi_{d-4}^\QQ(\C(M\times [0,1]),\C(M))\]
    is infinite-dimensional. In particular, we have $\phi^\QQ(M)\leq d-5$.
\end{maincor}

Just like embedding spaces play an important role in the study of diffeomorphism groups, spaces of concordance diffeomorphisms are closely related to spaces of \textit{concordance embeddings}. Fix an embedding of compact manifolds $M\subset N$, and let $\CEmb(M,N)$ denote the space of embeddings $M\times [0,1]\hookrightarrow N\times [0,1]$ fixed on $M\times \{0\}$ and that send $M\times \{1\}$ to $N\times \{1\}$ (see \eqref{CEmbIETSeq} for details). There is a stabilisation map
\[s_{M,N}:\CEmb(M,N)\longrightarrow \CEmb(M\times [0,1],N\times [0,1]),\]
and Goodwillie--Krannich--Kupers studied in \cite{KKGoodwillie} the connectivity of this map---roughly, the \textit{concordance embedding stable range} of $M\subset N$---when the handle codimension of $M$ in $N$ is at least $3$. As a corollary of \cref{main range} and the parametrised isotopy extension theorem, we determine the rational concordance embedding stable range for certain embeddings of submanifolds $M\subset N$ of handle codimension $\leq 2$.

\begin{maincor}\label{main range embeddings}
    Let $d\geq 6$, $N$ be a $d$-dimensional compact simply-connected spin smooth manifold and $M$ a compact neat submanifold of $N$ (both possibly with boundary), such that $N-M$ is connected with $\pi_1(N-M)\cong \ZZ$. Then, the concordance embedding stabilisation map 
    \[s_{M,N}:\CEmb(M,N)_\iota\longrightarrow \CEmb(M\times [0,1],N\times [0,1])_\iota\]where $(-)_\iota$ stands for the path-component of the standard inclusion $M\times [0,1]\subset N\times [0,1]$, is rationally $(d-4)$-connected, but not rationally $(d-3)$-connected. In fact, the rational vector space
    \[
    \pi_{d-3}^{\QQ}(\CEmb(M\times [0,1],N\times [0,1])_\iota,\CEmb(M,N)_\iota)
    \]
    is infinite-dimensional.
\end{maincor}

\begin{rmk}\label{main range embeddings long knots}
    The standard inclusion $M=D^{d-2}\subset D^d=N$, for $d\geq 6$, satisfies the hypothesis of \cref{main range embeddings}.
\end{rmk}

\begin{rmk}
    Writing $d=\dim(N)$ and $p$ for the handle dimension of $M$ in $N$, the main result \cite[Thm A]{KKGoodwillie} states that if $p\leq d-3$, then $s_{M,N}$ is $(2d-p-5)$-connected. Somewhat unexpectedly, we see from \cref{main range embeddings} that this formula does not generalise to $p=d-2$, as in this case $2d-p-5=d-3$. 
\end{rmk}

\subsection{Diffeomorphism groups of \texorpdfstring{$S^1\times N$}{S1xN} for \texorpdfstring{$N$}{N} simply-connected spin}

As mentioned before, we view $\BDiffb(S^1\times D^{d-1})$ as a \textit{stepping stone} to understanding $\BDiffb(M)$ for any compact spin manifold $M$ with fundamental group $\ZZ$. As a consequence of the surgery-pseudoisotopy approach together with our results on the rational concordance stable range, we can now justify this belief.

In the following result, write $L(-)\coloneq \Map(S^1,-)$ for the \emph{free loop space} functor, let $\tau_{\FLS}$ denote the \textit{free loop space involution} induced by complex conjugation on $S^1\subset \CC$ and, given an involution $\tau$ on a $\QQ$-vector space $A$, let $[A]_{\tau}$ denote the $C_2$-coinvariants of $A$ with respect to the involution $\tau$.

\begin{maincor}\label{BDiffSpinManifoldsMainCor} Let $N^{d-1}$ be a compact simply-connected spin smooth manifold (possibly with boundary) of dimension $d-1\geq 5$. Then, in degrees $2\leq *\leq d-5$, there is an isomorphism
\begin{align*}
    \pi_*^{\QQ}(\BDiff_\partial(S^1\times N))\cong &\hspace{4pt}\left\{
    \begin{array}{cl}
        0 & \text{if $d$ even,} \\
        \K^{\QQ}_{*+1}(\ZZ[\ZZ]) & \text{if $d$ odd,}
    \end{array}
    \right.\oplus \left\{
    \begin{array}{cl}
        \QQ & \text{if $\partial N=\emptyset$ and $*=2$,} \\
        0 & \text{otherwise,}
    \end{array}
    \right.\\[4pt]
    &\oplus \pi_{*-1}^{\QQ}(\bDiff_\partial(N))\oplus \pi_{*}^{\QQ}(\bDiff_\partial(N))\\[4pt]
    &\oplus \left[\frac{\H_*(LN;\QQ)}{\H_*(N;\QQ)\oplus \H_{*+1}(N;\QQ)}\right]_{(-1)^{d+1}\tau_{\FLS}}\oplus \bigoplus_{r=1}^\infty\H_*(LN;\QQ).
\end{align*}
    Here, the quotient is the cokernel of certain monomorphism $\widetilde{\H}_*(N;\QQ)\oplus \widetilde{\H}_{*+1}(N;\QQ)\hookrightarrow \widetilde{\H}_*(LN;\QQ)$ (see \cref{WeirdMonoLem}), and the first summand is the image of the map $\pi_*^{\QQ}(\BDiff_\partial(S^1\times D^{d-1}))\to \pi_*^{\QQ}(\BDiff_\partial(S^1\times N))$ induced by any embedded disc $D^{d-1}\subset N$. 
    
    In particular, for $2\leq k\leq d-5$, the rational vector space $\smash{\pi_k^{\QQ}(\BDiffb(S^1\times N))}$ is finite-dimensional if and only if $\H_k(LN;\QQ)= 0$. Additionally, $\smash{\pi_{d-4}^{\QQ}(\BDiffb(S^1\times N))}$ is infinite-dimensional if $\H_{d-4}(LN;\QQ)\neq 0$.
\end{maincor}

\begin{rmk}
    By surgery theory, understanding the homotopy type of $\BbDiffb(N)$, especially rationally, is often feasible---in fact, explicit rational models for this homotopy type have already been developed \cite[Thm. 4.24]{BerglundMadsen}. See Examples \ref{CPnExample} and \ref{HomotopySphereExample} for the special cases $N=\CP^n$ and $N=\Sigma$ a homotopy sphere.
\end{rmk}

\begin{rmk}\label{pi0S1xNRmk}
    The group $\pi_1(\BDiffb(S^1\times N))$ is abelian (or nilpotent) precisely when $\pi_0(\bDiff_\partial(N))$ is so, in which case, the formula in the statement above extends to $*=1$. 
\end{rmk}
\begin{rmk}
    The conclusion of \cref{BDiffSpinManifoldsMainCor} is only new in degrees above Igusa’s lower bound for the (integral) concordance stable range (roughly $\smash{d/3}$). Nevertheless, the result has several new consequences beyond exceeding this range. For instance, given that $\H_*(N;\QQ)$ is a split summand of $\H_*(LN;\QQ)$, in \cref{InfGenClosedManifold} we show that for essentially all simply-connected spin closed manifolds $N$ of dimension $d-1\geq 5$, the $\QQ$-vector space $\smash{\pi_*^{\QQ}(\BDiffb(S^1 \times N))}$ is infinitely generated in some degree $2 \leq * \leq d-4$.
\end{rmk}

Thus, one could argue that infinite generation of the rational homotopy groups $\pi_*^{\QQ}(\BDiffb(S^1 \times N))$ is \emph{to be expected}, and far from a striking phenomenon. In the (rational) concordance stable range, the source of infinite generation is the $S^1$-equivariant homology of the free loop space $L(S^1 \times N)$, i.e. the \emph{cyclic homology} of $S^1 \times N$, which is always either zero or infinitely generated. See \cite{FarrellNil} for a closely related result.


\subsection{Long knots in codimension \texorpdfstring{$2$}{2}}

Let us consider another object of interest in geometric topology, namely the space  $\Emb_\partial(D^{d-k}, D^d)$ of \emph{long knots} of codimension $k \geq 0$. In high dimensions, the homotopy type of these spaces has been extensively studied for $k \neq 2$ (when $k=2$, attention has so far mainly focused on its set of path components, as studied in high-dimensional knot theory; see e.g. \cite{Ranicki1998}). For $k = 0$, this space is simply the diffeomorphism group of the disc $\Diffb(D^d)$, which, as already explained, is understood rationally \cite{evendisc,krannichpseudo,krannich2021diffeomorphismsdiscssecondweiss} in an extensive range. For $k = 1$, the component of the unknot is equivalent to the classifying space $\BDiffb(D^d)$ as a consequence of Cerf’s lemma, and is therefore covered by the previous case. For $k \geq 3$, the rational homotopy type of this space is accessible via embedding calculus and its relation to the little discs operad and graph complexes \cite{Dwyer2012,Arone2014,Arone2015,fresse,BoavidadeBrito2018}), or integrally via the embedding surgery-pseudoisotopy program \cite{MElongknots}. Both of these approaches are unavailable in codimension $k = 2$.

Nevertheless, the parametrised isotopy extension theorem yields a fibre sequence
\[\Emb_\partial(D^{d-2}\times \RR^2,D^d)_u\longrightarrow \BDiffb(S^1\times D^{d-1})\longrightarrow \BDiffb(D^d).\]
The natural restriction map $\Emb_\partial(D^{d-2}\times \RR^2,D^d)\to \Emb_\partial(D^{d-2},D^d)$ is an equivalence for $d\geq 4$ (the homotopy fibre is the space of framings $\Omega^{d-2}\O(2)$, which is contractible since $\O(2)$ is $1$-truncated; see \cref{LongKnotsSection}\ref{restricting is we}). Thus, our results on solid tori in \cref{main even}, together with the work of \cite{evendisc,krannichpseudo,krannich2021diffeomorphismsdiscssecondweiss} on $\BDiffb(D^d)$, yield the following computation.

\begin{maincor}\label{main knots}
    For $d\geq 6$ and ${*}\geq 1$, we have an isomorphism
    \[\pi_*^\QQ(\Emb_\partial(D^{d-2},D^{d}),u)\cong \left\{\begin{array}{cl}
        0 &  \text{for } d \text{ even and } {*}<d-3,\\
        \mathrm{\K}^{\QQ}_{*}(\ZZ) &   \text{for } d \text{ odd and } {*}<d-4,
    \end{array}\right.\]where $u$ denotes the unknot. Furthermore, in degree ${*}=d-3$ for $d$ even and degree ${*}= d-4$ for $d$ odd, the left-hand side is an infinite-dimensional $\QQ$-vector space.
\end{maincor}

\begin{rmk} Fresse--Turchin--Willwacher \cite{fresse} studied the embedding calculus approximation of the rationalised presheaf $\Emb^\QQ_\partial(-, D^d)$, denoted $T_\infty \Emb^{\QQ}_\partial(D^{d-k}, D^d)$. Its homotopy groups are rational vector spaces which are, roughly and up to a degree shift, isomorphic to the homology of the hairy graph complex in codimension $k$--see Corollary~3 and equation~(10) of loc.cit. By inspection in the case $k = 2$, one verifies that $\pi_*(T_\infty \Emb^{\QQ}_\partial(D^{d-2}, D^d)_u)$ is finite-dimensional in degrees $1 \leq * \leq 2d-7$ (see Corollary~4 and equations~(1) and~(2) there). By the last part of \cref{main knots}, it follows that the natural comparison map $\Emb_\partial(D^{d-2}, D^d)_u \to T_\infty \Emb^{\QQ}_\partial(D^{d-2}, D^d)_u$ is not a rational equivalence if $d \geq 6$. See also \cite[Cor.~8.20]{KKsdisc}.
\end{rmk}

\subsection{On the orthogonal functor of topological Stiefel manifolds}
In \cite[Thm.~D]{krannich2021diffeomorphismsdiscssecondweiss}, the authors use their results on the rational homotopy type of $\BC(D^d)$ to determine that of the second derivative spectrum of the orthogonal functor $\Bt(V) \coloneq \BTop(V)$; see Section~1.5 of loc.cit. for a brief account of Weiss’ \emph{orthogonal calculus} \cite{WeissOrthCalc}. Analogously, our results admit a reinterpretation in this framework (see \cref{VtSection} for the explicit connection)--the relevant orthogonal functor is now
\[
\Vt_2: V\longmapsto \Top(V\oplus \RR^2)/\Top(V,V\oplus \RR^2)\simeq \hofib(\BTop(V,V\oplus \RR^2)\to \BTop(V\oplus \RR^2)).
\]
Here, $\Top(V, V \oplus \RR^2)$ denotes the topological group of homeomorphisms of $V \oplus \RR^2$ that fix the subspace $V$ pointwise. In particular, $\Vt_2(V)$ is the \emph{topological Stiefel manifold} of $V$-planes in $V \oplus \RR^2$. What is currently known about this orthogonal functor is limited to the following:
\begin{itemize}
    \item The zero-th orthogonal approximation of $\Vt_2$ is $P_0\Vt_2\simeq \Top/\Top(2)$ by \cite[Thm. B]{KirbySiebenmannTopmn}.

    \item The first derivative of $\Vt_2$ is $1$-connective\footnote{That is, its homotopy groups vanish in degrees $*\leq0 $.} and satisfies $\Omega^{\infty+1}(\Theta \Vt_2^{(1)})\simeq \Omega^\infty (\K(\bfS))$; see \cref{Theta1Lemma}.
\end{itemize}
As a consequence of \cref{main range embeddings}, we obtain the first results concerning the second derivative of $\Vt_2$.

\begin{maincor}\label{RelativeBTopMainCor}
    The second Weiss derivative $\Theta \Vt^{(2)}_2$ is rationally (at least) $(-1)$-connective. 
\end{maincor}
\begin{rmk}
    The proof of \cref{RelativeBTopMainCor} actually shows that, for $d\geq 4$, the unstable derivatives \(\smash{\Vt_2^{(2)}(\RR^d)}\) are rationally \((2d-2)\)-connected but not rationally \((2d-1)\)-connected. With this alone, we cannot conclude that \(\smash{\Theta \Vt_2^{(2)}}\) is not rationally \(0\)-connective, since this would also require understanding the structure maps \(\smash{\Vt_2^{(2)}(\RR^d)\to \Omega^2 \Vt_2^{(2)}(\RR^{d+1})}\) on \(\smash{\pi_{2d-1}^{\QQ}}\). Nevertheless, one can show that these homomorphisms for $d\geq 6$ \emph{even} are injective on an infinite-dimensional subspace of the domain.
\end{rmk}

\subsection{Strategy of the proof of \texorpdfstring{\cref{main even}}{Theorem A}}\label{strategy section} As mentioned at the beginning of this introduction, most of our results are consequences of \cref{main even}. We now outline the strategy of its proof. In the even-dimensional case $d=2n$, we follow the approach of \cite{BRW} by studying a variant of the \textit{Weiss fibre sequence}. More precisely, we look at a certain fibre sequence
\begin{equation}\label{WeissXgFibSeqIntroduction}
\BDiffb(S^1\times D^{2n-1})\longrightarrow \BDiffb(X_g)\longrightarrow \BEmbcong(X_g),
\end{equation}
or rather, the framed variant \eqref{plusConstructedWeissXg} of it, where $X_g$ denotes the connected sum of $S^1\times D^{2n-1}$ with $ (S^n\times S^n)^{\# g}$. This reduces the study of $\BDiffb(S^1\times D^{2n-1})$ to that of the remaining spaces in this sequence. The homology of the total space can be completely understood for $g$ large enough, as a consequence of the seminal work of Galatius and Randal-Williams \cite{GRWStableModuli,GRWII} on stable moduli spaces of even-dimensional manifolds. Our main contribution lies in the computation of the homology of the base space in a range doubling the one obtained in \cite{BRW}. The main novelty of our approach in this even-dimensional case is to use the work of the second named author \cite{MElongknots} to simplify the analysis of this embedding space by using \emph{trace methods}. 

\begin{rmk}[On embedding calculus]
    The approach proposed in \cite{BRW} is analogous to the highly successful one of Kupers--Randal-Williams \cite{evendisc} (building on insights of Weiss \cite{dalian}) for computing the rational homotopy groups of $\BDiffb(D^{2n})$. In \cite{evendisc}, the authors study the corresponding embedding space via \emph{embedding calculus}. Although a similar strategy could in principle be applied in our setting, we encountered significant computational obstacles in carrying it out. The main difficulty is the ubiquitous appearance of infinitely generated $\QQ$-vector spaces, which makes it hard to analyse various differentials in the relevant spectral sequences. In the range of \cref{main even}, our approach avoids many of these obstacles.
\end{rmk}

Our strategy for the odd-dimensional case $d=2n+1$ is analogous to the one above, and is inspired by the work of Krannich--Randal-Williams \cite{krannich2021diffeomorphismsdiscssecondweiss}. Given our knowledge of $\BDiffb(S^1\times D^{2n-1})$ by the previous case, and in view of the fibre sequence \eqref{rest fib seqe}, instead of studying $\BDiffb(S^1\times D^{2n})$ directly, we focus on the concordance space $\BC(S^1\times D^{2n-1})$. This space features also in another \textit{Weiss-type} fibre sequence
\begin{equation}\label{WeissYgXgFibSeqIntroduction}
\BC(S^1\times D^{2n-1})\longrightarrow \BDiffv(Y_g)\longrightarrow \BEmbcong(Y_g;X_g),
\end{equation}
where $Y_g$ denotes the boundary connected sum of $S^1\times D^{2n}$ with the handlebody $(S^n\times D^{n+1})^{\natural g}$. In turn, the homology of the total space of this fibre sequence can be completely understood for $g$ large enough, now as consequence of the work of first named author \cite{joaoTriads} on stable moduli spaces of odd-dimensional triads. We approach the base space of this sequence with the same method as above, relying on the results in \cite{MElongknots}. This strategy allows us to deduce \cref{main conc intro} for $d$ even. By combining the fibre sequence \eqref{rest fib seqe} with the computations above, we can deduce both \cref{main even,main conc intro} for $d$ odd.

\vspace{3mm}

\noindent \textbf{Structure of the paper.} For the reader’s convenience, we provide a brief overview of the paper.

\cref{preliminarySection} introduces the notation, spaces, and constructions that will play a role throughout. In particular, \cref{hAutSection} studies the homotopy type of the space of homotopy automorphisms of the pair $(Y_g;X_g)$. We also recall the analysis of the automorphism space of $X_g$ by Bustamante--Randal-Williams \cite[\S 3.1]{BRW}.

Sections \ref{MCGSection}, \ref{SelfEmbeddingXgSection}, and \ref{SelfEmbeddingYgSection} study the self-embedding spaces of $X_g$ and of the pair $(Y_g;X_g)$ (more precisely, their stably framed versions) appearing in the Weiss fibre sequences \eqref{WeissXgFibSeqIntroduction} and \eqref{WeissYgXgFibSeqIntroduction}. Section \ref{MCGSection} focuses on their path-components, i.e.\ the associated mapping class groups. The case of $X_g$ was analysed in \cite{BRW} using surgery theory; in \cref{YgXgMCGSection} we adapt related methods, based on \cite{krannichpseudo}, to treat the case of $(Y_g;X_g)$.

In \cref{SelfEmbeddingXgSection} we investigate the rational homotopy and homology type of the self-embedding space of $X_g$ in degrees up to roughly $2n-2$. We begin by analysing its higher rational homotopy groups as representations of the mapping class group, using a new approach based on trace methods and the embedding Weiss--Williams theorem of \cite{MElongknots}. An overview of this approach is given in \cref{RecollectionWWSection}, which we expect may be of independent interest. In \cref{homotopyOfPlusSectionXg} we then use this analysis to compute the homology of the embedding space. Section \ref{SelfEmbeddingYgSection} carries out the analogous analysis for the self-embedding space of the pair $(Y_g;X_g)$.

In \cref{ProofSection} we prove all of the main theorems stated in the introduction, along with several related results. \cref{AppendixWW} addresses some technicalities in the embedding Weiss--Williams approach of \cref{SelfEmbeddingXgSection}.\\[-5pt]

\vspace{3mm}

\noindent \textbf{Aknowledgments.} We would first like to express our deepest gratitude to Manuel Krannich and Oscar Randal-Williams, who first suggested this project to us. Manuel Krannich suggested the use of \cite{joaoTriads} to extend our computations to the odd-dimensional case, gave us helpful comments on the first drafts, and several ideas in this paper, including the approach of \cref{FrobeniusXgPseudoIsotopySection} and \cref{application to G}, grew out of illuminating conversations with him. Oscar Randal-Williams suggested to SME the application of \cite[Thm. A]{MElongknots} to the study of diffeomorphisms of solid tori, pointed out a mistake related to \cref{CEmbLongKnotLemma}, and provided valuable guidance at many stages of the project. We also thank Alexander Kupers, Florian Kranhold, Juan Muñoz-Echániz and Robin Stoll for helpful conversations and comments on the introduction. JLF was supported by the Deutsche Forschungsgemeinschaft (DFG, German Research Foundation) through the Research Training Group DFG~281869850 (RTG~2229) and partially by the European Union through the ERC grant MaFC (101221003). SME was partially supported by an EPSRC PhD Studentship, grant no.~2597647.

\section{Preliminaries}\label{preliminarySection}

\begin{nota}\label{manifolds definition}
    Let $n\geq 3$ and $g\geq 0$ be integers. We consider the following families of manifolds.
    \begin{enumerate}[label=(\roman*)]
        \item $\wg$ is the $g$-fold boundary connected sum of $S^n\times S^n\backslash \int(D^{2n})$, where $D^{2n}$ is an embedded $2n$-disc. We consider $\wg$ as a triad with the decomposition $\partial_0\wg\coloneq D^{2n-1}_+$ and $\partial_1\wg\coloneq D^{2n-1}_-$ of its boundary $\partial D^{2n}$. Here, $D^{2n-1}_+$ stands for the subspace of those tuples $(x_1,\dots,x_n)\in \partial D^{2n}$ where $x_{2n}\geq 0$ and $D^{2n-1}_-$ for the subspace of those tuples where $x_n\leq 0$.
        \item $\vg$ is the $g$-fold boundary connected sum of $S^n\times D^{n+1}.$ Its boundary is the $g$-fold connected sum of $S^n\times S^n$. Notice that $\wg$ embedds into $\partial \vg$ and its complement is a $2n$-disc $D^{2n}.$ We consider $\vg$ as a $4$-ad with the decomposition $\partial_0 \vg \coloneq D^{2n}_+$, $\partial_1\vg \coloneq D^{2n}_-$ and $\partial_2 \vg\coloneq \wg$. Here $D^{2n}_+$ stands for the subspace of $D^{2n}$ of those tuples $(x_1,\dots,x_n)\in D^{2n}$ where $x_{2n}\geq 0$ and $D^{2n-1}_-$ for the subspace of those tuples where $x_n\leq 0.$ Denote by $\partial_{ij}\vg$ the intersection of $\partial_i\vg$ and $\partial_j\vg$ for $i,j\in \{0,1,2\}$ and $\partial_{012}\vg$ for the common intersection $\partial_0\vg\cap \partial_1\vg\cap\partial_2\vg.$
        \item $\xg$ is the boundary connected sum of $S^1\times D^{2n-1}$ with $\wg.$ We consider $X_g$ as a triad with the following decomposition of its boundary $\partial(S^1\times D^{2n-1})\# \partial D^{2n}$: $\partial_0\xg\coloneq (S^1\times D^{2n-2}_+)\natural \partial_0\wg$ and $\partial_1\xg\coloneq (S^1\times D^{2n-2}_-)\natural \partial_1\wg$. The boundary connected sum above is assumed to be at a point in $S^1\times S^{2n-3}=(S^1\times D^{2n-2}_+)\cap (S^1\times D^{2n-2}_-).$
        \item $\yg$ is the boundary connected sum of $S^1\times D^{2n}$ with $\vg.$ Note that $\xg$ embeds into $\partial(\yg)$ with complement diffeomorphic to $S^1\times \operatorname{int}D^{2n-1}$, by seeing $S^1\times D^{2n-1}$ as a subspace of the boundary of $S^1\times D^{2n}$ given by the upper hemisphere $D^{2n-1}_+$ in the $D^{2n}$-coordinate. We consider $\yg$ as a $4$-ad with the following decomposition $\partial_0\yg \coloneq S^{1}\times D^{2n-1}_+$, $\partial_1\yg \coloneq S^1\times D^{2n-1}_-$ and $\partial_2\yg\coloneq \xg$. Similarly, denote the subspaces $\partial_{ij}\yg$ and $\partial_{012}\yg$ defined in the same way. 
    \end{enumerate}By definition, we have an embedding of triads $\iota:(\wg;\partial_0\wg,\partial_1\wg)\hookrightarrow (\xg;\partial_0\xg,\partial_1\xg)$ and an embedding of $4$-ads $\hat{\iota}:(\vg;\partial_0\vg,\partial_1\vg,\partial_2\vg)\hookrightarrow (\yg;\partial_0\yg,\partial_1\yg,\partial_2\yg)$ extending $\iota$ on $\partial_2.$ We fix a basepoint $x\in \partial_{01}(\wg)$ which induces basepoints on $\xg, \vg$ and $\yg$. As mentioned in \cite[Section 2.1]{BRW}, the inclusion $S^1\times D^{2n-1}\subset D^{2n}$ induces an embedding of triads $\kappa:(\xg;\partial_0\xg,\partial_1\xg)\hookrightarrow (\wg;\partial_0\wg,\partial_1\wg)$ such that $\kappa\circ \iota$ is isotopic to the identity. Similarly, $\kappa$ extends to an embedding of $4$-ads $\hat{\kappa}:(\yg;\partial_0\yg,\partial_1\yg,\partial_2\yg)\hookrightarrow (\vg;\partial_0\vg,\partial_1\vg,\partial_2\vg) $ and $\hat{\kappa}\circ \hat{\iota}$ is isotopic to the identity.
\end{nota}

\subsection{Embedding and automorphism spaces}\label{moduliSpacesDefnSection}
Let $(M;\partial_0 M,\partial_1 M)$ be a compact smooth manifold triad, we denote by $\Embo(M)$ the space of self-embeddings of $M$ extend the identity map on a neighbourhood of $\partial_0$ (but might send $\partial_1$ to the interior of $M$), equipped with the $C^\infty$-topology. Composition gives a (topological) monoid structure to this space. As mentioned in \cite[Section 5.1]{Bustamante2021FinitenessPO}, we have a fibre sequence, usually called the \textit{Weiss fibre sequence},
\begin{equation}\label{weiss fibre for triads}
    \BDiffb(\partial_1 M\times [0,1])\to \BDiffb(M)\to \BEmbop(M)
\end{equation}which deloops once to the right, where $\BEmbop(M)$ denotes the classifying space of the group-like submonoid $\Embop(M)$ of those path components of $\Embo(M)$, which are hit from $\Diffb(M).$ Denote also by $\bDiff_\partial(M)$ and by $\bEmbo(M)$ for the (realisation of the simplicial) spaces of block diffeomorphisms and self-embeddings of $M$, and by $\bEmbop(M)$ the union of all path components hit by $\bDiffb(M)$ (see \cite[Section 1.4]{krannichpseudo} for definition of block diffeomorphisms and embeddings) Finally, we denote by $\Aut_\partial(M)$ and $\Aut_{\partial_0}(M)$ for the group-like monoids of self homotopy equivalences of $M$ fixed in $\partial M$ or $\partial_0$, respectively. As observed in \cite[Section 1.5]{krannichpseudo}, the map $\Diffb(M)\to \Aut_\partial(M)$ factors through $\bDiffb(M)$ up to homotopy, since the analogous map $\Aut_\partial(M)\to \bAut_\partial(M)$ is a weak equivalence.

\subsubsection{\texorpdfstring{$4$}{4}-ads}
Similarly, for a compact smooth manifold $4$-ad $(N;\partial_0 N,\partial_1 N,\partial_2N)$, we denote by $\Embo(N;\partial_2 N)$ the space of self-embeddings of $N$ which are the identity on a neighbourhood of $\partial_0$ and fix $\partial_2 N$ setwise (but once again, $\partial_1 N$ is allowed to go to the interior of $N$, as long as $\partial_{12} N$ goes to the interior of $\partial_2 N$). Denote the topological group of diffeomorphisms of $N$ which are the identity on $\vb N\coloneq\partial_0 N\cup \partial_1 N$ by $\Diffv(N).$ As pointed out in \cite[Section 2.3]{krannich2021diffeomorphismsdiscssecondweiss}, there exists an analogous \textit{"Weiss fibre sequence"} of the form
\begin{equation}\label{weiss fibre for 4 ads}
    \BDiffv(\partial_1 N\times [0,1])\to \BDiffv(N)\to \BEmbop(N;\partial_2 N)
\end{equation}where $\vb (\partial_1 N\times [0,1])\coloneq \partial_1N\times \{0\}\cup_{\partial_{01}N\times [0,1]} \partial_1N\times \{1\}$, which deloops once. Here, $\Embop(N;\partial_2 N)$ denotes the path components of $\Embo(N;\partial_2 N)$ hit by $\Diffv(N).$ The sequences \eqref{weiss fibre for triads} and \eqref{weiss fibre for 4 ads} are compatible along the maps that restrict to $\partial_2N. $ Denote the block variants by $\bDiffv(N)$, $\bEmbo(N;\partial_2 N)$ and $\bEmbop(N;\partial_2N)$, defined in the same way as before. Denote by $\Aut_{\vb}(N;\partial_2N)$ the space of self-homotopy equivalences of the pair $(N,\partial_2 N)$ fixed in $\vb N$, with the compact-open topology.

\subsubsection{Stable framings}\label{defn of stable framings section} Consider the map $\sfr_d:\oo(d)\to \BO(d)$ given by the homotopy fibre of $\BO(d)\to \BO$. A \textit{stable framing} on a $d$-dimensional smooth manifold is a fibrewise isomorphism $\ell:TM\cong (\sfr_d)^*\gamma_d,$ where $\gamma_d$ denotes the universal $d$-dimensional vector bundle over $\BO(d)$. If $K\subset \partial M$ is a codimension $0$ submanifold, a stable framing $\ell$ on $M$ induces a fibrewise isomorphism $\ell|_{K}:TK\times [0,1]\cong TM|_K\cong (\sfr_d)^{*}\gamma_d,$ which induces an (essentially unique) stable framing $\ell_K:TK\cong (\sfr_{d-1})^*\gamma_{d-1}$ on $K$: this follows by the fact that the square
\[\begin{tikzcd}
    \oo(d-1)\arrow[d]\arrow[r] & \BO(d-1) \arrow[d] \\
    \oo(d) \arrow[r] &\BO(d)
\end{tikzcd}\]
is homotopy cartesian. Given a $d$-dimensional triad $(M;\partial_0 M,\partial_1 M)$ and a stable framing $\ell_\partial$ of $\partial M$, the spaces $\Diffb(M)$ and $\bDiffb(M)$ act (in the $A_\infty$ sense) on the space $\Bun_\partial(TM,(\sfr_d)^*\gamma_d;\ell_\partial)$, since the structure $\sfr_d$ is \textit{stable} in the sense of \cite[Section 1.8.1]{krannichpseudo}. We denote their homotopy quotients by $\BDiffsfr(M;\ell_\partial)$ and $\BbDiffsfr(M;\ell_\partial)$, respectively. Similarly, the spaces $ \Embop(M)$ and $\bEmbop(M)$ act on $\Bun_{\partial_0}(TM,(\sfr_d)^*\gamma_d;\ell_0)$ given a stable framing $\ell_0$ on $\partial_0M$. We denote the respective homotopy quotients by $\BEmbsfr(M)$ and $\BbEmbsfr(M).$ Similarly, for $d$-dimensional $4$-ad $(N;\partial_0 N,\partial_1 N,\partial_2N)$, we denote the homotopy quotients $\BDiffsfrv(N;\ell_v)$ and $\BbDiffsfrv(N;\ell_v)$ of the actions of $\Diffv(N)$ and $\bDiffv(N)$ on $\Bun_{\vb}(TN,(\sfr_d)^*\gamma_d;\ell_v).$ Denote the homotopy quotients $\BEmbsfr(N;\partial_2N;\ell_0)$ and $\BbEmbsfr(N;\partial_2N;\ell_0)$ of the actions of $\Embop(N;\partial_2N)$ and $\bEmbop(N;\partial_2 N)$ on $\Bun_{\partial_0}(TN,(\sfr_d)^*\gamma_d;\ell_0).$ As mentioned in \cite{krannich2021diffeomorphismsdiscssecondweiss}, we have a map of fibre sequences
\begin{equation}\label{weiss for sfr}
    \begin{tikzcd}[row sep = 15pt]
        \BDiffsfrv(\partial_1 N\times [0,1];\ell_v)\arrow[d] \arrow[r] & \BDiffsfrv(N;\ell_v)\arrow[r]\arrow[d] & \BEmbsfr(N;\partial_2 N;\ell_0) \arrow[d] \\
        \BDiffsfr(\partial_{12} N\times [0,1];\ell_v) \arrow[r] & \BDiffsfr(\partial_2 N,\ell_v) \arrow[r] & \BEmbsfr(\partial_2 N,\ell_0)
    \end{tikzcd}
\end{equation}with compatible deloopings once to the right, given a stable framing $\ell_v$ on $\vb N$.
\\\\
We now specify to the $4$-ad $\yg.$ Start by observing that $\partial(\partial_1\yg\times [0,1])$ is diffeomorphic to the union $S^1\times D^{2n-1}_-\cup_{S^1\times D^{2n-2}} S^1\times D^{2n-1}_+, $ which is diffeomorphic to $S^1\times D^{2n-1}. $ Thus, the space $\Diffv(\partial_1 \yg\times [0,1])$ is equivalent to the space $\Diff_{S^1\times D^{2n-1}}(S^1\times D^{2n})$ and thus equivalent to the concordance space $\conc_\partial(S^1\times D^{2n-1}).$ Moreover, since the map $\vb (S^1\times D^{2n})\to S^1\times D^{2n}$ is a homotopy equivalence, forgetting stable framings 
\[\BDiffsfrv(S^1\times D^{2n};\ell_v)\to \BDiffv(S^1\times D^{2n})\]is an equivalence, as the fibre is equivalent to $\Bun_{\vb}(T(S^1\times D^{2n}),(\sfr_{2n+1})^*\gamma_{2n+1},\ell_v)$ is contractible.

\subsubsection{Block embeddings spaces of $X_g$ and $Y_g$}\label{block embeddings definition section} In this section, we combine results of \cite{BRW,krannichpseudo} to describe the rational homotopy type of the block embedding spaces defined above in terms of certain spaces of homotopy automorphisms. Consider the block version of the sequence \eqref{weiss fibre for 4 ads} for $N=\yg.$ Observe that the block concordance space $\BbDiffv(S^1\times D^{2n})$ is contractible, since $\BbDiffb(S^1\times D^{2n})\simeq \Omega \BbDiffb(S^1\times D^{2n-1}).$ Thus, we have an equivalence $\BbDiffv(\yg)\simeq \BbEmbop(\yg;\xg),$ and thus we get a map 
\begin{equation}\label{block yg to aut}\BbEmbop(\yg;\xg)\to \BAutvp(\yg;\xg)\end{equation}to the classifying space of the group-like submonoid of those path components of $\Autv(\yg;\xg)$ hit by $\bDiffv(\yg).$ Moreover, in \cite[1650]{BRW} the authors explain how to get an analogous map 
\begin{equation}\label{block xg to aut}
    \BbEmbop(\xg)\to \BAutbp(\xg)
\end{equation}using that $\Aut_\partial(S^1\times D^{2n-1})$ is contractible. Moreover, the maps \eqref{block yg to aut} and \eqref{block xg to aut} are compatible with respect to restricting from $\yg$ to $\xg.$ Let $\ell$ be a stable framing on $\yg$, we use the subscript $(-)_\ell$ for the path component of $\ell$ in $\BbEmbsfr(\yg;\xg;\ell_0)$ and $\BbEmbsfr(\xg;\ell_0)$, where $l_0$ denotes the restriction of $\ell$ to $\partial_0$.

\begin{prop}\label{stable framed block embs are autos}
    For all $n\geq 3$ and $\ell$ a stable framing on $\yg$, the homotopy fibres of the horizontal maps induced by \eqref{block yg to aut} and \eqref{block xg to aut} in the square
    \[\begin{tikzcd}[row sep = 15pt]
        \BbEmbsfr(\yg;\xg;\ell_0)_\ell \arrow[r]\arrow[d] & \BAutvp(\yg;\xg) \arrow[d]\\
        \BbEmbsfr(\xg;\ell_0|_{\xg})_\ell \arrow[r] & \BAutbp(\xg)
    \end{tikzcd}\]are nilpotent and have finite homotopy groups, and thus are rationally contractible. 
\end{prop}
\begin{proof}
    We start by proving the claim for the upper horizontal map. Since $\BbDiffv(S^1\times D^{2n})\simeq \BbDiffsfrv(S^1\times D^{2n})$ is contractible, we see that the map 
    \[\BbDiffsfrv(\yg;l_0)_\ell\to \BbEmbsfr(\yg;\xg)_\ell\]is a weak equivalence. The claim now follows from \cite[Cor. 2.4]{krannichpseudo}, since $\xg\hookrightarrow\yg$ induces an equivalence of fundamental groupoids. By \cite[Cor. 6.2]{BRW}, the homotopy fibre of the lower horizontal map is equivalent to $\Map_*(\wg,\G)_J$, where $\G\coloneq \hocolim_n \Aut_*(S^n)$ and $J$ are the union of components of maps $\wg\to \G$ which factor up to homotopy through the $J$-homomorphism $J:\O\to\G $. Since $\wg$ is equivalent to a wedge of $n$-spheres, this mapping space is a product of $n$-fold loop spaces and thus nilpotent. Moreover, its homotopy groups are $2g$-fold products of shifts of homotopy groups of $\G$. These identify with certain stable homotopy groups of spheres, which are finite by Serre.
\end{proof}

\subsection{Quadratic modules}

Let $\pi\coloneq \pi_1(\xg,x)$ and $\bar{(-)}:\ZZ[\pi]\to \ZZ[\pi]$ be the ring involution induced by $\bar{g}=g^{-1}$, for $g\in \pi$. 

\begin{convention}\label{ConventionsModules}
    By \textit{$\ZZ[\pi]$-module}, we mean a left $\ZZ[\pi]$-module, and we see any such module $M$ as a right module by setting $x\cdot a\coloneq \bar{a}\cdot x$, and vice-versa. Given a module $M$, we define its dual $M^\vee$ to be the right $\ZZ[\pi]$-module of left $\ZZ[\pi]$-module maps $\Hom_{\ZZ[\pi]}(M,\ZZ[\pi])$, where the right module structure on this abelian group is given by $\alpha\cdot r(m)\coloneq \alpha(m)r.$ Given another module $N$, we can form the tensor product $M^\vee\otimes_{\ZZ[\pi]} N\coloneq M^\vee\otimes_\pi N$, which we see as an abelian group, and define the canonical map of abelian groups
    \(M^\vee\otimes_{\pi} N\to \Hom_{\ZZ[\pi]}(M,N)\) that takes $\alpha\otimes n$ to the map $(m\mapsto \alpha(m)\cdot n).$ This map is an isomorphism, whenever $M$ is a finitely generated projective $\ZZ[\pi]$-module.
\end{convention}

\begin{nota}\label{allTheNotationUngraded}
    We now set the notation of various modules that we will consider throughout the entire paper. First, consider the following $\ZZ[\pi]$ and $\QQ[\pi]$-modules.
    \begin{align*}
        \pix\coloneq \pi_n(X_g) \qquad \piy\coloneq \pi_n(Y_g) \qquad \pirel\coloneq\pi_{n+1}(Y_g,X_g)\cong \ker(\iota_*:\pix\to \piy) \\
    \pixq\coloneq \pi_n^\QQ(X_g) \qquad \piyq\coloneq \pi_n^\QQ(Y_g) \qquad \pirelq\coloneq\pi_{n+1}^\QQ(Y_g;X_g)\cong \ker(\iota_*:\pixq\to \piyq)
    \end{align*}where $\iota_*$ is the map on (rational) homotopy groups of the inclusion $X_g\hookrightarrow Y_g$. We also set the analogous notation for the pair $(V_g;W_{g,1})$, similar to the one in \cite[Section 2.5]{krannich2021diffeomorphismsdiscssecondweiss}.
    \begin{align*}
        \piw\coloneq \H_n(W_{g,1};\ZZ) \qquad \piv\coloneq \H_n(V_g;\ZZ) \qquad \piwrel\coloneq\H_{n+1}(V_g,W_{g,1};\ZZ)\cong \ker(\iota_*:\piw\to \piv) \\
    \piwq\coloneq \H_n(W_{g,1};\QQ) \qquad \pivq\coloneq \H_n(V_g;\QQ) \qquad \piwrelq\coloneq\H_{n+1}(V_g,W_{g,1};\QQ)\cong \ker(\iota_*:\piwq\to \pivq)
    \end{align*}Observe that the pair $(Y_g;X_g)$ is equivalent to the pair $(S^1\vee V_g,S^1\vee W_{g,1})$, from which we see that the first (resp. second) row of the notation for $(Y_g;X_g)$, as $\ZZ[\pi]$-modules, is obtained from the first (resp. second) row of the notation of $(V_g;W_{g,1})$ by applying $(-)\otimes_\ZZ\ZZ[\pi]$ (resp. $(-)\otimes_\QQ \QQ[\pi]$). Since $\iota:W_{g,1}\simeq \bigvee_{2g} S^n\hookrightarrow V_g\simeq \bigvee_g S^n$ is given by collapsing the summands represented by the spheres $\{*\}\times S^n\subset W_{g,1}$, the inclusion of the remaining summands defines a homotopy section of $\iota.$ Thus, we have splittings $\smash{\piw\cong \piv\oplus \piwrel}$ as $\ZZ$-modules and $\smash{\pix\cong \piy\oplus \pirel}$ as $\ZZ[\pi]$-modules, and also analogous splittings for the rationalised versions. Finally, we fix the \textit{standard basis} of $\pix$ $\{a_i,b_i\}_{i=1}^g$, where $a_i$ and $b_i$ are the classes representing, respectively, the inclusions of $S^n\times \{*\},\{*\}\times S^n\subset W_{1,1}=(S^n\times S^n)\backslash \int(D^{2n})$ of the $i$-th connected summand $W_{1,1}\subset X_g$. Note that $\{a_i\}_{i=1}^g$ is a basis of the summand $\piy$, and $\{b_i\}_{i=1}^g$ of the summand $\pirel.$ 
\end{nota}

In \cite[Section 2.2]{BRW}, the authors associate a quadratic module $ (\smash{\pix},\lambda_X,q_X)$ over $\ZZ[\pi]$ with form parameter $((-1)^n,\Lambda_n)$ in the sense of Bak \cite{Bak1969}, where $\Lambda_n$ is either $\Lambda^{\text{min}}_n\coloneq \{a-(-1)^n\bar{a}:a\in \ZZ[\pi]\}$ if $n\neq 3,7$, and $\langle\Lambda^{\text{min}},e\rangle$ where $e\in \pi$ is the identity element. By definition, this is the data of $\ZZ$-linear maps
\[\begin{tikzcd}
    \lambda_X:\pix\otimes \pix\to \ZZ[\pi] & q_X:\pix\to \ZZ[\pi]/\Lambda_n
\end{tikzcd}\]such that, for every $x,y\in \pix$ and $a,b\in \ZZ[\pi]$, we have:
\begin{enumerate}[label=(\roman*)]
    \item $\lambda_X(a\cdot x,b\cdot y)=a\cdot \lambda_X(x,y)\cdot \bar{b},$ 
    \item  $\overline{\lambda_X(x,y)}=(-1)^n\lambda_X(y,x),$
    \item $q_X(a\cdot x)=a\cdot q_X(x)\cdot \bar{a},$ and
    \item  $q_X(x+y)-q_X(x)-q_X(y)=\lambda_X(x,y) \text{ mod } \Lambda_n.$
\end{enumerate}In this case, $\lambda_X$ is given by the equivariant intersection form on the universal cover $\widetilde{X}_g.$ This quadratic module is \textit{non-degenerate}, in the sense that, the right $\ZZ[\pi]$-module map $\smash{\pix\to (\pix)^\vee}$ sending $x$ to $\lambda_X(-,x)$ is an isomorphism. Since $\pi_1(\xg)\cong \pi_1(\yg),$ observe that $\smash{\pirel}$ and $\smash{\piy}$ are isotropic subspaces with respect to $\lambda_X,$ that is, $q_X(x)=0$ and $\lambda_X(x,y)=0$ for $x,y$ in $\pirel$ or $\piy$. Moreover, the duality map $\smash{\pix\to (\pix)^\vee}$ restricts to  isomorphisms $\smash{\pirel\to (\piy)^\vee}$ and $\smash{\piy\to (\pirel)^\vee}$ of right $\ZZ[\pi]$-modules. 

Similarly, we have a quadratic module $(\piw,\lambda_W,q_W)$ over $\ZZ$ with form parameter $((-1)^n,\Lambda_n)$ where $\bar{(-)}:\ZZ\to \ZZ$ is the identity and $\Lambda_n$ is defined in the analogous way. Here, we have the isotropic subsets $\piwrel$ and $\piv$. We shall need the following improvement of \cite[Lemma 2.2]{BRW}.

\begin{lemma}\label{turn maps to embeddings}
    Let $n\geq 3$. An element $e\in \pix$ may be represented by a smooth embedding $\hat{e}:S^n\times D^n\hookrightarrow X_g$ if and only if $q_X(e)=0$. Moreover, every element $f\in \pirel$ may be represented by a smooth embedding of pairs $\hat{f}:(D^{n+1}\times D^n,S^n\times D^n)\hookrightarrow (\yg,\xg)$. The analogous claim holds for $\wg$ and $\vg.$
\end{lemma}
\begin{proof}
    The first claim is a special case of \cite[Lemma 2.2]{BRW}, so we focus on the second claim. Let $f\in \pirel$ be an element represented by a map of pairs $f:(D^{n+1},S^n)\to (Y_g;X_g).$ By a result of Hudson \cite{Hudson_1972} (best stated for our purposes in \cite[Thm. 5.13]{BP}), the map $f$ is homotopic to an embedding $f':(D^{n+1},S^n)\hookrightarrow (Y_g;X_g)$ since $n\geq 3$, $(D^{n+1},S^n)$ is $1$-connected, and $(Y_g;X_g)$ is $2$-connected. Additionally, as $D^{n+1}$ is contractible, the normal bundle of $f'$ is trivial, and hence, $f'$ may be extended to an embedding $\hat{f}:(D^{n+1}\times D^n,S^n\times D^n)\hookrightarrow (\yg,\xg)$, by choosing a framing of the normal bundle (here we also use that the normal bundle of $f'|_{S^n}$ in $X_g$ is the restriction of the normal bundle of $f'$ to $S^n$). The same proofs work verbatim for $W_{g,1}$ and $V_g$.
\end{proof}

Denote the orthogonal groups of these quadratic modules by $\Ug$ and $\Ugw.$ The maps $\ZZ\to\ZZ[\pi]\to \ZZ$ induced by the unique maps $0\to \ZZ\to 0$ of groups induce maps $\Ugw\to \Ug\to \Ugw.$ By naturality, self-embeddings of $\xg$ fixing $\partial_0\xg$ preserve this quadratic module. Moreover, if such self-embedding extends to $\yg$, we conclude that the isotropic subspace $\pirel$ is fixed set-wise. Denote by $\Ugext$ the subgroup of $\Ug$ consisting of those $A$ such that $A(\pirel)\subset \pirel.$ Define $\Ugwext$ in the analogous way. From \cref{turn maps to embeddings}, we deduce the following consequence.

\begin{lemma}\label{surjectivity from embeddings to unitary}
    Let $n\geq 3$. The horizontal maps of the commutative square
    \[\begin{tikzcd}
        \pi_0(\Embo(\yg;\xg))\arrow[d]\arrow[r] & \Ugext \arrow[d] \\
        \pi_0(\Embo(\xg)) \arrow[r] & \Ug
    \end{tikzcd}\]are surjective. The analogous claim holds for $\wg$ and $\vg.$
\end{lemma}
\begin{proof}
    The surjectivity of the bottom map is the content of \cite[Prop. 5.2]{BRW}. We focus on the surjectivity of the top map. Let $a_i$ and $b_i$ denote the elements of the standard basis of $\pix$ from \cref{allTheNotationUngraded}. By definition, the classes $b_i$ lie in $\pirel.$ Let $\rho\in \Ugext$. Since $\rho(b_i)\in \pirel$, we can apply \cref{turn maps to embeddings} to represent this class by an embedding $\beta_i:(D^{n+1}\times D^n,S^n\times D^{n})\to (Y_g;X_g)$. Moreover, we can isotope the embeddings $\beta_i$ to be pairwise disjoint, by applying \cite[Thm. C.3]{BP}, so we can assume that we have chosen $\beta_i$ above pairwise disjoint. This result only applies to $n\geq 4$, but it was communicated to us by Manuel Krannich and Alexander Kupers that loc.cit. can be extended to dimension $7$, using different methods, and will appear as part of forthcoming work of these authors. On the other hand, by following the proof of \cite[Prop. 5.2]{BRW}, we can represent $\rho(a_i)$ by embeddings $\alpha_i:S^n\times D^n\hookrightarrow X_g$, which are pairwise disjoint, and such that $\alpha_i|_{S^n\times \{0\}}$ and $\beta_j|_{S^n\times \{0\}}$ intersect transversely at exactly $1$ point if $i=j$, and are disjoint if $i\neq j.$ By proceeding as loc.cit., plumbing assembles the embeddings $\alpha_i$ and $\beta_i$ into an embedding $e:(V_g;W_{g,1})\hookrightarrow (Y_g;X_g)$ which sends $(\partial_0 V_g,\partial_{02} V_g)$ to $(\partial_0 Y_g,\partial_{02} Y_g)$. Hence, we can extend it to an element of $\Embo(\yg;\xg)$, whose induced map on $n$-th homotopy groups is $\rho.$
\end{proof}

\begin{rmk}[On dimension $7$]
    In the proof above, we relied on an extension of \cite[Thm. C.3]{BP} to dimension $7$ that is not yet available. This is however the only step where this extension is used. If one does not assume it, one must assume additionally that $n\geq 4$ or $d\geq 8$ in \cref{main even,main conc intro,main range} and Corollaries \ref{application to G}--\ref{RelativeBTopMainCor}. Additionally, \cref{main even} and \cref{main knots} are also valid for $d=6$, without assuming this extension. 
\end{rmk}

\subsection{Homotopy automorphisms of \texorpdfstring{$\xg$}{Xg} and \texorpdfstring{$\yg$}{Yg}}\label{hAutSection}

In this subsection, we study the spaces of homotopy automorphisms of $X_g$ and of the pair $(Y_g;X_g)$. The former was understood in \cite{BRW}, whose results we recall in the first subsubsection. We study the homotopy type of the latter using similar methods as the ones from loc.cit. The results we obtain will be used both in \cref{MCGSection} to understand the group of components of certain embedding spaces (once again, the case of $X_g$ is merely a recollection of the work of \cite{BRW}), and in \cref{SelfEmbeddingXgSection,SelfEmbeddingYgSection} to compute the higher rational homotopy groups of these embedding spaces. 

\subsubsection{The case of $X_g$}

In this section, we recall the main result of \cite[Section 3.1]{BRW} on the homotopy type of the space $\Aut_{\partial}(X_g)$. The authors consider the following fibre sequence
\begin{equation}\label{XgMappingSpaceFibSeq}
\begin{tikzcd}
\Map_{\partial} (\xg)\rar &\Map_{S^1}(\xg) \rar["\res_X"] &\Map_{S^1}(S^1\times S^{2n-2},\xg)    
\end{tikzcd}
\end{equation}given by restriction of maps, where $S^1\subset S^1\times S^{2n-2}$ is given by the inclusion of $S^1\times \{*\}$ induced by a fixed basepoint of $S^{2n-2}.$ We extract the following result from the main result Theorem $3.1$ of loc.cit.

\begin{prop}[\cite{BRW}]\label{homotopy groups of mapping of xg}
    The following statements hold for $k\geq 0:$
    \begin{enumerate}[label=(\roman*)]
        \item There is an isomorphism of $\ZZ[\pi_0(\Aut_{S^1}(X_g))]$-modules
        \[\pi_k(\Map_{S^1}(\xg),\id)\cong \Hom_{\ZZ[\pi]}(\pix,\pi_{k+n}(\xg))\]where $\pi_0(\Aut_{S^1}(X_g))$ acts on the left group using the conjugation action of $\Aut_{S^1}(X_g)$ on $\Map_{S^1}(X_g)$ and on the right group by post-composition with the induced map on $\pi_{k+n}(X_g)$ and pre-composition by the induced map on $\pi_n(X_g)$. 
        \item There is an isomorphism of abelian groups
        \[\pi_k(\Map_{S^1}(S^1\times S^{2n-2},\xg),\iota)\cong [\pi_{k+2n-1}(\xg)]_\pi\]where $\iota:S^1\times S^{2n-2}\hookrightarrow \xg$ denotes the inclusion and $[\pi_{k+2n-1}(\xg)]_\pi$ denotes the coinvariants of the action of $\pi=\pi_1(\xg)$ on $\pi_{k+2n-1}(\xg)$.
        \item After applying $(-)\otimes \QQ $, the map $\pi_k(\res_X)$ is surjective for every $k> 0.$ In particular, we have an isomorphism $\pi_k(\Map_{\partial} (\xg),\id)\otimes \QQ\cong \ker(\pi_k(\res_X)\otimes \QQ).$
    \end{enumerate}
\end{prop}
\begin{proof}
    The first two claims are precisely the first two isomorphisms of \cite[Thm. 3.1]{BRW} (the equivariance claim in (i) follows directly from the construction of the isomorphism in the proof of loc.cit), so we focus on the third claim. We briefly recall the description of $\pi_k(\res_X)$ from loc.cit. The quadratic pairing on $\pix$ is non-degenerate and hence we have the duality isomorphism of $\ZZ[\pi]$-modules $\smash{\pix\to (\pix)^{\vee}}$ given by taking $x$ to $\lambda_X(-,x)$. Under the canonical isomorphism \[\pix\otimes_{\pi}\pi_{k+n}(X_g)\to (\pix)^\vee\otimes_\pi \pi_{k+n}(X_g)\to \Hom_{\ZZ[\pi]}(\pix,\pi_{n+k}(X_g)),\]induced by the duality isomorphism and the canonical map in \cref{ConventionsModules} (where we use that $\pix$ is free and finitely generated), and the isomorphisms of (i) and (ii), the map $\pi_k(\res_X)$ agrees with the map $[-,-]:\pix\otimes_{\pi}\pi_{k+n}(X_g)\to [\pi_{k+2n-1}(\xg)]_\pi,$ given by taking $x\otimes y$ to the Whitehead product $[y,x]$ in $X_g$. We now show that $[-,-]$ is surjective after rationalising. Since the universal cover of $X_g$ is a wedge of spheres $\smash{\bigvee_{\ZZ}\bigvee_{i=1}^{2g}S^n}$, the rational homotopy groups $\smash{\pi_{*}^\QQ(X_g)}$ in degrees $*\geq 2$ are given by $(*-1)$-graded piece of the free graded Lie algebra $\mathrm{Lie}(\pixq)$ over $\QQ$, where $\pixq$ is seen as a graded vector space concentrated in degree $n-1$. Additionally, the Whitehead product on $X_g$ corresponds to the bracket on this free Lie algebra. Our claim now follows from the following classical fact about Lie algebras: Let $V$ be a graded vector space, then the bracket map $[-,-]:V\otimes \mathrm{Lie}(V)\to \mathrm{Lie}^{\geq 2}(V) $ is surjective, where the target is the subspace of decomposable elements.
\end{proof}

\subsubsection{The case of $(Y_g;X_g)$}

In this section, we analyse the homotopy groups of the space $\Autv(\yg;\xg)$ of self-homotopy equivalences of the pair $(\yg,\xg)$ which fix the pair $(S^1\times D^{2n-1}_+,S^1\times S^{2n-2})$ pointwise. This is essentially analogous to \cite[Section 3]{BRW}. We have a fibre sequence of the following form
\begin{equation}\label{YgXgMappingSpaceFibSeq}
\begin{tikzcd}
\Map_{\vb}(\yg;\xg)\rar &\Map_{S^1}(\yg;\xg) \rar["\res_Y"] &\Map_{S^1}((S^1\times D^{2n-1},S^1\times S^{2n-2}),(\yg,\xg))    
\end{tikzcd}
\end{equation}
given by restriction of maps. Here, the leftmost space is the endomorphism space of the pair $(\yg,\xg)$ under $(S^1\times D^{2n-1},S^1\times S^{2n-2})$, the middle space is the endomorphism space of $(\yg,\xg)$ under $(S^1,S^1)$ and the rightmost space is the mapping space of pairs under $(S^1,S^1),$ where the inclusion of $(S^1,S^1)$ in $(S^1\times D^{2n-1},S^1\times S^{2n-2})$ is induced by a fixed basepoint of $S^{2n-2}.$  

\begin{prop}\label{homotopy groups of mapping of pair}
    The following statements hold for $k\geq 0:$
    \begin{enumerate}[label=(\roman*)]
        \item There is an isomorphism of abelian groups
        \[\pi_k(\Map_{S^1}(\yg;\xg),\id)\cong \Hom_{\ZZ[\pi]}^{\mathrm{ext}}(\pix,\pi_{k+n}(\xg))\]where $(-)^\mathrm{ext}$ denotes the subgroup of those maps taking $\pirel$ to the subgroup $\pi_{k+n+1}(Y_g,X_g)\subseteq \pi_{k+n}(X_g).$
        \item There is an isomorphism of abelian groups
        \[\pi_k(\Map_{S^1}((S^1\times D^{2n-1},S^1\times S^{2n-2}),(\yg,\xg)),\iota)\cong [\pi_{k+2n}(\yg,\xg)]_\pi\]where $\iota:(S^1\times D^{2n-1},S^1\times S^{2n-2})\hookrightarrow (\yg,\xg)$ denotes the inclusion and $[\pi_{k+2n}(\yg,\xg)]_\pi$ denotes the coinvariants of the action of $\pi=\pi_1(\xg)$ on $\pi_{k+2n}(\yg,\xg)$.
        \item After applying $(-)\otimes \QQ $, the map $\pi_k(\res_Y)$ is surjective for every $k> 0.$ In particular, we have an isomorphism $\pi_k(\Map_{\vb}(\yg;\xg),\id)\otimes \QQ\cong \ker(\pi_k(\res_Y)\otimes \QQ).$
    \end{enumerate}Under these isomorphisms, the maps induced by restriction $\pi_k(\Map_{S^1}(\yg;\xg),\id)\to \pi_k(\Map_{S^1}(\xg),\id)$ and $\pi_k(\Map_{S^1}((S^1\times D^{2n-1},S^1\times S^{2n-2}),(\yg,\xg)),\iota)\to \pi_k(\Map_{S^1}(S^1\times S^{2n-2},\xg),\iota)$ agree with the inclusions, induced by the splittings $\pi_{*}(X_g)\cong \pi_{*}(Y_g)\oplus \pi_{*+1}(Y_g,X_g)$, \[\Hom_{\ZZ[\pi]}^{\mathrm{ext}}(\pix,\pi_{k+n}(\xg))\hookrightarrow \Hom_{\ZZ[\pi]}(\pix,\pi_{k+n}(\xg))\] and $[\pi_{k+2n}(\yg,\xg)]_\pi\hookrightarrow [\pi_{k+2n-1}(\xg)]_\pi$, respectively.
\end{prop}
\begin{proof}
    We start by proving the first claim, which is essentially a relative version of the analogous claim in \cite[Thm 3.1]{BRW}. Start by observing the there is an equivalence $\Map_{S^1}(Y_g;X_g)\simeq \Map_{*}((V_g;W_{g,1}),(Y_g;X_g))$ by restricting, where the latter space is the space of based maps of pairs. We can observe that the pair $(V_g;W_{g,1})$ is equivalent to the wedge sum \[\bigvee_{\mathrm{rk}(\H_{n+1}(V_g,W_{g,1}))} (D^{n+1},S^n)\vee \bigvee_{\mathrm{rk}(\H_{n}(V_g))}(S^n,S^n)\] of pairs. Hence, by taking homotopy groups at the constant map, we have an isomorphism
    \[\pi_k(\Map_{*}((V_g;W_{g,1}),(Y_g;X_g)),\mathrm{const})\cong \Hom_\ZZ(\piwrel,\pi_{n+1+k}(Y_g,X_g))\oplus \Hom_\ZZ(\piv,\pi_{n+k}(X_g)),\]using the following two facts, for any based pair of spaces $(X,X')$: 
    \begin{enumerate}
        \item We have an isomorphism $\pi_k(\Map_*((D^{n+1},S^n),(X,X'),\mathrm{const})\cong \Hom_\ZZ(\H_{n+1}(D^{n+1},S^n), \pi_{n+k}(X,X'))$.
        \item The restriction map $\Map_*((S^n,S^n),(X,X'))\to \Map_*(S^n,X')$ is an equivalence, and hence the homotopy group in degree $k$ of either of the spaces is isomorphic to $ \Hom_\ZZ(\H_n(S^n),\pi_{n+k}(X')).$
    \end{enumerate}To compute the homotopy groups of the space $\Map_{*}((V_g;W_{g,1}),(Y_g;X_g))$ based at the inclusion map, one observes that, using the wedge summand decomposition of above, the pair $(V_g;W_{g,1})$ admits a co-$H$-space structure, i.e. it is a comonoid in the homotopy category of based pairs of spaces, which implies that the mapping space above is a group-like $H$-space, and hence all its path components are equivalent. Now, using the isomorphisms $\pirel\cong \piwrel\otimes_\ZZ \ZZ[\pi]$ and $\pix\cong \piw\otimes_\ZZ \ZZ[\pi],$ we deduce the first claim from the isomorphism of abelian groups $\Hom_\ZZ(A,B)\cong \Hom_{\ZZ[\pi]}(A\otimes \ZZ[\pi],B)$, which holds for any abelian group $A$ and $\ZZ[\pi]$-module $B$. Moreover, one checks that the restriction map $\pi_k(\Map_{S^1}(\yg;\xg),\id)\to \pi_k(\Map_{S^1}(\xg),\id)$ agrees with the inclusion map in the statement, since it is given by restricting the maps of pairs to the source of the pair.

    The second claim follows analogously to the analogous claim in loc.cit., using the pushout square of pairs
    \[
    \begin{tikzcd}
    (D^{2n-1},S^{2n-2}) \arrow[d, "\alpha"] \arrow[r] & (D^{2n},D^{2n-1})  \arrow[d]\\
    (S^1\vee D^{2n-1},S^1\vee S^{2n-2}) \arrow[r,"\iota"] & (S^1\times D^{2n-1},S^1\times S^{2n-2}) 
    \end{tikzcd}
    \]where $\alpha$ is a map of pairs representing the unique homotopy class $[\alpha]\in \pi_{2n-1}(S^1\vee D^{2n-1},S^1\times S^{2n-2})$ mapping to the Whitehead product $[\iota_1,\iota_{2n-2}]\in \pi_{2n-2}(S^1\vee S^{2n-2})$ of the inclusion maps $\iota_1:S^1\hookrightarrow S^1\vee S^{2n-2}$ and $\iota_{2n-2}:S^{2n-2}\hookrightarrow S^1\vee S^{2n-2}$ under the connecting map of the long exact sequence of the pair $(S^1\vee D^{2n-1},S^1\vee S^{2n-2}).$ We leave the check to the reader. One also checks that the restriction map $\pi_k(\Map_{S^1}((S^1\times D^{2n-1},S^1\times S^{2n-2}),(\yg,\xg)),\iota)\to \pi_k(\Map_{S^1}(S^1\times S^{2n-2},\xg),\iota)$ agrees with the inclusion $[\pi_{k+2n}(\yg,\xg)]_\pi\hookrightarrow [\pi_{k+2n-1}(\xg)]_\pi$.

    We move now to the third claim. To do so, we consider the following commutative square
    \[\begin{tikzcd}
        \Hom^{\mathrm{ext}}_{\ZZ[\pi]}(\pix,\pi_{k+n}(X_g)) \arrow[d, hook]\arrow[r, "\pi_k(\res_Y)"] & \left[\pi_{k+2n}(Y_g,X_g)\right]_\pi \arrow[d, hook] \\
        \Hom_{\ZZ[\pi]}(\pix,\pi_{k+n}(X_g)) \arrow[r, "\pi_k(\res_X)"] & \left[\pi_{k+2n-1}(X_g)\right]_\pi
    \end{tikzcd}\]and proceed in the following way: First, we identify the images of the vertical maps after rationalisation and second, we show that the image of the left map is hit by the image of the right map under the bottom map. The proof of the first step will hinge on the following two classical facts about free graded Lie algebras: Let $V,A,B$ be graded $\QQ$-vector spaces:
    \begin{enumerate}
        \item\label{rightBracketed} As a $\QQ$-vector space, $\mathrm{Lie}(V)$ is generated by right bracketed words $[v_1,[v_2,[\cdots [v_{k-1},v_k]\cdots ]]]$ for $v_1,\cdots, v_k\in V$. This follows by inductively using the Jacobi identity.
        \item\label{kernel of Lie} The kernel of the projection map $\mathrm{Lie}(A\oplus B)\to \mathrm{Lie}(A)$ is generated by right bracketed as above, for $V=A\oplus B$, where at least one of the entries is an element of $B.$ This follows by the following argument: Any element $x$ in the kernel can be written as a sum of right bracketed words, as in \ref{rightBracketed}, where $v_i$ is either in $A$ or $B$. The image of such a sum in $\mathrm{Lie}(A)$ is the sum of those words where $v_i\in A$ for all $i$. Since $x$ is in the kernel, the latter sum vanishes in $\mathrm{Lie}(A)$ and hence it also vanishes in $\mathrm{Lie}(A\oplus B)$ (since the map $\mathrm{Lie}(A)\to \mathrm{Lie}(A\oplus B)$ is a section of the projection, and hence injective). Thus, each word present in the sum decomposition of $x$ must contain at least on element in $B$.    
    \end{enumerate}We start by identifying the images of the vertical maps in the square above. For the left map, we identify this image under the identification $\pix\otimes_{\pi} \pi_{k+n}(X_g)\cong \Hom_{\ZZ[\pi]}(\pix,\pi_{k+n}(X_g)).$ By seeing $\Hom^{\mathrm{ext}}_{\ZZ[\pi]}(\pix,\pi_{k+n}(X_g))$ as the direct sum 
    \[\Hom_{\ZZ[\pi]}(\pirel,\pi_{k+n+1}(Y_g,X_g))\oplus \Hom_{\ZZ[\pi]}(\piy,\pi_{n+k}(X_g)),\]we observe that the image of the left vertical map under the identification above is the sum of $\piy\otimes_\pi \pi_{k+n+1}(Y_g,X_g)$ and $\pirel\otimes_\pi \pi_{n+k}(X_g)$, using the duality isomorphisms above \cref{turn maps to embeddings} (see also \cref{ConventionsModules}). Observe that $\smash{\pi_*^\QQ(Y_g;X_g)}$ is, as a $\QQ$-vector space, the kernel of the map of Lie algebras $p:\mathrm{Lie}(\piy\oplus \pirel)\to \mathrm{Lie}(\piy)$, where we see $\piy$ and $\pirel$ as graded $\QQ$-vector spaces concentrated in degree $n-1$. Thus, by \ref{kernel of Lie}, it is generated by words of the form $[a,b]$ where either $a\in \pirel$ or $b$ is in this kernel. We now proceed with the second step. We recall that, under the identification above, the map $\pi_k(\res_X)$ is given by the Whitehead product map (see proof of \cref{homotopy groups of mapping of xg}). Observe that the image of the Whitehead bracket map of $\piy\otimes_\pi \pi_{k+n+1}(Y_g,X_g)$ are those classes which can be decomposed as sums of words $[x,y]$ where $y\in \pi_{k+n+1}(Y_g,X_g)$. Similarly, the image of $\pirel\otimes_\pi \pi_{n+k}(X_g)$ are those classes which can be decomposed as sums of words $[x,y]$ where $x\in \pirel$. Thus, the kernel of $p$ is hit by the Whitehead bracket map, once restricted to the image of the left vertical map. We conclude that the map $\pi_k(\res_Y)$ is surjective, since taking coinvariants by $\pi$ does not affect surjectivity. 
\end{proof}

The same proof as before gives the following analogous result. We leave the details to the reader.

\begin{prop}\label{homotopy groups of mapping of pair wg}
    The following statements hold for $k\geq 0:$
    \begin{enumerate}[label=(\roman*)]
        \item There is an isomorphism
        \(\pi_k(\Map_{*}(\vg;\wg),\id)\cong \Hom_{\ZZ}^{\mathrm{ext}}(\piw,\pi_{k+n}(\wg)).\)
        \item There is an isomorphism
        \(\pi_k(\Map_{*}((D^{2n},S^{2n-1}),(\vg,\wg)),\iota)\cong \pi_{k+2n}(\vg,\wg)\)where $\iota:(D^{2n},S^{2n-1})\hookrightarrow (\vg,\wg)$ denotes the inclusion.
    \end{enumerate}
\end{prop}

\subsection{The reflection involution on \texorpdfstring{$(Y_g;X_g)$}{(Yg,Xg)}}\label{ReflectionInvSection}

In this subsection, we define an involution $\rho_g$ on the $4$-ad $(Y_g;\partial_0 Y_g,\partial_1 Y_g,\partial_2 Y_g)$ of \cref{manifolds definition}(iv). Given any such involution, then conjugation by $\rho_g$ induces an involution on the spaces $\Diffv(Y_g)$, $\Diffb(X_g)$, $\Embcong(Y_g;X_g),$ and $\Embcong(X_g)$, and their block analogues (recall \cref{moduliSpacesDefnSection}). Taking classifying spaces induces a compatible $A_\infty$-involution\footnote{An \textit{$A_\infty$-involution} on a space $X$ is a morphism of $A_\infty$-algebras $C_2\to \End(X)$ to the endomorphism monoid of a space $X$} on the spaces $\BDiffv(Y_g)$, $\BDiffb(X_g)$, $\BEmbcong(Y_g;X_g),$ and $\BEmbcong(X_g)$, and their block analogues. We will define $\rho_g$ with the following design criteria, which will be the only properties from $\rho_g$ used in later sections:
\begin{enumerate}[label=$(\rho\arabic*)$]
    \item\label{action of rho on homology} The action of $\rho_g$ on the standard $\ZZ[\pi]$-basis $\{a_i,b_i\}_{i=1}^g$ of $\pix$ (recall \cref{allTheNotationUngraded}) takes $a_i$ to $-a_i$ and fixes $b_i.$
    \item\label{action on framings} There exists a stable framing $\ell$ on $Y_g$, such that the $A_\infty$-involutions by conjugation of $\rho_g$ on the spaces above admit compatible lifts to $A_\infty$-involutions on the spaces $\BDiffsfrv(Y_g)_\ell$, $\BDiffsfr(X_g)_\ell$, $\BEmbsfr(Y_g;X_g)_\ell,$ and $\BEmbsfr(X_g)_\ell$, and their block analogues.
\end{enumerate}To this end, we will first fix a model of $(Y_g;\partial_0 Y_g,\partial_1 Y_g,\partial_2 Y_g)$ as a submanifold of $\RR^{2n+1}$ together with fixed embeddings $e_i,f_i:S^n\hookrightarrow Y_g$ representing the classes $a_i$ and $b_i$, respectively, such that $(Y_g;\partial_0 Y_g,\partial_1 Y_g,\partial_2 Y_g)\subset (Y_{g+1};\partial_0 Y_{g+1},\partial_1 Y_{g+1},\partial_2 Y_{g+1})$ and these inclusions preserve the embeddings $e_i$ and $f_i.$ We will then define $\rho_g$ as reflection on the first coordinate of $\RR^{2n+1}$, which will fix $(Y_g;\partial_0 Y_g,\partial_1 Y_g,\partial_2 Y_g)$ setwise. Hence, the following design criterion will also hold. 

\begin{enumerate}[label=$(\rho\arabic*)$, resume]
    \item\label{equivarian} The inclusions $(Y_g;\partial_0 Y_g,\partial_1 Y_g,\partial_2 Y_g)\hookrightarrow (Y_{g+1};\partial_0 Y_{g+1},\partial_1 Y_{g+1},\partial_2 Y_{g+1}) $ are equivariant with respect to the actions of $\rho_g$ and $\rho_{g+1}$, respectively.
\end{enumerate}
\subsubsection{A standard model} We start by modelling $Y_g$ as a submanifold of $\RR^{2n+1}$. This will be done using the \textit{standard model} of $V_g$ as a submanifold of $\RR^{2n+1}$, defined in \cite[Section 2.5]{krannich2021diffeomorphismsdiscssecondweiss}. We briefly recall the properties of this construction that we will need. In loc.cit., the authors model $V_g\subset \RR^{2n+1}$ and define embeddings $e_i, f_i:S^n\to \partial V_g$ representing the standard basis $a_i,b_i\in \piw$ of the embeddings $S^n\times \{*\}, {*}\times S^n\subset W_{1,1} $ (see \cref{allTheNotationUngraded}) satisfying the following properties:
\begin{enumerate}[label=(\thesubsubsection.\alph*), align=left, leftmargin=*]
    \item\label{reflection of vg} The reflection along the first coordinate of $\RR^{2n+1}$ restricts to a diffeomorphism of $V_g$. This diffeomorphism fixes $f_i$ pointwise, and once restricted to $e_i$ is given by precomposition with the reflection along one coordinate of $S^n$ (in particular, it has degree $-1$).
    \item\label{image of vg is small} The submanifold $V_g$ is contained in $\RR^{2n}\times [\frac{1}{2},2g)$, and $V_g$ is contained in $V_{h}$, for $g\leq h$.
    \item The point $\frac{1}{2}\cdot\vec{e}_{2n+1}$ is in $\partial V_g,$ where $\vec{e}_i$ denotes the $i$-th basis vector of $\RR^{2n+1}.$
\end{enumerate}These properties follow directly by the construction in \cite[14]{krannich2021diffeomorphismsdiscssecondweiss}, by unravelling definitions.  We now define a model of $S^1\times D^{2n}$ as a submanifold of $\RR^{2n}\times (-2,0).$ Denote the standard inclusion $S^1\subset \RR^2$ by $\iota_{S^1}$ and see its normal bundle as a subset of $\RR^2$. We define the embedding $\phi_{S^1}:S^1\times D^1\times D^{2n-1}\to \RR^{2n-1}\times \RR^2$ defined by $\phi_{S^1}(z,t,x)=(x,\frac{1+t}{2}\iota_{S^1}(z))-\frac{3}{2}\vec{e}_{2n+1}$ to see $S^1\times D^{2n}$ as a submanifold of $\RR^{2n+1}$. Notice also that the image of $\phi_{S^1}$ is contained in $\RR^{2n}\times (-2,0)$, contains the point $-\frac{1}{2}\cdot\vec{e}_{2n+1},$ and is preserved setwise by the reflection along the first coordinate of $\RR^{2n+1}.$ Now, by choosing an appropriate tubular neighbourhood of the path $[0,1]\ni s\mapsto \frac{t-1}{2}\cdot \vec{e}_{2n+1}$ connecting $V_g$ with the image of $S^1\times D^{2n}\subset \RR^{2n+1}$ (and in $(0,1)$ is disjoint from both images by \ref{image of vg is small}), which is axially symmetric around the $\vec{e}_{2n+1}$-axis, then observe that the union of $S^1\times D^{2n}$ with $V_g$ along this tubular neighbourhood is diffeomorphic to $Y_g$. We define the submanifold $Y_g\subset \RR^{2n+1}$ to be this union. Observe that, by the definition above and \ref{reflection of vg}, the reflection along the first coordinate restricts to a diffeomorphism of $Y_g$. 
\begin{defn}\label{theReflectionInvolution}(Reflection involution)
    We define the involution $\rho_g:Y_g\to Y_g$ to be the reflection on the first coordinate of $\RR^{2n+1}$, restricted to the submanifold $Y_g\subset \RR^{2n+1}$.
\end{defn}

We can fix a $4$-ad structure on $Y_g$ by fixing a decomposition of $S^{2n-1}\cong \partial(D^1\times D^{2n-1})$ as two $(2n-1)$-dimensional discs $D_\pm$ such that the reflection on the first coordinate restricted to $S^1\times D^{2n}\subset \RR^{2n+1}$ restricts to a reflection on each disc, and setting $\partial_0 Y_g$ and $\partial_1 Y_g$ to be, respectively, $S^1\times D_+$ and $S^1\times D_-$. Using this $4$-ad structure, the involution $\rho_g$ is a diffeomorphism of $4$-ads. We observe that the design criteria \ref{action of rho on homology} and \ref{equivarian} follow directly from \ref{reflection of vg} and \ref{image of vg is small}. In the next subsubsection, we show that \ref{action on framings} holds.

\begin{rmk}\label{compatibility with V_g}
    The embedding $\phi_{S^1}$ extends to $D^2\times D^{2n-1}\cong D^{2n}$, using the standard inclusion to $\RR^{2n+1}$ translated by $-\frac{3}{2}\vec{e}_{2n+1}.$ If we take the union of this extended embedding with $V_g$ along the tubular neighbourhood of the path used above, we obtain another model $V_g'$ for $V_g$ as a submanifold of $\RR^{2n+1}$. By retracting the former embedding, we see that these two submanifolds are ambient isotopic. Moreover, this ambient isotopy can be arranged while preserving the symmetry along the first axis. Thus, we have $\langle\rho\rangle$-equivariant embeddings $V_g\subset Y_g\subset V_g',$ where the composite $V_g\subset V_g'$ is a $\langle\rho\rangle$-equivariant ambient isotopy.
\end{rmk}

\subsubsection{Action on stable framings}

In this section, we lift the action of $\rho_g$ on various automorphism and embedding spaces of $Y_g$ by conjugation to an action on the its variants with stable framings, as defined in \cref{defn of stable framings section}. We fix once and for all a standard stable framing on $Y_g$.

\begin{convention}\label{the standard framing on Y_g}(Standard framing)
    For the entirity of this paper, the stable framing $\ell$ on $Y_g$ will denote the stable framing induced by the inclusion $Y_g\subset \RR^{2n+1}$ and the standard framing on $\RR^{2n+1}$. This framing will be called the \textit{standard stable framing} on $Y_g$. Observe that the restriction of $\ell$ to $V_g\subset Y_g$ is the stabilisation of the \textit{standard framing} in the sense of \cite[Section 2.5.1]{krannich2021diffeomorphismsdiscssecondweiss}. By \cref{compatibility with V_g}, we see that the inclusion $Y_g\subset V_g$ takes the standard stable framing on $V_g$ to a stable framing homotopic to $\ell.$
\end{convention}

To lift the action of the reflection involution to stable framings, we will proceed analogously as \cite[Section 2.8.1]{krannich2021diffeomorphismsdiscssecondweiss}. We briefly recall their strategy, applied to our setting. We focus on lifting the action along the map $\BDiffsfrv(Y_g)_\ell\to \BDiffv(Y_g)$, as the remaining cases are analogous. We will use the following facts:
\begin{enumerate}[label=(\alph*)]
    \item We have a fibre sequence
    \[\BDiffv(Y_g)\to \BDiff(Y_g;\vb)\to \mathrm{BDiff_\partial^{ext}}(\vb Y_g)\]where the total space is the classifying space of the topological group of diffeomorphisms of $Y_g$ that fix $\vb Y_g$ setwise, the base space is the classifying space of the subgroup of $\Diffb(\vb Y_g)$ given by the union of the path components of those diffeomorphisms that extend to $Y_g$, and the right map restriction of diffeomorphisms. As for any fibre sequence of spaces, the loop space of the total space acts in the $A_\infty$-sense on the fibre by pointed maps. For the specific case of the sequence above, and under the canonical equivalence $\Diff(Y_g;\vb)\to \Omega\BDiff(Y_g;\vb)$, this action agrees with the action of $\Diff(Y_g;\vb )$ on $\BDiffv(Y_g)$ by conjugation of diffeomorphisms. Thus, the reflection involution on $\BDiffv(Y_g)$ agrees with the action of $\Omega\BDiff(Y_g;\vb)$ on this space, once restricted to $\Omega\B\langle\rho_g\rangle.$
    \item The group $\Diff(Y_g;\vb)$ acts on $\Bun(TY_g;\sfr^*\gamma_{2n})$, so we denote its homotopy orbits by $\BDiff^{\sfr}(Y_g;\vb)$. Note that we have an analogous fibre sequence as above, by replacing the spaces with their stably framed analogs. Thus, to lift the action of $\langle\rho_g\rangle$ on $\BDiffv(Y_g)$ to $\BDiffsfrv(Y_g)_\ell$, it suffices to solve the following lifting problem
    \begin{equation}\label{lift manuel}
        \begin{tikzcd}[row sep =15pt]
        & \BDiff^{\sfr}(Y_g;\vb)_\ell \arrow[d]\\
        \B\langle \rho_g\rangle \arrow[r]\arrow[ur, dashed] & \BDiff(Y_g;\vb)
    \end{tikzcd}
    \end{equation}by using the canonical action of $\Omega\BDiff^{\sfr}(Y_g;\vb)_\ell$ on $\BDiffsfrv(Y_g)_\ell$, and precomposing it with the lift $\langle \rho_g\rangle\to\Omega\B\langle \rho_g\rangle\to\Omega\BDiff^{\sfr}(Y_g;\vb)_\ell$.
\end{enumerate}
One can observe that the strategy of proof of \cite[Lemma 2.12]{krannich2021diffeomorphismsdiscssecondweiss} also extends to our setting so to show that the lift of \eqref{lift manuel} exists: one simply replaces $V_g$ by $X_g$, framings by stable framings and the structure $\pm \mathrm{fr}$ by $\pm \sfr:\oo(1)\to \oo(2n+1)$, using the fact that, by definition, the stable framing $\ell$ extends to $\RR^{2n+1}.$ One can check that such lifts can be found in a compatible way for the various diffeomorphism and embedding spaces used in this work, hence establishing \ref{action on framings}. We leave this check to the reader. We fix \textit{once and for all} such choice of lifts.

\section{Mapping class groups and their homology}\label{MCGSection}

In this section, we study various mapping class groups of (stably framed) self-embeddings of $X_g$ and $Y_g$. We start by recalling the work of Bustamante and Randal-Williams \cite{BRW} on the former case, and later obtain an analogous result for the odd-dimensional case, whose proof is heavily inspired by Krannich's work of the case for $V_g$ in \cite{krannichpseudo}.

\subsection{The case of \texorpdfstring{$X_g$}{Xg}}

In this subsection, we recall the work in \cite[Section 6.2]{BRW}. Recall that we have fixed the standard stable framing $\ell$ on $Y_g$, which induces a stable framing on $X_g$, which we will also denote by $\ell.$ Denote the group $\pi_1(\BEmbsfr(\xg;\ell_0),\ell)$ by $\Embmcgsfr$ (recall \cref{defn of stable framings section}). Given any stable framing $\ell$ on $X_g$, there exists a unique framing which stabilises to it. This produces a map $\pix\to \pi_n(\Fr^+(T\xg))\cong \text{Imm}^{\text{fr}}_n(X_g)$ from \cite[Section 2]{BRW}. Here, the rightmost identification follows from Hirsch-Smale theory and depends on the choice of a framing on $S^n\times D^{n}$. We impose that this framing extends to $D^{n+1}\times D^n$, which is always possible (take for example the framing induced by the embedding $S^n\times D^n\subset \RR^{2n})$.  Using this map, one can lift the map $q_X:\pix\to \ZZ[\pi]/\Lambda_n$ to a map $q_X^\ell:\pix\to \ZZ[\pi]/\Lambda^{\text{min}}_n$, by taking generic self-intersections of immersions (see loc.cit.). By unwrapping the definition of $q_X^\ell$, we see that $(\pix,\lambda_X,q_X^\ell)$ is a hyperbolic module with respect to the basis $\{a_i,b_i\}_{i=1}^g$ of \cref{allTheNotationUngraded}, that is, $q_X^\ell(a_i)=0=q_X^\ell(b_i)$, $\lambda_X(a_i,a_j)=0=\lambda_X(b_i,b_j)$ and $\lambda_X(a_i,b_j)=\delta_{ij}$ (see also the proof of Lemma 6.6 of loc.cit). The hyperbolic basis $\{a_i,b_i\}_{i=1}^g$ induces a map
\[\Embmcgsfr\to \Ugmin\]whose image we denote by $\Omega^{\min}_g.$ We state now the main input from \cite{BRW} regarding this map given by Corollary 5.8 and Proposition 6.7 in loc.cit.

\begin{teo}\label{mcg of xg}
    Let $g\geq 8$ if $n$ is even and $g\geq 2$ if $n$ is odd. We have an extension
    \[1\to L_g^{\sfr,\ell}\to \Embmcgsfr\to \Omega^{\min}_g\to 1\]where $L_g^{\sfr,\ell}$ is a finite group. Moreover, the subgroup $\Omega^{\min}_g\subset \Ugmin$ is a normal subgroup and the quotient $\Ugmin/\Omega^{\min}_g$ is a finite abelian group independent of $g$.
\end{teo}

\subsection{The case of \texorpdfstring{$(Y_g;X_g)$}{(Yg;Xg)}}

We now investigate the analogous mapping class group of stably framed self-embeddings of the pair $(Y_g;X_g)$. Our strategy is heavily inspired by and relies on both the even-dimensional case of the last subsection, and the work of Krannich in \cite{krannichpseudo} for the analogous problem for the pair $(V_g;W_{g,1}).$ 

\subsubsection{Homotopy mapping class groups of $\yg$}

Recall the map $\pi_0(\Embo(\xg))\to \Ug$ given by the action on the quadratic module $(\pix,\lambda,q).$ The discussion in \cite[Section 3.4]{BRW} proves that this map factors through forgetting to $\pi_0(\Autb(\xg)).$ In other words, homotopy automorphisms of $\xg$ fixing the boundary preserve the intersection form. By pre-composition, we obtain a map $\pi_0(\Autv(\vg;\xg))\to \Ug$. It is clear that the image lies in the subgroup $\Ugext$ and that this map factors $\pi_0(\Embo(\yg;\xg))\to \Ugext.$ The analogous claim holds once $\xg$ and $\yg$ are replaced by $\wg$ and $\vg.$

\begin{prop}\label{auts to ugext cartesian}
    The square of groups
    \[\begin{tikzcd}[row sep =15pt]
        \pi_0(\Autv(\yg;\xg))\arrow[r] \arrow[d] & \Ugext\arrow[d] \\
        \pi_0(\Autv(\vg;\wg))\arrow[r] & \Ugwext
    \end{tikzcd}\]induced by the inclusion of $4$-ads $\yg\hookrightarrow \vg$ is a pullback square with surjective horizontal maps.
\end{prop}
\begin{proof}
    To prove surjectivity of the horizontal maps, observe that they factor the maps present in \cref{surjectivity from embeddings to unitary}, which are surjective. We now prove that the square is a pullback of groups. We start by reducing this statement to proving that the canonical map between the horizontal kernels is an isomorphism. This follows from the following standard observation in group theory: a commutative square of groups whose horizontal maps are surjective
    \[\begin{tikzcd}[row sep =15pt]
        A\arrow[r, "f"]\arrow[d] & B \arrow[d] \\
        C \arrow[r, "g"] & D
    \end{tikzcd}\]
    is a pullback square if and only if the map $\ker(f)\to \ker(g)$ is an isomorphism. Recall that $\Ugext$ is a subgroup of $\GLpi$ and similarly for $\Ugwext$. By \cite[Thm. 3.1]{BRW}, we have the action of $\Map_{S^1}(\xg)$ on $\pix$ induces an isomorphism $\pi_0(\Map_{S^1}(\xg))\cong \Hom_{\ZZ[\pi]}(\pix,\pix).$ The analogous result for $\wg$ is given in \cite[Thm. 3.3]{BRW}. Therefore, we can verify that the horizontal kernels of the square above are isomorphic by checking the same property for the following square
    \[\begin{tikzcd}[row sep =15pt]
        \pi_0(\Autv(\yg;\xg))\arrow[r] \arrow[d] & \pi_0(\Map_{S^1}(\xg))\arrow[d] \\
        \pi_0(\Autv(\vg;\wg))\arrow[r] & \pi_0(\Map_{*}(\wg))
    \end{tikzcd}\]whose horizontal maps are induced by restricting automorphisms to $\partial_2.$ This restriction map factors through $\Map_{S^1}(\yg;\xg)$ from \cref{homotopy groups of mapping of pair}. Moreover, the map $\Map_{S^1}(\yg;\xg)\to \Map_{S^1}(\xg)$ is injective on $\pi_0$: in the description of \cref{homotopy groups of mapping of pair} and \cite[Thm. 3.1]{BRW}, the induced map on $\pi_0$ is the inclusion \(\Hom_{\ZZ[\pi]}^{\mathrm{ext}}(\pix,\pix)\to\Hom_{\ZZ[\pi]}(\pix,\pix) \), and the same holds for $\wg.$ We conclude that we can test isomorphic horizontal kernels of the square
    \[\begin{tikzcd}
        \pi_0(\Autv(\yg;\xg))\arrow[r] \arrow[d] & \pi_0(\Map_{S^1}(\yg;\xg))\arrow[d] \\
        \pi_0(\Autv(\vg;\wg))\arrow[r] & \pi_0(\Map_{*}(\vg;\wg))
    \end{tikzcd}\]to deduce that the kernels of the original square are isomorphic. The top horizontal map fits in the long exact sequence of the fibre sequence
    \[
        \Autv(\yg;\xg)\to \Aut_{S^1}(\yg;\xg)\to \Map_{S^1}((S^1\times D^{2n-1},S^1\times S^{2n-2}),(\yg,\xg))
    \]and similarly for the bottom map. Consider the following portion of this map of long exact sequences
    \[\begin{tikzcd}[row sep =15pt]
       \pi_1(\Map_{S^1}(\yg;\xg),\iota) \arrow[r]\arrow[d] & \left[\pi_{2n+1}(\yg,\xg)\right]_\pi \arrow[r] \arrow[d] &  \pi_0(\Autv(\yg;\xg))\arrow[r] \arrow[d] & \pi_0(\Map_{S^1}(\yg;\xg))\arrow[d] \\ \pi_1(\Map_{*}(\vg;\wg),\iota) \arrow[r] &
       \pi_{2n+1}(\vg,\wg)\arrow[r]  & \pi_0(\Autv(\vg;\wg))\arrow[r] & \pi_0(\Map_{*}(\vg;\wg)).
    \end{tikzcd}\]To show that the kernels of the rightmost horizontal maps are isomorphic, it suffices to show that the cokernels of the leftmost horizontal maps are isomorphic. We shall first describe these maps using \cite[Thm. 3.1]{BRW}. We start with the top map. The restriction map to $\partial_2$ induces the following commutative square
    \[\begin{tikzcd}[row sep =15pt]
        \pi_1(\Map_{S^1}(\yg;\xg),\iota)\cong \Hom_{\ZZ[\pi]}^{\mathrm{ext}}(\pix,\pi_{n+1}(\xg)) \arrow[r, "\res_Y"] \arrow[d] & \left[\pi_{2n+1}(\yg,\xg)\right]_\pi \arrow[d]\\
        \pi_1(\Map_{S^1}(\xg),\iota) \cong \Hom_{\ZZ[\pi]}(\pix,\pi_{n+1}(\xg))\arrow[r] & \left[\pi_{2n}(\xg)\right]_\pi.
    \end{tikzcd}\]where $\res_Y$ denotes the induced map $\pi_{1}(\res_Y)$ from \eqref{YgXgMappingSpaceFibSeq}. Once again, both vertical maps in this square are injective. From the description of the bottom map in \cite[Thm. 3.1]{BRW}, we can describe the top map as taking $\phi$ to the sum $\sum_{i=1}^g[a_i,\phi(b_i)]+(-1)^n[\phi(a_i),b_i],$ where $a_i$ denote the generators of $\piw\subset \pix$ representing the classes $S^n\times \{*\}$ in the $i$-th copy of $S^n\times S^n$ and $b_i$ the classes representing $\{*\}\times S^n$. Note that $b_i$ generate $ \pirel$  and $a_i$ generate $\piy$ as $\ZZ[\pi]$-modules. From the computation of $\pi_k(\xg)$ in the metastable range from \cite[Appendix A]{BRW}, we see that
    \[[\pi_{2n+1}(\yg,\xg)]_\pi\cong \bigoplus_{\substack{g+1\leq i\leq 2g}} \pi_{2n}(S^n)\{x_i\}\oplus \bigoplus_{\substack{1\leq i\leq 2g \\ g+1\leq j\leq 2g \\ a\in \ZZ,\ (a,i)<(0,j)}} \pi_{2n}(S^{2n-1})\{t^ax_i\otimes x_j\} \]where the first term is given by $x_i\circ f$ for $f\in \pi_{2n}(S^n)$ and $x_i=a_i$ for $i\leq g$ and $x_i=b_{i-g}$ for $g+1\leq i$, and the second term is given by $[t^ax_i,x_j]\circ f'$ for $f'\in \pi_{2n}(S^{2n-1}).$ We compute now the cokernel of $\res_Y.$ This is heavily inspired by \cite[Thm. 3.5]{BRW}. We start by showing that the image of $\res_Y$ spans the second summand. For fixed $1\leq i,j\leq g$ and $a\in \ZZ,$ consider the morphism $\phi:\pix\to \pi_{n+1}(\xg)$ defined by $\phi(b_i)=t^{-a}b_j\circ \eta$ and that sends all the other generators to $0$, where $\eta\in \pi_{2n}(S^{2n-1})$ is the suspension of the Hopf fibration element. This is in $\Hom_{\ZZ[\pi]}^{\mathrm{ext}}(\pix,\pi_{n+1}(\xg))$ since all the elements $b_i$ are sent to $\pi_{n+2}(\yg,\xg).$ Thus, $\res_Y(\phi)= [a_i,t^{-a}b_j\circ \eta]$, which equals $[a_i,t^{-a}b_j]\circ \eta$, since $\eta$ is a suspension, by \cite[Section A.2]{BRW}. On the other hand, $[a_i,t^{-a}b_j]\circ \eta=[t^aa_i,b_j]\circ \eta$, since we took coinvariants. Since $\eta$ generates $\pi_{2n}(S^{2n-1})$, the second factor when $i\leq g$ is spanned by the image of $\res_Y.$ For $g<i$, simply choose $\phi$ to be defined by $\phi(a_i)=t^ab_i\circ \eta$ and the remaining generators to $0$, which again lies in $\Hom_{\ZZ[\pi]}^{\mathrm{ext}}(\pix,\pi_{n+1}(\xg)).$ By the same argument of \cite[1644]{BRW}, the remaining elements generate $x_i\circ f$ for $f$ in the kernel of the suspension map $\Sigma:\pi_{2n}(S^n)\to \pi_{n+1}(S^{n+1}).$ We conclude that the cokernel of $\res_Y$ identifies with $\piwrel\otimes \Sigma\pi_{2n}(S^n).$ The exact same argument along with the computations in \cite[Appendix A.1]{BRW} identifies the cokernel of $\res_Y $ in the $\wg$ case with the same group. As mentioned in this reference, under these identifications the map between the cokernels can be seen to be the identity. Thus, we conclude that the horizontal kernels of the square in the claim are isomorphic and thus the square is a pullback square. This finishes the proof.
\end{proof}

\subsubsection{Isotopy classes of self-embeddings of $\yg$}

In this section, we describe the group $\pi_0(\Embop(\yg; \xg))$ up to extensions. We shall do so by combining \cref{auts to ugext cartesian} and results from \cite[Section 3]{krannichpseudo} and \cite[Section 3]{BRW}. Let $S\pi_n(\SO(n))\subset \pi_n(\SO(n+1))$ denote the image of the stabilisation map $\pi_n(\SO(n))\to \pi_n(\SO(n+1))$. Let $\Mg$ be the abelian group of sesquilinear forms $\mu:\piy\otimes \piy\to \ZZ[\pi]$ (that is, $\mu(a\cdot x,b\cdot y)=a\mu(x,y)\bar{b}$ for all $a,b\in \ZZ[\pi]$ and $x,y\in \piy$) which are $(-1)^{n+1}$-symmetric (that is, $\mu(x,y)=(-1)^{n+1}\overline{\mu(y,x)}$ for all $x,y\in \piy$) and whose values $\mu(x,x)$ lie in $\Lambda_n$ for all $x\in \piy.$ We have the following result.

\begin{prop}\label{extensions of mcg of yg}
    The action of $\pi_0(\Embop(\yg;\xg))$ on $\pix$ induces an extension
    \[1\to \piwrel\otimes S\pi_n(\SO(n))\to \pi_0(\Embop(\yg;\xg))\to \Ugext\to 1. \]Moreover, the projection $\Ugext\to \GL(\piy)$ fits into an extension of the form
    \[1\to \Mg\to \Ugext\to \GL(\piy)\to 1\]which admits a splitting.
\end{prop}
\begin{proof}
    We begin with the following remarks. Firstly, the map $\Diffv(\yg)\to \Embo(\yg;\xg)$ hits all path components. In other words, every self-embedding $e$ of $\yg$ is isotopic to a diffeomorphism: first isotope the embedding $e$ so that the image of $\partial_1Y_g$ lies in the interior of $Y_g$ and consider its complement. Observe that this complement is an $h$-cobordism from $S^1\times D^{2n-1}\cong \partial_1\yg\cup e(\partial_1\yg)$ to the complement of the self-embedding $e|_{\xg}.$ Since $\text{Wh}_1(\ZZ)$ is trivial, this $h$-cobordism is trivial and thus $e$ may be isotoped to be surjective and thus a diffeomorphism. Therefore, the decoration $\cong$ is redundant. Secondly, the map $\pi_0(\Embo(\yg;\xg))\to \pi_0(\bEmbo(\yg;\xg))$ is an isomorphism: one considers the map of fibre sequences
    \[\begin{tikzcd}[row sep =15pt]
        \Emb_{\xg\cup \partial_0}^{\cong}(\yg)\arrow[d] \arrow[r] & \Embo(\yg;\xg)\arrow[r] \arrow[d] & \Embop(\xg)\arrow[d] \\
        \bEmb_{\xg\cup \partial_0}^{\cong}(\yg) \arrow[r] & \bEmbo(\yg;\xg) \arrow[r] & \bEmbop(\xg)
    \end{tikzcd}\]where the decoration $\cong$ on the leftmost spaces refers to the path components hit by $\Diffb(\yg).$ By applying Morlet's lemma of disjunction \cite[Thm. 3.1]{Burghelea1975} (or \cite[Thm. 4.4]{krannich2021diffeomorphismsdiscssecondweiss}), the outer vertical map are $(n-2)$-connected (see \cite[Prop. 4.3]{BRW} for more details). Hence, they are isomorphisms on path components when $n\geq 3$. In this case, the middle map induces an isomorphism on path components. Thirdly, as observed in \cref{block embeddings definition section}, we have an equivalence $\bDiffv(\yg)\simeq \bEmbo(\yg;\xg).$ The analogous discussion holds for $\vg$ and $\wg$. Consider the following commutative diagram
    \begin{equation}\label{pullback of embds and ug}
    \begin{tikzcd}[row sep =15pt]
        \pi_0(\Embop(\yg;\xg)) \arrow[d] \arrow[r] & \pi_0(\Autv(\yg;\xg)) \arrow[d] \arrow[r] & \Ugext\arrow[d] \\
        \pi_0(\Embop(\vg;\wg)) \arrow[r] & \pi_0(\Autv(\vg;\wg)) \arrow[r] & \Ugwext.
    \end{tikzcd}\end{equation}By \cref{auts to ugext cartesian}, the rightmost square is a pullback square. To show the first extension, we start by showing that the outer square is a pullback square. Since the right square is cartesian, it suffices to prove that the horizontal kernels of the left square are isomorphic. To do so, it is equivalent to showing the same claim for the square
    \[\begin{tikzcd}[row sep =15pt]
        \pi_0(\bDiffv(\yg))\arrow[r] \arrow[d] & \pi_0(\Autvp(\yg;\xg)) \arrow[d]
        \\ \pi_0(\bDiffv(\vg))\arrow[r] & \pi_0(\Autvp(\vg;\wg))
    \end{tikzcd}\]by the discussion above. The fibre of the map $\bDiffv(\yg)\to \Autvp(\yg;\xg)$ is equivalent to the loop space of the block surgery structure space of the triad $\yg$ fixed on $\vb \yg$ (see \cite[Section 2]{krannichpseudo} for more details). Since this triad satisfies the "$\pi$-$\pi$"-condition, it follows that the surgery structure space is equivalent to $\Map_{\vb}(\yg,\G/\O)$. This is, in turn, equivalent to $\Map_{\vb}(\vg,\G/\O)$ and thus (after looping) to the fibre of $\bDiffv(\vg)\to \Autvp(\vg;\wg)$. So the square above is obtained by applying $\pi_0$ to a homotopy cartesian square. Thus, it is a pullback square and hence, the horizontal kernels are isomorphic. We conclude that the kernel of the map $\pi_0(\Embop(\yg;\xg))\to \Ugext$ is isomorphic to the kernel of the map $\pi_0(\Embop(\vg;\wg))\to \Ugwext$, which is isomorphic to $\piwrel\otimes S\pi_n(\SO(n))$ by \cite[Thm. 3.3.(iii)]{krannichpseudo}. This finishes the proof of the first extension.

    We now move to the proof of the second extension. We start with the following simple observations:
    \begin{enumerate}[label=(\roman*)]
        \item any $\ZZ[\pi]$-linear map $\phi:\pix\to \pix$ taking $\pirel$ to itself is completely determined by the maps $A=\phi|_{\pirel}:\pirel\to \pirel$, $A'=p_0\circ \phi|_{\piy}:\piy\to \piy$ and $B=p_1\circ \phi|_{\piy}:\piy\to \pirel$, where $p_0$ and $p_1$ are the projections $\pix\to \piy$ and $\pix\to \pirel;$
        \item such $\ZZ[\pi]$-linear map $\phi$ preserves the hyperbolic form $\lambda_X$ if and only if $A$ is invertible, $A=(A^{-1})^\vee$ and $A\circ B=(-1)^n(A\circ B)^\vee$. Here, $(-)^\vee$ denotes the dual morphism in the category of $\ZZ[\pi]$-modules and we use the identification $(\pirel)^\vee\cong \piy.$
        \item $\phi$ preserves the quadratic refinement $q_X$ if and only if $A$ preserves $q_X$ and $B+(A^{-1})^\vee:\piy\to \pix$ preserves the quadratic form.
    \end{enumerate}Under this notation, the map $p:\Ugext\to \GL(\piy)$ takes a map $\phi$ to the map $A$ and its clear that the map $\GL(\piy)\to \Ugext$ taking $A$ to $A+(A^{-1})^\vee$ is a splitting of this map. Moreover, the kernel of $p$ is exactly those $\phi$ such that $A=\id.$ Let $\phi\in \ker(p)$, then it follows from the remarks above that $B=(-1)^nB^\vee$ and $B+\id$ preserves $q$. From the second condition, we conclude that for $x\in \piy$,
    \[q(x)=q(x+B(x))=q(x)+q(B(x))-\lambda(x,B(x)) \mod \Lambda_n.\]Moreover from the fact that $q(B(x))=0$ since $q$ vanishes on the isotropic subspace $\pirel$, we conclude that $\lambda(x,B(x))$ vanishes in $\ZZ[\pi]/\Lambda_n$ and thus lives in $\Lambda_n$. Consider the map $r:\ker(p)\to \Mg$ that takes a map $\phi$ to the sesquilinear form $\mu_\phi$ given by $\mu_\phi(x,y)=\lambda(x,B(y))$. It follows from the discussion above that $\mu_\phi$ is $(-1)^{n+1}$-symmetric and $\mu_\phi(x,x)\in \Lambda_n$. In other words, this map associates $\phi$ to the adjoint of $B:\piy\to (\piy)^\vee$. Thus, the image of $\phi$ under $r$ totally determines $B$ and thus determines $\phi$ since $A=\id.$ In other words, the map $r$ is injective. Surjectivity follows from the fact that such a sesquilinear form $\mu$ determines a map $B_\mu$ since $\lambda$ is non-degenerate. Thus, $r$ is an isomorphism of groups, which finishes the proof of the second claim.
    \end{proof}

\subsubsection{Stably framed embeddings} Denote by $\pi_0(\Embop(\yg;\xg)_\ell)$ the image of the map $\pi_1(\BEmbsfr(\yg;\xg;\ell_0),\ell)\to \pi_1(\BEmbop(\yg;\xg)).$ This can be described as the isotopy classes of those self-embeddings $e$ such that $e^*\ell$ is homotopic to $\ell$. Denote by $\Ugextl$ the image of the subgroup $\pi_0(\Embop(\yg;\xg)_\ell)$ on $\Ugext$, and by $\Mgs\subset \Mg$ the submodule of sesquilinear forms $\mu$ such that the image of $\mu(x,x)$ along the augmentation $\Lambda_n\subset \ZZ[\pi]\to \ZZ$ is even, for all $x\in \piy$. Write $\Lambda_n^{\sfr}$ for the pre-image of $2\ZZ$ along this map.

\begin{prop}\label{extensions of mcg of yg sfr}
    The action of $\pi_0(\Embop(\yg;\xg)_\ell)$ on $\pix$ induces an extension
    \[1\to \piwrel\otimes \ker(S\pi_n(\SO(n))\to \pi_n(\SO))\to \pi_0(\Embop(\yg;\xg)_\ell)\to \Ugextl\to 1. \]Moreover, the projection $\Ugextl\to \GL(\piy)$ fits into an extension of the form
    \[1\to \Mgs\to \Ugextl\to \GL(\piy)\to 1.\]
\end{prop}
\begin{proof}
    We start by proving the first extension. Notice that $\pi_0(\Embop(\yg;\xg)_\ell)$ identifies with the kernel of the map $\pi_0(\Embop(\yg;\xg))\to \pi_0(\Bun_{\partial_0}(T\yg,\sfr^*\gamma_{2n+1})).$ Using the stable framing $\ell$, the latter identifies with $\piwrel\otimes \pi_n(\SO),$ as $\piwrel\cong (\piv)^\vee$ and the map above is a crossed homomorphism\footnote{Let $G$ be a group and $H$ an abelian group with a $G$-action. A function $h:G\to H$ is a \textit{crossed homomorphism} if $h(xy)=h(x)+x\cdot h(y)$.}. Under this identification, the map 
    \[f:\piwrel\otimes S\pi_n(\SO(n))\to  \pi_0(\Embop(\yg;\xg))\to \piwrel\otimes \pi_n(\SO),\]where the left map is the one appearing in \cref{extensions of mcg of yg}, is a homomorphism and identifies with the stabilisation $S\pi_n(\SO(n))\to \pi_n(\SO)$ tensored with $\piwrel$ (see \cite[Prop. 3.7]{krannichpseudo} for a proof of this fact, since this map factors through $\pi_0(\Embop(\vg;\wg)$). The kernel of the map $ \pi_0(\Embop(\yg;\xg)_\ell)\to \Ugextl$ identifies with the kernel of $f$, which proves the first extension. 

    We move on to the second extension. Let $P$ be the pullback of the following diagram
    \[\begin{tikzcd}[row sep =15pt]
        P\arrow[d]\arrow[r] & \Ugext \arrow[d]\\
        \Ugwextl \arrow[r] & \Ugwext
    \end{tikzcd}\]where $\Ugwextl\subset \Ugwext$ is the analogous image of $\pi_0(\Embop(\vg;\wg)_\ell).$ Since the bottom map is an inclusion, we can identify $P$ as a subgroup of $\Ugext.$ The square 
    \[\begin{tikzcd}[row sep =15pt]        \pi_0(\Embop(\yg;\xg)_\ell)\arrow[d] \arrow[r] & \pi_0(\Embop(\yg;\xg))\arrow[d] \\
        \pi_0(\Embop(\vg;\wg)_\ell)\arrow[r] & \pi_0(\Embop(\vg;\xg))    
    \end{tikzcd}\]is a pullback square of groups, since the leftmost groups are the kernels of maps to a common group, namely $\piwrel\otimes \pi_n(\SO).$ Moreover, this square maps to the previous square. We conclude that the map $\pi_0(\Embop(\yg;\vg)_\ell)\to \Ugext$ factors through $P$. In other words, $\Ugextl\subset P$. Since the square \eqref{pullback of embds and ug} is a pullback, so is the square
    \[\begin{tikzcd}[row sep =15pt] 
    \pi_0(\Embop(\yg;\xg)_\ell)\arrow[d] \arrow[r] & P\arrow[d] \\
    \pi_0(\Embop(\vg;\wg)_\ell)\arrow[r] & \Ugwextl .   
    \end{tikzcd}\]However, by definition, the bottom map is surjective and thus, so is the top map. In other words, $\Ugextl=P.$ Thus, we can describe the kernel $K$ of the map $\Ugextl\to \GL(\piy)$ using the pullback square
    \[\begin{tikzcd}[row sep =15pt] 
        K \arrow[d] \arrow[r] & \ker(\Ugext\to \GL(\piy))\cong \Mg \arrow[d] \\
        \ker(\Ugwextl\to \GL(\piv)) \arrow[r] & \ker(\Ugwext\to \GL(\piv)).
    \end{tikzcd}\]The bottom kernels were described by Krannich in \cite[Thm. 3.3/Prop 3.7]{krannichpseudo}. The bottom right kernel is isomorphic to the submodule of sesquilinear forms $\mu:\piv\otimes \piv\to \ZZ$ which are $(-1)^{n+1}$-symmetric and $\mu(x,x)\in \Lambda_n$ (see \cite[Rmk. 3.2]{krannichpseudo}). Under this identification, the right vertical map is induced by the augmentation map $\ZZ[\pi]\to \ZZ.$ The bottom left kernel is isomorphic to the submodule of those sesquilinear forms where $\mu(x,x)$ is even. Under this identification, it is clear that $K$ identifies with $\Mgs.$ This finishes the proof of the second extension.
\end{proof}

\subsection{Stable homology of mapping class groups}

In this section, we relate the stable homology with rational coefficients of the mapping class groups of stably framed embeddings of $X_g$ and $(Y_g;X_g)$ to that of groups of arithmetic nature. For the case of $X_g$, we deduce this relation from the work of Bustamante and Randal-Williams \cite{BRW}, and for the $(Y_g;X_g)$, we use the work of the previous subsection together with a result in functor homology due to Betley \cite{Betley1992}.

\subsubsection{The case of $X_g$} In this section, we extract the homology of the group $\Embmcgsfr$ from the work of Bustamante and Randal-Williams \cite{BRW}, as $g$ is taken to $\infty$. The space $(\BEmbmcginfs)^+$ is equivalent to the identity component of the loop space $\Omega\B(\amalg_g \BEmbmcgsfr)$ by the group-completion theorem, applied to the monoid $\amalg_g \BEmbmcgsfr$ defined via stacking of embeddings, which is possible since $n\geq 2$ (see also the proof of \cite[Lemma 7.7]{BRW}). In particular, it is a simple space. So is $\BUgmininf^+$ for the same reason.

\begin{prop}\label{stable mcg of xg is gw}
    For $n\geq 3$, the action of $\Embmcgsfr$ on $\pix$ induces a map of simple spaces
    \[\left(\BEmbmcginfs\right)^+\to \BUgmininf^+\]which induces an isomorphism on rational homotopy groups in all degrees.
\end{prop}
\begin{proof}
    This follows by combining \cite[Lemma 7.10]{BRW} (stated above as \cref{mcg of xg} with the fact that there is a fibre sequence 
    \[(\B\Omega_\infty^{\min})^+\to \BUgmininf^+\to \B Q\]where $Q$ is the quotient $\Ugmin/\Omega_g^{\min}$, which is finite abelian and independent of $g$ (see \cref{mcg of xg}), and thus its classifying space is rationally contractible.
\end{proof}

\begin{rmk}\label{computation of GW}
    The rational homotopy groups of $\BUgmininf^+$ were completely computed in \cite[Lemma 7.3]{BRW}, by combining the results of \cite{SteimleHebestreit,nineI,nineII,nineIII}. More precisely, we have that $\pi_k^\QQ(\BUgmininf^+)$ is isomorphic to $\QQ$ when $k=1$ or $2n+k\equiv 0,1$ mod $4$, and vanishes otherwise.
\end{rmk}

\subsubsection{The case of $(Y_g;X_g)$}\label{YgXgMCGSection} Denote the group $\pi_1(\BEmbsfr(\yg;\xg),\ell)$ by $\mcgy$ and the colimit over $g$ along the stabilisation maps by $\mcgyi$. The basis $a_i$ (recall \cref{allTheNotationUngraded}) induces an isomorphism $\GL(\piy)\cong \GL_g(\ZZ[\pi]).$ Just as before, the spaces $(\Bmcgyi)^+$ and $\BGL_\infty(\ZZ[\pi])^+$ are equivalent to loop spaces and are thus simple spaces.

\begin{prop}\label{stable mcg of yg are k theory}
    For $n\geq 3$, the action of $\mcgy$ on $\piy$ induces a map of simple spaces
    \[\left(\Bmcgyi\right)^+\to \BGL_\infty(\ZZ[\pi])^+\]which induces an isomorphism on rational homotopy groups in all degrees.
\end{prop}
\begin{proof}
    Recall that a map of simple spaces induces an isomorphism on rational homotopy groups if and only if it induces one on rational homology groups. We shall show that the map $\Bmcgyi\to \BGL_\infty(\ZZ[\pi])$ is a rational homology isomorphism. Consider the sequence of groups
    \begin{equation}\label{sequence of groups}
    \mcgy\to \pi_0(\Embop(\yg;\xg)_\ell)\to \Ugextl\to \GL(\piy)
    \end{equation}which factors the action of $\mcgy$ on $\piy.$ Recall that, by definition, the leftmost map is surjective. We shall show that the leftmost and middle map are rational homology equivalences after taking classifying spaces and the rightmost map induces a rational homology equivalence after taking the colimit over $g$. We start with the leftmost map. Consider the following long exact sequence
    \[\cdots \to  \pi_1(\Bun_{\partial_0}(T\yg;\sfr^*\gamma_{2n+1}),\ell)\to \mcgy\to \pi_0(\Embop(\yg;\xg))\to \cdots\]and let us show that the kernel of the map $\mcgy\to \pi_0(\Embop(\yg;\xg))$ is a finite group. Consider the following map of fibre sequences
    \[\begin{tikzcd}[row sep =15pt] 
        \Bun_{\partial_0}(T\yg;\sfr^*\gamma_{2n+1})\arrow[d]\arrow[r] & \BEmbsfr(\yg;\xg) \arrow[r]\arrow[d] & \BEmbop(\yg;\xg) \arrow[d] \\
        F \arrow[r] & \BbEmbsfr(\yg;\xg) \arrow[r] & \BAutvp(\yg;\xg).
    \end{tikzcd}\]As mentioned in the proof of \cref{extensions of mcg of yg}, the middle map is an isomorphism on $\pi_1.$ By \cref{stable framed block embs are autos}, $F$ has finite homotopy groups and thus the map $\pi_1(\Bun_{\partial_0}(T\yg;\sfr^*\gamma_{2n+1}),\ell)\to \mcgy$ factors thorugh a finite group. We conclude that the kernel $K$ of the leftmost map of \eqref{sequence of groups} is a finite group. The Serre spectral sequence of the fibre sequence $1\to \B K\to \Bmcgy\to \B\pi_0(\Embop(\yg;\xg)_\ell)\to 1$ and the fact that $\H_*(K;\QQ)$ vanishes in positive degrees and is isomorphic to $\QQ$ in degree $0$ implies that $\smash{\Bmcgy\to \B\pi_0(\Embop(\yg;\xg)_\ell)}$ induces an isomorphism on rational homology in all degrees. As for the middle map, by \cref{extensions of mcg of yg sfr} and the finiteness of $\ker(S\pi_n(\SO(n))\to \pi_n(\SO))$ (see \cite[Lemma 3.2]{framingswg}), we deduce that it is a rational homology equivalence after taking classifying spaces, by the same argument as the leftmost map. 
    
    We consider now the rightmost map. The second extension of \cref{extensions of mcg of yg sfr} yields a spectral sequence 
    \[E^2_{s,t}=\H_s(\GL(\piy);\H_t(\Mgs;\QQ))\Rightarrow \H_{s+t}(\Ugextl;\QQ) \]which is compatible with respect to stabilisation over $g$. We shall prove that 
    \begin{equation}\label{the group we want to vanish}
        \colim_g\H_s(\GL(\piy);\H_t(\Mgs;\QQ))
    \end{equation}vanishes for $t\geq 1$. Together with the fact that $\H_0(\Mgs;\QQ)\cong \QQ$, this implies that the homomorphism $\Ugextl\to \GL(\piy)$ induces a rational homology equivalence after taking classifying spaces and colimits over $g.$ This vanishing result is a consequence of a result of Betley \cite[Thm. 4.2]{Betley1992} (better stated for our purposes in \cite[Thm. F.3]{Djament2010}). We proceed to prove this claim. Let $V$ be a finitely generated projective $\ZZ[\pi]$-module and denote the abelian group $M^{\sfr}(V)$ of sesquilinear forms $\mu:V\otimes V\to \ZZ[\pi]$ which are $(-1)^{n+1}$-symmetric and $\mu(x,x)\in \Lambda^{\sfr}_n.$  The basis $\{a_i\}_i$ induces an isomorphism from the group \eqref{the group we want to vanish} and $\colim_g\H_s(\GL_g(\ZZ[\pi]);\H_t(M^{\sfr}(\ZZ[\pi]^g))).$ Let $t\geq 0$ and consider the functor $F:\Proj_{\ZZ[\pi]}^\text{op}\to \Mod_\QQ$ taking a finitely generated projective $\ZZ[\pi]$-module $V$ to $\H_t(M(V);\QQ)$. Our goal is to show that $F$ is \textit{polynomial} in the sense of \cite[Lecture 1 
 Defn. 3.1]{functorhomology}. A functor between additive categories $T:C\to D$ is called \textit{polynomial} if its $d$-th cross-effect functor vanishes for some $d\geq 0$. The class of such functors is closed under composition and subfunctors (see \cite[25]{functorhomology}). In our setting, the category $C$ will be the either $\Proj_{\ZZ[\pi]}$, the category of finitely generated projective left $\ZZ[\pi]$-modules, or its opposite. The functor $F$ is a composition of the functor $M^{\sfr}: \Proj^{\text{op}}_{\ZZ[\pi]}\to \text{Ab}$ taking $V$ to $M(V)$ and the functor $\H_t(-;\QQ):\text{Ab}\to \Mod_\QQ$. The latter identifies with the $t$-th symmetric power $\smash{\Lambda^t_\QQ(-\otimes \QQ)}$ which is polynomial (see \cite[Rmk. 3.3.4]{functorhomology}). Thus, it suffices to prove that $M^{\sfr}$ is polynomial. Observe that $M^{\sfr}$ is a subfunctor of the functor $T:\Proj^{\text{op}}_{\ZZ[\pi]}\to \text{Ab}$ taking $V$ to $\Hom_{\ZZ[\pi]}(V\otimes_{\ZZ[\pi]} V, \ZZ[\pi])$, since $\ZZ[\pi]$ is commutative and hence any sesquilinear form factors through $V\otimes_{\ZZ[\pi]} V$. (This tensor product is defined by using the right module structure on the first variable and the left modules structure on the second variable, and seen as a left $\ZZ[\pi]$-module by using the left module structure on the first variable.) The functor $T$ is in turn the composition of the functors $(V\mapsto V\otimes_{\ZZ[\pi]} V )$ and $(V\to \Hom_{\ZZ[\pi]}(V,\ZZ[\pi]))$, which are polynomial by Example 3.2 and Remark 3.3.4 of loc.cit. This implies that $T$ and the subfunctor $M^{\sfr}$ are polynomial. We may apply \cite[Thm. F.3]{Djament2010} for the ring $A=\ZZ[\pi]$ and the functor $F$ for $t\geq 1$, as $F(0)=0$. We conclude that \eqref{the group we want to vanish} vanishes for $t\geq 1$, and hence that the string of morphisms \eqref{sequence of groups} induces a rational homology isomorphism after taking classifying spaces and colimit over $g$. This completes the proof.
\end{proof}

\begin{rmk}\label{computation of K}
    The rational homotopy groups of $\BGL_\infty(\ZZ[\pi])^+$ agree with the rationalised algebraic $K$-theory groups of $\ZZ[\pi]$ in positive degrees. In particular, the Bass--Heller--Swan splitting in algebraic $K$-theory implies that $\pi_k^\QQ(\BGL_\infty(\ZZ[\pi])^+)\cong (\K_{k}^\QQ(\ZZ)\oplus \K^{\QQ}_{k-1}(\ZZ)) $ for $k\geq 1$, and thus, by the work of Borel, are isomorphic to $\QQ$ if $k\equiv 1,2$ mod $4$ and $k\geq 4$, and vanishes otherwise.
\end{rmk}

\subsection{Odd modules over mapping class groups}

In this section, we will introduce the notion of a \textit{gr-odd module} in the context of $\QQ[G]$-modules for certain groups $G$. We will establish closure properties of the class of gr-odd modules and prove a vanishing result for group (co)homology with odd coefficients. 

    \begin{defn}\label{gr odd modules}
        Let $G$ be a group and $\iota\in G$ be an element such that there exists a finite normal subgroup $L\lhd G$ such that $\iota$ lies in the center of $G/L$. Let $M$ be a module over $\QQ[G].$ We say that $M$ is \textit{odd} if $\iota\cdot m =-m$ for every $m\in M.$ We say that $M$ is \textit{gr-odd} if it admits a finite length filtration of $\QQ[G]$-submodules \[0=F_{-1}M \subseteq \cdots \subseteq F_iM\subseteq F_{i+1}M\subseteq \cdots \subseteq F_kM=M\]where for every $i,$ the module $F_{i+1}M/F_iM$ is odd. These filtrations will be called \textit{odd}. We say that $M$ is \textit{even} if $\iota\cdot m=m$ for every $m\in M.$ Analogously, we define \emph{gr-even} modules as those that have a finite length filtration where the filtration quotients are even.
    \end{defn}

    We start by recording favorable closure properties of the classes of gr-odd and gr-even modules.
    
    \begin{prop}\label{odd modules are cool}
    The class of gr-odd $\QQ[G]$-modules is closed under taking subgroups, quotients, and extensions. The same properties hold for gr-even modules.
    \end{prop}

    \begin{proof}
    We will prove the statement for the odd case, since the even case is completely analogous. Let $M$ be an gr-odd $\QQ[G]$-module and $F_iM$ a finite filtration realizing this condition. For a quotient $M\twoheadrightarrow Q,$ set $F_iQ$ to be the image of $F_iM$ under the quotient map. Then $F_iQ/F_{i-1}Q$ is a quotient of $F_iM/F_{i-1}M$ and thus, it is odd. For a submodule $N$, take $F_iN$ to be the kernel of $F_iM\to F_i(M/N)$ as just defined. We see that $F_iN/F_{i-1}N$ is a submodule of $F_iM/F_{i-1}M$ and the claim follows. Let now $0\to M'\to M\to M''\to 0$ be an extension of gr-odd modules $M'$ and $M''.$ By first filtering $M',$ we may assume without loss of generality that $M'$ is itself is odd. Then setting $F_iM$ to be the pullback of $M\to M''$ over $F_iM''$ for $i\geq 1$ and $F_0M=M',$ we have an exact sequence \(0\to M'\to F_iM\to F_iM''\to 0\), so we see that the associated graded is odd.
    \end{proof}

    One can easily see that the classes of gr-odd and gr-even modules are not closed under tensor products over $\QQ$ and mapping spaces. However, the next proposition shows that their interplay behaves as expected. 

    \begin{nota}
        Given two \smash{$\QQ[G]$}-modules $A$ and $B,$ we consider $A\otimes_{\QQ}B$ as a \smash{$\QQ[G]$}-module with the diagonal action. Similarly, we consider $\Hom_\QQ(A,B)$ as a \smash{$\QQ[G]$}-module with the action by conjugation.
    \end{nota}

    \begin{prop}\label{tensors and homs of odd/even}
        Let $A$ and $B$ be $\QQ[G]$-modules with $A$ gr-odd. Then if $B$ is gr-odd (\textit{resp.} gr-even) then both $A\otimes_\QQ B$ and $\Hom_\QQ(A,B)$ are gr-even (\textit{resp.} gr-odd). The analogous statement for $A$ gr-even also holds.
    \end{prop}
    \begin{proof}
        We will assume that $A$ and $B$ are both gr-odd and will prove that $A\otimes_\QQ B$ and $\Hom_\QQ(A,B)$ are gr-even. The remaining cases are analogous. We start with the tensor product. We will proceed by induction on the filtration length witnessing the gr-oddness of $B.$ The base case is when $B$ is odd. Let $F_*A$ be an odd filtration of $A,$ then by exactness of $(-)\otimes_\QQ B,$ we can define a filtration of $\QQ[G]$-modules $F_i(A\otimes_\QQ B)\coloneq F_iA\otimes_\QQ B$. Once again by the exactness of the tensor product, we have
        \[F_i(A\otimes_\QQ B)/F_{i-1}(A\otimes_\QQ B)\cong (F_iA/F_{i-1}A)\otimes_\QQ B \]as $\QQ[G]$-modules. One can easily see by bilinearity that the tensor product of two odd modules is even. Assume now that the claim holds for all gr-odd modules $B'$ with odd filtration length $\leq n.$ Let $B$ be an gr-odd module along with an odd filtration $F_iB$ of length $n+1.$ Observe that we have a short exact sequence of $\QQ[G]$-modules
        \[0\to A\otimes_\QQ F_nB\to A\otimes_\QQ B\to A\otimes_\QQ (B/F_nB)\to 0.\]Both leftmost and rightmost modules are gr-even by the induction hypothesis and the base case, respectively. The claim follows from \ref{odd modules are cool}. The case of $\Hom_\QQ(A,B)$ follows the same strategy, since $\Hom_\QQ(-,-)$ is exact in both variables and that $\Hom_\QQ(A,B)$ is even when $A$ and $B$ are odd.
    \end{proof}

    \subsubsection{Vanishing results and the centre kills trick}

    In this subsubsection, we establish the key property of gr-odd modules. This property may be viewed as a graded analogue of the classical \emph{centre kills trick} from group homology. The idea of using this argument was inspired by \cite[Lem. A.3]{manuel4k+2}.

    
    \begin{prop}\label{center kills for odd}
    If $M$ is a gr-odd $\QQ[G]$-module, then the group $\H_i(G;M)$ vanishes for all $i\geq 0$. The same conclusion holds for cohomology.
\end{prop}
\begin{proof}
    We will prove this statement by induction on the odd filtration length of $M$. The base case is, once again, when $M$ is odd. Let $L$ be a finite normal subgroup of $G$ such that $\iota$ is in the center of $Q\coloneq G/L$. Notice that the $L$-coinvariants of $M$ form a $\QQ[Q]$-module. By the \textit{"centre kills trick"} (see, e.g., \cite[Lemma 5.4]{dupont2001scissors} for a proof), since the element $\iota\in Q$ is central, the group $\H_i(Q;[M]_{L})$ is $2$-torsion. On the other hand, it is a rational vector space so it vanishes. Consider the Lyndon--Hochschild--Serre spectral sequence
    \[E^2_{s,t}=\H_s(Q;\H_t(L;M))\Rightarrow \H_{s+t}(G;M)\]for the extension $1\to L\to G\to Q\to 1.$ By finiteness of $L,$ the only non-trivial entries are of the form $E^2_{s,0}=\H_s(Q;[M]_{L}).$ But these vanish and therefore so do $\H_i(G;M).$ Assume now that the claim holds for $M'$ with odd filtration length $\leq k$. Let $F_iM$ be an odd filtration of $M$ of length $k+1$. Then the long exact sequence on homology groups of the extension $0\to F_kM\to M\to M/F_kM\to 0$ implies the claim.
    \end{proof}

    \subsubsection{The element $-\id$} In order to use the vanishing results from above, we choose a specific $\iota$ for the mapping class groups considered in this section, satisfying the hypothesis of \cref{gr odd modules}.
    \begin{lemma}\label{-id exists}
        The subgroup $\Ugextl\subset \GL(\pix)$ contains the element $-\id$ for all stable framings $\ell$.
    \end{lemma}
    \begin{proof}
        We know that $-\id\in \Ugext$ as it lies in the image of the splitting $\GL(\piy)\to \Ugext$. Recall from \cref{extensions of mcg of yg sfr} that we have the following pullback square
         \[\begin{tikzcd}[row sep =15pt] 
        \Ugextl\arrow[d]\arrow[r] & \Ugext \arrow[d]\\
        \Ugwextl \arrow[r] & \Ugwext
        \end{tikzcd}\]so it suffices to show that $-\id\in \GL(\piw)$ is contained in $\Ugwextl $. This is exactly \cite[Lem. 3.9]{krannichpseudo}.
    \end{proof}

    \begin{nota}
        We fix, once and for all, an element $\iota_Y\in \pi_1(\BDiffsfrv(\yg),\ell)$ lifting $-\id\in \Ugextl$, which is possible by \cref{extensions of mcg of yg sfr} and the fact that $\smash{\pi_1(\BDiffsfrv(\yg),\ell)\to \mcgy}$ is surjective. Denote also the image of $\iota_Y$ in $\smash{\pi_1(\BDiffsfr(\xg),\ell)}$ by $\iota_X$. Since $-\id$ is central in $\GL(\pix)$, it is so too in $\Ugextl$. The groups $\smash{\mcgy}$ and $\smash{\Embmcgsfr}$ along with the elements $\iota_Y$ and $\iota_X$ satisfy the hypothesis of \cref{gr odd modules} by \cref{extensions of mcg of yg sfr} and \cref{mcg of xg}. When the context is clear, we supress the subscripts of $\iota_X$ and $\iota_Y$.   
    \end{nota}

\section{Self-embedding spaces of \texorpdfstring{$X_g$}{Xg}}\label{SelfEmbeddingXgSection}

A major ingredient in the approach of Bustamante--Randal-Williams is the computation of the rational homotopy type of $\BEmbsfr(X_g)_\ell$ in the range $* \leq n-2$. In this section, we extend this computation to degrees up to approximately $* \leq 2n - 2$, the main result being \cref{HomotopyBEmbsfr}. In \cref{homotopyOfPlusSectionXg}, we use this last result to analyse the homotopy type of the plus construction $\BEmbsfr(X_\infty)^+_\ell$ appearing in the Weiss fibre sequence \eqref{plusConstructedWeissXg}; the conclusion is \cref{HomotopyPlusConstruction}. We now give a brief overview this section.

In order to study the rational homotopy type of the space $\BEmbsfr(X_g)_\ell$, we consider the fibre sequence
\begin{equation}\label{BEmbsfrFibreSeq}
    \begin{tikzcd}
       \bEmbmodEmb(X_g) \rar & \BEmbsfr(X_g)_\ell\rar["\iota"] &\BbEmbsfr(X_g)_\ell,
    \end{tikzcd}
\end{equation}
where the left hand term denotes the homotopy fibre of $\BEmbop(X_g)\to \BbEmbop(X_g)$, and hence that of $\iota$ too. By Proposition \ref{stable framed block embs are autos}, the base space is rationally equivalent to $\BAutbp(X_g)$, whose higher homotopy groups are described in \cite[Thm. 3.1]{BRW}. By Morlet's lemma, the fibre is $(n-2)$-connected \cite[Prop. 4.3]{BRW}. The main input that allows us to extend the understanding of this homotopy type up to degrees $*\leq 2n-2$ is the work of the second-named author in \cite{MElongknots}, which we recall in \cref{RecollectionWWSection} for convenience. The computation of this left-hand term---outlined in \cref{PlanpEmbXgSection}---will be carried out in Sections \ref{FLShomologySection} and \ref{mcgXgsection}.

We also need to deal with the connecting homomorphisms in the long exact sequence of the fibration \eqref{BEmbsfrFibreSeq}. For this, note that for each positive integer $d\in \NN$, the circle $S^1$ is its own $d$-fold cover and, similarly, $X_{dg}$ is the corresponding $d$-fold cover of $X_g$ for every $g\geq 0$. This observation gives rise to \textit{Frobenius maps}
\begin{equation}\label{FrobeniusBEmbSfr}
  \begin{array}{lc}
    \varphi_d: \BEmbsfr(X_g)_\ell\longrightarrow \BEmbsfr(X_{dg})_\ell, &\\[6pt]
    \varphi_d: \BbEmbsfr(X_g)_\ell\longrightarrow \BbEmbsfr(X_{dg})_\ell, & \quad g\geq 0,\quad  d\in \NN,\\[4pt]
    \varphi_d: \bEmbmodEmb(X_g)\longrightarrow \bEmbmodEmb(X_{dg}), &
  \end{array}
\end{equation}
given by ``lifting'' a self-embedding of $X_g$ to $X_{dg}$. These maps are natural in $g\geq 0$, satisfy that the composite $\varphi_d\circ \varphi_{d'}$ is homotopic to $\varphi_{d\cdot d'}$ for every $d,d'\in \NN$, and are compatible with the fibre sequences \eqref{BEmbsfrFibreSeq} for varying $g\geq0$. In particular, the connecting homomorphisms
\begin{equation}\label{ConnectingMapSelfEmb}
\delta_*: \pi_{*+1}^{\QQ}(\BbEmbsfr(X_g)_\ell)\longrightarrow \pi_*^{\QQ}\left(\bEmbmodEmb(X_g)\right)
\end{equation}
are compatible with the homomorphisms induced by the Frobenius maps \eqref{FrobeniusBEmbSfr}---this additional compatibility will allow us to determine these connecting homomorphisms in \cref{HomotopyOfEmbeddingsXgSection} (it will turn out that, for our purposes, we will only need to understand $\delta_n$). The analysis of the Frobenius homomorphisms on the domain and codomain of \eqref{ConnectingMapSelfEmb} will take place in \cref{FrobeniusXgSection}.

Finally, in order to understand the homotopy type of the plus construction $\BEmbsfr(X_\infty)_\ell^{+}$, we will need to compute its homology and, consequently, that of $\BEmbsfr(X_g)_\ell$. In \cref{homotopyOfPlusSectionXg}, we convert the knowledge we 
have of its homotopy groups into information about its homology. We will need to understand the action of 
$\Embmcgsfr = \pi_1(\BEmbsfr(X_g)_\ell)$ on the higher homotopy groups of the base and the fibre in \eqref{BEmbsfrFibreSeq} to do this. The analysis of this mapping class group action on the fibre takes place in \cref{mcgXgsection}.

\subsection{Recollections on pseudoisotopy theory}\label{RecollectionWWSection}

The purpose of this section is to recall some results and constructions from \cite{GoodwillieHC-,WWI,WJR,MElongknots} that we will need later. We aim to keep the exposition concise, but readers already familiar with the contents of these references may safely skip this section and consult it only when referenced in the text.

\subsubsection{The Weiss--Williams theorems}\label{RecollectionWWtheorems}

The original Weiss--Williams theorem \cite{WWI} aims to describe the homotopy type of the pseudoisotopy space $\bDiffmodDiff(M)\coloneq \hofib(\BDiff_\partial(M)\to \BbDiffb(M))$ of a compact, smooth $d$-manifold $M$. To do so, the authors introduced certain (Borel) $C_2$-spectrum $\Hsp(M)$---the \textit{h-cobordism spectrum} of $M$---which is non-connective and natural in codimension-zero embeddings. Its $1$-connective cover $\Hsp^s(M)$ is known as the \textit{simple} $h$-cobordism spectrum. Then, they showed that there is a map
\begin{equation}\label{WWMap}
\Phi^{\Diff}_M: \bDiffmodDiff(M)\longrightarrow \Omega^{\infty}\big(\Hsp^s(M)_{hC_2}\big)
\end{equation}
that is $(\phi(M)+1)$-connected, where $\phi(M)$ denotes the concordance stable range of $M$.

We now recall an analogue of this theorem for embedding spaces that was established in \cite[Thm. A]{MElongknots}. Fix a compact, codimension-zero (for simplicity), neat submanifold $\iota:P\hookrightarrow M$. The goal now is to describe the homotopy type of the \textit{pseudoisotopy embedding space}
\[
\pEmb(P,M)\coloneq \hofib_{\iota}\left(\Embo(P,M)\to \bEmbo(P,M)\right).
\]
Write $\nu P$ for an open tubular neighbourhood of $P$ in $M$, and consider the induced embedding of compact manifolds $M-\nu P\hookrightarrow M$. The \textit{concordance embedding spectrum} of $\iota: P\hookrightarrow M$ is the $C_2$-spectrum
\begin{equation}\label{CEspDefn1}
\CEsp(P,M)\coloneq \hofib(\Hsp(M-\nu P)\to \Hsp(M)).
\end{equation}
Then, there exists a map
\begin{equation}\label{WWCEMap}
\Phi^{\Emb}_{(P,M)}: \pEmb(P,M)\longrightarrow \Omega^\infty\big(\CEsp(P,M)_{hC_2}\big)
\end{equation}
which is $\phi_{\CEmb}(P,M)$-connected if $d\geq 4$ and the handle dimension $p$ of $P$ relative to $\partial_0P\coloneq \partial P\cap \partial M$ satisfies $p \leq d-3$. Here, $\phi_{\CEmb}(P,M)$ is the concordance embedding stable range of $\iota: P\hookrightarrow M$. By \cite{KKGoodwillie}, 
\begin{equation}\label{CEmbRangeBound}
\phi_{\CEmb}(P,M)\geq 2d-p-5 \quad \text{if} \quad p\leq d-3.
\end{equation}
Moreover, the maps $\Phi^{\Diff}_{(-)}$ and $\Phi^{\Emb}_{(-)}$ of \eqref{WWMap} and \eqref{WWCEMap} give rise to a map of fibre sequences
\[
\begin{tikzcd}[row sep =18pt]     \pEmb(P,M)\rar\ar[dd,"\Phi^{\Emb}_{(P,M)}"] & \bDiffmodDiff(M-\nu P)\dar["\Phi^{\Diff}_{M-\nu P}"]\rar & \bDiffmodDiff(M)\dar["\Phi^{\Diff}_{M}"]\\
    &\Omega^{\infty}\big(\Hsp^s(M-\nu P)_{hC_2}\big)\dar\rar&\Omega^{\infty}\big(\Hsp^s(M)_{hC_2}\big)\dar\\
    \Omega^\infty(\CEsp(P,M))\rar & \Omega^{\infty}\big(\Hsp(M-\nu P)_{hC_2}\big)\rar&\Omega^{\infty}\big(\Hsp(M)_{hC_2}\big),
\end{tikzcd}
\]
where the top fibre sequence arises as the homotopy fibre of the map between the ordinary and block isotopy extension fibre sequences, and the bottom right subsquare is cartesian as long as $M-\nu P\to M$ is $2$-connected (cf. \cite[Cor. 5.6]{WWI})---this is the case under the assumption $p\leq d-3$.

\subsubsection{The map $\Phi^{\Emb}$ and orthogonal calculus}\label{OrthCalcRecallSection}
We will need to understand the particular construction of the map $\Phi^{\Emb}_{(P,M)}$, which we now recall. For this, we will use Weiss' \textit{orthogonal calculus} \cite{WeissOrthCalc}, which studies functors $\cJ\to \Spc$ from the ($\infty$-)category $\mathcal{J}$ of inner product vector spaces and linear isometric embeddings between them, to based spaces $\Spc$. The only thing we will need to recall from this theory is that, to each such orthogonal functor $F: \cJ\to\Spc$, there is associated a $C_2=\O(1)$-spectrum $\Theta F^{(1)}$---the \textit{first derivative of} $F$---and a diagram of orthogonal functors
\begin{equation}\label{OrthTower}
\begin{tikzcd}[row sep =12pt] 
     &&& T_1F(-)\dar[color=blue!65!black] &\lar[color=blue!65!black] \Omega^\infty\left((S^{\sigma\cdot(-)}\otimes \Theta F^{(1)})_{h\O(1)}\right)\\
     F(-)\ar[rrr, "\eta_0"]\arrow[urrr, bend left = 15pt, "\eta_1"] &&&T_0F(-)\rar[equal, shorten >=12ex]&\hspace{-70pt}F(\RR^\infty),
\end{tikzcd}
\end{equation}
where the indicated blue arrows form a fibre sequence of orthogonal functors. Here,
\begin{itemize}[itemsep = 3pt]
    \item $\sigma$ is the one-dimensional $\O(1)$-sign representation, $S^{\sigma\cdot V}$ denotes the one point compactification of $\sigma\cdot V\coloneq \sigma\otimes V$, and $\O(1)$ acts diagonally on the smash $S^{\sigma\cdot V}\otimes \Theta F^{(1)}$;

    \item the zeroth stage $T_0F(-)$ is given by
    $
    T_0F(V)\coloneq\operatorname{hocolim}_k F(V\oplus \RR^k)
    $, and thus admits a canonical equivalence from the constant orthogonal functor with \textit{value at infinity} $F(\RR^\infty)\coloneq\operatorname{hocolim}_k F(\RR^k)$. The map $\eta_0: F(V)\to T_0F(V)$ is simply the inclusion map;

    \item a model of the $\O(1)$-spectrum $\Theta F^{(1)}$ is as follows: the space $F^{(1)}(V)\coloneq\hofib(F(V)\to F(V\oplus \RR))$ inherits an $\O(1)$-action by declaring $-1\in \O(1)$ to act on $V$ and $V\oplus \RR$ by $-1$ on all coordinates. There are $\O(1)$-equivariant maps $s_V: S^1\wedge F^{(1)}(V)\longrightarrow F^{(1)}(V\oplus \RR)$ given, roughly, by performing a 180º rotation of $V\oplus \RR^2$ about the $2$-plane $0\oplus \RR^2$ (cf. \cite[\S A.1]{MElongknots} for an explicit description). The $\O(1)$-(pre-)spectrum $\Theta F^{(1)}$ has $F^{(1)}(\RR^n)$ as its $n$-th space, and $s_{\RR^n}$ as the structure map.
\end{itemize}

Given such an orthogonal functor $F:\cJ\to \Spc$, there is a map
\begin{equation}\label{PhiFInftyMap}
\begin{tikzcd}
    \Phi^F_\infty: F^{(1)}(\RR^\infty)\coloneq\hofib(F(0)\to F(\RR^\infty))\rar["\eta_1"] & \hofib(T_1F(0)\to T_1F(\RR^\infty))\simeq \Omega^\infty\big(\Theta F^{(1)}_{h\O(1)}\big),
\end{tikzcd}
\end{equation}
where the second equivalence follows from the observation that $T_1 F(\RR^\infty)\to T_0F(\RR^\infty)=F(\RR^\infty)$ is an equivalence (for the infinite-dimensional sphere $S^{\sigma\cdot\RR^\infty}$ is contractible). The map $\Phi^{\Emb}_{(P,M)}$ of \eqref{WWCEMap} is $\Phi^F_\infty$ for some appropriate orthogonal functor $F$ that we now introduce.

\begin{defn}\label{BdAndEorthFunctorsDefn}
    \begin{enumerate}[label = (\roman*)]
        \item For $Q$ a compact smooth manifold, consider the orthogonal functor
\[
\Bd_Q: \cJ\longrightarrow \Spc, \quad \Bd_Q(V)\coloneq \BDiff^b_\partial(Q\times V),
\]
where $\Diff^b_\partial(Q\times V)$ is the group-like topological monoid of diffeomorphisms of $Q\times V$ (rel boundary) which are bounded in the $V$-direction (see e.g. \cite[Defn. 2.4]{MElongknots}). Weiss and Williams define the $h$-cobordism spectrum $\Hsp(Q)$ as the first orthogonal derivative of $\Bd_Q$:
\begin{equation}\label{HspDefn}
    \Hsp(Q)\coloneq\Theta \Bd_Q^{(1)}.
\end{equation}
The involution $\iota_{\H}$ on $\Hsp(Q)$ inhereted from the $\O(1)$-action on $\Theta\Bd_Q^{(1)}$ is the \textit{h-cobordism involution}.
\item A codimension-zero embedding $Q\hookrightarrow Q'$ induces a natural transformation $\Bd_Q\to \Bd_{Q'}$. We set
\begin{equation}\label{CEspDefn}
\sfE_{(P,M)}\coloneq \hofib(\Bd_{M-\nu P}\longrightarrow \Bd_M), \quad \text{so that}\quad \Theta \sfE_{(P,M)}^{(1)}= \CEsp(P,M).
\end{equation}
We will also write $\iota_{\H}$ for the involution on $\CEsp(P,M)$ inhereted from the $\O(1)$-action on $\Theta\sfE_{(P,M)}^{(1)}$.
    \end{enumerate}

\end{defn}

Consider now the following commutative diagram:
\begin{equation}\label{CartesianIETBounded}
\begin{tikzcd}[column sep = 20pt, row sep = 12pt]
    \BbDiffb(M-\nu P)\rar[hook]\dar & \BbDiff^b_\partial((M-\nu P)\times \RR^\infty)\dar & \lar[hook', "\sim"'] \BDiff^b_\partial((M-\nu P)\times \RR^\infty)\dar\rar[equal] & \Bd_{M-\nu P}(\RR^\infty)\dar\\
    \BbDiffb(M)\rar[hook] & \BbDiff^b_\partial(M\times \RR^\infty) & \lar[hook', "\sim"']  \BDiff^b_\partial(M\times \RR^\infty)\rar[equal] & \Bd_{M}(\RR^\infty),
\end{tikzcd}
\end{equation}
where $\bDiff_\partial^b(Q\times V)$ is the topological monoid of bounded (in the $V$-direction) block diffeomorphisms of $Q\times V$ (cf. \cite[Defn. 2.8]{MElongknots}). That the middle horizontal maps are equivalences is \cite[Thm. B]{WWI}, and the left subsquare is homotopy cartesian if $M-\nu P\to M$ is $2$-connected by \cite[Cor. 5.5]{WWI}. Thus, taking vertical homotopy fibres, we see that
\begin{equation}\label{bEmbAsT0}
\bEmbcong(P,M)\simeq \sfE_{(P,M)}(\RR^\infty) \quad \text{if $p\leq d-3$.}
\end{equation}

\begin{defn}\label{PhiEmbDefn}
    The map $\Phi^{\Emb}_{(P,M)}$ of \eqref{WWCEMap} is the map $\Phi^{F}_\infty$ of \eqref{PhiFInftyMap} for $F=\sfE_{(P,M)}$, which takes the form
    \[
    \pEmb(P,M)\overset{\eqref{bEmbAsT0}}{\simeq} \hofib(\sfE_{(P,M)}(0)\to \sfE_{(P,M)}(\RR^\infty))\xrightarrow{\Phi^{\sfE_{(P,M)}}_\infty} \Omega^\infty\big(\Theta (\sfE_{(P,M)})^{(1)}_{h\O(1)}\big)\overset{\eqref{CEspDefn}}{=} \Omega^\infty\big(\CEsp(P,M)_{C_2}\big). 
    \]
\end{defn}

\begin{rmk}
    The original Weiss--Williams map $\Phi^{\Diff}_M$ of \eqref{WWMap} is not quite $\Phi^{\Bd_M}_\infty$, but they are related: using obstruction theory \cite[\S4]{WWI}, they showed that the dashed arrow in the commutative diagram
    \[
    \begin{tikzcd}
        \left(\frac{\bDiff}{\Diff}\right)_\partial(M)\dar[dashed, "\Phi^{\Diff}_M"]\rar[hook] &\frac{\bDiff^b_\partial(M\times \RR^\infty)}{\Diffb(M)} &\lar[hook', "\sim"'] \frac{\Diff^b_\partial(M\times \RR^\infty)}{\Diffb(M)}\dar["\Phi^{\Bd_M}_\infty"]\\
        \Omega^\infty\left(\Hsp^s(M)_{hC_2}\right)\ar[rr] && \Omega^\infty\left(\Hsp(M)_{hC_2}\right)
    \end{tikzcd}
    \]
    exists---this map is, by definition, $\Phi^{\Diff}_M$.
\end{rmk}

\subsubsection{On the homotopy type of $\Hsp(M)$ and $\CEsp(P,M)$}

Ignoring the $C_2$-actions on $\Hsp(M)$ and $\CEsp(P,M)$, the homotopy types of these spectra are closely related to Waldhausen's algebraic $K$-theory---in this section, we recall this connection. 

For $Q$ a compact smooth manifold, let $\H(Q)$ denote the \textit{space of (smooth) $h$-cobordisms starting at $Q$} (cf. \cite[p. 296]{VogellInvolution}). For instance, by the $s$-cobordism theorem, if $\dim Q\geq 5$, then the set of path components $\pi_0(\H(Q))$ is in bijection with the Whitehead group $\mathrm{Wh}(\pi_1 Q)$ of the fundamental group(oid) of $Q$. The codimension-zero embedding $M-\nu P\hookrightarrow M$ induces an ``extension-by-the-identity'' map $\H(M-\nu P)\to \H(M)$, which, as long as $d\geq 5$ and $d-p\geq 3$\footnote{In general, \eqref{CEmbIETSeq} is a fibre sequence upon looping the total space and the base.}, features in an isotopy extension fibre sequence
\begin{equation}\label{CEmbIETSeq}
\begin{tikzcd}
    \CEmb(P,M)\rar & \H(M-\nu P)\rar & \H(M).
\end{tikzcd}
\end{equation}
Here $\CEmb(P,M)$ denotes the space of \textit{concordance embeddings} of $P$ in $M$---that is, the space of all embeddings $\varphi: P\times [0,1]\hookrightarrow M\times [0,1]$ such that
\begin{itemize}
    \item[(i)]  $\varphi^{-1}(M\times \{i\})=P\times \{i\}$ for $i=0,1$, and

    \item[(ii)] $\varphi$ agrees with the identity in a neighbourhood of $P\times \{0\}\cup \partial_0 P\times [0,1]$.
\end{itemize}
Under the usual handle codimension assumption $p\leq d-3$, the space $\CEmb(P,M)$ is connected by Hudson's Theorem \cite{HudsonThm}. Let us assume this condition from now on.

There are stabilisation maps
\begin{equation}\label{ConcordanceStabilisationMapsDefn}
\Sigma_{\H}: \H(Q)\longrightarrow \H(Q\times I), \qquad \Sigma_{\CEmb}: \CEmb(P,M)\longrightarrow\CEmb(P\times I, M\times I)
\end{equation}
roughly given by crossing an $h$-cobordism or a concordance embedding with the interval $I$---see \cite[p.298]{VogellInvolution} for $\Sigma_H$ and \cite{KKGoodwillie} for $\Sigma_{\CEmb}$. The spaces of \textit{stable} $h$-cobordisms (resp. concordance embeddings) are
\begin{equation}\label{sCEdefn}
\cH(Q)\coloneq\hocolim_{k} \H(Q\times I^k), \qquad \sCEnopartial(P,M)\coloneq \hocolim_{k}\CEmb(P\times I^k,M\times I^k),
\end{equation}
where the bonding maps in the filtered systems are $\Sigma_H$ and $\Sigma_{\CEmb}$, respectively. Since the fibre sequence \eqref{CEmbIETSeq} is compatible with these stabilisation maps, we also obtain a stable isotopy extension fibre sequence
\begin{equation}\label{StableCEmbIETSeq}
\begin{tikzcd}[row sep = 8pt]
    \sCEnopartial(P,M)\rar\dar[phantom, "\vsimeq"] & \cH(M-\nu P)\rar\dar[phantom, "\vsimeq"] & \cH(M)\dar[phantom, "\vsimeq"]\\
    \Omega^\infty\CEsp(P,M)\rar & \Omega^\infty\Hsp(M-\nu P)\rar & \Omega^\infty\Hsp(M).
\end{tikzcd}
\end{equation}
The middle, right (and hence left) vertical equivalences in \eqref{StableCEmbIETSeq} follow from Propositions 1.8, 1.10 and 1.12 in \cite{WWI} (see also \cite[Rem. 3.10]{MElongknots}).

\begin{rmk}\label{IsotopyExtensionBdQRmk}
    In fact, more is shown in loc.cit.: upon looping once\footnote{This is presumably not necessary, but it is a slightly technical condition to get rid of.}, there are equivalences $\Omega \H(Q\times I^k)\xrightarrow{\sim} \Omega^{1+k}\Bd_Q^{(1)}(\RR^k)$ and $\Omega \CEmb(P\times I^k,M\times I^k)\xrightarrow{\sim} \Omega^{1+k}\sfE_{(P,M)}^{(1)}(\RR^k)$ which are compatible with the ones in \eqref{StableCEmbIETSeq} in the sense that there are homotopy commutative diagrams
    \begin{equation}\label{ConcStabilisationInclusion}
    \begin{tikzcd}[row sep = 15pt]
        \Omega \H(Q)\dar["\vsim"]\rar["\Sigma"] & \Omega \H(Q\times I)\rar["\Sigma"]\dar["\vsim"] & \Omega \H(Q\times I^2)\dar["\vsim"]\rar &\dots\rar &\Omega\cH(Q)\dar["\vsim"]\\
        \Omega\Bd_Q^{(1)}(0)\rar["s^{\vee}"] & \Omega^{2}\Bd_Q^{(1)}(\RR)\rar["s^{\vee}"] & \Omega^{3}\Bd_Q^{(1)}(\RR^2)\rar &\dots\rar &\Omega^{\infty+1}\Hsp(Q),
    \end{tikzcd}
    \end{equation}
    natural in codimension-zero embeddings, and
\begin{equation}\label{CEmbStabilisationInclusion}
    \begin{tikzcd}[row sep = 15pt]
        \Omega \CEmb(P,M)\dar["\vsim"]\rar["\Sigma"] & \Omega \CEmb(P\times I,M\times I)\rar\dar["\vsim"] &\dots\rar &\Omega\sCEnopartial(P,M)\dar["\vsim"]\\
        \Omega\sfE_{(P,M)}^{(1)}(0)\rar["s^{\vee}"] & \Omega^{2}\sfE_{(P,M)}^{(1)}(\RR)\rar &\dots\rar &\Omega^{\infty+1}\CEsp(P,M).
    \end{tikzcd}
    \end{equation}
    Moreover, the diagram \eqref{CEmbStabilisationInclusion} is the homotopy fibre of the map of diagrams \eqref{ConcStabilisationInclusion} induced by $M-\nu P\hookrightarrow M$.
\end{rmk}

One of Waldhausen’s most celebrated results, proved together with Jahren and Rognes in \cite{WJR}, is the \emph{parametrised stable $s$-cobordism theorem}, which establishes an equivalence
\begin{equation}\label{WJReq}
\cH(M)\simeq \Omega\Wh(M)\coloneq \hofib\big(Q_+M=\Omega^\infty\Sigma^\infty_+M\xrightarrow{\nu} A(M)\big)
\end{equation}
which is natural in codimension-zero embeddings. Here $A(-)\simeq \Omega^\infty\Ksp(\Sp_{(-)})$ denotes Waldhausen's algebraic $K$-theory of spaces---for $M$ based connected, we have a natural equivalence $A(M)\simeq \Omega^\infty\Ksp(\mathbf{S}[\Omega M])$---and $\nu$ is induced by the composition of the unit map $M_+\otimes \bfS\to M_+\otimes \Ksp(\bfS)$ and the assembly $M_+\otimes \Ksp(\bfS)\to \Ksp(\Sp_M)$; this map is well-known to be split injective. As a consequence of \eqref{WJReq} and the fibre sequence \eqref{StableCEmbIETSeq}, we see that the homotopy type of the infinite loop space of the concordance embedding spectrum is given by
\begin{equation}\label{CEWJReq}
    \Omega^\infty\CEsp(P,M)\simeq \sCEnopartial(P,M)\simeq\tohofib\Bigg(\begin{tikzcd}[scale=0.4, column sep = 15pt, row sep = 12pt]
    Q_+(M-P)\dar\rar["\nu"]&A(M-P)\dar\\
    Q_+M\rar["\nu"]&A(M)
\end{tikzcd}\Bigg).
\end{equation}
In the next section, we recall how to describe the right-hand term in terms of simpler trace invariants.

\subsubsection{Trace invariants and Goodwillie's isomorphism}
\textit{Trace methods} are concerned with the study of \emph{topological cyclic homology} \cite{BokHsiMadTC, NikolausScholze}, denoted $\TC(-)$, and related invariants as an approximation to algebraic $K$-theory. The invariants that will be relevant to us feature in the following commutative diagram of functors from $\EE_1$-ring spectra, say, to spectra:
\begin{equation}\label{TraceInvariantsDiagram}
\begin{tikzcd}
        \Ksp(-)\ar[rr, bend left=20pt, "\mathrm{Trc}^-"]\ar[rrr, bend left=30pt, "\mathrm{Tr}"]\rar["\mathrm{Trc}"'] & \TC(-)\rar &\HC^-(-)\rar & \HH(-).\\
        &&\Sigma\HC(-)\uar["N"']&
\end{tikzcd}
\end{equation}
Here $\HH(-)$ denotes \textit{Hochschild homology}, $\HC(-)$ is \emph{cyclic homology} and $\HC^-(-)$ is \textit{negative cyclic homology}. The map $\mathrm{Trc}$ is the cyclotomic trace map of Bökstedt--Hsiang--Madsen, $\mathrm{Tr}$ is the Dennis trace map, and $N$ is the $S^1$-norm map under the well-known equivalences $\HC(-)\simeq \HH(-)_{hS^1}$ and $\HC^-(-)\simeq \HH(-)^{hS^1}$ \cite{KrauseNikolausTHHLectures}. These invariants will be relevant to us because of the following remarkable property: given a ring map $R\to S$ such that $\pi_0(R)\to \pi_0(S)$ is surjective with kernel a nilpotent ideal, then the squares in the following diagram are cartesian as indicated:
\begin{equation}\label{DGMdiagram}
\begin{tikzcd}[column sep=huge, row sep=15pt] 
\Ksp(R) \arrow[r] \arrow[d]
  & \TC(R)  \arrow[r] \arrow[d]
  & \HC^{-}(R) \arrow[r, leftarrow] \arrow[d]
  & \Sigma\HC(R) \arrow[d] \\
\Ksp(S) \arrow[r]
  & \TC(S) \arrow[r]
  & \HC^{-}(S) \arrow[r, leftarrow]
  & \Sigma\HC(S).
\arrow[from=1-1, to=2-2, phantom, "\lrcorner", very near start]
\arrow[from=1-2, to=2-3, phantom, "\lrcorner_{\QQ}", very near start]
\arrow[from=1-4, to=2-3, phantom, "{}_{\QQ}\llcorner", very near start]
\end{tikzcd}
\end{equation}
That the leftmost square is cartesian is the celebrated theorem of Dundas–Goodwillie–McCarthy \cite{DundasGoodwillieMcCarthy}. Earlier, Goodwillie \cite{GoodwillieHC-} had shown that the right-hand subsquare, as well as the combined square obtained from the left and middle subsquares, are cartesian after rationalisation.

We are only concerned with ring spectra of the form $R=\bfS[\Omega X]$, i.e. the spherical group ring of $\Omega X$, for a connected based space $X$; in this case, the Hochschild homology $\HH(\bfS[\Omega X])$ is equivalent to $\mathrm{H}\ZZ\otimes LX_+$, the (integral) homology of the free loop space of $X$. In particular, it also follows that
\begin{equation}\label{HCasFLShomology}
\HH_*(\bfS[\Omega X])\cong \H_*(LX) \quad\text{and}\quad \HC_*(\bfS[\Omega X])\cong \H_*^{S^1}(LX).
\end{equation}
To conclude, let $\iota:P\hookrightarrow M$ with $M$ connected and such that $p\leq d-3$, so that the map of ring spectra $(R\to S)=(\bfS[\Omega(M-P)]\to \bfS[\Omega M])$ is $1$-connected. Then, by \eqref{CEWJReq}, \eqref{DGMdiagram} and \eqref{HCasFLShomology}, we obtain
\begin{equation}\label{CEhtpyFLShomologyeq}
    \pi_*^{\QQ}(\CEsp(P,M)) \cong \frac{\H^{S^1}_{*+1}(LM, L(M-P); \QQ)}{\H_{*+2}(M,M-P;\QQ)}.
\end{equation}
Here, the right-hand side really stands for the cokernel of the homomorphism induced on $\H_{*+2}(-;\QQ)$ by
\begin{equation}\label{ShiftHC-Map}
\begin{tikzcd}
\Sigma^\infty(M/M-P)\rar[hook, "c"] & \left(\Sigma^\infty(LM/L(M-P))\right)^{hS^1} & \lar["N"', "\simeq_{\QQ}"] \left(\Sigma^{\infty+1}(LM/L(M-P))\right)_{hS^1},
\end{tikzcd}
\end{equation}
where $c$ is induced by the ($S^1$-equivariant) inclusion $(M,M-P)\to (LM,L(M-P))$ as constant loops, and hence is split injective (with left inverse induced by the evaluation map $\mathrm{ev}_1: LX\to X$).

Our computation of the rational homotopy groups of the pseudoisotopy embedding space $\pEmb(P,M)$ will crucially use the isomorphism \eqref{CEhtpyFLShomologyeq}. However, since the codomain of the Weiss--Williams map $\Phi^{\Emb}_{(P,M)}$ in \eqref{WWCEMap} is the infinite loop space of $\CEsp(P,M)_{hC_2}$ rather than that of $\CEsp(P,M)$ itself, we will need to trace the effect of the involution $\iota_{\H}$ on $\CEsp(P,M)$ through the isomorphism \eqref{CEhtpyFLShomologyeq}---we do this in the next section.

\subsubsection{The involution on $\CEsp(P,M)$}\label{InvolutionRecollectionSection}

By the homotopy orbits spectral sequence, given a finite group $G$ and $G$-spectrum $X$, the rational homotopy groups $\pi_*^{\QQ}(X_{hG})$ of the homotopy orbit space of $X$ are isomorphic to the $G$-coinvariants $[\pi_*^\QQ(X)]_{G}$ of the rational homotopy groups of $X$. Thus, in order to compute the rational homotopy groups of $\Omega^{\infty}(\CEsp(P,M)_{hC_2})$, which is our goal, it suffices to undertand the effect of the $h$-cobordism involution on (the rational homotopy groups of) $\Omega^\infty\CEsp(P,M)\simeq \sCEnopartial(P,M)$. We now briefly recall how this was done in \cite{VogellInvolution} and \cite[\S 5]{MElongknots}.

The first step is to introduce an involution on the stable spaces $\cH(Q)$ and $\sCEnopartial(P,M)$ that can be compared to the concordance involution $\iota_{\H}$ on $\Omega^\infty\Hsp(Q)$ and $\Omega^\infty\CEsp(P,M)$, respectively. This was done in \cite[\S2]{VogellInvolution}, who constructed geometric involutions on the $h$-cobordism spaces $\H(Q\times I^k)$ which assembled to a (homotopy) involution on $\cH(Q)$---we shall also write $\iota_{\H}$ for these involutions on $\H(Q\times I^n)$ and $\cH(Q)$, as in \cite[Prop. 5.8]{MElongknots} they were shown to be compatible (up to homotopy) with the one on $\Omega^\infty\Hsp(Q)$. Equipping $\CEmb(P,M)\simeq\hofib(\H(M-\nu P)\to \H(M))$ and $\sCEnopartial(P,M)\simeq \hofib(\cH(M-\nu P)\to \cH(M))$ with the induced involutions $\iota_{\H}$ on the homotopy fibre, it also follows that the equivalence $\Omega^\infty\CEsp(P,M)\simeq \sCEnopartial(P,M)$ is (homotopy) $C_2$-equivariant. We record some constructions and facts related to $\iota_{\H}$ that we will need:

\begin{enumerate}[label = (\arabic*)]
    \item Here is an explicit model for the involution $\iota_{\H}$ on the $h$-cobordism space $\H(Q)$: Vogell introduced in \cite[\S2]{VogellInvolution} a model of $\H(Q)$ in terms of \textit{partitions}---a zero-simplex in $\H(Q)$ is a triple $(W,F,V)$ giving rise to a decomposition $Q\times [-1,1]=W\cup_F V$, where $W$ and $V$ are compact codimension-zero submanifolds of $Q\times [-1,1]$ which are $h$-cobordisms $W: M\times \{-1\}\leadsto F$ and $V: F\leadsto M\times \{1\}$, and $F$ is a codimension-one sumanifold of $Q\times[-1,1]$ meeting $\partial Q\times \{0\}$ transversely. A $p$-simplex in $\H(Q)$ is a bundle of such triples over $\Delta^p$. Then, $\iota_{\H}: \H(Q)\to \H(Q)$ is induced by the reflection $r$ on $[-1,1]$, i.e., it sends a partition $(W,F,Q)$ to $(r(V),r(F),r(W))$, its image under $\Id_Q\times r: Q\times [-1,1]\cong Q\times [-1,1]$. With this model, it is evident that $\iota_{\H}$ is natural in codimension-zero embeddings, and thus induces an involution $\iota_{\H}$ on the homotopy fibre $\CEmb(P,M)\simeq \hofib(\H(M-\nu P)\to \H(M))$---we will explicitly model the $h$-cobordism involution $\iota_{\H}$ on $\CEmb(P,M)$ in \cref{AppendixCEhCobInv}. \label{iotaHItem1}

    \item The $h$-cobordism involution $\iota_{\H}$ on the space $\H(Q)$ admits an alternative simpler description on its (based) loop space. Consider the diagram
\begin{equation}\label{OmegaHFibSeq}
\begin{tikzcd}
    \C(Q)\rar["a"] & E(Q)\rar["p"] & \H(Q),
\end{tikzcd}
\end{equation}
where $\C(Q)\cong \Diff_{Q\times \{-1\}\cup\partial Q\times [-1,0]}(Q\times [-1,0])$ is the group of concordance diffeomorphisms and $E(Q)$ is the space of embeddings of $Q\times [-1,0]$ into $Q\times [-1,1]$ that are fixed on a neighbourhood of $Q\times \{-1\}\cup \partial Q\times [-1,0]$. The map $p$ is given by sending such an embedding $\varphi$ to the partition $(\varphi(Q\times [-1,0]), \varphi(Q\times 0), Q\times [-1,1]- \varphi(Q\times [-1,0)))$. The map $a$ is the action map by precomposition on the identity embedding. The space $E(Q)$ is contractible for it is a space of collars, and, by \cite[Prop. 2.1]{VogellInvolution}, this diagram is in fact a fibre sequence. Thus, it follows that $\Omega \H(Q)$ is equivalent to $\C(Q)$. On $\C(Q)$, there is a natural involution $\iota_{\C}$---the \textit{concordance involution}---which sends a concordance diffeomorphism $F=(f,g):Q\times I\to Q\times I$ to $\overline{F}=(\overline{f}, \overline{g})$, where
    \begin{equation}\label{ConcInvDefn}
    \overline{f}(x,t)= f_1^{-1}\circ f(x,1-t), \qquad \overline{g}(x,t)=1-g(x,1-t), \qquad (x,t)\in Q\times I,
    \end{equation}
    and $f_1=f\mid_{M\times 1}: Q\cong Q$. By \cite[Prop. 2.2]{VogellInvolution} (see also \cite[Warn. 5.11]{MElongknots}), the involutions $\iota_{\C}$ and $\iota_{\H}$ agree on $\C(Q)\simeq \Omega \H(Q)$ up to a sign---more precisely, writing $\sigma$ for the $C_2$-sign representation and $\Omega^\sigma(-)\coloneq \Map_*(S^\sigma,-)$ for the $\sigma$-twisted loop space, there is a $C_2$-equivariant equivalence
    \begin{equation}\label{antiequivarianceIotaC}
    \Omega^\sigma (\H(Q),\iota_{\H})\overset{\sim}{\longrightarrow} (\C(Q), \iota_{\C}), \quad \text{so $\iota_{\C}=-\iota_{\H}$ on $\pi_*(\H(Q))$.}
    \end{equation}
    In particular, the restriction map $(\C(M),\iota_{\C})\to (\CEmb(P,M),\iota_{\H})$ is anti-equivariant.

    \item The stabilisation maps $\Sigma: \C(Q)\to \C(Q\times I)$ and $\Sigma: \H(Q)\to \H(Q\times I)$ are \emph{anti}-equivariant with respect to $\iota_{\C}$ and $\iota_{\H}$, respectively (cf. \cite[App. I, Lem.]{HatcherSSeq} and \cite[Lem. C.1]{MElongknots}). In particular, the maps $\H(Q\times I^n)\hookrightarrow \cH(Q)\simeq \Omega^{\infty}\Hsp(Q)$ and $\CEmb(P\times I^n,M\times I^n)\hookrightarrow \sCEnopartial(P,M)\simeq \Omega^\infty\CEsp(P,M)$ are $(-1)^n$-equivariant on homotopy groups when equipping all of the spaces with $\iota_{\H}$.\label{C2EquivarianceConcStabilisation}
\end{enumerate}

The second step in understanding $\iota_{\H}$ on $\CEsp(P,M)$ is to compare the homotopy involution on $\sCEnopartial(P,M)$ with an involution on the relative algebraic $K$-theory space on the right-hand side of \eqref{CEWJReq}. This comparison was carried out by Vogell \cite{VogellInvolution}, who introduced \textit{the canonical involution} $\tau_\epsilon$ on $\smash{\Omega^\infty \mathrm{K}(\Sp_X)\simeq A(X)}$, natural in $X$, which moreover descends to the smooth Whitehead space $\smash{\Wh(X)}=\Omega^\infty\hocofib\big(\Sigma^\infty_+ X \xrightarrow{\nu} \mathrm{K}(\Sp_X)\big)$. One might then hope to compare this canonical involution with involutions on the trace invariants featuring in \eqref{TraceInvariantsDiagram}, so to make the isomorphism \eqref{CEhtpyFLShomologyeq} $C_2$-equivariant with respect to $\iota_{\H}$ on the left-hand side, and an appropriate involution on the free loop space homology. This is indeed what happens.

\begin{prop}\label{involutionlemma}
Let $\iota: P\hookrightarrow M$ be such that $p\leq d-3$, so that the isomorphism
\[
    \pi_*^{\QQ}(\CEsp(P,M)) \cong \frac{\H^{S^1}_{*+1}(LM, L(M-P); \QQ)}{\H_{*+2}(M,M-P;\QQ)}
    \]
    in \eqref{CEhtpyFLShomologyeq} holds. Complex conjugation on $S^1\subset\CC$ induces a natural involution $\tau_{\FLS}$ on the free loop space $LX$, and hence on the homology groups $\H_*(LX)$ and $\H_*^{S^1}(LX)$. This involution descends to the quotient on the right-hand side of the above isomorphism. If $M$ is connected\footnote{Otherwise, deal with each component separately.}, then this isomorphism is $C_2$-equivariant for $\iota_{H}$ on the left and $(-1)^{d+1}\tau_{\FLS}$ on the right. Hence, it follows that
    \[
    \pi_*^{\QQ}(\pEmb(P,M))\cong \left[\frac{\H^{S^1}_{*+1}(LM, L(M-P); \QQ)}{\H_{*+2}(M,M-P;\QQ)}\right]_{(-1)^{d+1}\tau_{\FLS}} \quad \text{for $*\leq 2d-p-6$.}
    \]
\end{prop}

\begin{proof}
First note the free loop space $LX$ is naturally an $\O(2)$-space, and that $\tau_{\FLS}$ on $\H_*^{S^1}(LX)$ is induced by the residual $\O(2)/S^1$-action on the Borel construction $(LX)_{hS^1}$. Now, for a $2$-connected map $X\to Y$ of connected, based spaces, Goodwillie's isomorphism \eqref{DGMdiagram} for the ring map $\bfS[\Omega X]\to\bfS[\Omega Y]$ becomes
\begin{equation}\label{GoodwillieIso}
    \begin{tikzcd}[column sep = 23pt, scale = 0.95]
\pi_*^\QQ A(X,Y)\rar["\mathrm{Trc}^-", "\cong"'] & \HC^-_*(\bfS[\Omega X],\bfS[\Omega Y];\QQ)&\lar["N"', "\cong"] \HC_{*-1}(\bfS[\Omega X], \bfS[\Omega Y];\QQ)\rar["\cong"] & \H^{S^1}_{*-1}(LX,LY;\QQ).
\end{tikzcd}
\end{equation}
This isomorphism can be made $C_2$-equivariant as follows: given a (simplicial) ring $R$ with anti-involution (e.g. $R=\QQ[\Omega X]\simeq_{\QQ}\bfS[\Omega X]$ with the flip $C_2$-action), the ordinary and negative cyclic homologies of $R$, $\HC_*(R)$ and $\HC^-_*(R)$, inherit canonical involutions (cf. \cite[\S 5.2]{LodayHC} and \cite[p. 177]{LodderInvolution}). By \cite[Lem. 1.8]{LodderInvolution}, the trace map $\mathrm{Trc}^-$ of \eqref{GoodwillieIso} is equivariant with respect to the canonical involutions on both sides, while the map $N: \HC_{*-1}(R)\to \HC_*^-(R)$ is anti-equivariant \cite[p. 195]{LodderInvolution} (i.e. equivariant up to a minus sign)---alternatively, this is explained by the fact that $\Sigma\HC(\bfS[\Omega X])$ on the domain of the $S^1$-norm map should really be $\mathrm{H}\ZZ\otimes (S^{\operatorname{ad}S^1}\otimes \Sigma^\infty_+LX)_{hS^1}$, and the residual $C_2=\O(2)/S^1$-action on the adjoint representation is that of the sign representation. Finally, the isomorphism $\HC_{*}(\Omega X)\cong \H^{S^1}_*(LX;\QQ)$ is equivariant with respect to the canonical involution on $\HC$ and $\tau_{\FLS}$ on the right; this follows by a generalisation of \cite[Cor. 7.3.14]{LodayHC} and the results in Section 7 loc.cit., using the work of Dunn \cite{Dunn1989}. Thus, the canonical involution on $\HC^-_*(\bfS[\Omega X],\bfS[\Omega Y];\QQ)$ agrees with the FLS involution $\tau_{\FLS}$ on the rational homotopy groups of $\left(\Sigma^\infty(LM/L(M-P))\right)^{hS^1}$. The map $c$ of \eqref{ShiftHC-Map} is equivariant for the trivial $C_2$-action on the domain and $\tau_{\FLS}$ on the codomain, so it follows that $\tau_{\FLS}$ descends to the quotient in the isomorphism \eqref{CEhtpyFLShomologyeq}.

To conclude, we recall that for each spherical fibration $\xi$ over $X$, Vogell defined a natural involution $\tau_{\xi}$ on the $A$-theory space $A(X)$ (which gives $\tau_\epsilon$ when $\xi=\epsilon$ is the trivial $0$-dimensional fibration), with the property that when $X=M$ is a smooth $d$-manifold and $\xi$ is the spherical fibration associated to the $1$-stabilised tangent bundle $\tau_M\oplus \epsilon^1$, then the connecting map $(\Omega A(M), \tau_{\xi})\twoheadrightarrow (\cH(M),\iota_H)$ of the Waldhausen--Jahren--Rognes fibre sequence is $C_2$-equivariant on homotopy groups. But by \cite[Lem.~3.2]{InvolutionAtheoryBustamanteFarrell}, $\tau_\xi=(-1)^{d}\tau_{\epsilon}$ on rational homotopy groups $\pi_*^{\QQ}(A(M))$. It follows that the connecting homomorphism
\[
(\pi_{*+2}^{\QQ}(A(M,M-P)),(-1)^{d}\tau_\epsilon)\twoheadrightarrow (\pi_*^{\QQ}(\sCE(P,M)),\iota_H)\cong (\pi_*^{\QQ}(\CEsp(P,M)),\iota_H)
\]
is $C_2$-equivariant. The anti-equivariance of the norm map $N$ in \eqref{GoodwillieIso} turns $(-1)^d\tau_{\epsilon}$ into $(-1)^{d+1}\tau_{\FLS}$, establishing the first claim. The addendum now follows from the fact that the map $\Phi^{\Emb}_{(P,M)}$ in \eqref{WWCEMap} is $(2d-p-5)$-connected by \eqref{CEmbRangeBound}, so
\[\pi_*^{\QQ}(\pEmb(P,M))\cong \pi_*^{\QQ}\left(\CEsp(P,M)_{hC_2}\right)\cong \left[\pi_*^{\QQ}(\CEsp(P,M))\right]_{\iota_{\H}} \quad \text{for $*\leq 2d-p-6$.}\]\end{proof}

\subsubsection{Pseudoisotopy self-embeddings of $X_g$}\label{PlanpEmbXgSection}
Let us specialise to the case where $\iota:P\hookrightarrow M$ is a 
fixed self-embedding $\iota: X_g\hookrightarrow X_g$ isotopic to the identity, fixed on $\partial_0 (X_g)$ and with a fixed identification $S^1\times D^{2n-1}\cong X_g-\nu(\iota(X_g))$ rel $\partial_1(X_g)=S^1\times D^{2n-2}_-$. In this case, the handle dimension of $X_g$ relative to $\partial_0 (X_g)$ is $p=n$, which satisfies the handle codimension condition $n=p\leq d-3=2n-3$ as $n\geq 3$. In the last sections, we have explained how to understand the (rational) homotopy type of $\pEmb(X_g)\coloneq\hofib_\iota(\Embo(X_g)\to \bEmbo(X_g))$, but we are really after that of $\bEmbmodEmb(X_g)$. However, by Hudson's theorem \cite{HudsonThm}, the map $\Embo(X_g)\to \bEmbo(X_g)$ is an isomorphism on $\pi_0$, and hence 
\begin{equation}\label{pEmbVsbEmbModEmbEq}
\Omega \bEmbmodEmb(X_g)\simeq \pEmb(X_g).
\end{equation}
Thus, by the conclusion of \cref{involutionlemma}, we obtain
\begin{equation}\label{XgPseudoisotopyVsFLS}
\pi_*^{\QQ}\left(\bEmbmodEmb(X_g)\right)\cong \left[\frac{\H^{S^1}_{*}(LX_g, LS^1); \QQ)}{\H_{*+1}(X_g,S^1;\QQ)}\right]_{-\tau_{\FLS}} \quad \text{for $*\leq 3n-5$.}
\end{equation}
We will need to carry out the following tasks related to the isomorphism above:
\begin{itemize}
    \item[\S\ref{FLShomologySection}.] Compute the right-hand side of \eqref{XgPseudoisotopyVsFLS} (cf. \cref{FLSS1homology}(iii)).

    \item[\S\ref{mcgXgsection}.] Trace the action of the mapping class group $\Embmcgsfr$ and the reflection involution $\rho$ of \cref{ReflectionInvSection} on the homotopy groups of $\bEmbmodEmb(X_g)$ through the isomorphism \eqref{XgPseudoisotopyVsFLS} (cf. \cref{MCGrepBlockModEmb}).

    \item[\S\ref{FrobeniusXgSection}.] Understand the Frobenius homomorphisms of \eqref{FrobeniusBEmbSfr} on the left-hand side of \eqref{XgPseudoisotopyVsFLS}.
    
\end{itemize}

\subsection{The \texorpdfstring{$S^1$}{S1}-equivariant homology of \texorpdfstring{$LX_g$}{LXg}} \label{FLShomologySection}

This section is devoted to computing the right-hand side of \eqref{XgPseudoisotopyVsFLS} in degrees $* \leq 2n-2$. For use in the next section, it is convenient to carry out this computation as a representation of $\hMCGXg \coloneq \pi_0(\Aut_{*}(X_g))$. Since $n \geq 3$, we have an identification $\pi_0(\Aut_*(X_g)) \cong \pi_0(\Aut_*(X_g, S^1))$, and hence $\hMCGXg$ acts on $\smash{\H^{S^1}_*(LX_g, LS^1;\QQ)}$ and $\H_{*+1}(X_g, S^1;\QQ)$ by naturality on the pair $(X_g, S^1)$. We begin by introducing some notation.

\begin{nota}\label{NotationLXgSection}
    \begin{enumerate}[label = (\roman*)]
        \item All throughout this section and the following, $\otimes$ and $\otimes_\pi$ stand for $\otimes_{\QQ}$ and $\otimes_{\QQ[\pi]}$, respectively, and $[A]^{\pi}$ (resp. $[A]_\pi$) denotes the $\pi$-(co)invariants of a $\QQ[\pi]$-module $A$. 

        \item \label{HXgNotation}The main $\hMCGXg$-representation we will need to consider is
\[
H_{X_g}\coloneq\pix\otimes\QQ\cong \H_n(\widetilde{X}_g;\QQ),
\]
as introduced in \cref{allTheNotationUngraded}. The fundamental group $\pi_1 (X_g) = \pi \cong \ZZ\langle t\rangle$ is a normal subgroup of $\hMCGXg$. The generator $t$ is represented by the homotopy automorphism of $X_g \simeq S^1 \vee \bigvee^{2g} S^n$ that fixes the $S^1$–summand and wraps each of the $S^n$–summands once around $S^1$. Equivalently, it corresponds to the subgroup $\{+1\}\ltimes \pi \trianglelefteq \{\pm1\}\ltimes \GL_{2g}(\ZZ[\pi]) \cong \hMCGXg$, consisting of diagonal matrices whose diagonal entries are all equal to $t^{i}$ for some $i\in \ZZ$. The restricted $\QQ[\pi]$-module structure on $H_{X_g}$ is visibly induced by the usual action of $\pi_1 (X_g)$ on its higher homotopy groups, and as $\QQ[\pi]$-modules,
\[
H_{X_g}\cong \QQ[\pi]\otimes H_g,\quad \text{where $H_g\coloneq\H_n(W_{g,1};\QQ)$.}
\]
We will also regard $H_g$ as a $\QQ[\hMCGXg]$-module via the evident isomorphism $H_g\cong [H_{X_g}]_{\pi}$.

\item Given a $\QQ$-vector space $V$ and $i,k\in \ZZ_{\geq 0}$, we will write $S^k_{\scriptscriptstyle (i)}(V)$ for the $\QQ$-vector space given by the $ki$-graded piece of the free graded commutative $\QQ$-algebra on $V$, when $V$ is seen as a graded $\QQ$-vector space concentrated in degree $i$. For instance, we have 
\[
S^*_{\scriptscriptstyle (i)}(V)= \left\{\begin{array}{cl}
    \mathrm{Sym}^*(V) & \text{if $i$ is even,} \\
    \Lambda^*(V) & \text{if $i$ is odd,}
\end{array}\right.
\]
where $\mathrm{Sym}(W)$ and $\Lambda(W)$ denote the symmetric and exterior algebras of a $\QQ$-vector space $W$.
    \end{enumerate}
\end{nota}

We begin by computing the ordinary rational homology of the free loop space $LX_g$, which decomposes as $LX_g = \coprod_{r\in\ZZ} L_r X_g$, where $L_r X_g$ is the component of unbased loops representing $t^r \in \pi_1 X_g \cong \ZZ\langle t\rangle$.

\begin{prop}\label{FLShomology} Let $H_{X_g}$ denote the $\QQ[\hMCGXg]$-module introduced in \cref{NotationLXgSection}\ref{HXgNotation}.
    \begin{enumerate}
        \item[(i)] For $r\neq 0$, consider the $\QQ[\hMCGXg]$-modules
        $$
A_r\coloneq \frac{H_{X_g}}{\langle 1-t^r\rangle}, \qquad B_r\coloneq\frac{\sn(H_{X_g})}{\sn(1-t^r)}.
        $$
The rational homology of $L_{r}X_g$ in degrees $*\leq 2n-2$ is given, as a $\QQ[\hMCGXg]$-module, by
        $$
\H_*(L_{r}X_g;\QQ)\cong \left\{\begin{array}{cl}
    \QQ & *=0, 1,  \\[3pt]
    [A_r]_\pi\cong H_g & *=n-1, \\[3pt]
    [A_r]^{\pi}\cong H_g\langle 1+t+\dots+t^{r-1}\rangle & *=n,\\[3pt]
    {\left[\sn(A_r)\right]}_{\pi}\oplus [B_r]_\pi & *=2n-2, \\[3pt]
    0 & \text{otherwise}.
\end{array}
\right.
        $$

        \item[(ii)] The rational homology of $L_{0}X_g$ in degrees $*\leq 2n-2$ is given, as a $\QQ[\hMCGXg]$-module, by
        \[
\H_*(L_{0}X_g;\QQ)\cong \left\{\begin{array}{cl}
    \QQ & *=0, 1,  \\
    {[H_{X_g}]}_{\pi}\cong H_g & *=n-1, n, \\
   C_0 & *=2n-2, \\
    0 & \text{otherwise},
\end{array}
\right.
        \]
        where $C_0$ is some quotient of the $\QQ[\hMCGXg]$-module $[H_{X_g}^{\otimes 2}]_{\pi}$.
    \end{enumerate}
\end{prop}

\begin{proof}
    Let us first deal with the case $r\neq 0$. Consider the fibre sequence
    $$
\Omega_r X_g\longrightarrow L_rX_g\longrightarrow X_g,
    $$
    where $\Omega_r X_g$ is the path-component of the based loop space $\Omega X_g$ corresponding to $t^r$ under $\pi_0(\Omega X_g)\cong \ZZ\langle t\rangle$. Let us also write $t$ for a based loop in $X_g$ representing $t$ in $\pi$. One verifies that the connecting map of this fibration is, under the equivalence $t^r_\#: \Omega_r X_g\simeq \Omega_0 X_g$ and up to a sign, the Whitehead bracket
    $$
[t^r,-]: \pi_* X_g\longrightarrow \pi_* X_g, \quad x\longmapsto (1-t^r)\cdot x= x-t^r\cdot x.
    $$
Moreover, since the universal cover of $X_g$ is a wedge of spheres $\bigvee_{\ZZ}\bigvee_{i=1}^{2g}S^n$, the rational homotopy groups of $X_g$ in degrees $*\geq 2$ are given, as a graded $\QQ$-vector space, by the suspension of the free graded Lie $\QQ$-algebra $\mathrm{Lie}^*(H_{X_g})$ of $H_{X_g}$, seen as concentrated in degree $n-1$. It follows that the fundamental group of $L_r X_g$ is $\pi$, and that, as graded $\QQ$-vector space,
$$
\pi_{*\geq 2}^{\QQ} (L_rX_g)\cong A_r[n-1]\oplus B_r[2n-2]\oplus\ \text{($*\geq 3n-3$)-terms.}
$$
Here we are using the fact that $1-t^r\in \QQ[\pi]$ is a not a zero divisor if $r\neq 0$, that
$$
\mathrm{Lie}^2(1-t^r): \pi^\QQ_{2n-1} (X_g)\longrightarrow \pi^\QQ_{2n-1}(X_g), \quad [x,y]\mapsto [(1-t^r)x, (1-t^r)y]
$$
is injective by the same reason, and that $\Lie^2(H_{X_g}[n-1])\cong \sn(H_{X_g})$.

From this, and observing that the $(2n-2)$-th (rational) Postnikov invariant of the universal cover $\widetilde{L_rX_g}$ lies in a zero group, we deduce that the $(3n-4)$-th (rational) Postnikov trunction of $\smash{\widetilde{L_rX_g}}$ is
$$
\tau_{\leq 3n-4} \widetilde{L_rX_g}\simeq_\QQ K(A_r,n-1)\times K(B_r,2n-2).
$$
By the fibration sequence $\widetilde{L_rX_g}\to L_r X_g\to B\pi\simeq S^1$, the homology group $\H_{p+q}(L_r X_g;\QQ)$ is given by 
\[
\H_p(S^1; \underline{\H_q(\widetilde{L_rX_g};\QQ)})\cong\left\{\begin{array}{cl}
 \mathrm{coker}(1-t: \H_q(\widetilde{L_rX_g};\QQ)\to \H_q(\widetilde{L_rX_g};\QQ)) & p=0,  \\
  \ker(1-t: \H_q(\widetilde{L_rX_g};\QQ)\to \H_q(\widetilde{L_rX_g};\QQ)) & p=1\\
  0, & \textit{otherwise.}
\end{array}
\right.
\]
The claim in part (i) readily follows.

For $r=0$, consider the spectral sequence associated to the fibration $\Omega_0 X_g\to L_0 X_g\to X_g$, with $E^2$-page
$$
\H_p(X_g; \underline{\H_q(\Omega_0X_g;\QQ)})\cong\left\{
\begin{array}{cl}
    [\H_q(\Omega_0 X_g;\QQ)]_\pi & p=0, \\
    {[\H_q(\Omega_0 X_g;\QQ)]}^{\pi} & p=1,\\
    {[H_{X_g}\otimes \H_q(\Omega_0 X_g;\QQ)]}_\pi & p=n,\\
    0, & \textit{otherwise.}
\end{array}
\right.
$$
Since the map $\Omega\widetilde{X}_g\to \Omega_0 X_g$ is an equivalence and the homology of the domain is the free graded tensor coalgebra $T^c(H_{X_g})$, we need only consider the following two differentials in the previous spectral sequence:
\begin{align*}
\alpha=d^n: [H_{X_g}]_\pi &\longrightarrow [\H_{n-1}(\Omega_0 X_g;\QQ)]_\pi\cong [H_{X_g}]_\pi,\\
\beta=d^n:  [H_{X_g}^{\otimes 2}]_\pi &\longrightarrow [\H_{2n-2}(\Omega_0 X_g;\QQ)]_\pi\cong [H_{X_g}^{\otimes 2}]_\pi.
\end{align*}
\begin{claim}
    The map $\alpha$ is the zero map. 
\end{claim}
\begin{proof}[Proof of Claim]
    To see this, consider the map of fibre sequences
\begin{equation}\label{tildeXgFibSeqMap}
\begin{tikzcd}[row sep =15pt]
    \Omega \widetilde{X}_g\rar\dar["\simeq"] & L\widetilde{X}_g\rar\dar & \widetilde{X}_g\dar\\
    \Omega_0 X_g\rar &L_0 X_g\rar &X_g.
\end{tikzcd}
\end{equation}
which gives rise to a map of corresponding spectral sequences. We obtain a commutative square
$$
\begin{tikzcd}[row sep =15pt]
    H_{X_g}\dar[two heads]\rar["d^n"] & H_{X_g}\dar[two heads]\\
    {[H_{X_g}]}_\pi \rar["\alpha"] & {[H_{X_g}]}_\pi,
\end{tikzcd}
$$
where the top horizontal map is the corresponding differential in the spectral sequence associated to the top fibration of (\ref{tildeXgFibSeqMap}). But this fibration is obtained as the (homotopy) colimit of the fibrations
\begin{equation}\label{tildeXgNFibSeq}
\begin{tikzcd}
\Omega \widetilde{X}_g^{N}\rar & L \widetilde{X}_g^{N}\rar & \widetilde{X}_g^{N},
\end{tikzcd}
\end{equation}
where $\widetilde{X}_g^{ N}\simeq \vee_{|i|\leq N}\vee^{2g} S^n$. Thus, the differential $d^n: H_{X_g}\to H_{X_g}$ above is obtained as the colimit over $N\in \NN$ of the analogous differentials \(d^n_N: \bigoplus_{|i|\leq N} H_g\to \bigoplus_{|i|\leq N} H_g\) for the fibre sequences \eqref{tildeXgNFibSeq}. It suffices to show that each of the $d^n_N$'s are the zero map. This in turn follows from a dimension-counting argument involving the well-known cohomology of the free loop space of a finite wedge of $n$-spheres (a.k.a. the Hochschild homology of $\C^*(\widetilde{X}_g^N;\QQ)\simeq \H^*(\widetilde{X}_g^N;\QQ)$).
\end{proof}

To conclude, let us point out that all of the computations in this argument were of $\hMCGXg$-representations, and that the $\QQ[\hMCGXg]$-module $C_0$ in the statement is the cokernel of the map $\beta$. 
\end{proof}

We now deal with the $S^1$-equivariant homology of $LX_g$ and the right-hand side of \eqref{XgPseudoisotopyVsFLS}.

\begin{prop}\label{FLSS1homology}
Let $A_r$ and $B_r$ and $C_0$ be the grade $\QQ[\hMCGXg]$-modules of \cref{FLShomology}.
\begin{enumerate}
    \item[(i)] Let $r\neq 0$. As a $\QQ[\hMCGXg]$-module, the rational $S^1$-equivariant homology of $L_r X_g$ in degrees $*\leq 2n-2$ is 
    $$
\H_*^{S^1}(L_{r}X_g;\QQ)\cong \left\{\begin{array}{cl}
    \QQ & *=0,  \\[3pt]
    {[H_{X_g}]}_{\pi}\cong H_g & *=n-1, \\[3pt]
    {[\sn(A_r)]}_{\pi}\oplus [B_r]_\pi & *=2n-2, \\[3pt]
    0 & \text{otherwise}.
\end{array}
\right.
        $$

        \item[(ii)] As a $\QQ[\hMCGXg]$-module, the rational $S^1$-equivariant homology of $L_0 X_g$ in degrees $*\leq 2n-2$ is given by
        $$
\H_*^{S^1}(L_{0}X_g;\QQ)\cong \left\{\begin{array}{cl}
    \QQ & *=0,   \\
    {[H_{X_g}]}_{\pi}\cong H_g & *=n-1, \\
   C_0 & *=2n-2, \\
    0 & \text{otherwise}.
\end{array}
\right.
        $$

\item[(iii)] Therefore, for $*\leq 2n-2$, the $\QQ[\hMCGXg]$-module appearing on the right-hand side of \eqref{XgPseudoisotopyVsFLS} is given by
\[
\left[\frac{\H^{S^1}_{*}(LX_g, LS^1); \QQ)}{\H_{*+1}(X_g,S^1;\QQ)}\right]_{-\tau_{\FLS}}\cong \left\{\begin{array}{cl}
    \bigoplus_{r> 0} {[A_r]_{\pi}}\cong \bigoplus_{r> 0} H_g & *=n-1, \\[3pt]
     C \oplus \bigoplus_{r> 0}\left({\left[\sn(A_r)\right]}_{\pi}\oplus [B_r]_\pi\right) & *=2n-2, \\[3pt]
    0 & \text{otherwise},
\end{array}
\right.
    \]
    where $C$ denotes the coinvariants of $C_0$ under some involution.
\end{enumerate}
    \end{prop}

\begin{proof}
    For part (i), consider the homotopy orbit spectral sequence
    \begin{equation}\label{S1equivLrSS}
\H_*(\B S^1;\H_*(L_rX_g;\QQ))\implies \H_*^{S^1}(L_rX_g;\QQ).
    \end{equation}
    The $d^2$-differentials $\QQ\to \QQ$ going from bidegree $(2k,0)$ to $(2k-2,1)$ for $k>0$ are all isomorphisms: this is the case for the spectral sequence for the $S^1$-equivariant homology of $L_r S^1\simeq S^1$, which is a retract of (\ref{S1equivLrSS}). 

    By considering the cohomology version of (\ref{S1equivLrSS}) and exploiting its multiplicative structure, we also see that the $d^2$-differentials $H_g\to H_g$ from bidegree $(2k,n-1)$ to $(2k-2,n)$, for $k>0$, are isomorphisms too. Thus, (\ref{S1equivLrSS}) collapses on the $E^3$-page in degrees $*\leq 2n-2$, and the claim (i) follows immediately.

    A similar argument as for part (i) also holds for (ii). Notice that both computations are as $\QQ[\hMCGXg]$-modules.
    
    For part (iii), we first show that for $*\leq 2n-2$, we have an isomorphism of $\QQ[\hMCGXg]$-modules
    \begin{equation}\label{HS1XgcomputationProof}
    \frac{\H^{S^1}_{*}(LX_g, LS^1); \QQ)}{\H_{*+1}(X_g,S^1;\QQ)}\cong \left\{\begin{array}{cl}
    \bigoplus_{r\in \ZZ} {[A_r]_{\pi}}\cong \bigoplus_{r\in\ZZ} H_g & *=n-1, \\[3pt]
     C_0 \oplus \bigoplus_{r\neq 0}\left({[\sn(A_r)]}_{\pi}\oplus [B_r]_\pi\right) & *=2n-2, \\[3pt]
    0 & \text{otherwise},
\end{array}
\right.
    \end{equation}
    where we recall that the left-hand side above stands for the cokernel of the homomorphism 
    \begin{equation}\label{UnitMapHCXg}
    \begin{tikzcd}[column sep = 18pt, scale = 0.95]
        \H_{*+1}(X_g,S^1;\QQ)\rar[hook, "c"] & \HC^{-}_{*+1}(\bfS[\Omega X_g], \bfS[\Omega S^1];\QQ) & \lar["N"'{pos=0.4}, "\cong"{pos=0.4}] \HC_{*}(\bfS[\Omega X_g], \bfS[\Omega S^1];\QQ)\cong \H^{S^1}_*(LX_g,LS^1;\QQ)
    \end{tikzcd}
    \end{equation}
    induced by \eqref{ShiftHC-Map} on $\H_{*+1}(-;\QQ)$ for $\iota: P\hookrightarrow M$ our fixed self-embedding $\iota: X_g\hookrightarrow X_g$. It follows that
$$
\begin{tikzcd}
    \H_*(X_g,S^1;\QQ)\ar[equal, rr, bend left = 20pt]\rar[hook,"c"]\ar[hook,dr, "\eqref{UnitMapHCXg}"'] &\H_*(LX_g,LS^1;\QQ)\rar[two heads, "\mathrm{ev}_1"] & \H_*(X_g,S^1;\QQ)\\
    &\H^{S^1}_{*-1}(LX_g,LS^1;\QQ)\uar["B"]&
\end{tikzcd}
$$
commutes, where $B$ is the composition of the norm map $N$ with the natural transformation $\HC^-(-)\to \HH(-)$. This map $B$ fits in Connes' Periodicity exact sequence (see e.g. \cite[Thm. 2.2.1]{LodayHC})
\begin{equation}\label{ConnesExactSeq}
\begin{tikzcd}
    \H_{*+1}^{S^1}(LX_g, LS^1;\QQ)\rar & \H_{*-1}^{S^1}(LX_g,LS^1;\QQ)\rar["B"] &\H_*(LX_g, LS^1;\QQ)\rar["j"] &\H_*^{S^1}(LX_g,LS^1;\QQ),
\end{tikzcd}
\end{equation}
a.k.a. the Gysin exact sequence for the $S^1$-fibration $(LX_g,LS^1)\to (LX_g,LS^1)_{hS^1}$ (cf. \cite[Thm. 7.2.3(b)]{LodayHC}). Thus, by the computations in \cref{FLShomology} and parts (i) and (ii) of this proof, we see that $B$ is an isomorphism in degrees $*\leq n$. The isomorphism \eqref{HS1XgcomputationProof} now follows from these observations and the fact that $c: (X_g, S^1)\to (LX_g, LS^1)$ only hits the degree-zero component $(L_0 X_g, L_0 S^1)$, so, by dimension-counting, it must induce an isomorphism in $\H_n$ onto this component. Observe also that $\smash{\H_*^{S^1}(LS^1;\QQ)}$ is supported in degree $0$ and, in this degree, there is an isomorphism $\smash{\H_0^{S^1}(LX_g;\QQ)\cong \H_0^{S^1}(LS^1;\QQ)}$, so $\smash{\H_0^{S^1}(LX_g,LS^1;\QQ)=0}$.

Also note that a consequence of our previously mentioned computations and the exact sequence \eqref{ConnesExactSeq} is that the map $j:\H_*(LX_g, LS^1;\QQ)\to\H_*^{S^1}(LX_g,LS^1;\QQ)$ is surjective for $*\leq 2n-1$ and moreover, if $*\leq 2n-2$, then it is an isomorphism precisely when the codomain is non-zero. Thus, since $j$ is clearly equivariant for $\tau_{\FLS}$, it suffices to understand the effect of this involution on $\H_*(LX_g, LS^1;\QQ)$. On $LX_g$, $\tau_{\FLS}$ restricts to isomorphisms $L_r X_g\cong L_{-r}X_g$ for all $r$, and thus it identifies the $r$ and $-r$ summands, for $r\neq 0$, and restricts to some involution on $C_0$. This establishes (iii).
   \end{proof}

\subsection{The mapping class group action on pseudoisotopy self-embeddings of \texorpdfstring{$X_g$}{Xg}}\label{mcgXgsection}
From the fibre sequence \eqref{BEmbsfrFibreSeq}, there is an action of $\Embmcgsfr=\pi_1(\BEmbsfr(\xg),\ell)$ on the homotopy groups of $\bEmbmodEmb(\xg)$. That fibre sequence  is $\brho$-equivariant, where $\rho$ is the reflection involution of \cref{ReflectionInvSection}, so it follows that this action upgrades to one of the semidirect product $\mcgRhoXgsfr$ by $(\gamma,\rho^i)\cdot \sigma\coloneq\gamma\cdot (\rho^i\cdot \sigma)$. In this section, we trace this group action through the isomorphism \eqref{XgPseudoisotopyVsFLS}, the right-hand side of which we have just computed in \cref{FLSS1homology}(iii). Every embedding or automorphism space appearing in this section naturally carries a pointed $\brho$-action induced by conjugation by $\rho$. We will often leave this implicit.

We first get rid of the stable framings by considering the map of fibre sequences
\begin{equation}\label{FibSeqMapsfrtonotsfr}
\begin{tikzcd}[row sep =15pt]
       \bEmbmodEmb(X_g) \rar\dar[equal] & \BEmbsfr(X_g)_\ell\dar\rar &\BbEmbsfr(X_g)_\ell\dar\\
       \bEmbmodEmb(X_g) \rar & \BEmbcong(X_g)\rar &\BbEmbcong(X_g),
    \end{tikzcd}
\end{equation}
where the right square is cartesian. Hence, the action of $\Embmcgsfr$ on the homotopy groups of $\bEmbmodEmb(X_g)$ factors through $\Embmcgsfr\to \Embmcgnofr\coloneq\pi_1(\BEmbcong(\xg))$, and thus we only need to understand the action of $\mcgRhoXg$ arising from the bottom $\brho$-equivariant fibre sequence. 

However, note that the space we understand the homotopy type of is not $\bEmbmodEmb(X_g)$, but rather its loop space $\pEmb(X_g)$ (cf. \eqref{pEmbVsbEmbModEmbEq}). The homotopy groups of the latter space admit a natural $\mcgRhoXg$-action.

\begin{cons}\label{bEmbmodEmbAction}
Consider the fibre sequence $\pEmb(X_g)\xrightarrow{i}\Embcong(X_g)\xrightarrow{q}\bEmbcong(X_g)$. Given a mapping class $[\phi]\in \Embmcgnofr$ and a homotopy class $[\sigma]\in \pi_k(\pEmb(X_g))$, these can be represented, respectively, by a \textit{diffeomorphism} $\phi$ of $X_g$ and a map of pairs $\sigma: (D^{k},\partial D^{k})\to (\pEmb(X_g),*)$. Write $\psi$ for the diffeomorphism $\phi\rho^i$. It is clear that $q$ and pointwise conjugation by $\psi$, denoted $\psi_\#$, commute, and thus $q(\psi_\#(i\sigma))=\psi_\#(qi\sigma)$. But since $qi$ is null-homotopic (rel basepoints), it follows that $q(\psi_\#(i\sigma))$ is null-homotopic (rel $\partial D^{k}$). Using this null-homotopy and the fact that $i$ and $q$ give a fibre sequence, it follows that $\psi_\#(i\sigma)$ lifts, uniquely up to homotopy, to some map $\psi_\#\sigma: (D^{k},\partial D^{k})\to (\pEmb(X_g),*)$. Set \[([\phi],\rho^i)\cdot[\sigma]\coloneq[\psi_\#\sigma]=[(\phi\rho^i)_\#\sigma].\]
\end{cons}
\begin{rmk}\label{WellDefinedActionRmk} One verifies that the rule defined in \cref{bEmbmodEmbAction} does not depend on the representatives $\phi$ and $\sigma$, and so indeed gives a $\mcgRhoXg$-action on $\pi_k(\pEmb(X_g))$. We assumed $\phi$ to be a diffeomorphism for simplicity, but in general the homotopy class $[(\phi\rho^i)_{\#}\sigma]$ is well-defined if $\phi$ is just an embedding, provided a choice of homotopy right inverse---that is, an embedding $\phi^{-1}$ together with an isotopy $\gamma$ from $\phi \circ \phi^{-1}$ to the identity $\Id_{X_g}$. Such inverses exist by assumption, and the space of such pairs $(\phi^{-1}, \gamma)$ is contractible\footnote{Indeed, given a unital group-like topological monoid $M$ and an element $m \in M$, the map $m \cdot - : M \to M$ is a homotopy equivalence. The space of homotopy right inverses to $m$ is the homotopy fibre of this map at the identity element, which is contractible since $m \cdot -$ is an equivalence.}.
\end{rmk}

\begin{lemma}\label{pEmbvsbEmbmodEmbMCGActionXg}
For every $*\geq 0$, the isomorphism $\pi_*(\pEmb(X_g))\cong \pi_{*+1}\big(\bEmbmodEmb(X_g)\big)$ induced by the equivalence $\pEmb(X_g)\simeq \Omega \bEmbmodEmb(X_g)$ of \eqref{pEmbVsbEmbModEmbEq} is $\mcgRhoXg$-equivariant for the action on the left-hand side described in \cref{bEmbmodEmbAction}, and the usual action on the right-hand side arising from the bottom $\brho$-equivariant fibre sequence in \eqref{FibSeqMapsfrtonotsfr}.
\end{lemma}

\begin{proof}
    The equivalence mentioned in the statement arises from the map of fibre sequences
    \begin{equation}\label{etaDiagram}
    \begin{tikzcd}[row sep =15pt]        \pEmb(X_g)\dar["\eqref{pEmbVsbEmbModEmbEq}","\vsim"'] \rar["i"] &  \Embcong(X_g)\rar["q"] \dar["\eta", "\vsim"']& \bEmbcong(X_g)\dar["\eta", "\vsim"']\\
    \Omega\bEmbmodEmb(X_g)\rar["i"] & \Omega \BEmbcong(X_g)\rar["q"] &\Omega \BbEmbcong(X_g),
    \end{tikzcd}
    \end{equation}
    where $\eta$ stands for the canonical $\EE_1$-equivalence $\eta: M\xrightarrow{\sim}\Omega\B M$ associated to any unital group-like topological monoid $M$. We need an alternative description of the $\mcgRhoXg$-action on the homotopy groups of the left bottom term that resembles that of \cref{bEmbmodEmbAction}.
    
    For $F\xrightarrow{i} E\xrightarrow{q} B$ a fibre sequence of based spaces, the usual action of $\pi_1(E)$ on $\pi_{*\geq 1}(F)$ can be recovered from the looped fibre sequence $\Omega F\to \Omega E\to \Omega B$ as follows: given homotopy classes $[\gamma]\in \pi_1(E)$ and $[\sigma]\in \pi_{k}(\Omega F)$, these can be represented by a loop $\gamma\in \Omega E$ and a map $\sigma: (D^{k},\partial D^{k})\to (\Omega F, *)$. Applying pointwise concatenation of loops, we obtain a map $\gamma_\#(i\sigma)\coloneq\gamma\cdot i\sigma\cdot \overline{\gamma}: D^{k}\to \Omega E$. Note that on the boundary $\partial D^{k}$, $\gamma_{\#}(i\sigma)$ takes the value $\gamma \cdot *\cdot \overline\gamma$, which is canonically null-homotopic. Thus, we obtain a map $\widetilde{\gamma}_\# (i\sigma): (D^{k},\partial D^{k})\to (\Omega E,*)$. But $q(\widetilde{\gamma}_\# (i\sigma))=\widetilde{q(\gamma)}_\#(qi\sigma)$, where the latter is obtained from $q(\gamma)_\#(qi\sigma)$ as before, and since $qi$ is null-homotopic, then $q(\widetilde{\gamma}_\#(i\sigma))$ is null-homotopic (rel $\partial D^{k}$). It follows that $\widetilde{\gamma}_\# (i\sigma)$ lifts, up to homotopy, to some map $\widetilde{\gamma}_\# \sigma: (D^{k},\partial D^{k})\to (\Omega F, *)$, and thus we set $[\gamma]\cdot [\sigma]\coloneq[\widetilde{\gamma}_\# \sigma]\in \pi_{k}(\Omega F)$. One readily verifies that this action of $\pi_1(E)$ on $\pi_k(\Omega F)$ recovers the standard action of $\pi_1(E)$ on $\pi_{k+1}(F)$ under the natural isomorphism $\pi_k(\Omega F)\cong \pi_{k+1}(F)$, and is therefore well-defined. Moreover, if $G$ is a discrete group and the fibre sequence $F\to E\to B$ is $G$-equivariant, then setting $([\gamma],g)\cdot [\sigma]\coloneq[\gamma]\cdot [g\cdot \sigma]$ defines an action of $\pi_1(E)\rtimes G$ on $\pi_{k}(\Omega F)$ which recovers the one just described for the fibre sequence $F\times_G EG\to E\times_G EG\to B\times_G EG$ obtained from $F\to E\to B$ by applying the Borel construction.

    Taking $F\to E\to B$ to be the bottom fibre sequence in \eqref{FibSeqMapsfrtonotsfr}, it is clear that the description just given of the $\Embmcgnofr$-action on the homotopy groups of $\bEmbmodEmb(X_g)$ differs from that of \cref{bEmbmodEmbAction} only in that conjugation by an embedding $\phi$ is replaced with conjugation by the loop $\eta(\phi)$, and hence both agree under the equivalence \eqref{pEmbVsbEmbModEmbEq} (for $\eta: M\xrightarrow{\sim}\Omega\B M$ identifies the multiplication on $M$ with loop concatenation on $\Omega\B M$). The $\brho$-equivariance follows since all the maps involved in \eqref{etaDiagram} are $\brho$-equivariant.
\end{proof}

So we only need to trace the $\mcgRhoXg$–action on $\pi_*(\pEmb(X_g))$ from \cref{bEmbmodEmbAction} through the isomorphism \eqref{XgPseudoisotopyVsFLS}. Recall that this isomorphism arises from the $(3n-5)$-connected map \eqref{WWCEMap}
\begin{equation}\label{WWMapXg}
\Phi^{\Emb}_{X_g} \colon \pEmb(X_g) \longrightarrow \Omega^\infty(\CEspo(X_g)_{hC_2}),
\end{equation}
where $\CEspo(X_g)$ denotes the concordance embedding spectrum $\CEsp(P,M)$ associated to our fixed embedding $\iota\colon X_g \hookrightarrow X_g$ (we retain $\partial_0$ in the notation $\CEspo(X_g)$ to avoid confusion with the concordance diffeomorphism spectrum $\mathbf{C}(X_g)=\Sigma^{-1}\Hsp(X_g)$). Because of the somewhat opaque definition of this spectrum (cf.\ \eqref{CEspDefn1}), it is not immediate whether the homotopy groups of the target in \eqref{WWMapXg} carry a natural $\mcgRhoXg$–action. We will resolve this after rationalisation (indeed, inverting $2$ would suffice, but we rationalise for consistency).

\begin{cons}\label{CEmbAction} Given $([\phi],\rho^i)\in \Embmcgnofr$ and $[\sigma]\in \pi_k(\CEmbo(X_g))$, represent them, respectively, by a diffeomorphism $\phi$ of $X_g$ and a map of pairs $\sigma: (D^{k},\partial D^{k})\to (\CEmbo(X_g),*)$, and let $\psi\coloneq \phi\circ \rho^i$. For a concordance embedding $\varphi\in \CEmbo(X_g)$, let us write $\psi_\#\varphi\coloneq (\psi\times \mathrm{Id}_I) \circ \varphi\circ (\psi^{-1}\times \mathrm{Id}_I)\in \CEmbo(X_g)$. Then, if $\psi_\#\sigma: (D^{k},\partial D^k)\to (\CEmbo(X_g),*)$ is the map given by $(\psi_\#\sigma)(v)=\psi_\#(\sigma(v))$ for $v\in D^{k}$, we set 
\[([\phi], \rho^i)\cdot[\sigma]\coloneq[\psi_\#\sigma]=[\phi_{\#}\rho^i_{\#}\sigma]\in \pi_k(\CEmbo(X_g)).\]
This action is clearly compatible with the stabilisation maps $\Sigma: \CEmbo(X_g)\to \CEmbo(X_g\times I)$ of \eqref{ConcordanceStabilisationMapsDefn}, and hence there is an induced $\mcgRhoXg$-action on the homotopy groups of the stable space $\sCE(X_g)\coloneq \hocolim_{k}\CEmbo(X_g\times I^k)$. Moreover, this action is also visibly compatible with the $h$-cobordism involution $\iota_{\H}$ modelled in \cref{uCEmbDefn}(ii), and hence there is an induced $\mcgRhoXg$-action on the groups
\begin{equation}\label{TransferMCGAction}
\left[\pi_*^{\QQ}(\sCE(X_g))\right]_{\iota_{\H}}\overset{\eqref{StableCEmbIETSeq}}{\cong} \left[\pi_*^{\QQ}(\CEspo(X_g))\right]_{\iota_{\H}}\cong \pi_*^{\QQ}(\CEspo(X_g)_{hC_2}).
\end{equation}
\end{cons}

The following is one of the main results in this section, whose proof will be our next immediate goal. 

\begin{teo}\label{BoringAlphaProp}
    The Weiss--Williams map \eqref{WWMapXg}
    \[
    \pi_*^{\QQ}\big(\Phi^{\Emb}_{X_g}\big): \pi_*^{\QQ}\big(\pEmb(X_g)\big)\longrightarrow\pi_*^{\QQ}\left(\Omega^\infty(\CEspo(X_g)_{hC_2})\right)
    \]
    is $\mcgRhoXg$-equivariant on rational (in fact, $\ZZ[\tfrac{1}{2}]$-local) homotopy groups in degrees $*\leq 3n-6$, for the action from \cref{bEmbmodEmbAction} on the domain and the action from \cref{CEmbAction} on the target.
\end{teo}

\begin{rmk}
    As mentioned in the statement of \cref{BoringAlphaProp}, the conclusion also holds for $\ZZ[\tfrac{1}{2}]$-local homotopy groups. The only properties of the rationalisation hypothesis that we will use are:
    \begin{itemize}
        \item the natural map $\nu:(\Omega^\infty X)_{hC_2}\xrightarrow{\simeq_{\QQ}}\Omega^\infty(X_{hC_2})$ is a rational equivalence  for any $C_2$-spectrum $X$;
        \item the group $\pi_0\left(\C(S^1\times D^{2n-1})\right)$ is $2$-torsion for $2n\geq 6$, and so vanishes after rationalisation (cf. \eqref{HatcherWagoner}). 
    \end{itemize}
    The $\ZZ[\tfrac{1}{2}]$-local analogues of both of these facts clearly hold. In particular, a $\ZZ[\tfrac{1}{2}]$-local version of \cref{IntermediateBoringAlphaProp} also holds. For clarity of notation, and because this is the only case required later, we choose to work rationally throughout.
\end{rmk}

We first relate the $\mcgRhoXg$–actions on $\pEmb(X_g)$ and on $\CEmbo(X_g)$ from \cref{bEmbmodEmbAction} and \cref{CEmbAction}, respectively, via a map that is (at least initially) unrelated to the Weiss--Williams map \eqref{WWMapXg}. Consider the $\brho$-equivariant fibre sequence
$$
\begin{tikzcd}
    \Embo(X_g\times I)\rar[hook] & \CEmbo(X_g)\rar["\mathrm{res}_1"] &\Embop(X_g),
    \end{tikzcd}
$$
where $\opartial(X_g\times I)=\partial_0(X_g\times I)\coloneq X_g\times \{0,1\}\cup \partial_0(X_g)\times I$. Looping the base and total space, we obtain 
\begin{equation}
   \begin{tikzcd}
    \Omega\CEmbo(X_g)\rar["\mathrm{res}_1"] & \Omega\Embop(X_g)\rar["\Gamma"]&  \Embo(X_g\times I),
\end{tikzcd} 
\end{equation}
where $\Gamma$ is, by inspection, the map that sends a loop $\{\varphi_t\}_{t\in [0,1]}$ (based at the identity) of self-embeddings of $X_g$ to the embedding\footnote{For $\varphi$ to be an actual embedding, one needs $\Omega$ in the domain of $\Gamma$ to denote based \textit{smooth} loops—see, for instance, \cite[\S4]{MElongknots} for a proper model of the map $\Gamma$.} $\varphi: X_g \times I \xhookrightarrow{} X_g \times I$ given by $\varphi(x,t) = (\varphi_t(x), t)$. Note that $\Gamma\circ \mathrm{res}_1$ is indeed null-homotopic: replace the interval $I=[0,1]$ in the codomain of $\Gamma$ by $[-1,1]$, and model $\Gamma$ by the map sending $\{\varphi_t\}_{t\in [0,1]}$ to $\varphi(x,t)=(x,t)$ if $t\leq 0$ and $\varphi(x,t)=(\varphi_t(x),t)$ if $t\geq 0$. Now, given a loop of concordance embeddings $\{\Phi_t: X_g\times [0,1]\hookrightarrow X_g\times [0,1]\}_{t\in [0,1]}$ with $\{\varphi_t\}_{t\in [0,1]}=\res_1(\{\Phi_t\}_{t\in [0,1]})$, and writing $\mathrm{tr}_{r}: X_g\times \RR\to X_g\times \RR$ for the translation map $\mathrm{tr}_r(x,t)=(x,t+r)$, the nullhomotopy $\Gamma\circ \mathrm{res}_1\simeq *$ is 
\[
[0,1]\ni s\longmapsto \left((x,t)\mapsto \left\{
\begin{array}{cc}
    (x,t), & -1\leq t\leq s-1,\\
    \mathrm{tr}_{s-1}(\Phi_s(x,t-s+1)), & s-1\leq t\leq s,\\
    (\varphi_t(x),t), & s\leq t\leq 1.
\end{array}
\right.\right).
\]

With this in mind, observe that there is a map of $\brho$-equivariant fibre sequences
\begin{equation}\label{CEvsPseudoisotopyMap}
\begin{tikzcd}
\Omega\CEmbo(X_g)\rar["\mathrm{res}_1"]\ar[dd, "\alpha"] & \Omega\Embop(X_g)\ar[dd, equal]\rar["\Gamma"]&  \Embo(X_g\times I)\dar["\beta"]\\
&&\bEmbo(X_g\times I)\\
    \Omega\pEmb(X_g)\rar["i"] & \Omega \Embcong(X_g)\rar["\beta"] &\Omega \bEmbcong(X_g),\ar[u, "\widetilde{\Gamma}"', "\vsim"]
\end{tikzcd}
\end{equation}
where the block version of the graphing map $\Gamma$, denoted $\widetilde{\Gamma}$, is an equivalence (essentially by definition). The map $\alpha$ is the induced map on horizontal homotopy fibres of the right-hand commutative square.

\begin{lemma}\label{ActionCompatibilityProp}
    The map $\pi_{*}(\alpha): \pi_*(\Omega\CEmbo(X_g))\to \pi_*(\Omega\pEmb(X_g))$ is $\mcgRhoXg$-equivariant for the actions of Constructions \ref{CEmbAction} and \ref{bEmbmodEmbAction}, respectively. 
\end{lemma}

\begin{proof}
    This follows by a direct verification. Given $([\phi],\rho^{i})\in \mcgRhoXg$ and $[\sigma]\in \pi_{k-1}(\Omega\CEmbo(X_g))$, choose representatives $\phi$ and $\sigma\colon (D^{k},\partial D^{k})\to (\CEmbo(X_g),*)$ as in \cref{CEmbAction}, and set $\psi=\phi\circ \rho^{i}$. Pointwise conjugation by $\psi$, written $\psi_{\#}$, is well-defined in each space of the top fibre sequence in \eqref{CEvsPseudoisotopyMap}: on the base it is given by conjugation by $\psi\times \Id_I$, and on the total space and the fibre by conjugation by $\psi$ and by $\psi\times \Id_I$, respectively. Both $\mathrm{res}_1$ and $\Gamma$ commute with $\psi_{\#}$, hence $\Gamma(\psi_{\#}\mathrm{res}_1(\sigma))=\psi_{\#}\Gamma(\mathrm{res}_1(\sigma))$. Since $\Gamma\circ \mathrm{res}_1$ is null-homotopic, the map $\psi_{\#}\mathrm{res}_1(\sigma)$ lifts along this null-homotopy to a map $(D^{k},\partial D^{k})\to (\CEmbo(X_g),*)$, unique up to homotopy; this lift must be $\psi_{\#}(\sigma)$.

    Replacing the base space of the bottom fibre sequence in \eqref{CEvsPseudoisotopyMap} by $\bEmbo(X_g\times I)$ and replacing $b$ by $\widetilde{\Gamma}\circ \beta$, it is clear that the action described in \cref{CEmbAction} is compatible with that of \cref{bEmbmodEmbAction}, since the null-homotopy of $\widetilde{\Gamma}\circ i\circ \beta\circ \alpha = \beta\circ \Gamma\circ \mathrm{res}_1$ is compatible with the null-homotopy of $\Gamma\circ \mathrm{res}_1$.
    \end{proof}

\begin{rmk}
    Note that \cref{ActionCompatibilityProp} does not deal with the map $\pi_0(\CEmbo(X_g))\to \pi_0(\pEmb(X_g))$, but, luckyly, both of these groups are zero by Hudson's theorem \cite{HudsonThm}. 
\end{rmk}

As foreshadowed, the map $\alpha$ of \eqref{CEvsPseudoisotopyMap} is indeed closely related to the Weiss--Williams map \eqref{WWMapXg}.

\begin{prop}\label{IntermediateBoringAlphaProp}
   If $2n\geq 6$, the following diagram commutes on rational homotopy groups
   $$
\begin{tikzcd}
    \Omega \CEmbo(X_g)\dar[hook]\ar[rr,"\alpha"] && \Omega \pEmb(X_g)\ar[d,"\Phi^{\Emb}_{X_g}"]\\
\Omega\sCE(X_g)\rar[dash, "\eqref{StableCEmbIETSeq}", "\sim"']&\Omega^{\infty+1}\big(\CEspo(X_g)\big)\rar["q"] & \Omega^{\infty+1}\big(\CEspo(X_g)_{hC_2}\big),
\end{tikzcd}
   $$   
   where the left vertical map is the inclusion into the colimit \eqref{CEmbStabilisationInclusion}, and $q$ is induced by the natural quotient map $q: X\to X_{hC_2}$ for $X$ a $C_2$-spectrum. In fact, the diagram is homotopy commutative when looped once.
\end{prop}

\begin{proof}[Proof of \cref{BoringAlphaProp} given \cref{IntermediateBoringAlphaProp}]
    In the diagram of \cref{IntermediateBoringAlphaProp}, equip the homotopy groups of the top-right term with the $\mcgRhoXg$-action of \cref{bEmbmodEmbAction}, and those of the remaining corners with the actions described in \cref{CEmbAction}. We now verify which maps in this diagram are $\mcgRhoXg$-equivariant on rational homotopy groups. 

The map $\alpha$ is equivariant by \cref{ActionCompatibilityProp}. The stabilisation map $\CEmbo(X_g)\hookrightarrow \sCE(X_g)$ is equivariant by construction. The bottom horizontal composition is equivariant, since it fits into the commutative square
\[
\begin{tikzcd}
    \sCE(X_g)\dar["q"]\rar[dash, "\eqref{StableCEmbIETSeq}", "\sim"'] &
    \Omega^{\infty}\big(\CEspo(X_g)\big)\dar["q"] \\
    \sCE(X_g)_{hC_2}\rar["\nu", "\simeq_{\QQ}"'] &
    \Omega^{\infty}\big(\CEspo(X_g)_{hC_2}\big),
\end{tikzcd}
\]
where the left vertical map is equivariant because $\iota_{\H}$ is compatible with the $\mcgRhoXg$-action, and where $\nu$ is equivariant by construction, as the isomorphism \eqref{TransferMCGAction} transferring the action is induced by $\nu$.

Since $\alpha$ is surjective on rational homotopy groups in degrees $*\leq 3n-6$ (both vertical maps in the diagram of \cref{IntermediateBoringAlphaProp} are $(3n-5)$\nobreakdash–connected, and $q$ is surjective on homotopy groups), it follows that $\pi_*^{\QQ}(\Phi^{\Emb}_{X_g})$ is indeed $\mcgRhoXg$-equivariant for $*\leq 3n-6$.
\end{proof}

So our goal now is to prove \cref{IntermediateBoringAlphaProp}. To do so, we establish an analogue for concordance diffeomorphisms (\cref{BdQProp}), and obtain \cref{IntermediateBoringAlphaProp} by taking fibres. Consider the following map $\alpha$, obtained as the horizontal homotopy fibre map of the right subsquare in
\begin{equation}
\begin{tikzcd}
\Omega\C(Q)\rar["\mathrm{res}_1"]\ar[dd, "\alpha"] & \Omega\Diffb(Q)\ar[dd, equal]\rar["\Gamma"]&  \Diffb(Q\times I)\dar["\beta"]\\
&&\bDiffb(Q\times I)\\
    \Omega^2\left(\frac{\bDiff}{\Diff}\right)_{\partial}(Q)\rar["i"] & \Omega \Diffb(Q)\rar["\beta"] &\Omega \bDiffb(Q),\ar[u, "\widetilde{\Gamma}"', "\vsim"]
\end{tikzcd}
\end{equation}
where notation is completely analogous to that in \eqref{CEvsPseudoisotopyMap}. (Hence, for $Q=X_g$, this diagram maps compatibly to that in \eqref{CEvsPseudoisotopyMap}.) Consider also the composition
\[
\begin{tikzcd}
    j:\left(\frac{\bDiff}{\Diff}\right)_\partial(Q)\rar[hook] &\frac{\bDiff^b_\partial(Q\times\RR^\infty)}{\Diffb(Q)} & \lar["\sim"', hook'] \frac{\Diff^b_\partial(Q\times\RR^\infty)}{\Diffb(Q)}=\colon \Bd_Q^{(1)}(\RR^\infty),
\end{tikzcd}
\]
where the right equivalence is \cite[Thm. B]{WWI}, and $\Bd_Q$ is the orthogonal functor of \cref{BdAndEorthFunctorsDefn}(i). 

Finally, let $\cC(Q)\coloneq \hocolim_{k}\C(Q\times I^k)$ denote the space of \textit{stable} concordance diffeomorphisms of $Q$, where the maps in the defining filtered system are the concordance stabilisation maps of \eqref{ConcordanceStabilisationMapsDefn}. The equivalences $\Omega\H(Q)\xrightarrow{\sim} \C(Q)$ induced by the fibre sequence \eqref{OmegaHFibSeq} are compatible with the stabilisation maps, and hence induce a natural equivalence of stable spaces
\begin{equation}\label{cCvscH}
    \Omega \cH(Q)\overset{\simeq}{\longrightarrow} \cC(Q).
\end{equation}

The following is an analogue of \cref{IntermediateBoringAlphaProp} for concordance diffeomorphisms.

\begin{prop}\label{BdQProp}
There is a homotopy commutative square
   \[
\begin{tikzcd}[row sep = 15pt]
    \Omega \C(Q)\dar[hook, "\mathrm{stab}."]\ar[rr,"j\circ \alpha"] && \Omega^{2}\left(\frac{\Diff^b_\partial(Q\times \RR^\infty)}{\Diff_\partial(Q)}\right)\ar[dd,"\Phi^{\Bd_Q}_{\infty}"]\\    
\Omega\mathcal{C}(Q)\dar[dash, "\vsim", "\eqref{cCvscH}"'] &&\\
\Omega^2\cH(Q)\rar[dash,"\eqref{StableCEmbIETSeq}", "\sim"']&\Omega^{\infty+2}(\Hsp(Q))\rar["q"] & \Omega^{\infty+2}\big(\Hsp(Q)_{hC_2}\big),
\end{tikzcd}
   \]
   which is natural in codimension-zero embeddings. Here, $\Phi^{\Bd_Q}_\infty$ is the map \eqref{PhiFInftyMap} for $F=\Bd_Q$.
\end{prop}

\begin{proof}
    Given the technical nature of this proof, we defer it to \cref{AppendixBoringAlpha}.
\end{proof}

\begin{rmk}
    It is very likely that one can get rid of one of the loopings in the square of \cref{BdQProp} and, moreover, that such delooped square actually factors through another commutative square
    \[
    \begin{tikzcd}
    \C(Q)\dar[hook, "\mathrm{stab.}"]\ar[rr,"\B(\alpha)"] && \Omega\left(\frac{\bDiff}{\Diff}\right)_\partial(Q)\dar["\Phi^{\Diff}_Q"]\\
\mathcal{C}(Q)\rar[dash,"\eqref{StableCEmbIETSeq}", "\sim"']&\Omega^{\infty+1}(\Hsp^s(Q))\rar["q"] & \Omega^{\infty+1}\big(\Hsp^s(Q)_{hC_2}\big),
\end{tikzcd}
    \]
    where $\Phi^{\Diff}_Q$ is the Weiss--Williams map from \eqref{WWMap}. Both of these claims seem more involved to show and we will not need them, so we shall not pursue this direction.
\end{rmk}

\begin{proof}[Proof of \cref{IntermediateBoringAlphaProp}]
    We first argue that the homotopy fibre of the map between the natural commutative diagrams of \cref{BdQProp} induced by the embedding $(Q\hookrightarrow Q')=(S^1\times D^{2n-1}\hookrightarrow X_g)$ yields the one appearing in the statement of \cref{IntermediateBoringAlphaProp} looped once more. 

    The bottom horizontal map follows by \eqref{StableCEmbIETSeq} and the definition of $\CEspo(X_g)\coloneq \hofib(\Hsp(S^1\times D^{2n-1})\to \Hsp(X_g))$. The left vertical map follows by the usual isotopy extension fibre sequence \eqref{CEmbIETSeq} and its stable version (see \cref{IsotopyExtensionBdQRmk}).
    For the top horizontal map, there is a commutative diagram
    \[
    \begin{tikzcd}[row sep =15pt]
        \Omega^2\CEmbo(X_g)\dar["\alpha"] \rar& \Omega\C(S^1\times D^{2n-1})\dar["\alpha"]\rar & \Omega\C(X_g)\dar["\alpha"]\\
        \Omega^2\pEmb(X_g)\rar["i"]\dar[equal] & \Omega^2\left(\frac{\bDiff}{\Diff}\right)_\partial(S^1\times D^{2n-1})\dar["j"] \rar&\Omega^2\left(\frac{\bDiff}{\Diff}\right)_\partial(X_g)\dar["j"]\\
        \Omega^2\pEmb(X_g)\rar["j\circ i"]& \Omega^2\left(\frac{\Diff^b_\partial(S^1\times D^{2n-1}\times \RR^\infty)}{\Diffb(S^1\times D^{2n-1})}\right)\rar&\Omega^2\left(\frac{\Diff^b_\partial(X_g\times \RR^\infty)}{\Diffb(X_g)}\right),
    \end{tikzcd}
    \]
    where the top map is one of fibre sequences by the isotopy extension theorem, and the bottom right subsquare is cartesian as explained in \eqref{CartesianIETBounded}. Finally for the right vertical map, note from \cref{PhiEmbDefn} that the map $\smash{\Phi^{\Emb}_{X_g}}$ is defined to be $\smash{\Phi_{\infty}^{\sfE_{X_g}}}$ up to the equivalence \eqref{bEmbAsT0} in the domain (which is a consequence of the bottom right subsquare in the above diagram being cartesian). Since the maps $\smash{\Phi_\infty^F}$ of \eqref{PhiFInftyMap} preserve fibre sequences of orthogonal functors, the claim follows. 

    Thus, we have established that the diagram in the statement of \cref{IntermediateBoringAlphaProp} is homotopy commutative upon looping once, so the original diagram must be commutative on homotopy groups in degrees $*\geq 1$. It remains to argue the commutativity on $\pi_0^{\QQ}(-)$, i.e. that the diagram of rational vector spaces
    \begin{equation}\label{2n6homotopysquare}
\begin{tikzcd}
\pi_1^{\QQ}\left(\CEmbo(X_g)\right)\dar[hook]\ar[r,"\alpha_*"] & \pi_1^{\QQ}\left(\pEmb(X_g)\right)\ar[d,"\Phi^{\Emb}_*\coloneq (\Phi^{\Emb}_{X_g})_*"]\\
\pi_1^{\QQ}\left(\CEspo(X_g)\right)\rar[two heads, "q_*"] & \left[\pi_1^{\QQ}\left(\CEspo(X_g)\right)\right]_{\iota_{\H}}
\end{tikzcd}
    \end{equation}
    commutes. (Observe that since $\CEmbo(X_g)$ is an $\EE_1$-space, $\pi_1\left(\CEmbo(X_g)\right)$ is abelian and thus can indeed be rationalised.) If $2n>6$, this follows immediately from the fact that $\CEmbo(X_g)$ is $(n-2)$-connected, so we only need to deal with the case $2n=6$. To this end, consider the commutative square
\[
\begin{tikzcd}
\pi_1^{\QQ}\left(\C(X_g)\right)\ar[dd,hook]\ar[r,"\alpha_*"] & \pi_2^{\QQ}\left(\left(\frac{\bDiff}{\Diff}\right)_\partial(X_g)\right)\ar[d, "j_*"]\\
&\pi_2^{\QQ}\left(\frac{\Diff^b_\partial(X_g\times \RR^\infty)}{\Diff_\partial(X_g)}\right)
\ar[d,"\Phi^{\Diff}_*"]\\
\pi_2^{\QQ}\left(\Hsp(X_g)\right)\rar[two heads, "q_*"] & \left[\pi_2^{\QQ}\left(\Hsp(X_g)\right)\right]_{\iota_{\H}}
\end{tikzcd}
    \]
   obtained from the one in \cref{BdQProp} for $Q=X_g$ when taking $\pi_0(-)$ and rationalising appropriately. This square maps to \eqref{2n6homotopysquare} and we further claim that the group homomorphism
    \[
r_*:\pi_1^{\QQ}\left(\C(X_g)\right)\longrightarrow\pi_1^{\QQ}\left(\CEmbo(X_g)\right)
    \]
    is surjective. If so, then \eqref{2n6homotopysquare} is commutative by an elementary diagram chase. To see this, note that $r_*$ fits in a long exact sequence
    \[
    \begin{tikzcd}
        \pi_1^{\QQ}\left(\C(S^1\times D^{2n-1})\right)\rar&\pi_1^{\QQ}\left(\C(X_g)\right)\rar["r_*"] &\pi_1^{\QQ}\left(\CEmbo(X_g)\right)\rar &\pi_0^{\QQ}\left(\C(S^1\times D^{2n-1})\right)
    \end{tikzcd}
    \]
    and hence it suffices to show that $\pi_0^{\QQ}\left(\C(S^1\times D^{2n-1})\right)=0$. This in turn follows from the work of Hatcher--Wagoner \cite{HatcherWag} (with corrections by Igusa \cite[Thm. 8.a.1]{IgusaCorrection}) who, for $M$ a compact smooth manifold of dimension $\dim M\geq 6$, established an exact sequence of abelian groups 
   \begin{equation}\label{HatcherWagoner}
       \begin{tikzcd}
           \mathrm{Wh}_1^+(\pi_1 M; \pi_2 M)\oplus \mathrm{Wh}_1^+(\pi_1 M; \mathbf{F}_2)\rar & \pi_0 (\C(M))\rar & \mathrm{Wh}_2(\pi_1M)\rar &0.
       \end{tikzcd}
   \end{equation}
   For $M=S^1\times D^{2n-1}$ (so $\pi_1 M=\ZZ$ and $\pi_2 M=0$), the abelian groups on the left and right terms in \eqref{HatcherWagoner} are either zero or $2$-torsion. Thus, $\pi_0^{\QQ}(\C(S^1\times D^{2n-1}))=0$ (in fact, $\pi_0(\C(S^1\times D^{2n-1}))[\tfrac{1}{2}]=0$), as desired.
\end{proof}

Now that \cref{BoringAlphaProp} has been established, we need to understand the $\QQ[\mcgRhoXg]$-action of \cref{CEmbAction} on $\pi_*^{\QQ}(\sCE(X_g))$ (or rather, on its $\iota_{\H}$-coinvariants). For the remainder of the section, fix $([\phi],\rho^i)\in \mcgRhoXg$ and choose a representative diffeomorphism $\phi$ of $X_g$ that fixes a neighbourhood of $\partial X_g\cup(S^1\times D^{2n-1})$. Looping the fibre sequence \cref{StableCEmbIETSeq} and noting \eqref{cCvscH} yields a fibre sequence
\begin{equation}\label{rSurjectionFibSeq}
\begin{tikzcd}
    \mathcal{C}(S^1\times D^{2n-1})\rar["i"] &\mathcal{C}(X_g)\rar["r"] &\sCE(X_g)
\end{tikzcd}
\end{equation}
 which is split since $\cC(-)\simeq \Omega\cH(-)$ is a homotopy functor and the embedding $S^1\times D^{2n-1}\hookrightarrow X_g$ lies in the homotopy class of the standard inclusion $S^1\hookrightarrow S^1\vee \vee^{2g}S^n$ (which is a section of the obvious collapse map). Thus the restriction map $r$ is surjective on homotopy groups. Moreover, conjugation by $\psi\coloneq\phi\rho^i$ is also defined on the space $\mathcal{C}(X_g)$, and $r$ clearly commutes with $\psi_\#$. We claim that $\psi_\#$ on $\mathcal{C}(X_g)$ is induced by functoriality of the homotopy functor $\Omega^2\Wh(-)\simeq \Omega\cH(-)$. 

\begin{lemma}
    Let $\psi$ be a diffeomorphism of a compact manifold $M$. Then, there is a commutative diagram
    $$
\begin{tikzcd}[column sep = 35pt]
    \mathcal{C}(M)\dar["\psi_\#"] &\lar["\sim", "\eqref{cCvscH}"']\Omega\mathcal{H}(M)\dar["\psi_*"]\rar[dash, "\sim"', "\eqref{WJReq}"] 
    &\Omega^2\Wh(M)\dar["\psi_*"]\\
    \mathcal{C}(M)&\lar["\sim", "\eqref{cCvscH}"']\Omega\mathcal{H}(M)\rar[dash, "\sim"', "\eqref{WJReq}"] &\Omega^2\Wh(M).
\end{tikzcd}
    $$
    Here $\psi_*$ are the maps induced by $\psi$ from the functoriality of the homotopy functors $\mathcal{H}(-)$ and $\Wh(-)$. Moreover, if $N$ is a codimension-zero submanifold of $M$ on which $\psi$ restricts to a diffeomorphism $\psi\mid_N=\xi$, then the above diagram is compatible with
    $$
\begin{tikzcd}
    \mathcal{C}(N)\dar["\xi_\#"] &\lar["\sim"']\Omega\mathcal{H}(N)\dar["\xi_*"]\rar[dash, "\sim"] 
    &\Omega^2\Wh(N)\dar["\xi_*"]\\
    \mathcal{C}(N)&\lar["\sim"']\Omega\mathcal{H}(N)\rar[dash, "\sim"] &\Omega^2\Wh(N)
\end{tikzcd}
    $$
    via the obvious maps induced by the inclusion $N\subset M$.
\end{lemma}

\begin{proof}
    The Waldhausen--Jahren--Rognes equivalence $\mathcal{H}(-)\simeq \Omega\Wh(-)$ of \eqref{WJReq} is functorial in codimension-zero embeddings (in particular diffeomorphisms), so the right hand square in the diagram above commutes. 

    Regarding the commutativity of the left-hand subsquare, consider the diagram
\begin{equation}\label{CollarDiagram}
    \begin{tikzcd}
  \C(M)\dar["\psi_\#"]\rar["a"] &E(M)\dar["\psi_\#"]\rar["p"] & \H(M)\dar["\psi_*"]  \\
    \C(M)\rar["a"] &E(M)\rar["p"] & \H(M),  
    \end{tikzcd}
\end{equation}
    where the rows are the fibre sequence in \eqref{OmegaHFibSeq} inducing $\Omega \H(M)\xrightarrow{\sim} \C(M)$. This diagram is visibly commutative and compatible with stabilisation. Thus, we obtain the desired commutativity of the left subsquare in the statement. The addendum easily follows too: that the right subsquares are compatible follows from the fact that the Waldhausen--Jahren--Rognes equivalence is functorial with respect to codimension-zero embeddings. The left subsquares are compatible because \eqref{CollarDiagram} is also functorial in such embeddings.
\end{proof}

Therefore, given $([\phi],\rho^i)\in \mcgRhoXg$ as before, the diffeomorphism $\psi=\phi\rho^i$ restricts to a diffeomorphism on $S^1\times D^{2n-1}\subset X_g$ by the assumption that $\phi$ is fixed on this submanifold and the reflection $\rho$ preserves it. Thus, by the previous result we obtain a homotopy commutative diagram
\[
\begin{tikzcd}
\Omega\sCE(X_g)\dar["\psi_\#"]\rar[dash, "\sim"] & \Omega^{3}\Wh(X_g,S^1)\dar["\psi_*"]\\
   \Omega\sCE(X_g)\rar[dash, "\sim"] & \Omega^{3}\Wh(X_g,S^1).
\end{tikzcd}
\]
Finally, Goodwillie's isomorphism \eqref{GoodwillieIso} is functorial in the pair $(X,Y)$ as it is induced by the natural diagram \eqref{DGMdiagram}---this isomorphism is the one that allows us to access the rational homotopy of the relative Whitehead space $\Wh(X,Y)$ in \cref{involutionlemma}. All in all, we have proved:

\begin{teo}\label{ActionTheorem}
    The action of $\mcgRhoXgsfr$ on the homotopy groups of the fibre in
$$
\begin{tikzcd}
\bEmbmodEmb(X_g) \rar& \BEmbsfr(X_g)_\ell\rar &\BbEmbsfr(X_g)_\ell
\end{tikzcd}
$$
factors through the homomorphism $\mcgRhoXgsfr\to \mcgRhoXg$. Under the isomorphism 
$$
\pi_*^{\QQ}\big(\bEmbmodEmb(X_g)\big)\cong \left[\frac{\H^{S^1}_{*}(LX_g, LS^1; \QQ)}{\H_{*+1}(X_g,S^1;\QQ)}\right]_{-\tau_{\FLS}}, \quad *\leq 3n-5,
$$
of \eqref{XgPseudoisotopyVsFLS}, the action of $([\phi],\rho^i)\in \mcgRhoXg$ (where $[\phi]$ is represented by a diffeomorphism $\phi$ of $X_g$) on the left-hand side corresponds, on the right-hand side, to the map
$$
(\phi\rho^i)_*: \left[\frac{\H^{S^1}_{*}(LX_g, LS^1; \QQ)}{\H_{*+1}(X_g,S^1;\QQ)}\right]_{-\tau_{\FLS}}\overset\cong\longrightarrow \left[\frac{\H^{S^1}_{*}(LX_g, LS^1; \QQ)}{\H_{*+1}(X_g,S^1;\QQ)}\right]_{-\tau_{\FLS}}
$$
induced by functoriality on $(X_g,S^1)$.
\end{teo}

\begin{proof}
    Note that when $*=0$ or $1$ there is nothing to prove as $\bEmbmodEmb(X_g)$ is $(n-2)$-connected. The rest of the cases follow from \cref{BoringAlphaProp} and the work in this section.
\end{proof}

Recall that in \cref{FLSS1homology}(iii) we computed the group $\smash{\big[\frac{\H^{S^1}_{*}(LX_g, LS^1; \QQ)}{\H_{*+1}(X_g,S^1;\QQ)}\big]_{-\tau_{\FLS}}}$ as a representation of $\hMCGXg=\pi_0(\Aut_*(X_g))$. Moreover, there are natural homomorphisms
\begin{equation}\label{GrestrictionHomomorphism}
\begin{tikzcd}
    \mcgRhoXgsfr\rar &\mcgRhoXg\rar["\nu"] &\hMCGXg,
    \end{tikzcd}
\end{equation}
where $\nu([\phi],\rho^i)\coloneq[\phi\rho^i]$. Thus, all the $\hMCGXg$-representations involved in that statement are naturally $\mcgRhoXgsfr$-representations by restriction along \eqref{GrestrictionHomomorphism}. Together with \cref{ActionTheorem}, we conclude:

\begin{cor}\label{MCGrepBlockModEmb}
    Let $H_{X_g}$, $A_r$, $B_r$ and $C$ be the $\QQ[\hMCGXg]$-modules of \cref{FLShomology,FLSS1homology}, regarded now as $\QQ[\mcgRhoXgsfr]$-modules by restriction along \eqref{GrestrictionHomomorphism}. Then, for all $n\geq 3$ and $*\leq 2n-2$, there is an isomorphism of $\QQ[\mcgRhoXgsfr]$-modules
    \[
\pi_*^{\QQ}\big(\bEmbmodEmb(X_g)\big)\cong \left\{\begin{array}{cl}
    \bigoplus_{r> 0} {[A_r]_{\pi}}\cong \bigoplus_{r> 0} H_g & *=n-1, \\[3pt]
     C \oplus \bigoplus_{r> 0}\left({[\sn(A_r)]}_{\pi}\oplus [B_r]_\pi\right) & *=2n-2, \\[3pt]
    0 & \text{otherwise},
\end{array}
\right.
    \]
    where $C$ is some quotient of $[H_{X_g}^{\otimes 2}]_{\pi}$
.
    \end{cor}

\subsection{The \texorpdfstring{$\ZN$}{Z[N]}-module structures}\label{FrobeniusXgSection} For each positive integer $d\in \NN$, the circle $S^1$ is its own $d$-fold cover and, similarly, $X_{dg}$ is the corresponding $d$-fold cover of $X_g$ for every $g\geq 0$. This observation gives rise to maps
\[
\varphi_d: \BEmbsfr(X_g)_\ell\longrightarrow \BEmbsfr(X_{dg})_\ell, \quad d\in \NN,
\]
given by ``lifting'' a self-embedding of $X_g$ to $X_{dg}$. These maps are natural in $g\geq 0$ and satisfy that the composite $\varphi_d\circ \varphi_{d'}$ is homotopic to $\varphi_{d\cdot d'}$ for every $d,d'\in \NN$. Letting $g\to\infty$, these (homotopy classes of) maps induce self-maps of $\BEmbsfr(X_\infty)_\ell$ that equip its homotopy groups with the structure of a $\ZN$-module---also known as a \textit{cyclotomic module}. 

Similarly, the homotopy groups of $\BbEmbsfr(X_\infty)_\ell$ and $\bEmbmodEmb(X_\infty)$ inherit compatible $\ZN$-module structures that make the long exact sequence associated to
\[\begin{tikzcd}
       \bEmbmodEmb(X_\infty) \rar & \BEmbsfr(X_\infty)_\ell\rar["\iota"] &\BbEmbsfr(X_\infty)_\ell,
    \end{tikzcd}\]
one of $\ZN$-modules---in particular, the connecting maps
\[
\delta_*: \pi_{*+1}^\QQ(\BbEmbsfr(X_\infty)_\ell)\longrightarrow \pi_*^\QQ\left(\bEmbmodEmb(X_\infty)\right)
\]
are $\ZN$-module maps. We will use this additional structure in later sections to understand $\delta$ and, consequently, the homotopy groups of $\BEmbsfr(X_\infty)_\ell$. 

In this section, we focus on the $\ZN$-module structures of the vector spaces featuring in the connecting maps $\delta_*$ above. To do so, we will study the Frobenius maps $\varphi_d$ unstably, i.e., in the homotopy groups of the spaces $\BbEmbsfr(X_g)_\ell$ and $\bEmbmodEmb(X_g)$, for all $g\geq0$ and $d\in \NN$.

\subsubsection{The $\ZN$-module structure of pseudoisotopy embeddings} \label{FrobeniusXgPseudoIsotopySection}

In this section, we study the Frobenius map  
\begin{equation}\label{FrobeniusPseudoIsotopy}
\varphi_d \colon \pi_{n-2}^{\QQ}(\pEmb(X_g))  \;\longrightarrow\; 
\pi_{n-2}^{\QQ}(\pEmb(X_{dg})) ,
\end{equation}
which, by \cref{MCGrepBlockModEmb}, is of the form 
$\varphi_d \colon \bigoplus_{r>0} H_g \longrightarrow \bigoplus_{r>0} H_{dg}$. One possible approach is to trace the effect of $\varphi_d$ through our computation of 
$\pi_{n-2}^{\QQ}(\pEmb(X_g)) $ in the previous sections, which, as explained in \cref{RecollectionWWSection}, proceeds via the Weiss--Williams map, the Waldhausen--Jahren--Rognes equivalence, and finally the cyclotomic trace.  
This method is complicated by the intricacy of the Waldhausen--Jahren--Rognes equivalence, which makes it difficult to relate $\varphi_d$ to modern $K$-theoretic transfer maps. Instead, we adopt a more direct approach to understanding \eqref{FrobeniusPseudoIsotopy}, using the work of Goodwillie~\cite{GoodwillieCalcI}. (We note, however, that this direct method does not work for the Frobenius map \eqref{FrobeniusPseudoIsotopy} in degree $2n-3$---the next non-vanishing homotopy group---for which the earlier, more involved approach seems better suited. As we will not need that case, we restrict ourselves to the direct method.) It occurred to us to use this more direct approach after an illuminating conversation about Goodwillie’s work with Manuel Krannich, to whom we are grateful.

Let us explain the strategy first. Recall the map $\alpha: \Omega \CEmbo(X_g)\to \Omega \pEmb(X_g)$ of \eqref{CEvsPseudoisotopyMap}. By the construction of this map, it is readily seen that it is compatible with the Frobenii, i.e.
\[
\begin{tikzcd}
     \Omega \CEmbo(X_g)\rar["\alpha"]\dar["\varphi_d"] &  \Omega \pEmb(X_g)\dar["\varphi_d"]\\
     \Omega \CEmbo(X_{dg})\rar["\alpha"]&  \Omega \pEmb(X_{dg})
\end{tikzcd}
\]
commutes. It follows from \cref{IntermediateBoringAlphaProp} that the map $\alpha$ induces the right isomorphism in
\begin{equation}\label{iotaCeq}
\left[\pi_*^{\QQ}\big(\CEmbo(X_g\times I^{2m})\big)\right]_{\iota_{\H}}\cong \left[\pi_*^{\QQ}\big(\CEmbo(X_g)\big)\right]_{\iota_{\H}}\cong\pi_*^{\QQ}\big(\pEmb(X_g)\big) , \qquad *\leq 3n-6, 
\end{equation}
which is compatible with the Frobenii. The left isomorphism is induced by the concordance stabilisation map $\Sigma^{2m}: \CEmbo(X_g)\to\CEmbo(X_g\times I^{2m})$, which is $(3n-5)$-connected, clearly compatible with the Frobenii, and $\iota_{\H}$-equivariant by \ref{C2EquivarianceConcStabilisation}. Thus, it is enough for us to compute $\smash{\pi_{n-2}^{\QQ}(\CEmbo(X_g\times I^{2m}))}$, for sufficiently large $m\geq 0$, in a way that we can identify the action of $\iota_{\H}$ and understand the Frobenius maps $\varphi_d$.

For this, we import the work of Goodwillie in \cite[\S 3]{GoodwillieCalcI}, which we shall now recall. Writing $\uSpc^{[1]}$ for the arrow category $\mathrm{Fun}([1],\uSpc)$, there is defined a (homotopy) functor
\[
L: \uSpc^{[1]}\longrightarrow \uSpc, \quad (Y\to X)\longmapsto Y\times_{X}LX\simeq \operatorname{holim}(Y\to X\xleftarrow{\mathrm{ev}_1}LX).
\]
The inclusion $c: X\to LX$ as the constant loops (a section of $\mathrm{ev}_1$) gives rise to a map $Y\to L(Y\to X)$, and we denote by $L(Y\to X)/Y$ its (homotopy) cofibre. There is another functor
\[
\Psi: \uSpc^{[1]}\longrightarrow \Spc, \quad (Y\to X)\longmapsto \Omega^2Q(L(Y\to X)/Y)=\Omega^{\infty+2}\Sigma^\infty(L(Y\to X)/Y).
\]
 Write $\Psi(X)$ for $\Psi(\id_X)$ and $\overline{\Psi}(Y\to X)$ for the homotopy fibre of the natural map $\Psi(Y\to X)\to \Psi(X)$. Note that $\Psi(X)\simeq \Omega^2 Q(LX/X)$ and that for a map $Y\to X$, the factorisation
\[
\begin{tikzcd}[row sep = 15pt]
    Y\dar[equal]\rar[equal] &Y\dar\rar &X\dar[equal]\\
    Y\rar &X\rar[equal] &X
\end{tikzcd}
\]
gives rise to natural maps $\Psi(Y)\to \Psi(Y\to X)\to \Psi(X)$.

For $(Y\to X)=(N\xhookrightarrow{}M)$ an embedding of compact manifolds, Goodwillie constructs a commutative diagram (cf. page 24 loc.cit.)
\[
\begin{tikzcd}[row sep = 15pt]
    \C(N)\rar["\tau"]\dar & \Psi(N)\rar & \Psi(N\hookrightarrow M)\dar\\
\C(M)\ar[rr,"\tau"] && \Psi(M),
\end{tikzcd}
\]
whose vertical homotopy fibre is denoted by $\overline{\tau}$. When $P\subset M$ is a compact submanifold of $M$, and $N\xhookrightarrow{}M$ is the evident inclusion $M-\nu P\hookrightarrow M$, this map $\overline{\tau}$ is of the form
\begin{equation}\label{BarTauPM}
\overline{\tau}: \Omega\CEmb(P,M)\longrightarrow \overline{\Psi}(M-P\to M).
\end{equation}

\begin{lemma}[Lemma 3.19, \cite{GoodwillieCalcI}]\label{GoodwillieDisjunctionTau} If $3\leq k<\dim M$ and $M$ is obtained from $M-\nu P$ by attaching handles of index $k$, then $\eqref{BarTauPM}$ is $(2k-5)$-connected. 
\end{lemma}

The case of interest for us is when $N\xhookrightarrow{} M$ is the inclusion $S^1\times D^{2n-1+2m}\xhookrightarrow{} X_g\times I^{2m}$. In this case, \cref{GoodwillieDisjunctionTau} tells us that, for $n\geq 3$, the map $\overline{\tau}: \Omega \CEmbo(X_g\times I^{2m})\to \overline{\Psi}(S^1\to X_g)$ is $(2n-5)$-connected. Since $n-3\leq 2n-6$ for $n\geq 3$, we obtain an isomorphism
\begin{equation}\label{taubarmap}
\overline{\tau}_*: \pi_{n-2}(\CEmbo(X_g\times I^{2m}))\xrightarrow{\cong} \pi_{n-3}(\overline{\Psi}(S^1\to X_g)).
\end{equation}
We will use this isomorphism, described in \cite[pp. 23--24]{GoodwillieCalcI}, to study the Frobenii on $\pi_{n-2}\CEmbo(X_g\times I^{2m})$.

\begin{rmk}
A quick word on why $X_g\times I^{2m}$: as explained in \cite[p.~26]{GoodwillieCalcI}, the map $\tau\colon \C(M)\to \Psi(M)$ factors through the stable concordance space $\cC(M)$, and hence $\overline{\tau}$ factors through $\Omega\sCEnopartial(P,M)$. Since the descriptions of $\tau$ and $\overline{\tau}$ are simpler in a fixed finite dimension, we work in such a dimension. For \cref{rProp}, we require this dimension $N\coloneq d+2n$ large enough that Igusa’s lower bound for the concordance stable range of $X_g\times I^{2m}$ (approximately $N/3$) exceeds $n-2$. Thus, we replace $X_g$ by $X_g\times I^{2m}$ for $m$ sufficiently large; however, we continue to simply write \(X_g\) for the homotopy type 
of \(X_g\times I^{2m}\).

\end{rmk}

To this end, we first need to compute $\pi_{n-3}^{\QQ}(\overline{\Psi}(S^1\to X_g))$; of course, we already know by \eqref{CEhtpyFLShomologyeq} and \eqref{HS1XgcomputationProof} that it is isomorphic to $\bigoplus_{r\neq 0} H_g$, but we need to fix an identification that we can trace the Frobenii through. Consider the commutative square
\begin{equation}\label{PsiDiagram}
\begin{tikzcd}[column sep = 7pt]
    &\pi_{n-2}^{\QQ}(\C(X_g\times I^{2m}))\dar["\tau_*"]\ar[rrrr, "r"] &&&& \pi_{n-2}^{\QQ}(\CEmbo(X_g\times I^{2m}))\dar["\overline{\tau}_*", "\cong"']\\
    \bigoplus_{r\neq 0} H_g \rar[phantom, "\cong"] &\pi_{n-2}^{\QQ}(\Psi(X_g))\ar[rrrr,"\delta"] &&&& \pi_{n-3}^{\QQ}(\overline{\Psi}(S^1\to X_g)),
\end{tikzcd}
\end{equation}
where the bottom left isomorphism is $\pi_{n-2}^{\QQ}(\Psi(X_g))=\pi_{n}^s(LX_g/X_g)\otimes \QQ\cong \H_n(LX_g,X_g;\QQ)$ together with \cref{FLShomology}.

\begin{lemma}\label{PsiLemma}
    The map $\delta$ in \eqref{PsiDiagram} is an isomorphism.
\end{lemma}
\begin{proof}
    By definition, $\delta$ is the connecting map in a long exact sequence
    \[
    \adjustbox{scale=0.98,center}{
    \begin{tikzcd}[column sep = 17pt]
        \pi_{n-2}^{\QQ}(\Psi(S^1\to X_g))\rar &\pi_{n-2}^{\QQ}(\Psi(X_g))\rar["\delta"] & \pi_{n-3}^{\QQ}(\overline{\Psi}(S^1\to X_g))\rar & \pi_{n-3}^{\QQ}(\Psi(S^1\to X_g))\rar["a"] & \pi_{n-3}^{\QQ}(\Psi(X_g)).
    \end{tikzcd}}
    \]
    We will show that i) the leftmost term $\pi_{n-2}^{\QQ}\Psi(S^1\to X_g)=0$ and that ii) the rightmost map $a$ is an isomorphism. Recall also that $\Psi(X_g)=\Omega^2Q(LX_g/X_g)$ and $\Psi(S^1\to X_g)=\Omega^2Q(L(S^1\to X_g)/S^1)$, so we need to understand the map $L(S^1\to X_g)/S^1\to LX_g$ in rational homology. Note that, just like $LX_g$, the space $L(S^1\to X_g)$ decomposes as a disjoint union of connected components $\coprod_{r\in \ZZ}L_r(S^1\to X_g)$, and by definition we have maps of fibrations
    \[
    \begin{tikzcd}
        \Omega_rX_g\rar\dar[equal] & L_r(S^1\to X_g)\dar\rar &S^1\dar\\
        \Omega_rX_g\rar & L_rX_g\rar & X_g.
    \end{tikzcd}
    \]
Their respective Serre spectral sequences are
\begin{align}
   \H_p(S^1;\underline{\H_q(\Omega_r X_g;\QQ)})&\implies \H_{p+q}(L_r(S^1\to X_g);\QQ), \notag \\
   \H_p(X_g;\underline{\H_q(\Omega_r X_g;\QQ)})&\implies \H_{p+q}(L_rX_g;\QQ), \label{LrXgSS}
\end{align}
and the map of $E^2$-pages $\H_p(S^1;\underline{\H_q(\Omega_r X_g;\QQ)})\to \H_p(X_g;\underline{\H_q(\Omega_r X_g;\QQ)})$ is an isomorphism if $p\leq n-1$. We only need to consider the case when $q=n-1$, as the homology of $\Omega_r X_g\simeq \Omega_0 X_g$ is concentrated in degrees that are multiples of $n-1$. In this degrees, we have the Hurewicz isomorphism $\H_{n-1}(\Omega_r X_g;\QQ)\cong \pi_{n-1}^{\QQ}(\Omega_r X_g)\cong H_{X_g}$, and thus the action of $\pi=\pi_1 S^1=\pi_1 X_g$ on $\underline{\H_{n-1}(\Omega_r X_g;\QQ)}$ is the one arising from the usual action of $\pi_1$ of the total space of a fibration on its fibre. Thus, $\underline{\H_{n-1}(\Omega_r X_g;\QQ)}\cong H_{X_g}$ as a $\QQ[\pi]$-module. Noting that $[H_{X_g}]^\pi=0$, both i) and ii) follow.
\end{proof}

\begin{lemma}\label{rProp}
    For all $n\geq 3$ and $m\geq 0$ such that $\min(\tfrac{2(n+m)-4}{3}, \tfrac{2(n+m)-7}{2})\geq n-1$, the map $r$ in \eqref{PsiDiagram} is surjective. (By \cref{PsiLemma}, then so is $\tau_*$).
\end{lemma}
\begin{proof}
By Igusa's stability theorem \cite{IgusaStableRange}, the left vertical map in
\[
\begin{tikzcd}
    \C(X_g\times I^{2m})\dar[hook]\rar["r"] &\CEmbo(X_g\times I^{2m})\dar[hook]\\
    \cC(X_g)\rar[two heads,"r"] & \sCE(X_g)
\end{tikzcd}
\]
is $\min(\tfrac{2(n+m)-4}{3}, \tfrac{2(n+m)-7}{2})$-connected, whereas the right vertical map is $3(n+m)-5\geq n-1$. As explained in \eqref{rSurjectionFibSeq}, the bottom horizontal map $r$ is surjective on homotopy groups, so the result follows.
\end{proof}

From now on, let $m\geq 0$ satisfy the assumption in \cref{rProp}. We now define Frobenius homomorphisms 
\begin{equation}\label{PsiFrobenius}
\varphi_d: \pi_{n-2}^\QQ(\Psi(X_g))\longrightarrow \pi_{n-2}^\QQ(\Psi(X_{dg}))
\end{equation}
which will be compatible with the map $\tau_*$ in \eqref{PsiDiagram}. Since $r$ is surjective by \cref{rProp} and $\delta$ is an isomorphism by \cref{PsiLemma}, the map $\delta^{-1}\circ \overline{\tau}_*$ will be Frobenius-preserving as long as we equip the right-bottom term in \eqref{PsiDiagram} with the Frobenii
\[
\begin{tikzcd}
\overline{\varphi}_d: \pi_{n-3}^\QQ(\overline{\Psi}(S^1\to X_g))\rar["\delta^{-1}", "\cong"'] & \pi_{n-2}^\QQ(\Psi(X_g))\rar["\varphi_d"] & \pi_{n-2}^\QQ(\Psi(X_{dg}))\rar["\delta", "\cong"'] & \pi_{n-3}^\QQ(\overline{\Psi}(S^1\to X_{dg})).
\end{tikzcd}
\]
To define $\varphi_d$, let $p: X_{dg}\times I^{2m}\to X_g\times I^{2m}$ be a fixed $d$-fold cover. Then $Lp: LX_{dg}\to LX_g$ is a fibration with compact fibres (either empty or $d$-many points), so we can consider the diagram
\[
\begin{tikzcd}
    \Sigma^\infty_+X_g\rar["p^!"]\dar["\Sigma^\infty_+c"] & \Sigma^\infty_+X_{dg}\dar["\Sigma^\infty_+c"]\\
    \Sigma^\infty_+LX_g\rar["(Lp)^!"] & \Sigma^\infty_+X_{dg},
\end{tikzcd}
\]
where $c: (-)\to L(-)$ stands for the inclusion as the constant loops, and $p^!$ and $Lp^!$ are the Becker--Gottlieb transfers of the fibrations $p$ and $Lp$, respectively. By formal properties of the Becker--Gottlieb transfer, this square of spectra is homotopy commutative and hence, upon taking vertical cofibres and $\Omega^{\infty+2}(-)$, induces a map $(Lp)^!: \Omega^2Q(LX_g/X_g)\to \Omega^2Q(LX_{dg}/X_{dg})$ which, on homotopy groups gives
\begin{equation}\label{PsiFrobeniusHtpy}
\varphi_d\coloneq\pi_k((Lp)^!): \pi_{k+2}^s(LX_g,X_g)\longrightarrow \pi_{k+2}^s(LX_{dg},X_{dg}).
\end{equation}
Rationally, this gives the desired \eqref{PsiFrobenius}. We must now check that $\tau_*$ is compatible with these Frobenii.

\begin{cons}\label{taumapconstruction}
Let us first recall the description of the map 
\[
\tau_*: \pi_{k}(\C(M))\longrightarrow \pi_{k}(\Psi(M))\cong \pi_{k+2}^s(LM,M)\cong \Omega^{\fr}_{k+2}(LM,M)
\]
given in \cite[pp. 23--24]{GoodwillieCalcI}. Here $\Omega^{\fr}_*$ stands for framed bordism theory, and the last isomorphism in the above equation is the one by Pontryagin--Thom. For simplicity, assume that $M$ is stably framed\footnote{The homomorphism $\tau_*$ factors through $\pi_k (\mathcal{C}(M))$ as shown in \cite[pp. 26--27]{GoodwillieCalcI}. If $M$ is not stably framed, we can replace it with the disc bundle $D(\nu_M)$ of a normal bundle for $M$ since Burghelea--Lashof construct in \cite[p. 10]{BurgheleaLashofTransfer} a factorisation $\C(M)\to \C(D(\xi))\to\mathcal{C}(M)$ for any vector bundle $\xi$ over $M$.}.

Given a concordance diffeomorphism $F=(f,g):M\times I\xrightarrow{\cong} M\times I$, a point $(x,s,t)\in M\times I\times I$ is a \textit{crossing} for $F$ if it is either an
\begin{itemize}
    \item \emph{ordinary crossing}: $s<t$, $f(x,s)=f(x,t)$ and $g(x,s)>g(x,t)$;
    \item \emph{infinitesimal crossing}: $s=t$, $(Df)_{(x,s)}(0,1)=0$ and $(Dg)_{(x,s)}(0,1)< 0$.
\end{itemize} 
A homotopy class in $\pi_k(\C(M))$ can be represented by a smooth map $F: D^k\times M\times I\to M\times I$, seen as a family of maps $F_z: M\times I\to M\times I$ for each $z\in D^k$, such that $F_z\in C(X_g)$ for every $z\in D^k$ and $F_z=\Id_{M\times I}$ if $z\in \partial D^k$. Under the generic transversality assumptions on $F$ of Hypothesis 3.18 loc.cit., the set 
\[
\mathcal{M}_F\coloneq\{(z,x,s,t)\in D^k\times M\times I\times I: \text{$(x,s,t)$ is a crossing for $F_z$}\}
\]
is a smooth $(k+2)$-dimensional manifold with boundary $\partial \mathcal{M}_F$, and the boundary points correspond precisely to the infinitesimal crossings of $F$. It is also stably framed, for the normal bundle of $\mathcal{M}_F- \partial \mathcal{M}_F$ in $D^k\times M\times I\times I$ can be identified with the pullback of the tangent bundle of $M$ along
\[
\mathcal{M}_F\longrightarrow M, \quad (z,x,s,t)\longmapsto f_z(x,s)=f_z(x,t),
\]
map which is homotopic to the evident projection $\mathcal{M}_F\to M$ via $\{f_z(x,rs)\}_{r\in [0,1]}$. (Here we use that $M$ is stably framed by assumption.) Moreover, there is a natural map
\[
q_F: (\mathcal{M}_F,\partial{M}_F)\longrightarrow (LM,M), \quad (z,x,s,t)\mapsto \big([0,1]\ni r\mapsto f_z(x,s+r(t-s))\big),
\]
where we regard $M\subset LM$ as the constant loops. Then:
\[
\tau_*: \pi_k(\C(M))\longmapsto \pi_k (\Psi(M))\cong \Omega_{k+2}^{\fr}(LM,M), \quad [F]\longmapsto [q_F: (\mathcal{M}_F, \partial \mathcal{M}_F)\to (LM,M)].
\]
This concludes the description of $\tau_*$.
\end{cons}

We now show that $\tau_*$ is Frobenius-preserving and explicitly describe the Frobenii on the target.

\begin{prop}\label{tauFrobeniusPreservingProp}
    The map $\tau_*: \pi_k(\C(X_g\times I^{2m}))\to \pi_k(\Psi(X_g))$ of \cref{taumapconstruction} is Frobenius-preserving if we equip the target with the Frobenii of \eqref{PsiFrobeniusHtpy}, i.e., for every $d\in \NN$, the following square commutes:
    \begin{equation}\label{tauFrobeniusPreservingSquare}
    \begin{tikzcd}
        \pi_k(\C(X_g\times I^{2m}))\ar[rr,"\varphi_d"]\dar["\tau_*"] &&\pi_k(\C(X_{dg}\times I^{2m}))\dar["\tau_*"]\\
        \pi_{k+2}^s(LX_g,X_g)\ar[rr, "\varphi_d=(Lp)^{!}"] && \pi_{k+2}^s(LX_{dg},X_{dg}).
    \end{tikzcd}
    \end{equation}
    Under the isomorphism $\pi_{n}^s(LX_g,X_g)\otimes \QQ\cong H_n(LX_g,X_g;\QQ)\cong \bigoplus_{r\neq 0} H_g$, the Frobenius $\varphi_d=(Lp)^!$ is:
    \[
    \begin{tikzcd}
\bigoplus_{r\neq 0} H_g^{(r)}\ar[rr,"\bigoplus_{r\neq0} \alpha_r"] && \bigoplus_{s\neq 0} H_{dg}^{(s)},
\end{tikzcd}
    \]
    where the subscripts $(-)^{(r)}$ and $(-)^{(s)}$ only serve bookkeeping purposes, and $\alpha_r$:
    \begin{itemize}
        \item[(i)] is zero if $d$ does not divide $r$;
        \item[(ii)] is injective if $r=di$, and factors as
        \[
        \begin{tikzcd}
        H_{g}^{(di)}\rar["\Delta"] & \bigoplus_{j=0}^{d-1} H_g\cong H_{dg}^{(i)}\rar[hook] &\bigoplus_{s\neq 0} H_{dg}^{(s)}.
        \end{tikzcd}
        \]
    \end{itemize}
\end{prop}

For the first part of \cref{tauFrobeniusPreservingProp}, fix a finite-sheeted cover $p: \widetilde{X}\to X$ (we only need the special case $p: X_{dg}\to X_g$). Consider the homomorphism
\begin{align*}
    \ell_p: \Omega_{k+2}^{\fr}(LX,X)&\longrightarrow \Omega_{k+2}^{\fr}(L\widetilde{X},\widetilde{X}),\\
    [q: (W,\partial W)\to (LX,X)]&\longmapsto [q\times_{\id} \id: (W,\partial W)\times_{(LX,X)}(L\widetilde{X},\widetilde{X})\to (L\widetilde{X},\widetilde{X})],
\end{align*}
which is well-defined as $L\widetilde{X}\to \widetilde{X}$ is also finite-sheeted, and consider the diagram
\begin{equation}\label{tauFrobeniusFactorisation}
\begin{tikzcd}
\pi_k(\C(X)) \ar[rr,"\varphi_p"]\dar["\tau_*"] \ar[rrd, phantom, "\circled{1}", pos = 0.6] &&\pi_k(\C(\widetilde{X}))\dar["\tau_*"]\\
\Omega_{k+2}^{\fr}(LX,X)\ar[rr,"\ell_d"]\dar["PT", "\cong"'] \ar[rrd, phantom, "\circled{2}"]&& \Omega_{k+2}^{\fr}(L\widetilde{X},\widetilde{X})\dar["PT", "\cong"']\\
        \pi_{k+2}^s(LX,X)\ar[rr, "(Lp)^{!}"'] && \pi_{k+2}^s(L\widetilde{X}, \widetilde{X}),
\end{tikzcd}
\end{equation}
where $\varphi_p$ stands for the lifting map. When $p$ is our preferred cover $X_{dg}\to X_g$, the outer square in \eqref{tauFrobeniusFactorisation} is \eqref{tauFrobeniusPreservingSquare}. We shall show that both $\circled{1}$ and $\circled{2}$ commute.

\begin{lemma}\label{circled1lemma}
    The square $\circled{1}$ in \eqref{tauFrobeniusFactorisation} commutes. 
\end{lemma}
\begin{proof}
    This is nearly immediate. Represent a homotopy class $[F]\in \pi_k(\C(X))$ by a smooth map $F: D^k\times X\times I\to X\times I$ as in \cref{taumapconstruction}, and let $\widetilde{F}: D^k\times \widetilde{X}\times I\to \widetilde{X}\times I$ be its lift (thus representing $\varphi_p([F])$). Then, the map
    \begin{align*}
    (\mathcal{M}_{\widetilde{F}},\partial \mathcal{M}_{\widetilde{F}})&\longrightarrow (\mathcal{M}_{F},\partial \mathcal{M}_{F})\times_{(LX,X)}(L\widetilde{X},\widetilde{X})\\
    (z,\widetilde{x},s,t)&\longmapsto ((x,p(\widetilde{x}),s,t), q_{\widetilde{F}}(z,\widetilde{x},s,t))
    \end{align*}
    is a diffeomorphism over $(L\widetilde{X},\widetilde{X})$, as desired.
\end{proof}

\begin{lemma}\label{circled2lemma}
    The square $\circled{2}$ in \eqref{tauFrobeniusFactorisation} commutes. More generally, given a finite-sheeted cover $p: E\to B$, we can define the homomorphism
    \[
    \ell_p: \Omega^{\fr}_k(B)\longrightarrow \Omega^{\fr}_k(E), \quad [q:W\to B]\longmapsto [q\times_{\id}\id: W\times_B E\to E].
    \]
    Then, if there exists a fibrewise smooth embedding $\psi: E\subset B\times \RR^n$, the following square commutes:
    \[
    \begin{tikzcd}
        \Omega_{k}^{\fr}(B)\ar[rr,"\ell_p"]\dar["PT", "\cong"'] \ar[rrd, phantom, "\circled{2}"]&& \Omega_{k}^{\fr}(E)\dar["PT", "\cong"']\\
        \pi_{k}^s(B)\ar[rr, "p^{!}"'] && \pi_{k}^s(E).
    \end{tikzcd}
    \]
\end{lemma}

\begin{rmk}

    When $p\colon E\to B$ is a finite-sheeted cover, the homomorphism $\ell_p$ is well defined since $W\times_B E$ is a closed, stably framed $k$-manifold. If $p\colon E\to B$ were a smooth fibre bundle with fibres closed, stably framed $d$-manifolds, then $W\times_B E$ would be a $(k+d)$-manifold, and $\ell_p\colon \Omega_k^{\fr}(B)\to \Omega_{k+d}^{\fr}(E)$ would necessarily raise degree. Thus, the assumption in \cref{circled2lemma} that $p\colon E\to B$ is a finite cover seems to be necessary.

\end{rmk}

\begin{rmk}

    
    For a finite-sheeted cover $p: \widetilde{X}\to X$ of compact (smooth) manifolds, the cover $Lp:L\widetilde{X}\to LX$ admits such a fibrewise smooth embedding as in the statement of \cref{circled2lemma}: indeed, fixing an embedding $\phi: \widetilde{X}\subset \RR^n$, the map
    \[
    \psi: L\widetilde{X}\longrightarrow LX\times \RR^n, \quad \gamma\longmapsto (p\circ\gamma, \phi(\gamma(1))),
    \]
    is such a fibrewise smooth embedding (so long as we replace $L(-)$ with the homotopy equivalent space of smooth maps $C^\infty(S^1,-)$). 
\end{rmk}
\begin{proof}[Proof of \cref{circled2lemma}]
    We will need to recall the constructions of the Becker–Gottlieb transfer and the Pontryagin–Thom isomorphism. For the former, let $p:E \to B$ be a smooth fibre bundle with compact manifold fibres (we will later specialise to the case of a finite cover), and let $\tau$ and $\nu$ denote the vertical tangent and normal bundles of a fibrewise smooth embedding $\psi:E \subset B \times \RR^n$. We obtain
    \[
    \begin{tikzcd}
    S^n\wedge B_+=\mathrm{Th}(B\times \RR^n)\rar["\psi^!"] & \mathrm{Th}(\nu)\rar[hook] & \mathrm{Th}(\tau\oplus\nu)=S^n\wedge E_+,
    \end{tikzcd}
    \]
where $\mathrm{Th}(-)$ stands for the Thom space and $\psi^!$ is the Pontryagin--Thom collapse map along $\psi$. This construction stabilises as $n\to\infty$ to the Becker--Gottlieb transfer for $p$:
\[
p^!: \Sigma^\infty_+ B\longrightarrow \Sigma^\infty_+ E.
\]
A property of the Becker--Gottlieb transfer that we will need \cite[Eq.~3.2]{BeckerGottlieb} is that if $p':E' \to B'$ is another smooth fibre bundle and $h:E \to E'$ is a map of smooth fibre bundles, then the following square commutes:
\begin{equation}\label{BGtransfer1}
\begin{tikzcd}
    \Sigma^\infty_+ B\dar["p^!"]\rar["h"] & \Sigma^\infty_+ B'\dar["(p')^!"]\\
    \Sigma^\infty_+ E\rar["h"] & \Sigma^\infty_+ E'.
\end{tikzcd}
\end{equation}

As for the Pontryagin--Thom isomorphism
\[
PT:\Omega_k^{\fr}(X)\overset{\cong}\longrightarrow \pi_k^s(X),
\]
given a map $q: W\to X$ representing a framed bordism class (so $W$ is a stably framed closed $k$-manifold), we can find an embedding $\varphi: W\subset \RR^m$ with trivial normal bundle. Then, the composition
\[
\begin{tikzcd}
S^m=\Th(\RR^m)\rar["\varphi^!"] & \Th(\nu_\varphi)\simeq W_+\wedge S^{m-k}\rar["q\wedge\id"] & X_+\wedge S^{m-k} 
\end{tikzcd}
\]
stabilises, as $m\to \infty$, to a map $\mathbf{S}^k\to \Sigma^\infty_+ W$ representing $PT([q: W\to X])\in \pi_k^s(X)$.

With this in mind, let $p:E\to B$ be a finite-sheeted cover as in the statement, let $q:W\to B$ represent a framed bordism class in $\Omega^{\fr}_k(B)$, and choose an embedding $\varphi: W\subset \RR^m$. This yields
\[
\begin{tikzcd}[column sep = 45pt]
\Phi: W\times_B E\rar[hook,"\id\times_{B}\psi"] &W\times_B (B\times \RR^n)=W\times \RR^n\rar[hook,"\varphi\times \id"] & \RR^m\times \RR^n=\RR^{m+n}.
\end{tikzcd}
\]
whose normal bundle is again trivial. We need to show that the outer square in
\begin{equation}\label{BGvsBorddiagram}
\begin{tikzcd}[column sep = 45pt]
    S^{m+n}\dar["\Phi^!"]\rar["(\varphi\times \id_{\RR^n})^!"] & S^{m-k+n}\wedge W_+\dar[dashed, "\mathrm{pr}_W^!"]\rar["\id\wedge q_+"] & S^{m-k+n}\wedge B_+\dar["p^!"]\\
    \mathrm{Th}(\nu_{\Phi})\rar[equal] & S^{m-k+n}\wedge (W\times_B E)_+\rar["\id\wedge (\mathrm{pr}_E)_+"] & S^{m-k+n}\wedge E_+
\end{tikzcd}
\end{equation}
is homotopy commutative. The dashed arrow $\mathrm{pr}_W^!$ is the (unstable) Becker--Gottlieb transfer for the finite cover $\mathrm{pr}_W: W\times_BE\to W$. Then, the right subsquare commutes since it is an unstable instance of \eqref{BGtransfer1} for the map of smooth fibre bundles
\[
\begin{tikzcd}
W\times_BE\dar["\mathrm{pr}_W"]\rar["\mathrm{pr}_E"] & E\dar["p"]\\
W\rar["q"] & B.
\end{tikzcd}
\]
Regarding the left subsquare in \eqref{BGvsBorddiagram}, note that the (unstable) Becker--Gottlieb transfer for $\mathrm{pr}_W$ arises from the fibrewise smooth embedding $\id\times_B\psi: W\times_B E\hookrightarrow W\times \RR^n$. Thus, because the Pontryagin--Thom collapse map is functorial on compositions of embeddings, it follows that the left subsquare is commutative by definition of $\Phi$. Of course, one needs to check that the commutativity of the diagram \eqref{BGvsBorddiagram} stabilises appropriately as $m\to \infty$, but this is routine. This concludes the proof.
\end{proof}

\begin{proof}[Proof of \cref{tauFrobeniusPreservingProp}]
    Lemmas \ref{circled1lemma} and \ref{circled2lemma} establish the commutativity of \eqref{tauFrobeniusPreservingSquare}. It remains to show that the Becker--Gottlieb transfer $(Lp)^!: \pi_n^s(LX_g,X_g)\otimes \QQ\to \pi_n^s(LX_{dg},X_{dg})\otimes \QQ$ yields the claimed formula under the isomorphisms $\pi_n^s(LX_g,X_g)\otimes \QQ\cong \bigoplus_{r\neq 0}H_g$ and $\pi_n^s(LX_{dg},X_{dg})\otimes \QQ\cong \bigoplus_{r\neq 0}H_{dg}$.

    Observe that $p: X_{dg}\to X_g$ sends the generator of $\pi_1 (X_{dg})$ to $t^d\in \pi$, and hence $Lp$ factors as
    \[
    \begin{tikzcd}[column sep = 40pt]
    Lp: LX_{dg}=\coprod_{s\in \ZZ} L_sX_{dg}\rar["\coprod_{s\in \ZZ}L_sp"] &\coprod_{s\in \ZZ} L_{ds}X_g\rar[hook] &\coprod_{r\in \ZZ}L_rX_g=LX_g.
    \end{tikzcd}
    \]
    Thus, the connected component $L_rX_g$ is hit by $Lp$ if and only if $d$ divides $r$. This proves claim (i).

    It remains to understand the Becker--Gottlieb transfer of the cover $L_ip: L_i X_{dg}\to L_{di}X_g$ on $H_n(-,\QQ)$. Note that $(L_i p)^!: H_n(L_{di}X_g;\QQ)\to H_n(L_iX_{dg};\QQ)$ is readily seen to be injective since, by formal properties of the Becker--Gottlieb transfer \cite[Eq. 3.4]{BeckerGottlieb}, the composition $(L_ip)_*\circ (L_ip)^!$ is given by multiplication by $d$, which is invertible over $\QQ$. Recall that from our computation of $H_n(LX_g;\QQ)$ in \cref{PsiLemma} using the spectral sequence \eqref{LrXgSS}, it follows that the map $(\mathrm{ev}_1)_*: H_n(L_rX_g;\QQ)\to H_n(X_g;\QQ)$ is an isomorphism for all $r\in \ZZ$. Moreover, by \eqref{BGtransfer1}, the left subsquare in the diagram
    \[
    \begin{tikzcd}[row sep = 10pt]
    H_n(L_{di}X_g;\QQ)\rar["(\mathrm{ev}_1)_*"]\ar[dd,"(L_ip)^!"] & H_n(X_g;\QQ)\ar[dd,"p^!"]\rar["\cong"] &H_g\dar["\Delta"]\\
    &&\bigoplus_{j=0}^{d-1}H_g\dar[equal]\\
    H_n(L_{i}X_{dg};\QQ)\rar["(\mathrm{ev}_1)_*"] & H_n(X_{dg};\QQ)\rar["\cong"] &H_{dg}
    \end{tikzcd}
    \]
    is commutative. The right subsquare is easily seen to commute, yielding claim (ii).
\end{proof}

To finish our computation of the cyclotomic structure on $\pi_{n-2}^{\QQ}(\pEmb(X_g))$, by \eqref{iotaCeq}, it remains to understand the action of the $h$-cobordism involution $\iota_{\H}$ on $\pi_{n-2}^{\QQ}(\CEmbo(X_g\times I^{2m}))\cong \bigoplus_{r\neq 0}H_g$. Since the map $r:(\Omega \C(X_g), \iota_{\C})\to (\Omega \CEmbo(X_g),\iota_{\H})$ is anti-equivariant by \cref{iotaCvsiotaHCor} and surjective by \cref{rProp}, it suffices to equip $\bigoplus_{r\neq 0}H_g$ with an involution for which the map $\tau_*$ is anti-equivariant. For this, recall the free loop space involution $\tau_{\FLS}$ on $LX_g$ induced by complex conjugation on $S^1\subset \CC$.

\begin{lemma}\label{TauVsInvolutionLemma}
    Let $M$ be a compact manifold. The map $\tau_*: \pi_*(\C(M))\longrightarrow \pi_{*+2}^s(LM,M)$ is $C_2$-equivariant for $\iota_{\C}$ on the domain and $\tau_{\FLS}$ on the codomain. 
\end{lemma}

\begin{proof}
    As explained in \cref{taumapconstruction}, we may assume that $M$ is stably framed. The concordance involution $\iota_{\C}$ sends a concordance diffeomorphism $F=(f,g):M\times I\to M\times I$ to $\overline{F}=(\overline{f}, \overline{g})$, where
    \[
    \overline{f}(x,t)= f_1^{-1}\circ f(x,1-t), \qquad \overline{g}(x,t)=1-g(x,1-t), \qquad (x,t)\in M\times I,
    \]
    and $f_1=f\mid_{M\times 1}: M\cong M$. It is straightforward to check that if $(x,s,t)\in M\times I\times I$ is a crossing for $F$ (in the sense of \cref{taumapconstruction}), then $(x,1-t,1-s)$ is a crossing for $\overline{F}$. It follows that if $F: D^k\times M\times I\to M\times I$ is a family of concordance embeddings representing a homotopy class in $\pi_k C(M)$ (and which satisfies the generic transversality assumtion), then there is a diffeomorphism
    \[
    \xi: (\mathcal{M}_F,\partial \mathcal{M}_F)\longrightarrow (\mathcal{M}_{\overline{F}},\partial \mathcal{M}_{\overline{F}}), \quad (z,x,s,t)\longmapsto (z,x,1-t,1-s).
    \]
    Moreover, if $(z,x,s,t)\in \mathcal{M}_F$, then the loop $q_{\overline{F}}(z,x,1-t,1-s)$ is $f^{-1}_1(\overline{q_F(z,x,s,t)})$. But the diffeomorphism $f_1$ is homotopic to the identity (via $\{f\mid_{M\times r}\}_{0\leq r\leq 1}$), and hence $q_F$ and $\tau_{\FLS}\circ q_{\overline{F}}\circ \xi$ are homotopic.
\end{proof}

The upshot of this section is the following.

\begin{prop}\label{FrobXgPseudoIsotopyCor}
    For each $d\in \NN$, the Frobenius map
    \[
    \varphi_d: \pi_{n-1}^{\QQ}\big(\bEmbmodEmb(X_g)\big)\longrightarrow \pi_{n-1}^{\QQ}\big(\bEmbmodEmb(X_{dg})\big)
    \]
    is isomorphic to
      \[
    \begin{tikzcd}
\bigoplus_{r> 0} H_g^{(r)}\ar[rr,"\bigoplus_{r>0} \alpha_r"] && \bigoplus_{s> 0} H_{dg}^{(s)},
\end{tikzcd}
    \]
    where the subscripts $(-)^{(r)}$ and $(-)^{(s)}$ only serve bookkeeping purposes, and $\alpha_r$:
    \begin{itemize}
        \item[(i)] is zero if $d$ does not divide $r$;
        \item[(ii)] is injective if $r=di$, and factors as
        \[
        \begin{tikzcd}
        H_{g}^{(di)}\rar["\Delta"] & \bigoplus_{j=0}^{d-1} H_g\cong H_{dg}^{(i)}\rar[hook] &\bigoplus_{s> 0} H_{dg}^{(s)}.
        \end{tikzcd}
        \]
    \end{itemize}
\end{prop}
\begin{proof}
    The isomorphism $\pi_{*+1}(\bEmbmodEmb(X_g))\cong \pi_*(\pEmb(X_g))$ is Frobenius-preserving, so it suffices to prove the claim for the Frobenius of the latter. Under the isomorphism $H_n(LX_g,X_g;\QQ)\cong \bigoplus_{r\neq 0} H_g$, $-\tau_{\FLS}$ on the left becomes the involution on the right-hand side that identifies $H_g^{(r)}$ with $H_g^{(-r)}$ via the $(-1)$-map. Thus, by \cref{TauVsInvolutionLemma} and \eqref{iotaCeq} (noting $\iota_{\C}=-\iota_{\H}$), we recover our isomorphism \(\pi_{n-2}^{\QQ}(\pEmb(X_g))\cong \bigoplus_{r>0} H_g\). The claim now follows from the second part of \cref{tauFrobeniusPreservingProp} and because all the maps in
\[
\begin{tikzcd}
    \pi_{n-2}^{\QQ}(\C(X_g))\rar[two heads, "r"] &\pi_{n-2}^{\QQ} (\CEmbo(X_g))\rar[two heads, "\alpha"] & \pi_{n-2}^\QQ(\pEmb(X_g))
\end{tikzcd}
\]
are Frobenius-preserving.
\end{proof}

There are two additional pieces of data that, together with the Frobenii, will be useful to keep track of:
\begin{enumerate}[label = (\alph*)]
    \item There is an embedding $\iota_0 : X_g \hookrightarrow X_{dg}$ which, under the evident homotopy equivalences $X_g \simeq S^1 \vee (\vee^{2g} S^n)$ and $X_{dg} \simeq S^1 \vee (\vee_{i=0}^{d-1} (\vee^{2g} S^n))$, lies in the homotopy class of the map $S^1 \vee (\vee^{2g} S^n) \to S^1 \vee (\vee_{i=0}^{d-1} (\vee^{2g} S^n))$ that is the identity on the $S^1$ summand, and on the other summand is the canonical inclusion $\vee^{2g} S^n \to \vee_{i=0}^{d-1} (\vee^{2g} S^n)$ as the $i=0$ summand. See \cref{XginXdgEmbFig}. This embedding gives rise to a map of fibre sequences
    \begin{equation}\label{iota0map}
    \begin{tikzcd}
        \pEmb(X_g)\rar\dar["(\iota_0)_*"] & \Embsfr(X_g)\rar\dar["(\iota_0)_*"] & \bEmbsfr(X_g)\dar["(\iota_0)_*"]\\
        \pEmb(X_{dg})\rar & \Embsfr(X_{dg})\rar& \bEmbsfr(X_{dg}),
    \end{tikzcd}
    \end{equation}
    where, the middle and right vertical maps are induced by extension by the identity on the complement of the embedding $\iota_0$. We would like to understand the map 
    \[
    (\iota_0)_*: \pi_{n-2}^{\QQ}(\pEmb(X_g))\longrightarrow \pi_{n-2}^{\QQ}(\pEmb(X_{dg})).
    \]
\end{enumerate}
\begin{figure}[h]
    \centering
    \includegraphics[scale=0.10]{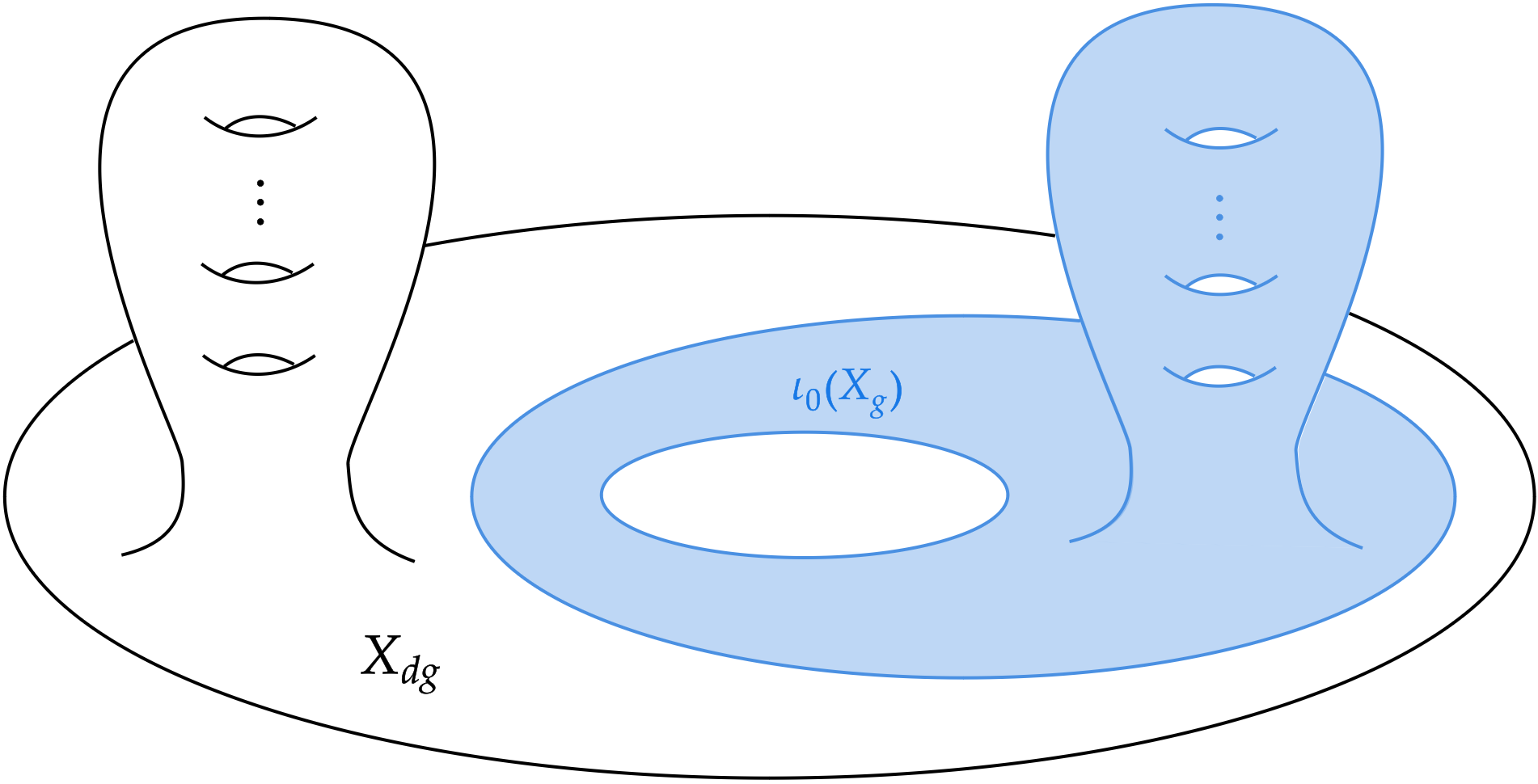}
    \caption{Depiction of the image of $\iota_0:X_g\hookrightarrow X_{dg}$ for $d=2$ and $2n=2$.}
    \label{XginXdgEmbFig}
\end{figure}
\begin{enumerate}[label = (\alph*)]\setcounter{enumi}{1}
    \item Conjugation by the rotation $C_d$-action on the manifold $X_{dg}$ gives rise to another $C_d$-action on the bottom fibre sequence in \eqref{iota0map}. We should keep track of the $C_d$-action on $\pi_{n-2}^{\QQ}(\pEmb(X_{dg}))$.
\end{enumerate}

We finish this section by recording this data.

\begin{lemma}\label{ExtraDataLemma}
    Under the isomorphism $\pi_{n-1}^{\QQ}\big(\bEmbmodEmb(X_{dg})\big)\cong \bigoplus_{r>0} H_{dg}$ of \cref{FrobXgPseudoIsotopyCor},
    \begin{itemize}
        \item[(a)] the map $(\iota_0)_*: \pi_{n-1}^{\QQ}\big(\bEmbmodEmb(X_g)\big)\to \pi_{n-1}^{\QQ}\big(\bEmbmodEmb(X_{dg})\big)$ becomes
        \[
        \begin{tikzcd}[column sep = 10pt]
            \bigoplus_{r>0}H_g\ar[rrr,hook, "j=0"] &&&\bigoplus_{r>0} \bigoplus_{j=0}^{d-1}H_g\rar[equal] & \bigoplus_{r>0}H_{dg};
        \end{tikzcd}
        \]
        
        \item[(b)] the $C_d$-action on the group $\pi_{n-1}^{\QQ}\big(\bEmbmodEmb(X_{dg})\big)$ is the diagonal $C_d$-action on $\bigoplus_{r>0} H_{dg}$, where $C_d$ acts on $H_{dg}\cong \bigoplus_{j=0}^{d-1}H_g$ by permuting the summands.
    \end{itemize}
\end{lemma}

\begin{proof}
    Both (a) and (b) are almost immediate from the work we have already done: for (a), note that there is a commutative diagram
    \[
    \begin{tikzcd}[row sep = 15pt]
        \pi_* (\C(X_g))\rar["\tau_*"]\dar["(\iota_0)_*"] & \pi_{*+2}^s(LX_g,X_g)\dar["(\iota_0)_*"]\\
        \pi_* (\C(X_{dg}))\rar["\tau_*"] & \pi_{*+2}^s(LX_{dg},X_{dg}),
    \end{tikzcd}
    \]
    where the right vertical map $(\iota_0)_*$ is the one induced by the map of pairs $(L\iota_0,\iota_0): (LX_g,X_g)\to (LX_{dg},X_{dg})$. Now in our identification $H_{n}(LX_g,X_g;\QQ)\cong \bigoplus_{r\in \ZZ}H_g$, the $r$-th summand corresponds to the $n$-th rational homology of the pair $(L_rX_g,X_g)$, and the map $(\iota_0)_*$ on it behaves as the inclusion $H_g\xhookrightarrow{} \bigoplus_{j=0}^{d-1}H_g\cong H_{dg}$ of the $j=0$-summand. Then, the claim follows by the same argument as in the previous proof. 

    For (b), again note that the map $\tau_*: \pi_*(\C(X_{dg}))\to \pi_{*+2}^s(LX_{dg}, X_{dg})$ is $C_d$-equivariant and, under the isomorphism $H_{n}(LX_{dg},X_{dg};\QQ)\cong \bigoplus_{r\in \ZZ}H_{dg}$, the $C_d$-action becomes the diagonal action described in the statement. The same line of reasoning as before applies.
\end{proof}

\subsubsection{The \texorpdfstring{$\ZN$}{Z[N]}-module structure on block embeddings of \texorpdfstring{$X_g$}{Xg}}\label{FrobBlockXgSection}
Recall from \cite[Cor. 6.2]{BRW} that there is a map
$$
\BbEmbsfr(X_g)_\ell\xrightarrow{\simeq_\QQ} \BAut_\partial^{\cong}(X_g)
$$
which is a rational equivalence; this map is also Frobenius-preserving, so it will be enough for us to study these in the codomain. Recall from \cite[Eq. (2)]{BRW} that there is a fibration sequence
\begin{equation}\label{MapXgFibSeq}
\begin{tikzcd}
    \Map_\partial(X_g,X_g)\rar & \Map_{S^1}(X_g,X_g)\rar & \Map_{S^1}(S^1\times S^{2n-2}, X_g),
\end{tikzcd}
\end{equation}
where the decoration on the middle and right mapping spaces demands maps to be fixed on $S^1\times *\subset S^1\times S^{2n-2}=\partial X_g$. This fibration sequence is compatible with the lifting Frobenius maps, and since the fibre inclusion is injective in rational homotopy groups, we will only need to understand the Frobenii on the rational homotopy groups of the middle term.

\begin{lemma} \label{MapS1FrobLemma}
Fix a $d$-th root of unity $\zeta\in S^1$. Then, there is a commutative diagram
\[
\begin{tikzcd}[column sep = 70pt]
\Map_{S^1}(\xg,\xg)\ar[r, "\varphi_d"]\dar["(\iota_0)_*"] & \Map_{S^1}(X_{dg}, X_{dg})\\
    \Map_{*}(X_{dg}, X_{dg})\rar["{(1,\zeta,\dots, \zeta^{d-1})}"]&\prod_{i=0}^{d-1}\Map_{S^1}(X_{dg}, X_{dg}),\uar["{\circ}"]
\end{tikzcd}
\]
where $(\iota_0)_*$ denotes the "extension-by-the-identity" map induced by the embedding $\iota_0: X_g\hookrightarrow X_{dg}$ of \cref{XginXdgEmbFig}, the horizontal bottom maps are conjugation by the rotation action of $\zeta^i$, and the right vertical map stands for composition of self-maps. Thus, on homotopy groups, the $d$-th Frobenius is given by
\[
\varphi_{d}(x)=\sum_{i=0}^{d-1}t^i\cdot(\iota_0)_*(x), \quad x\in \pi_{*}^\QQ\Map_{S^1}(X_{dg},X_{dg})\cong H_{g}\otimes \pi_{*+n}^{\QQ}(X_g).
\]
Moreover, it is injective.

\end{lemma}

\begin{proof} There is a commutative diagram
\begin{equation}\label{ZetaEq}
\begin{tikzcd}[column sep = 50pt]
    \Map_{S^1}(\xg,\xg)\ar[rr, "\varphi_d"]\dar["\vsim", "\mathrm{res}"'] && \Map_{S^1}(X_{dg}, X_{dg})\\
    \Map_{*}(\vee^{2g}S^n, X_{g})\rar["(\iota_0)_*"]&\Map_{*}(\vee^{2g}S^n, X_{dg})\rar["{(1,\zeta,\dots, \zeta^{d-1})}"]&\prod_{i=0}^{d-1}\Map_{\zeta^{i}}(\vee^{2g}S^n, X_{dg})\uar["\vsim"', "e"],
\end{tikzcd}
\end{equation}
where $(\iota_0)_*$ denotes postcomposition by $\iota_0: X_g\to X_{dg}$, $\Map_{\zeta^i}$ stands for the mapping space that sends the wedge point $*\in \vee^{2g}S^n$ to $\zeta^i\in S^1\subset X_g$, the left vertical map ``res'' is given by restriction to the core of the $n$-handles of $X_g$, the vertical map $e$ sends $(f_0,\dots, f_{d-1})$ to the map $f:X_{dg}\to X_{dg}$ which is the identity on $S^1$ and is $f_i$ when restricted to the $i$-th wedge of cores $\vee_{\zeta^i}^{2g}S^n\subset X_{dg}$, and finally the lower horizontal maps $\zeta^i$ are given by post-composition by the rotation $C_d$-action on $X_{dg}$. One immediately verifies that the bottom composition in \eqref{ZetaEq} coincides with the one in the diagram of the statement. The last claim in the statement holds since the bottom composition in \eqref{ZetaEq} is visibly injective on homotopy groups. 
\end{proof}

\begin{cor}\label{BbEmbFrobXgProp}
    The $d$-th Frobenius on $\pi_*^{\QQ}(\BbEmbsfr(X_g)_\ell)$ is injective and is given by
    \[
    \varphi_{d}(x)=\sum_{i=0}^{d-1}t^i\cdot (\iota_0)_*(x), \quad x\in \pi_*^{\QQ}(\BbEmbsfr(X_g)_\ell),
    \]
    where $(\iota_0)_*$ is as in \cref{iota0map} and $t^i$ is conjugation by the rotation action of $C_d=\langle t\rangle/\langle t^d\rangle$ on $X_{dg}$.
\end{cor}
\begin{proof}
    This follows from \cref{MapS1FrobLemma} and since $\Map_{\partial}(X_g,X_g)\to \Map_{S^1}(X_g,X_g)$ is injective on $\pi_*^{\QQ}$.
\end{proof}

\subsection{Homotopy of embeddings spaces of \texorpdfstring{$\xg$}{Xg}}\label{HomotopyOfEmbeddingsXgSection}

As mentioned at the beginning of this section, we consider the fibre sequence \eqref{BEmbsfrFibreSeq} and use it to compute the rational homotopy groups of $\BEmbsfr(X_g)_\ell$ in the range of degrees $\leq 2n-2$. In this range, we have
    \[
        \pi^{\QQ}_{*\geq 2}\left(\BbEmbsfr(X_g)_\ell\right)\cong \left\{ \begin{array}{cl}
             \ker\left([-,-]_\pi:\pi_n^\QQ(\xg)\otimes_\pi\pi_{2n-1}^\QQ(\xg)\to [\pi_{3n-2}^\QQ(\xg)]_\pi\right)    & *=n, \\[4pt]
             \ker\left([-,-]_\pi:\pi_n^\QQ(\xg)\otimes_\pi\pi_{3n-2}^\QQ(\xg)\to [\pi_{4n-3}^\QQ(\xg)]_\pi\right) & *=2n-1, \\[4pt]
             0 & \text{otherwise,}
        \end{array}\right.
    \]where $[-,-]_\pi$ is the map obtained by taking $\pi$-coinvariants of the Whitehead bracket map $[-,-]:\pi_{p-1}(\xg)\otimes \pi_{q-1}(\xg)\to \pi_{p+q-1}(\xg)$, which is $\pi$-equivariant using the diagonal action on the source, by \cite[Thm. 3.1]{BRW} (see \cref{homotopy groups of mapping of xg}) and \cref{stable framed block embs are autos}. 
    
    We observe that this isomorphism is equivariant with respect to the action of $\Embmcgsfr\cong \pi_1(\BbEmbsfr(X_g)_\ell)$ on its higher homotopy groups (i.e. the lefthand side) and the diagonal action on the tensor product $\pi_n^\QQ(X_g)\otimes_\pi\pi_{2n-1}^\QQ(X_g)$ by the following argument: since this isomorphism is the composition of the induced map of $\BEmbsfr(X_g)_\ell\to \BAut_\partial(X_g)$ with the isomorphism in \cref{homotopy groups of mapping of xg}, it suffices to prove that the action of $\pi_1(\BAut_\partial(X_g))\cong \pi_0(\Aut_\partial(X_g))$ on its higher rational homotopy groups agrees with the diagonal action in $\smash{\pi_n^\QQ(X_g)\otimes_\pi \pi_{n+*}^\QQ(X_g)} $ under the isomorphism of \cref{homotopy groups of mapping of xg}. By loc.cit, it suffices to show that the latter action agrees with the conjugation action in loc.cit(i) (and the fact that, for any group-like topological monoid, the action of $\pi_0(G)$ on $\pi_k(\B G)$, under the isomorphism $\pi_k(\B G)\cong \pi_{k-1}(G)$, is the conjugation action), on $\Hom_{\QQ[\pi]}(\pi_n^\QQ(X_g),\pi_{n+*}^\QQ(X_g))$ under the isomorphism
    \[\pi_n^\QQ(X_g)\otimes_\pi \pi_{n+*}^\QQ(X_g)\to \Hom_{\QQ[\pi]}(\pi_n^\QQ(X_g),\pi_{n+*}^\QQ(X_g))\]taking $x\otimes y$ to $\lambda_X(-,x)\cdot y$. This follows from the fact that any $\phi\in \Aut_\partial(X_g)$ preserves the form $\lambda_X$.

    To compute the homotopy groups of the total space of \eqref{BEmbsfrFibreSeq}, we must understand the connecting map between the homotopy groups of the base and the fibre in degree $n$. The next lemma considers this problem.

    \begin{lemma}\label{ConnectingMapsAreZero}
        The connecting map
        \(\delta_n: \pi_n^{\QQ}\left( \BbEmbsfr(X_g)_\ell\right)\to \pi_{n-1}^{\QQ}\big(\bEmbmodEmb(X_g)\big)\)
        is zero. 
    \end{lemma}

    \begin{proof}
        First observe that, by \cref{FrobXgPseudoIsotopyCor}(i), the Frobenii on the target of $\delta_n$ are such that for any $y\in \pi_{n-1}^{\QQ}\big(\bEmbmodEmb(X_g)\big)$ there exists some $d\in \NN$ such that $\varphi_d(y)=0$. Thus, given $x\in \pi_n^{\QQ}\left(\BbEmbsfr(X_g)_\ell\right)$, let $d\in\NN$ be such that $\varphi_{d}(\delta_n(x))=0$. Since the Frobenii fit in maps of fibre sequences
        \begin{equation}\label{varphidFibSeqMap}\begin{tikzcd}
        \bEmbmodEmb(X_g)\rar\dar["\varphi_d"] & \BEmbsfr(X_g)_\ell\rar\dar["\varphi_d"] & \BbEmbsfr(X_g)_\ell\dar["\varphi_d"]\\
        \bEmbmodEmb(X_{dg})\rar & \BEmbsfr(X_{dg})_\ell\rar& \BbEmbsfr(X_{dg})_\ell,
    \end{tikzcd}\end{equation}
    it follows that $\delta_n$ commutes with them, so \(0=\varphi_{d}(\delta_n(x))=\delta_n(\varphi_d(x))=\sum_{i=0}^{d-1}\delta_n\left(t^i\cdot (\iota_0)_*(x)\right)
    \), where the last equality is \cref{BbEmbFrobXgProp}. Note that $\delta_n$ is $C_d$-equivariant, for the bottom fibre sequence in \eqref{varphidFibSeqMap} is so, and by \eqref{iota0map}, we also get that $(\iota_0)_*$ and $\delta_n$ commute. All in all, we get
    \[
    0=\sum_{i=0}^{d-1}t^i\cdot(\iota_0)_*(\varphi_d(x)).
    \]
    By \cref{ExtraDataLemma}, this can only happen if $\varphi_d(x)=0$. But since $\varphi_d$ on $\pi_n^{\QQ}\left(\BbEmbsfr(X_g)_\ell\right)$ is injective by \cref{BbEmbFrobXgProp}, we must have that $x=0$, establishing the claim.
    \end{proof}

    From the long exact sequence of \eqref{BEmbsfrFibreSeq}, we obtain the following computation of the homotopy groups of the embedding space in question.

    \begin{prop}\label{HomotopyBEmbsfr}
        The rational homotopy groups of $\BEmbsfr(\xg)_\ell$ in degrees $2\leq *\leq 2n-2$ are given as $\QQ[\mcgRhoXgsfr]$-modules by
        \[\pi_{*}^{\QQ}\left(\BEmbsfr(X_g)_\ell\right)\cong \left\{ \begin{array}{cl}
             \bigoplus_{r> 0} H_g, & *=n-1 \\[4pt]          
             \ker\left([-,-]_\pi:\pi_n^\QQ(\xg)\otimes_\pi\pi_{2n-1}^\QQ(\xg)\to [\pi_{3n-2}^\QQ(\xg)]_\pi\right),    & *=n \\[4pt]
             \coker(\delta_{2n-1}), & *=2n-2\\[4pt]
             0 & \text{otherwise.}
        \end{array}\right.\]
    \end{prop}

\subsection{Homology of embedding spaces of \texorpdfstring{$\xg$}{Xg}}\label{homotopyOfPlusSectionXg} In this section, we wish to calculate the homotopy groups of the space $\BEmbsfr(\xg)_\ell^+$ in a range. To do so, we start by computing the homology of the universal cover of the latter space before taking the plus-construction and colimit over $g$, which we denote by $\BEmbsfr(\xg)_\ell^\sim$. We start by setting notation to simplify our exposition.

\begin{nota}\label{l m n notation}
    Recall \cref{HomotopyBEmbsfr} and let $L$, $M$ and $N$ denote the following $\QQ[\mcgRhoXgsfr]$-modules:
    \[
    L\coloneq \pi_{n-1}^\QQ\left(\BEmbsfr(X_g)_\ell\right), \qquad M\coloneq \pi_{n}^\QQ\left(\BEmbsfr(X_g)_\ell\right), \qquad N\coloneq \pi_{2n-2}^\QQ\left(\BEmbsfr(X_g)_\ell\right).
    \]
    \comment{
    \begin{align*}
        L &\coloneq \pi_{n-1}^\QQ\left(\BEmbsfr(X_g)_\ell\right)\\
        M &\coloneq \pi_{n}^\QQ\left(\BEmbsfr(X_g)_\ell\right)\\
        N &\coloneq \pi_{2n-2}^\QQ\left(\BEmbsfr(X_g)_\ell\right)
    \end{align*}}
    
\end{nota}

We start by noticing that the Postnikov truncation $\tau_{\leq 2n-3}\BEmbsfr(\xg)_\ell^\sim$ is given by $K(L,n-1)\times K(M,n)$ since there are no possible $k$-invariants for degree reasons. Thus, by looking at the Serre spectral sequence of the map from $(2n-2)$-nd Postnikov truncation to the $(2n-3)$-rd of $\BEmbsfr(\xg)_\ell^\sim$, we obtain: 

\begin{prop}\label{HomologyUniversalCOver}
    The rational homology groups of $\BEmbsfr(\xg)_\ell^\sim$ as $\QQ[\mcgRhoXgsfr]$-modules in degrees $*\leq 2n-2$ are given by 
    \[\widetilde{\H}_{*}\left(\BEmbsfr(X_g)_\ell^\sim;\QQ\right) \cong \left\{ \begin{array}{cl}
             L & *=n-1, \\[4pt]          
             M   & *=n, \\[4pt]
             \sn(L)\oplus \coker(d:L\otimes M\to N), & *=2n-2, \\[4pt]
             0 & \text{otherwise,}
        \end{array}\right.\]where $d$ is a certain map.
\end{prop}

\begin{lemma}\label{plusLemma}
     Let $G$ be a group, $k\geq 1$ and let $f:X\to Y$ be a $G$-equivariant map of simple $G$-spaces\footnote{That is, simple spaces equipped with a $G$-action.} which induces an isomorphism on $\H_*(-;\QQ)$ for $*\leq k$ and a surjection for $*=k+1,k+2$. Moreover, assume that there is an isomorphism of $\QQ[G]$-modules $\H_{k+1}(X;\QQ)\cong\H_{k+1}(Y;\QQ)\oplus A$, where the map to the first factor is $\H_{k+1}(f;\QQ)$ and $A$ is a $\QQ[G]$-module. Then, $f$ induces an isomorphism on $\pi_*^\QQ(-)$ for $*\leq k$ and $\pi_{k+1}^\QQ(X)\cong \pi_{k+1}^\QQ(Y)\oplus A$ as $\QQ[G]$-modules.
\end{lemma}
\begin{proof}
    We start by proving the following subclaim: for any integer $i$ and a map $\phi:\H_n(X;\QQ)\to V$ of $\QQ$-vector spaces, there exists a map $\Phi:X\to K(V,n)$ such that $\phi=\text{hur}^{-1}\circ \H_n(\Phi;\QQ)$ where $\text{hur}^{-1}:\H_n(K(V,n);\QQ)\to \pi_n^\QQ(K(V,n))\cong V$ is the inverse of the rational Hurewicz map. To prove this, recall that the Kronecker pairing map $\kappa:\H^n(X;V)\to \Hom_\QQ(\H_n(X;\QQ),V)$ is an isomorphism, by the universal coefficient theorem. Moreover, a class in $\H^n(X;V)$ determines a unique homotopy class of maps $X\to K(V,n)$. Choose $\Phi$ to correspond to the class that maps to $\phi$ under $\kappa$. One sees that the induced map $\H_n(\Phi;\QQ)$ is precisely $\text{hur}\circ \phi$.   
    
    Using this subclaim, we can construct the map $F=(f,\Phi):X\to Y\times K(A,k+1)$ whose induced map on $\H_{k+1}(-;\QQ)$ witnesses the isomorphism $\H_{k+1}(X;\QQ)\cong\H_{k+1}(Y;\QQ)\oplus A$ in the statement. Thus, $F$ induces an isomorphism on $\H_*(-;\QQ)$ for $*\leq k+1$. Moreover, as $\H_{k+2}(K(A,k+1);\QQ)$ vanishes (since $k\geq 1$), the map $F$ induces a surjection on $\H_{k+2}(-;\QQ)$. On the other hand, as both source and target of $F$ are simple spaces, we deduce from \cite[Cor. 3.4]{zeemancomparison} for $C$ the class of torsion groups that $F$ is rationally $(k+2)$-connected.

    To deal with the $G$-equivariance, note that by construction of $F=(f,\Phi)$, and since $f$ itself is $G$-equivariant, it suffices to show that the map $\pi_{k+1}^{\QQ}(\Phi)$ is $G$-equivariant: this map fits in a commutative diagram
    \[
\begin{tikzcd}
    \pi_{k+1}^{\QQ}(X)\dar["\mathrm{hur}"]\ar[rr,"\pi_{k+1}^{\QQ}(\Phi)"]&& \pi_{k+1}^{\QQ}(K(A,k+1))\dar["\mathrm{hur}", "\cong"']\rar[equal, "a"] &V\\
    \H_{k+1}(X;\QQ)\ar[rr,"\H_{k+1}(\Phi;\QQ)"] && \H_{k+1}(K(A,k+1);\QQ).&
\end{tikzcd}
    \]
   By construction, \(a\circ \mathrm{hur}^{-1}\circ H_{k+1}(\Phi;\QQ)\) agrees with the projection \(H_{k+1}(X;\QQ)\to A\), which is \(G\)-equivariant by assumption; since both \(a\) and the Hurewicz homomorphism are \(G\)-equivariant, the claim follows.
\end{proof}

Let $\firstxg$ denote the $\QQ[\brho]$-module given by
\[
\firstxg \coloneq \colim_{g}\, \bigl[\sn(L) \oplus \coker(d : L \otimes M \to N)\bigr]_{\Embmcgsfr}
\]where $d$ is the map in \cref{HomologyUniversalCOver}. Observe also that conjugation by $\rho = \left(\begin{smallmatrix}-1 & 0 \\ 0 & 1\end{smallmatrix}\right)^{\oplus g}$ (with respect to the $\ZZ[\pi]$-basis $\{a_1,b_1,\dots, a_g,b_g\}$) on $\Ugextl$ induces a $\brho$-action on the colimit $\BUgmininf$. Equipped with this action, the standard map $\BEmbmcginfs \to \BUgmininf$ is easily seen to be $\brho$-equivariant.

\begin{teo}\label{HomotopyPlusConstruction}
    In the range $*\leq 2n-2$, we have an isomorphism of $\QQ[\brho]$-modules
    \[\pi_{*}^\QQ\left(\BEmbsfr(X_\infty)_\ell^+\right) \cong \pi_*^\QQ\left(\BUgmininf^+\right) \oplus\left\{ \begin{array}{cl}
             \firstxg & *=2n-2, \\[4pt]
             0 & \text{otherwise.}
        \end{array}\right.\]
\end{teo}
\begin{proof}
    We want to use \cref{plusLemma} for the map of simple spaces $r:\BEmbsfr(X_\infty)_\ell^+\to \BUgmininf^+$, with $A=\firstxg$. To do so, first show that $r$ induces an isomorphism on $\H_*(-;\QQ)$ for $*\leq 2n-3$ and a surjection for $*=2n-2,2n-1.$ By definition, $r$ factors as $\BEmbsfr(X_\infty)_\ell^+\to (\BEmbmcginfs)^+\to \BUgmininf^+$, and the second map is a rational equivalence by \cref{stable mcg of xg is gw}. This reduces to showing that the first map induces an isomorphism on $\H_*(-;\QQ)$ for $*\leq 2n-3$ and a surjection for $*=2n-2,2n-1.$ Consider the Serre spectral sequence with rational coefficients for the universal cover fibre sequence
    \begin{equation}\label{universal cover fibre sequence for BEmb ifn}
        \BEmbsfr(X_\infty)_\ell^\sim\to \BEmbsfr(X_\infty)_\ell\to \BEmbmcginfs.
    \end{equation}Observe that the $\Embmcgsfr$-modules $L$ and $M$ from \cref{l m n notation} are gr-odd in the sense of \cref{gr odd modules}: recall that $L$ is a direct sum of $H_g$, which is odd, and $M$ is a submodule of $\pixq\otimes_\pi\pi_{2n-1}^\QQ(X_g)$, which is a tensor product of an odd module with an even module, hence odd by \cref{tensors and homs of odd/even}. We conclude from \cref{center kills for odd} that the homology groups of $\Embmcgsfr$ with coefficients in $L$ or $M$ vanish in all degrees. Using \cref{HomologyUniversalCOver}, we can then compute the groups $E^2_{s,t}$ of the spectral sequence of \eqref{universal cover fibre sequence for BEmb ifn} for $t\leq 2n-3$ in the following way: for $t=0$, then $E^2_{s,0}\cong \H_s(\BEmbmcginfs;\QQ)$. For $t=n-1$, we have $E^2_{s,n-1}$ is the homology of $\BEmbmcginfs$ in degree $s$ with coefficients in the colimit of the $\Embmcgsfr$-modules $L$ over $g$, which vanishes as it is the colimit of the groups $\H_s(\Embmcgsfr;L)$, which vanishes. Similarly, we see that $E_{s,n}^2$ vanishes for all $s\geq 0$. Otherwise, we see that $E^2_{s,t}$ vanishes, since $\smash{\widetilde{\H}_{*}\left(\BEmbsfr(X_g)_\ell^\sim;\QQ\right)}$ vanishes by \cref{HomologyUniversalCOver}. We conclude that for $*\leq 2n-3$, the map $\H_*(\BEmbsfr(X_\infty)_\ell;\QQ)\to \H_*(\BEmbmcginfs;\QQ)$ is an isomorphism, hence also for the induced map after plus-construction as the homology remains unchanged. 
    
    We argue now that the latter is surjective for $*=2n-2,2n-1,$ which uses the explicit computation of $\pi_*^\QQ((\BEmbmcginfs)^+)$ from \cite[1677]{BRW} as an input. Recall that the spaces $\smash{\BEmbsfr(X_\infty)^+_\ell}$ and $\smash{(\BEmbmcgsfr)^+}$ are equivalent to the identity component of the group completion of the monoids given by the disjoint unions of the spaces $\{\BEmbsfr(X_g)_\ell\}_{g\geq 0}$ and $\{\BEmbmcgsfr\}_{g\geq 0}$, respectively---in particular, they are path-connected. Thus, the map $\pi:\BEmbsfr(X_\infty)^+_\ell\to(\BEmbmcginfs)^+$ is equivalent to an $\EE_1$-map of connected $\EE_1$-spaces. However, by the Milnor-Moore theorem, we know that the $\QQ$-algebra $\H_*(\mathbb{X};\QQ)$ of a connected $\EE_1$-space $\mathbb{X}$ is isomorphic to the universal enveloping algebra of the Lie algebra $\pi_*^{\QQ}(\mathbb{X})$, induced by the $\EE_1$-structure on $\mathbb{X}.$ Thus, we deduce that for any $k\geq 0$, the vector space $\H_k(\mathbb{X};\QQ)$ is generated by products of classes in $\H_l(\mathbb{X};\QQ)$ for $l<k$ and $\pi_{k}^\QQ(\mathbb{X}).$ From \cite[1677]{BRW} (together with \cref{stable mcg of xg is gw}), it follows that $\pi_*((\BEmbmcginfs)^+)$ vanishes for $*=2n-2,2n-1$. Thus, the vector spaces $\H_{*}((\BEmbmcginfs)^+;\QQ)$ for $*\leq 2n-2,2n-1$ are generated by products of classes in degrees $l\leq 2n-3$. However, those are hit by the map $\H_l(\pi;\QQ)$ by the discussion above (and the fact that $\pi$ is an $\EE_1$-map). Thus, we conclude that the map $\H_*(\pi;\QQ)$ is surjective for $*\leq 2n-1.$

    We finish by showing that the kernel of the map $\H_{2n-2}(\pi;\QQ)$ is isomorphic to $\firstxg$, which establishes the whole claim by applying \cref{plusLemma}. Let us return to the Serre spectral sequence of \eqref{universal cover fibre sequence for BEmb ifn}. Observe that, by \cref{HomologyUniversalCOver}, $E^2_{0,2n-2}$ is isomorphic to $\firstxg$. In degree $2n-2$, there is only one possible non-trivial differential hitting $E^2_{0,2n-2}$, namely the $d^{2n-1}$ from $E^2_{2n-1,0}$. However, the surjective map $\H_{2n-1}(\BEmbsfr(X_\infty)_\ell;\QQ)\to \H_{2n-1}(\BEmbmcginfs;\QQ)\cong E^2_{2n-1,0}$ factors through the edge homomorphism $ E^\infty_{2n-1,0}\hookrightarrow E^2_{2n-1,0}.$ Hence the latter is surjective and thus an isomorphism. We conclude that there are no non-trivial differentials out of $E^2_{2n-1,0}.$ From that, we see that the kernel of $\H_{2n-2}(\pi;\QQ)$ is isomorphic to $E^{2}_{0,2n-2}$ and hence, isomorphic to $\firstxg$. Thus, we obtain an extension of $\QQ[\brho]$-modules
    \[
\begin{tikzcd}
    0\rar &\firstxg\rar & \H_{2n-2}(\BEmbsfr(X_\infty)_\ell^+;\QQ)\ar[rr, two heads, "\H_{2n-2}(\pi;\QQ)"] && \H_{2n-2}((\BEmbmcginfs)^+;\QQ)\rar & 0,
\end{tikzcd}
    \]
    which, by Maschke's theorem, splits since $\brho$ is a finite group. Moreover, this splitting can be chosen of the form $\H_{2n-2}(\BEmbsfr(X_\infty)_\ell^+;\QQ)\cong \H_{2n-2}((\BEmbmcginfs)^+;\QQ)\oplus \firstxg$, where the map to the first factor is $\H_{2n-2}(\pi;\QQ)$, as required in the assumption of \cref{plusLemma}. This finishes this proof.
\end{proof}

\section{Self-embedding spaces of \texorpdfstring{$(\yg;\xg)$}{(Yg;Xg)}} \label{SelfEmbeddingYgSection}

In this section, we study the rational homotopy type of $\BEmbsfr(\yg;\xg)_\ell$ in degrees $*\leq 2n-2$ (cf. \cref{HomotopyBEmbsfrY}), and hence that of the plus construction $\BEmbsfr(Y_\infty;X_\infty)^+_\ell$ in \cref{HomotopyPlusConstructionY}. We end with \cref{mainInputWatanabe}, which describes some features of $\BEmbsfr(Y_\infty;X_\infty)^+_\ell\to \BEmbsfr(X_\infty)^+_\ell$ on $\smash{\pi_{2n-2}^{\QQ}}$.

The approach we take to understand the rational homotopy groups of $\BEmbsfr(\yg;\xg)_\ell$ is completely analogous to that in the preceding section for $\BEmbsfr(\xg)_\ell$. Namely, we consider the fibre sequence
\begin{equation}\label{BEmbsfrFibreSeqYg}
    \begin{tikzcd}
       \bEmbmodEmb(\yg; X_g) \rar & \BEmbsfr(\yg;X_g)_\ell\rar["\iota"] &\BbEmbsfr(\yg;X_g)_\ell,
    \end{tikzcd}
\end{equation}
understand the left and right terms individually, and later worry about the connecting maps. For the left-hand term, we again appeal to the pseudoisotopy theory reviewed in \cref{RecollectionWWSection}. The splitting result \cref{splittingprop} reduces the analysis of its rational homotopy groups as representations of $\mcgRhoYg$, where $\mcgy\coloneq \pi_1(\BEmbsfr(\yg;X_g)_\ell)$ and $\rho$ is the reflection involution of \cref{ReflectionInvSection}, to the work already carried out in \cref{mcgXgsection}. This yields \cref{MCGactionYg}. In \cref{HomotopyOfEmbeddingsYgSection} we analyse the connecting maps in the long exact sequence of \eqref{BEmbsfrFibreSeqYg}, which requires understanding the Frobenius lifting maps on both base and fibre---the latter is treated in \cref{FrobeniusYgSection}.

\subsection{Pseudoisotopy self-embeddings of \texorpdfstring{$(Y_g;X_g)$}{(Yg,Xg)}} \label{PseudoisotopyEmbYgXgSection}

We dedicate this section to the computation of the rational homotopy groups of the pseudoisotopy self-embedding space $\bEmbmodEmb(Y_g;X_g)$ in degrees approximately $*\leq 2n-2$, including their description as $\QQ[\mcgRhoYgsfr]$-modules. Note that, as in the case of $X_g$, there is a map of $\brho$-equivariant fibre sequences
\[
\begin{tikzcd}
    \bEmbmodEmb(Y_g;X_g)\dar[equal]\rar&\BEmbsfr(Y_g;X_g)_\ell\rar\dar&\BbEmbsfr(Y_g;X_g)_\ell\dar\\
    \bEmbmodEmb(Y_g;X_g)\rar &\BEmbcong(Y_g;X_g)\rar &\BbEmbcong(Y_g;X_g),
\end{tikzcd}
\]
and hence the action of $\mcgRhoYgsfr$ on the homotopy groups of the fibre factors through the group homomorphism $\mcgRhoYgsfr\to \mcgRhoYg$, where $\mcgynosfr\coloneq \pi_1(\BEmbcong(Y_g;X_g))$, so we restrict ourselves to describing the action of the latter semidirect product.

Fix a self-embedding of $4$-ads $\iota\in \Embcong(Y_g;X_g)$ extending $\iota: X_g\hookrightarrow X_g$ as fixed in \cref{PlanpEmbXgSection}, and write $\pEmb(Y_g;X_g)$ for the homotopy fibre of the map $\Embcong(Y_g;X_g)\to \bEmbcong(Y_g;X_g)$ at $\iota$. It readily follows that there is an equivalence
\[
\Omega\bEmbmodEmb(Y_g;X_g)\simeq \pEmb(Y_g;X_g).
\]
Given $([\phi], \rho^i)\in \mcgRhoYg$, represent $[\phi]$ by a diffeomorphism $\phi$ of $Y_g$ fixing $\partial^vY_g=\partial_0Y_g\cup \partial_1 \yg$, and let $\psi\coloneq \phi \rho^i$. Conjugation by $\psi$ gives rise to an action of $\mcgRhoYg$ on the homotopy groups of $\pEmb(Y_g;X_g)$ similar to (and compatible with) the one in \cref{bEmbmodEmbAction}. Moreover, by an argument analogous to the one in \cref{pEmbvsbEmbmodEmbMCGActionXg}, it follows that this action recovers the one on the homotopy groups of $\bEmbmodEmb(Y_g;X_g)$ under the above equivalence. Moreover, this latter space is connected, for it fits in a fibre sequence
\[
\begin{tikzcd}
    \bEmbymodEmby(\yg) \rar & \bEmbmodEmb(\yg;\xg)\rar["\partial_2"] &\bEmbmodEmb(\xg),
\end{tikzcd}
\]
where both base and fibre are connected by Hudson's theorem. Consequently, we only need to describe the $\QQ[\mcgRhoYg]$-modules $\pi_*^{\QQ}(\pEmb(Y_g;X_g))$ (under the conjugation action just described).

To this end, looping the above fibre sequence yields another fibre sequence
\begin{equation}\label{PseudoisoFibreSeqYgXg}
\begin{tikzcd}
       \pEmby(\yg) \rar["j"] & \pEmb(\yg;\xg)\rar["\partial_2"] &\pEmb(\xg),
\end{tikzcd}    
\end{equation}
where the homotopy groups of the base and fibre admit natural actions of $\mcgRhoYg$ as follows: on the base, it is given by restricting the action of $\mcgRhoXg$ from \cref{bEmbmodEmbAction} along the homomorphism $\partial_2: \mcgRhoYg\to \mcgRhoXg$. On the fibre, it arises from the observation that conjugation by $\psi=\phi\rho^i$ still makes sense on the spaces $\Emb_{\partial_+}(Y_g)$ and $\bEmb_{\partial_+}(Y_g)$, and hence on the homotopy fibre space $\pEmby(Y_g)$. Both of these actions make the maps $j$ and $\partial_2$ in \eqref{PseudoisoFibreSeqYgXg} $\mcgRhoYg$-equivariant on homotopy groups.

\begin{prop}\label{splittingprop}
    In degrees $*\leq 3n-6$, there is a map
    \[
    \psigma: \pi_*^{\QQ}\left(\pEmb(X_g)\right)\longrightarrow \pi_*^{\QQ}\left(\pEmb(Y_g;X_g)\right)
    \]
    which is a $\mcgRhoYg$-equivariant section of $\pi_*^{\QQ}(\partial_2)$. In particular, there is an isomorphism of $\QQ[\mcgRhoYg]$-modules
    \[
    \pi_*^{\QQ}(j)+\psigma: \pi_*^{\QQ}\left(\pEmby(Y_g)\right)\oplus \pi_*^{\QQ}\left(\pEmb(X_g)\right)\xrightarrow{\cong} \pi_*^{\QQ}\left(\pEmb(Y_g;X_g)\right).
    \]
\end{prop}
\begin{rmk}
    This proposition also holds $\ZZ[\tfrac{1}{2}]$-locally---replace $\pi_*^{\QQ}(-)$ by $\pi_*(-)\left[\tfrac{1}{2}\right]$ in the proof below. In fact, an even more general statement is true: inverting $2$, there is an equivalence of truncated spaces
    \[
    \tau_{\leq 3n-6} \big(\pEmb(Y_g;X_g)\big)\simeq_{[\frac{1}{2}]} \tau_{\leq 3n-6} \big(\pEmby(Y_g) \times \pEmb(X_g)\big).
    \]
    This is a consequence of an analogue of the embedding Weiss--Williams theorem \eqref{WWCEMap} for relative embedding spaces, and the following general principle in orthogonal calculus: given an orthogonal functor $F: \cJ\to \Spc$, the first orthogonal derivative of the functor $F^{(1)}(-)\coloneq \hofib(F(-)\to F(-\oplus \RR))$ is equivalent to $\smash{\operatorname{Ind}_{e}^{\O(1)}\operatorname{Res}^{\O(1)}_{e}\Theta F^{(1)}\simeq_{[\frac{1}{2}]} \Theta F^{(1)}\oplus \Sigma^{\sigma-1}\Theta F^{(1)}}$. We will not need this, so we omit the details.
\end{rmk}

Before proceeding with the proof, let us introduce the \emph{relative concordance embedding space} $\CEmbo(Y_g;X_g)$, consisting of embeddings $\varphi: (Y_g;X_g)\times[0,1]\hookrightarrow (Y_g;X_g)\times [0,1]$ such that
\begin{itemize}
    \item[(i)] $\varphi^{-1}(Y_g\times \{i\})=Y_g\times \{i\}$ for $i=0,1$,
    \item[(ii)] $\varphi$ agrees with the identity on a neighbourhood of $Y_g\times\{0\}\cup \partial_0 Y_g\times [0,1]$, and
    \item[(iii)] $\varphi(X_g\times [0,1])\subset X_g\times [0,1]$.
\end{itemize}
Restricting such an embedding to $X_g\times [0,1]$ gives a map $\partial_2: \CEmbo(Y_g;X_g)\to \CEmbo(X_g)$. Moreover, there is another map $\alpha: \Omega \CEmbo(Y_g;X_g)\to \Omega\pEmb(Y_g;X_g)$ defined as the map induced on horizontal homotopy fibres by the right commutative subsquare in
\[
\begin{tikzcd}
\Omega\CEmbo(Y_g;X_g)\rar["\mathrm{res}_1"]\ar[dd, "\alpha"] & \Omega\Emb_{\partial_0}(Y_g;X_g)\ar[dd, equal]\rar["\Gamma"]&  \Emb_{\partial_0}(Y_g\times I;X_g\times I)\dar["\beta"]\\
&&\bEmbo(Y_g\times I;X_g\times I)\\
    \Omega\pEmb(Y_g;X_g)\rar["i"] & \Omega \Embo(Y_g;X_g)\rar["\beta"] &\Omega \bEmb_{\partial_0}(Y_g;X_g).\ar[u, "\widetilde{\Gamma}"', "\vsim"]
\end{tikzcd}
\]
There is a conjugation $\mcgRhoYg$-action on the homotopy groups of $\CEmbo(Y_g;X_g)$, and just like in \cref{ActionCompatibilityProp}, one easily verifies that $\alpha$ is $\mcgRhoYg$-equivariant for these actions.

\begin{proof}[Proof of \cref{splittingprop}]
    In this proof, every $\mcgRhoXg$-representation is regarded as a $\mcgRhoYg$-representation by restriction along $\mcgRhoYg\to\mcgRhoXg$. We first describe a space-level section of $\partial_2: \CEmbo(Y_g;X_g)\to \CEmbo(X_g)$ which is $\mcgRhoYg$-equivariant on homotopy groups for the conjugation action just described on the domain, and for the restricted conjugation $\mcgRhoXg$-action on the codomain described in \cref{CEmbAction}. For this, observe that the $4$-ad $(Y_g, \partial_0 Y_g, \partial_1 Y_g, \partial_2 Y_g)=(Y_g, S^1\times D^{2n-1}_+, S^1\times D^{2n-1}_-, X_g)$ is diffeomorphic to a $4$-ad on $\overline{Y}_g\coloneq Y_g\cup_{X_g\times \{0\}}X_g\times [0,1]$ given by
    \[
    \partial_0 \overline{Y}_g\coloneq S^1\times D^{2n-1}_+\cup \partial_{0}X_g\times [0,1], \quad  \partial_1 \overline{Y}_g\coloneq S^1\times D^{2n-1}_-\cup\partial_{1}X_g\times [0,1],\quad \partial_2 \overline{Y}_g\coloneq X_g\times \{1\}.
    \]
    (One has to smooth corners of $\overline{Y}_g$ appropriately, but let us ignore this issue.) Then, the section
    \[
    \sigma: \CEmbo(X_g)\longrightarrow \CEmbo(\overline{Y}_g;X_g)\simeq \CEmbo(Y_g;X_g)
    \]
    can be described as follows: given a concordance embedding $\varphi\in \CEmbo(X_g)$, i.e. a self-embedding $\varphi: X_g\times I\hookrightarrow X\times I$ such that $\varphi$ is the identity on $X_g\times \{0\}\cup \partial_0 X_g\times I$, left-stabilisation gives a self-embedding $\Sigma_\ell(\varphi): X_g\times [0,1]\times I\hookrightarrow X_g\times [0,1]\times I$ such that
    \begin{itemize}
        \item $\Sigma_\ell(\varphi)$ is pointwise fixed on $(X_g\times [0,1]\times \{0\})\cup (X_g\times \{0\}\times I)\cup (\partial_0 X_g\times [0,1]\times I)$, and
        \item $\Sigma_\ell(\varphi)\mid_{X_g\times \{1\}\times I}=\varphi$.
    \end{itemize}
    Thus, extending $\Sigma_\ell(\varphi)$ by the identity on $Y_g\subset \overline{Y_g}$ gives a point $\sigma(\varphi)\in \CEmbo(\overline{Y}_g;X_g)$. Moreover, by construction, we get that $r\circ\sigma(\varphi)\coloneq\sigma(\varphi)\mid_{X_g\times \{1\}\times I}=\varphi$, i.e., $\sigma$ is a right section of $\rho$ as desired. 

    To see that $\sigma_*$ is $\mcgRhoYg$-equivariant, observe that given $([\phi],\rho^i)\in\mcgRhoYg$, we can represent $[\phi]$ by some diffeomorphism $\phi$ on $\overline{Y_g}$ which, when restricted to $X_g\times [0,1]\subset \overline{Y}_g$, is given by a product diffeomorphism $\xi\times \Id_{[0,1]}$ for some diffeomorphism $\xi\in \Diff_{\partial}(X_g)$. Moreover, the involution $\rho$ of $Y_g$ extends to $\overline{Y}_g$ by setting it to be $\rho\mid_{X_g}\times \Id_{[0,1]}$ on $X_g\times [0,1]\subset\overline{Y}_g$. Then, since $\sigma(\varphi)$ is only supported on $X_g\times [0,1]\times I$, region on which $\phi\rho^i\times \Id_I$ is given by $\xi\rho^i\mid_{X_g}\times \Id_{[0,1]\times I}=\Sigma_\ell(\xi\rho^i\mid_{X_g}\times \Id_{I})$, it follows that conjugation with $\phi\rho^i\times \Id_{I}$ commutes with $\sigma$, establishing $\mcgRhoYg$-equivariance.

   Let us write $\lambda$ for the $\mcgRhoXg$-equivariant isomorphism
    \[
    \begin{tikzcd}
    \lambda:\pi_*^{\QQ}(\CEmbo(X_g))\rar["\mathrm{stab.}", "\cong"'] &\pi_*^{\QQ}(\sCE(X_g))\rar[dash, "\eqref{StableCEmbIETSeq}","\cong"'] & \pi_*^{\QQ}(\CEspo(X_g)), \quad *\leq 3n-6.
    \end{tikzcd}
    \]
    Since both $\pEmb(Y_g;X_g)$ and $\pEmb(X_g)$ are connected, we only need to describe the splitting map $\psigma$ of the statement in degrees $1\leq *\leq 3n-6$. In such degrees, it is given by the composition
    \begin{equation}\label{sigmasimsplitting}
    \begin{tikzcd}
        \pi_*^{\QQ}(\pEmb(X_g))\rar["\Phi^{\Emb}_{X_g}"]\dar[dashed, "\psigma"] & \pi_*^{\QQ}(\CEspo(X_g)_{hC_2})\rar["N"] & \pi_*^{\QQ}(\CEspo(X_g))\dar["\cong"', "\lambda^{-1}"]\\
        \pi_*^{\QQ}(\pEmb(Y_g;X_g)) & \pi_*^{\QQ}(\CEmbo(Y_g;X_g))\lar["\alpha"'] & \pi_*(\CEmbo(X_g))\lar["\sigma"'] 
    \end{tikzcd}
    \end{equation}
    where $N: A_{hC_2}\to A^{hC_2}\to X$ stands for the norm map of a $C_2$-spectrum $A$. All of the homomorphisms in this composition are $\mcgRhoYg$-equivariant: $\smash{\Phi^{\Emb}_{X_g}}$ by \cref{BoringAlphaProp}, the norm map $N$ because the $\mcgRhoXg$-action on $\smash{\pi_*^{\QQ}(\CEspo(X_g))}$ commutes with the $h$-cobordism involution $\iota_{\H}$ (as explained in \cref{CEmbAction}), the isomorphism $\lambda$ by definition, $\sigma$ by the argument above, and $\alpha$ by construction. To conclude, $\sigma^{(\sim)}$ is indeed a section of $\partial_2: \pi_*^{\QQ}(\pEmb(Y_g; X_g))\to \pi_*^{\QQ}(\pEmb(X_g))$ since
    \[
    \partial_2\circ \alpha\circ\sigma\circ \lambda^{-1}\circ N\circ \Phi^{\Emb}_{X_g}=\alpha\circ \lambda^{-1}\circ N\circ \Phi_{X_g}^{\Emb}=\Id,
    \]
    where, in the first equality we are using that $\partial_2$ and $\alpha$ commute by construction, and in the second equality we invoke the commutativity of the square in \cref{IntermediateBoringAlphaProp} and the fact that, $\ZZ[\tfrac{1}{2}]$-locally, the norm map $N: A_{hC_2}\to A$ is a section of the quotient map $q: A\to A_{hC_2}$.
\end{proof}

Therefore, for $*\leq 3n-5$, we only need to understand the $\QQ[\mcgRhoYg]$-modules
\[
\pi_*^{\QQ}\big(\bEmbymodEmby(Y_g)\big)\cong \left[\frac{\H^{S^1}_{*}(LY_g, LS^1; \QQ)}{\H_{*+1}(Y_g,S^1;\QQ)}\right]_{\tau_{\FLS}}, \quad \pi_*^{\QQ}\big(\bEmbmodEmb(X_g)\big)\cong \left[\frac{\H^{S^1}_{*}(LX_g, LS^1; \QQ)}{\H_{*+1}(X_g,S^1;\QQ)}\right]_{-\tau_{\FLS}},
\]
where both of these isomorphisms follow from \cref{involutionlemma} (notice the absence of a minus sign in front of $\tau_{\FLS}$ on the right). The $\QQ[\mcgRhoXg]$-module on the right-hand side of the right isomorphism was identified in \cref{ActionTheorem} to be induced by naturality on the pair $(X_g,S^1)$. Exactly the same analysis applies to the isomorphism on the left:

\begin{prop}\label{ActionThmYg}
    The action of $\mcgRhoYgsfr$ on the homotopy groups of the space $\bEmbymodEmby(Y_g)$ factors through the homomorphism $\mcgRhoYgsfr\to \mcgRhoYg$. Under the isomorphism 
\begin{equation}\label{YgbEmbmodEmbvsFLS}
\pi_*^{\QQ}\big(\bEmbymodEmby(Y_g)\big)\cong \left[\frac{\H^{S^1}_{*}(LY_g, LS^1; \QQ)}{\H_{*+1}(Y_g,S^1;\QQ)}\right]_{\tau_{\FLS}}, \quad *\leq 3n-3,
\end{equation}
the action of $([\phi], \rho^i)\in \mcgRhoYg$ (where $[\phi]$ is represented by a diffeomorphism $\phi$ of $Y_g$ fixed on $\vb Y_g$) on the left-hand side corresponds, on the right-hand side, to the map
\[
(\phi\rho^i)_*: \left[\frac{\H^{S^1}_{*}(LY_g, LS^1; \QQ)}{\H_{*+1}(Y_g,S^1;\QQ)}\right]_{\tau_{\FLS}}\overset\cong\longrightarrow \left[\frac{\H^{S^1}_{*}(LY_g, LS^1; \QQ)}{\H_{*+1}(Y_g,S^1;\QQ)}\right]_{\tau_{\FLS}}
\]
induced by functoriality on $(Y_g,S^1)$.
\end{prop}

It remains to compute the groups on the right-hand side of \eqref{YgbEmbmodEmbvsFLS} as representations of the homotopy mapping class group $\hMCGYg\coloneq \pi_0(\Aut_*(Y_g))\cong \pi_0(\Aut_*(Y_g, S^1))$. To this end, consider the $\QQ[\hMCGYg]$-module
\[
H_{Y_g}\coloneq\piy[-1]\otimes\QQ\cong H_n(\widetilde{Y}_g;\QQ)[-1],
\]
which lives in degree $n-1$, defined analogously to $H_{X_g}$. As for $\hMCGXg$, the fundamental group $\pi_1 Y_g=\pi\cong \ZZ\langle t\rangle$ is a normal subgroup of $\hMCGYg$, and the restricted $\QQ[\pi]$-module structure on $H_{Y_g}$ is induced by the usual action of $\pi_1 Y_g$ on its higher homotopy groups. As $\QQ[\pi]$-modules,
\[
H_{Y_g}\cong \QQ[\pi]\otimes H_{V_g},\quad \text{where $H_{V_g}\coloneq H_n(V_{g,1};\QQ)[-1]$.}
\]
Finally, we will also regard $H_{V_g}$ as a $\QQ[\hMCGYg]$-module via the obvious isomorphism $H_{V_g}\cong [H_{Y_g}]_{\pi}$. The following is a complete analogue of \cref{FLSS1homology}(iii), and its proof carries over verbatim.

\begin{prop}\label{YgXgpseudoisotopycomputation}
For $r\in \ZZ$, consider the graded $\QQ[\hMCGYg]$-modules
\[
A_r^{Y_g}\coloneq \frac{H_{Y_g}}{\langle 1-t^r\rangle}, \qquad B_r^{Y_g}\coloneq\frac{\sn(H_{Y_g})}{\sn(1-t^r)}.
\]
Then, in degrees $*\leq 2n-2$, there is an isomorphism of $\QQ[\hMCGYg]$-modules
\[
 \left[\frac{\H^{S^1}_{*}(LY_g, LS^1; \QQ)}{\H_{*+1}(Y_g,S^1;\QQ)}\right]_{\tau_{\FLS}}\cong\left\{\begin{array}{cl}
    \bigoplus_{r> 0} {[A_r^{Y_g}]_{\pi}}\cong \bigoplus_{r> 0} H_{V_g} & *=n-1, \\[6pt]
     C^{Y_g} \oplus \bigoplus_{r> 0}\left({\left[\sn(A_r^{Y_g})\right]}_{\pi}\oplus [B_r^{Y_g}]_\pi\right) & *=2n-2, \\[6pt]
    0 & \text{otherwise},
\end{array}
\right.
\]
where $C^{Y_g}$ is some quotient of $[H^{\otimes 2}_{Y_g}]_{\pi}$. 
\end{prop}

All in all, we can describe the rational homotopy groups of $\bEmbmodEmb(Y_g;X_g)$ as a $\QQ[\mcgRhoYg]$-module as follows. Consider the diagram
\begin{equation}\label{bighomomorphismdiagram}
\begin{tikzcd}   \mcgRhoYgsfr\rar\dar[shorten >=3pt] & \mcgRhoYg\dar\rar["\nu"] & \hMCGYg\\
\mcgRhoXgsfr\rar & \mcgRhoXg\rar["\nu"] & \hMCGXg,
\end{tikzcd}
\end{equation}
where $\nu([\phi],\rho^i)=[\phi\rho^i]$ in both cases.
\begin{cor}\label{MCGactionYg}
     Let $H_{X_g}$, $A_r$, $B_r$ and $C$ be the $\hMCGXg$-representations of \cref{FLShomology,FLSS1homology}, and let $\smash{H_{Y_g}}$, $\smash{A_r^{Y_g}}$, $\smash{B_r^{Y_g}}$ and $\smash{C^{Y_g}}$ be the $\hMCGYg$-representations of \cref{YgXgpseudoisotopycomputation}. Regard all of these representations as $\mcgRhoYgsfr$-representations by pulling them back along \eqref{bighomomorphismdiagram}. Then, for $*\leq 2n-2$, there is isomorphism of $\QQ[\mcgRhoYgsfr]$-modules
     \[
\scalebox{0.96}{$\pi_*^{\QQ}\big(\bEmbmodEmb(Y_g;X_g)\big)\cong\left\{\begin{array}{cl}
    \bigoplus_{r> 0} \big([A_r]_{\pi}\oplus[A_r^{Y_g}]_{\pi}\big)\cong \bigoplus_{r> 0} (H_g\oplus H_{V_g}) & *=n-1, \\[4pt]
    \parbox{0.60\linewidth}{\(
      C\oplus C^{Y_g} \oplus \bigoplus_{r> 0}\big(
        [\sn(A_r)]_\pi \oplus [\sn(A_r^{Y_g})]_\pi \oplus
        [B_r]_\pi \oplus [B_r^{Y_g}]_\pi
      \big)
    \)} & *=2n-2, \\[4pt]
    0 & \text{otherwise}.
\end{array}\right.$}
\]
Moreover, under this isomorphism and the one in \cref{MCGrepBlockModEmb}, the map $\partial_2: \pi_*^{\QQ}\big(\bEmbmodEmb(Y_g;X_g)\big)\to \pi_*^{\QQ}\big(\bEmbmodEmb(X_g)\big)$ is given by the evident projection onto the summands involving $H_g$, $A_r$, $B_r$ and $C$.
    \end{cor}

\subsection{The \texorpdfstring{$\ZZ[\NN^{\times}]$}{Z[N]}-module structures}\label{FrobeniusYgSection}

As in the case of $X_g$, it will be useful to understand the actions of the Frobenii on the homotopy groups $\bEmbmodEmb(Y_g;X_g)$ and $\BbEmbsfr(Y_g;X_g)$ in order to understand some of the connecting maps in the long exact sequence of homotopy groups of the fibration \eqref{BEmbsfrFibreSeqYg}.

\subsubsection{The \texorpdfstring{$\ZZ[\NN^{\times}]$}{Z[N]}-module structure on pseudoisotopy of \texorpdfstring{$(Y_g;X_g)$}{(Yg;Xg)}}

The fibre sequence \(\smash{\pEmby(Y_g)\xrightarrow{j}  \pEmb(Y_g;X_g)\xrightarrow{\partial_2} \pEmb(X_g)}\) of \eqref{PseudoisoFibreSeqYgXg} is clearly compatible with the Frobenius maps $\varphi_d$. Moreover, we saw in \cref{splittingprop} that this fibration splits on rational homotopy groups in the range $*\leq 3n-6$. The splitting $\psigma$ is also compatible with the Frobenii:

\begin{lemma}
    If $*\leq 3n-6$, the following square commutes:
    \[
    \begin{tikzcd}
        \pi_*^{\QQ}(\pEmb(X_g))\rar["\psigma"] \dar["\varphi_d"]& \pi_*^{\QQ}(\pEmb(Y_g;X_g))\dar["\varphi_d"] \\
        \pi_*^{\QQ}(\pEmb(X_{dg}))\rar["\psigma"] & \pi_*^{\QQ}(\pEmb(Y_{dg};X_{dg})).
    \end{tikzcd}
    \]
    Thus, for $*\leq 3n-6$ and $m\geq 0$, the following isomorphisms are compatible with the Frobenius maps:
    \begin{align*}
    \pi_*^{\QQ}\pEmb((Y_g;X_g))&\overset{\ref{splittingprop}}{\cong} \pi_*^{\QQ}(\pEmb(X_g))\oplus \pi_*^{\QQ}(\pEmby(Y_g))\\
    &\overset{\eqref{iotaCeq}}\cong \left[\pi_*^{\QQ}(\CEmbo(X_g\times I^{2m}))\right]_{\iota_{\H}}\oplus \left[\pi_*^{\QQ}(\CEmby(Y_g\times I^{2m}))\right]_{\iota_{\H}}.
    \end{align*}
\end{lemma}
\begin{proof}
    First observe that $\sigma: \pi_*\CEmbo(X_g)\to \pi_*\CEmbo(Y_g;X_g)$ as in the proof of \cref{splittingprop} is Frobenius-preserving: indeed, it is clear $\Sigma_\ell\circ\varphi_d=\varphi_d\circ \Sigma_\ell$, where $\Sigma_\ell$ is the left-stabilisation map involved in the construction of $\sigma$. As $\sigma$ is obtained from $\Sigma_\ell$ by extending by the identity on $Y_g\subset \overline{Y}_g$, the claim follows.

    But now recall the definition of $\sigma^{(\sim)}$, given by the composition in \eqref{sigmasimsplitting}. All the maps involved are compatible with the Frobenius: the norm map $N$ since the Frobenius commutes with the $h$-cobordism involution, the maps $i$ and $\alpha$ by construction, $\sigma$ by the argument above. Finally, the Weiss--Williams map $\Phi^{\Emb}_{X_g}$ is also compatible with the Frobenius since the remaining maps in the commutative square of \cref{IntermediateBoringAlphaProp} are compatible with the Frobenius, and because the horizontal maps in that square are surjective on rational homotopy groups below the stable range $*\leq 3n-6$.
\end{proof}

For our purposes, we will only need to understand the action of the Frobenius $\varphi_d$ on $\smash{\pi_*^{\QQ}(\pEmb(Y_g; X_g))}$ when $* = n - 2$, and hence, by the previous result, that on $\pi_{n-2}^{\QQ}(\CEmbo(X_g\times I^{2m}))$ and $\pi_{n-2}^{\QQ}(\CEmby(Y_g\times I^{2m}))$ for any $m\geq 0$. For sufficiently large $m$, the effect of $\varphi_d$ on the former group was carefully analysed in \cref{FrobeniusXgPseudoIsotopySection} (cf. \cref{FrobXgPseudoIsotopyCor} and \cref{ExtraDataLemma}), and the exact same analysis applies to the latter groups. The upshot is:

\begin{prop}
     For each $d\in \NN$, the Frobenius map
    \[
    \varphi_d: \pi_{n-1}^{\QQ}\big(\bEmbmodEmb(Y_g;X_g)\big)\longrightarrow \pi_{n-1}^{\QQ}\big(\bEmbmodEmb(Y_{dg};X_{dg})\big)
    \]
    is isomorphic to
      \[
    \begin{tikzcd}
\bigoplus_{r> 0} H_g^{(r)}\oplus H_{V_g}^{(r)}\ar[rr,"\bigoplus_{r>0} \alpha_r\oplus \beta_r"] && \bigoplus_{s> 0} H_{dg}^{(s)}\oplus H_{V_{dg}}^{(s)},
\end{tikzcd}
    \]
    where the subscripts $(-)^{(r)}$ and $(-)^{(s)}$ only serve bookkeeping purposes, and $\alpha_r$ and $\beta_r$:
    \begin{itemize}
        \item[(i)] are zero if $d$ does not divide $r$;
        \item[(ii)] are injective if $r=di$, and factor as
        \[
        \begin{tikzcd}[row sep = 8pt]
        \alpha_r: H_{g}^{(di)}\rar["\Delta"] & \bigoplus_{j=0}^{d-1} H_g\cong H_{dg}^{(i)}\rar[hook] &\bigoplus_{s> 0} H_{dg}^{(s)},\\
        \beta_r: H_{V_g}^{(di)}\rar["\Delta"] & \bigoplus_{j=0}^{d-1} H_{V_g}\cong H_{V_{dg}}^{(i)}\rar[hook] &\bigoplus_{s> 0} H_{V_{dg}}^{(s)}.
        \end{tikzcd}
        \]
    \end{itemize}

    Moreover, the embedding $\iota_0: X_g\xhookrightarrow{}X_{dg}$ and the $C_d$-action on $X_{dg}$ of \cref{ExtraDataLemma} extend to an embedding $\iota_0: (Y_g;X_g)\xhookrightarrow{}(Y_{dg},X_{dg})$ and a $C_d$-action on $(Y_{dg},X_{dg})$. Under the above isomorphisms,
    \begin{itemize}
        \item[(a)] the map $(\iota_0)_*: \pi_{n-1}^{\QQ}\big(\bEmbmodEmb(Y_g;X_g)\big)\to \pi_{n-1}^{\QQ}\big(\bEmbmodEmb(Y_{dg}; X_{dg})\big)$ becomes
        \[
        \begin{tikzcd}[column sep = 9pt]
            \bigoplus_{r>0}H_g\oplus H_{V_g}\ar[rrr,hook, "j=0"] &&&\bigoplus_{r>0} \bigoplus_{j=0}^{d-1}H_g\oplus H_{V_g}\rar[equal] & \bigoplus_{r>0}H_{dg}\oplus H_{V_{dg}};
        \end{tikzcd}
        \]
        
        \item[(b)] the $C_d$-action on $\pi_{n-1}^{\QQ}\big(\bEmbmodEmb(Y_{dg};X_{dg})\big)$ is the diagonal $C_d$-action on $\bigoplus_{r>0} H_{dg}\oplus H_{V_{dg}}$, where $C_d$ acts on $H_{dg}\oplus H_{V_{dg}}\cong \bigoplus_{j=0}^{d-1}H_g\oplus H_{V_g}$ by permuting the summands.
    \end{itemize}
\end{prop}

\subsubsection{The \texorpdfstring{$\ZN$}{Z[N]}-module structure on block embeddings of \texorpdfstring{$(Y_g;X_g)$}{(Yg;Xg)}}

Recall from \cref{stable framed block embs are autos} that there is a rational equivalence
\[
\BbEmbsfr(Y_g;X_g)_\ell\xrightarrow{\simeq_{\QQ}} \BAutvp(Y_g;X_g),
\]
which is compatible with the Frobenius maps. Thus, it suffices to study the effect of these on the homotopy groups of the codomain; this will be particularly easy given what we did in \cref{FrobBlockXgSection}. 

\begin{prop}
    The $d$-th Frobenius on $\pi_{n}^{\QQ}(\BbEmbsfr(Y_g;X_g)_\ell)$ is injective and given by
    \[
    \varphi_d(y)=\sum_{i=0}^{d-1}t^i \cdot(\iota_{0})_{*}(y),
    \]
    where $(\iota_0)_*$ is the "extension-by-the-identity" map induced by the embedding $\iota_0: (Y_g;X_g)\hookrightarrow(Y_{dg};X_{dg})$, and the group $C_d=\langle t\rangle/\langle t^d\rangle$ acts on $\BbEmbsfr(Y_g;X_g)_\ell$ as conjugation by the rotation action on $(Y_g;X_g)$.
\end{prop}
\begin{proof}
    By \cref{homotopy groups of mapping of pair}, the map $\BAutvp(Y_g;X_g)\longrightarrow \BAutbp(X_g)$ is injective on rational homotopy groups. Since the $d$-th Frobenius on the target of this map is injective, we conclude the same for the source. By the discussion above, the claim for $\pi_{n}^{\QQ}(\BbEmbsfr(Y_g;X_g)_\ell)$ holds. The description of the Frobenius now follows from \cref{BbEmbFrobXgProp}, the observation that the square
    \[
    \begin{tikzcd}
        \BbEmbsfr(Y_g;X_g)_{\ell}\dar["(\iota_0)_*"]\rar["r"] & \BbEmbsfr(X_g)_{\ell}\dar["(\iota_0)_*"]\\
        \BbEmbsfr(Y_{dg};X_{dg})_{\ell}\rar["r"] & \BbEmbsfr(X_{dg})_{\ell}
    \end{tikzcd}
    \]
    is visibly commutative and, moreover, the right vertical map is $C_d$-equivariant, and the fact that $r$ is injective on rational homotopy groups by the discussion above.
\end{proof}

\subsection{Homotopy of embedding spaces of \texorpdfstring{$(Y_g;X_g)$}{(Yg,Xg)}}\label{HomotopyOfEmbeddingsYgSection}

We proceed analogously to \cref{HomotopyOfEmbeddingsXgSection} to compute (up to some ambiguities) the rational homotopy groups of $\BEmbsfr(Y_g;X_g)_\ell$ in a range, using the fibre sequence \eqref{BEmbsfrFibreSeqYg}. Once again, there is one connecting map in the long exact sequence of \eqref{BEmbsfrFibreSeqYg} that we will need to consider. The following lemma is analogous to \cref{ConnectingMapsAreZero}, and its proof is completely identical.

\begin{lemma}\label{ConnectingMapsAreZeroYg}
        The following connecting map is zero:
        \[\delta_n^Y: \pi_n^{\QQ}\left( \BbEmbsfr(Y_g;X_g)_\ell\right)\longrightarrow \pi_{n-1}^{\QQ}\left(\bEmbmodEmb(Y_g;X_g) \right).\]
    \end{lemma}

From the result above, \cref{MCGactionYg}, and the long exact sequence of \eqref{BEmbsfrFibreSeqYg}, we obtain the following.

\begin{prop}\label{HomotopyBEmbsfrY}
        The rational homotopy groups of $\BEmbsfr(\yg;\xg)_\ell$ in degrees $2\leq *\leq 2n-2$ are given as $\QQ[\mcgRhoYgsfr]$-modules by
        \[\pi_{*}^\QQ\left(\BEmbsfr(Y_g;X_g)_\ell\right)\cong \left\{ \begin{array}{cl}
             \bigoplus_{r> 0} (H_g\oplus H_{V_g}) & *=n-1, \\[4pt]          
             \pi_{n}^\QQ(\BAutvp(\yg;\xg))    & *=n, \\[3pt]
             \coker(\delta_{2n-1}^Y) & *=2n-2 \\[3pt]
             0 & \text{otherwise.}
        \end{array}\right.\]
    \end{prop}

\subsection{Homology of embedding spaces of \texorpdfstring{$(Y_g;X_g)$}{(Yg,Xg)}}\label{HomologyEmbYgSection}

In this section, we compute the homotopy groups of $\BEmbsfr(Y_\infty;X_\infty)_\ell^+$ in roughly in the range $*\leq 2n-2$. As in \cref{homotopyOfPlusSectionXg}, we start by computing the homology of the universal cover of the latter space before taking the plus-construction and colimit over $g$, which we denote by $\BEmbsfr(Y_g;X_g)_\ell^{\sim}.$

\begin{nota}\label{notation L M N of Y}
    Let $L^Y, M^Y, N^Y$ denote the $\QQ[\mcgRhoYgsfr]$-modules
    \[
        L^Y \coloneq \pi_{n-1}^\QQ\left(\BEmbsfr(Y_g;X_g)_\ell\right), \hspace{15pt}
        M^Y \coloneq \pi_{n}^\QQ\left(\BEmbsfr(Y_g;X_g)_\ell\right), \hspace{15pt}
        N^Y \coloneq \pi_{2n-2}^\QQ\left(\BEmbsfr(Y_g;X_g)_\ell\right).
    \]
\end{nota}

The following computation follows verbatim as \cref{HomologyUniversalCOver}.

\begin{prop}\label{HomologyUniversalCOverYg}
    The rational homology groups of $\BEmbsfr(\yg;\xg)_\ell^\sim$ in degrees $*\leq 2n-2$ are given as $\QQ[\mcgRhoYgsfr]$-modules by 
    \[\widetilde{\H}_{*}\left(\BEmbsfr(Y_g;X_g)_\ell^\sim;\QQ\right) \cong \left\{ \begin{array}{cl}
             L^Y & *=n-1, \\[4pt]          
             M^Y   & *=n, \\[4pt]
             \sn(L^Y)\oplus \coker(d:L^Y\otimes M^Y\to N^Y) & *=2n-2, \\[4pt]
             0 & \text{otherwise,}
        \end{array}\right.\]where $d$ is a certain map.
\end{prop}

 Let $\firstyg$ denote the $\QQ[\mcgRhoYgsfr]$-module given by
 \[
 \firstyg\coloneq \colim_g \smash{\left[\sn(L^Y)\oplus \coker(d:L^Y\otimes M^Y\to N^Y)  \right]_{\mcgy}}
 \]where $d$ is the map from \cref{HomologyUniversalCOverYg}. As before, conjugation by $\rho=(-1)^{\oplus g}$ induces a $\brho$-action on $\BGL_g(\ZZ[\pi])$, and hence on $\BGL_\infty(\ZZ[\pi])$, that makes the map $\Bmcgyi\to \BGL_\infty(\ZZ[\pi])$ be $\brho$-equivariant.

\begin{teo}\label{HomotopyPlusConstructionY}
    In the range $*\leq 2n-2$, we have an isomorphism of $\QQ[\brho]$-modules
    \[\pi_{*}^\QQ\left(\BEmbsfr(Y_\infty;X_\infty)_\ell^+\right) \cong \pi_*^\QQ\left(\BGL_\infty(\ZZ[\pi])^+\right) \oplus\left\{ \begin{array}{cl}
             \firstyg & *=2n-2, \\[4pt]
             0 & \text{otherwise.}
        \end{array}\right.\]
\end{teo}
\begin{proof}
    We prove this statement by applying \cref{plusLemma} to the map $p:\BEmbsfr(Y_\infty;X_\infty)_\ell^+\to \BGL_\infty(\ZZ[\pi])^+$ of simple spaces, by proving that this map induces an isomorphism on $\H_*(-;\QQ)$ for $*\leq 2n-3$ and a surjection in degrees $*=2n-2,2n-1.$ As the map $(\Bmcgyi)^+\to \BGL_\infty(\ZZ[\pi])^+$ is a rational equivalence by \cref{stable mcg of yg are k theory}, it suffices to prove the same claim for the map $\BEmbsfr(Y_\infty;X_\infty)_\ell^+\to (\Bmcgyi)^+.$ Consider the Serre spectral sequence of the universal cover fibre sequence
    \[\BEmbsfr(Y_\infty;X_\infty)_\ell^\sim\to \BEmbsfr(Y_\infty;X_\infty)_\ell\to \Bmcgyi.\]Once again as in the proof of \cref{HomotopyPlusConstruction}, the $\mcgy$-modules $L^Y$ and $M^Y$ are gr-odd in the sense of \cref{gr odd modules}. Thus, the rational homology of $\mcgy$ with coefficients in $L^Y$ and $M^Y$ vanishes in all degrees. In the same way as in \cref{HomotopyPlusConstruction}, we conclude that the spectral sequence of the fibre sequence above has the following behaviour: for $t\leq 2n-2$, the group $E^2_{s,t}$ vanishes if $t\neq 0, 2n-2$. For $t=0$, we have $E^2_{s,0}\cong \H_s(\Bmcgyi;\QQ)$. The entry $E^2_{0,2n-2}$ is isomorphic to $\firstyg.$ From this behaviour, we automatically conclude that the map $\BEmbsfr(Y_\infty;X_\infty)_\ell\to \Bmcgyi$ is an isomorphism on rational homology in degrees less than $2n-2$ and surjective in degree $2n-2,$ hence showing the first part of the original claim. 
    
    It remains to show that $\H_{2n-1}(p;\QQ)$ is a surjection. From the discussion above, we know that $p$ is rationally $(2n-2)$-connected, as it is a map of simple spaces that is homologically $(2n-2)$-connected. Recall from the proof of \cref{HomotopyPlusConstruction} that the homology of a connected $\EE_1$-space is isomorphic as an algebra to the universal enveloping algebra of its rational homotopy Lie algebra. Thus, to show that $\H_{2n-1}(p;\QQ)$ is surjective (and given that it is surjective on the lower degrees), it suffices to prove that $\pi_{2n-1}^\QQ(p)$ is surjective. Consider the following commutative square induced by the inclusion $\iota:(V_g;W_{g,1})\hookrightarrow (Y_g;X_g)$
    \begin{equation}\label{square with vg}
        \begin{tikzcd}[row sep = 15pt]
            \BEmbsfr(V_\infty;W_{\infty,1})^+_\ell \arrow[d] \arrow[r] & \BGL_\infty(\ZZ)^+\arrow[d]\\
            \BEmbsfr(Y_\infty;X_\infty)_\ell^+\arrow[r] & \BGL_\infty(\ZZ[\pi])^+,
        \end{tikzcd}
    \end{equation}where the top left corner is the path component of the stable framing given by the restriction of $\ell$ along $\iota$, the left vertical map is given by extending embeddings and stable framings, and the right vertical map is induced by extension of scalars. We deduce that the bottom map is surjective on rational homotopy groups in degree $2n-1$ by the following two observations:
    \begin{itemize}
        \item \textit{The map $\pi_{2n-1}^\QQ(\BGL_\infty(\ZZ)^+)\to \pi_{2n-1}^\QQ(\BGL_\infty(\ZZ[\pi])^+)$ is an isomorphism}: this follows by identifying the source and target with the $K$-theory groups $\K_{2n-1}^{\QQ}(\ZZ)$ and $\K^{\QQ}_{2n-1}(\ZZ[\pi]),$ and by applying the Bass--Heller--Swan isomorphism and the fact that $\K_{2n-2}^{\QQ}(\ZZ)=0.$
        
        \item \textit{The map $\BEmbsfr(V_\infty;W_{\infty,1})^+_\ell\to \BGL_\infty(\ZZ)^+$ is rationally $(3n-5)$-connected}: this follows from \cite[Thm. 7.1]{krannich2021diffeomorphismsdiscssecondweiss} by replacing the source by the variant with stable framings instead of framings (which does not affect the result by Lemma 2.8 in loc.cit, and using the fact that $\ell$ is the standard framing of Section 2.5.1 of loc.cit), taking the colimit over $g$ on the source and plus-constructing (making it an $\EE_1$-space), hence obtaining rational connectivity in the desired range. 
    \end{itemize}
    Thus, we observe that the bottom map of \eqref{square with vg} is surjective on rational homotopy in degree $2n-1$. Hence, we conclude that $\H_{2n-1}(p;\QQ)$ is surjective. We now apply \cref{plusLemma} for the map $p$ and the group $A=\firstyg$. To check that the kernel of $\H_{2n-1}(p;\QQ)$ is isomorphic to $\firstyg$, we proceed verbatim as in \cref{HomotopyPlusConstruction}. Hence, we obtain the desired isomorphism as $\QQ[\langle\rho\rangle]$-modules, which finishes the proof.   
\end{proof}

\subsection{The sources of infinite generation \texorpdfstring{$\firstxg$}{LXg} and \texorpdfstring{$\firstyg$}{LYg}}\label{InfGenSourcesSection}

In this subsection, we analyse the $\QQ$-vector spaces $\firstxg$ and $\firstyg$ from \cref{HomotopyPlusConstruction,HomotopyPlusConstructionY}. The main result of this subsection (and the only one carried to the remaining sections) is the following.

\begin{teo}\label{mainInputWatanabe}
    The induced map $\pi_{2n-2}^\QQ(\BEmbsfr(Y_\infty;X_\infty)_\ell^+)\to \pi_{2n-2}^\QQ(\BEmbsfr(X_\infty)_\ell^+)$ satisfies the following properties:
    \begin{enumerate}[itemsep=0pt,label=(\roman*)]
        \item\label{surj of watanabe} The map is surjective.
        \item\label{target watanabe} The kernel and target are infinitely generated as a $\QQ$-vector spaces. 
        \item\label{ly is in -1} The summand $\firstyg$ of the source is in the $(-1)$-eigenspace of the reflection involution $\rho$.
    \end{enumerate}
\end{teo}

To prove this result, we focus on the $\QQ$-vector spaces $\firstxg$ and $\firstyg$. First, observe that, under the identifications of \cref{HomotopyPlusConstruction,HomotopyPlusConstructionY}, the map in \cref{mainInputWatanabe}, restricts to the map $\firstyg\to \firstxg$ which is given by taking colimits along $g$ of the maps 
\begin{equation}\label{map on universal covers}
    \widetilde{\H}_{2n-2}\left(\BEmbsfr(Y_g;X_g)_\ell^\sim;\QQ\right)\to \widetilde{\H}_{2n-2}\left(\BEmbsfr(X_g)_\ell^\sim;\QQ\right)
\end{equation}after taking coinvariants by $\mcgy$ on the source and $\Embmcgsfr$ on the target (recall \cref{HomologyUniversalCOver,HomologyUniversalCOverYg}). We conclude that the map \eqref{map on universal covers} is an extension of the following maps (recall \cref{l m n notation,notation L M N of Y}):
\begin{enumerate}[label=\protect\circled{\arabic*}]
    \item\label{the lambda 2 map} The map $\sn(L^Y)\to \sn(L)$ induced by the projection $L^Y=\bigoplus_{r>0}(H_g\oplus H_{V_g})\to \bigoplus_{r>0}H_g$ in each summand.
    \item\label{the cokernels map} The induced map on the cokernels of the connecting maps $\delta$ and $\delta^Y$ of \cref{HomotopyBEmbsfr,HomotopyBEmbsfrY}, respectively, induced by the map $\bEmbmodEmb(Y_g;X_g)\to \bEmbmodEmb(X_g)$ on $\pi_{2n-2}^{\QQ}.$
\end{enumerate}

We start by estimating the rank of the coinvariants of $\sn(L^Y)$ and $\sn(L)$ as modules over $\mcgy$ and $\Embmcgsfr$, which is the key ingredient to establish \ref{target watanabe} of \cref{mainInputWatanabe}. More precisely, we establish the following result. A directed system $A_1\to A_2\to \cdots $ is \textit{stably infinitely generated} if the its colimit $A_\infty$ is infinitely generated.

\begin{lemma}\label{lambda and ker are infinite}
    The systems of $\QQ$-vector spaces $\smash{\left[\sn(L)\right]_{\Embmcgsfr}}$ and $\smash{\ker\left(\left[\sn(L^Y)\right]_{\mcgy}\to \left[\sn(L)\right]_{\Embmcgsfr}\right)}$ indexed by $g\geq 0$ are stably infinitely generated.
\end{lemma}
\begin{proof}
    We start by proving that $\Embmcgsfr$-coinvariants of $\smash{\sn(L)}$ are infinitely generated, for every $g\geq 1$. First, we observe that the $\Embmcgsfr$-action on the latter vector space factors through the map $\Embmcgsfr\to \Ugwmin$: we see that it factors through $\Embmcgsfr\to \hMCGXg$ where $\hMCGXg=\pi_0(\Aut_*(X_g))$, as a consequence of \cref{MCGrepBlockModEmb}. Moreover, recall from \cref{FLSS1homology} that the $\QQ[\hMCGXg]$-module $L=\bigoplus_{r>0} H_g $ is a direct sum of the $\QQ[\hMCGXg]$-modules $H_g$. Now, these $\QQ[\hMCGXg]$-modules clearly factor through the map $\pi_0(\Aut_*(X_g))\to \pi_0(\Aut_*(W_{g,1}))\cong \GL_{2g}(\ZZ).$ Thus, the action of $\Embmcgsfr$ on $L$ factors through the map to $\GL_{2g}(\ZZ)$ and thus, it factors through the map to $\Ugwmin$. Thus, it suffices to prove that the $\Ugwmin$-coinvariants of this module are infinitely generated. To establish this claim, we produce a surjective map $\sn(L)\to I$ of $\QQ[\Ugwmin]$-modules, where $I$ is infinitely generated as a $\QQ$-vector space and whose $\Ugwmin$-action is trivial. Denote the $r$-th $H_g$ summand of $L$ by $\smash{H_g^{(r)}}$. Let $\lambda:H_g\otimes H_g\to \QQ$ be the standard $(-1)^n$-hyperbolic form, then let $\smash{\lambda_L:L\otimes L\to \bigoplus_{r,s>0}\QQ}$ be given by the sum $\smash{\lambda^{(r,s)}=\lambda:H_g^{(r)}\otimes H_g^{(s)}\to \QQ^{(r,s)}} $, using the isomorphism $\smash{L\otimes L\cong \bigoplus_{r,s>0} H_g^{(r)}\otimes H_g^{(s)}}$. Observe that $\lambda_L$ is $\Ugwmin$-equivariant (where the target is seen with trivial action), since by definition the isomorphisms in $\Ugwmin$ commute with the $(-1)^n$-hyperbolic form. Now, let $\sigma_L:L\otimes L\to L\otimes L$ be the involution that interchanges the coordinates and acts by $(-1)^{n-1}$. Recall that $\sn(L)$ is, by definition, the cokernel of $\sigma_L.$ We can now see that the following square commutes
    \[\begin{tikzcd}[row sep =15pt]
        L\otimes L \arrow[r, "\lambda_L"] \arrow[d, "\sigma_L"] & \bigoplus_{r,s>0}\QQ \arrow[d, "\sigma"]\\
        L\otimes L \arrow[r, "\lambda_L"] & \bigoplus_{r,s>0}\QQ
    \end{tikzcd}\]where $\sigma$ is the involution given by the sum of the maps $-\id:\QQ^{(r,s)}\to \QQ^{(s,r)}$ by the following argument: let $a,b\in H_g$ and $a^{(r)}, b^{(s)}$ be the corresponding classes in $H_g^{(r)}$ and $H_g^{(s)}$, then $\lambda_L(a^{(r)},b^{(s)})=(-1)^{n}\lambda_L(b^{(r)},a^{(s)})\in\QQ^{(r,s)}.$ Thus 
    \[\sigma(\lambda_L(a^{(r)},b^{(s)}))=-\lambda_L(a^{(s)},b^{(r)})=-(-1)^n\lambda(b^{(s)},a^{(r)})= (-1)^{n-1}\lambda_L(b^{(s)},a^{(r)})=\lambda_L(\sigma_L(a^{(r)},b^{(s)})).\]
    Denote by $I$ the quotient of $\bigoplus_{r,s>0}\QQ$ by $\sigma$, so the map $\lambda_L$ descends to a $\Ugwmin$-equivariant map $\lambda_L:\sn(L)\to I$, where the action on $I$ is trivial. Now, the map $\lambda_L$ is surjective, since the $(-1)^{n}$-hyperbolic form $\lambda$ is non-trivial. It remains to check that $I$ is infinitely generated: the map $\sigma$ identifies $\smash{\QQ^{(r,s)}}$ with $\smash{\QQ^{(s,r)}}$, so $I$ is isomorphic to $\bigoplus_{r>s>0}\QQ$ as a $\QQ$-vector space, and hence infinitely generated. Thus, we conclude that the coinvariants of $\sn(L)$ are infinitely generated, as they admit a surjective map to an infinitely generated $\QQ$-vector space $I$. Observe that the map $\lambda_L$ is compatible with stabilisation induced by the inclusion $H_g\hookrightarrow H_{g+1}$. Thus, the map $\lambda_L$ descends to a surjective map from the colimit to $I$. Hence, the colimit of this system is infinitely generated.

    We proceed to prove, by a similar argument, that the second $\QQ$-vector space in the statement is infinitely generated for every $g\geq 1$. Start by observing that it suffices to prove the same claim after we replace the target by the $\mcgy$-coinvariants of $\sn(L)$, as the map $\smash{[\sn(L^Y)]_{\mcgy}\to [\sn(L)]_{\Embmcgsfr}}$ factors through the map $\smash{[\sn(L^Y)]_{\mcgy}\to [\sn(L)]_{\mcgy}}$, induced by \ref{the lambda 2 map}. By the same argument as above, we see that the $\mcgy$-representations $L^Y$ and $L$ factor through the map $\mcgy\to \GL_g(\ZZ)$, by using \cref{MCGactionYg}. Moreover, this action is given by the sum of the actions on $H_{V_g}\oplus H_g$ and $H_g$, respectively. As $\GL_g(\ZZ)$-representations, we have a splitting $\sn(L^Y)\cong \sn(L)\oplus \sn(L')\oplus (L\otimes L')$, where $L'\coloneq \bigoplus_{r>0}H_{V_g}.$ Thus, the $\GL_g(\ZZ)$-coinvariants of $\sn(L^Y)$ split as a sum of the coinvariants of these summands. As the map $\sn(L^Y)\to \sn(L)$ is induced by the projection $L^Y=L\oplus L'\to L$, we see that the subspace $[L\otimes L']_{\GL_g(\ZZ)}\subset [\sn(L^Y)]_{\GL_g(\ZZ)}$ is contained in the kernel of $\smash{[\sn(L^Y)]_{\GL_g(\ZZ)}\to [\sn(L)]_{\GL_g(\ZZ)}}.$ Hence, we shall prove that this subspace is infinitely generated. Once again, it suffices to construct a surjective $\GL_g(\ZZ)$-equivariant map $L\otimes L'\to I$ to an infinitely generated vector space, where $\GL_g(\ZZ)$ acts trivially on the target. We proceed similarly by defining the map $  \smash{L\otimes L'\hookrightarrow L\otimes L\overset{\lambda_L}{\to} \bigoplus_{r,s>0}\QQ}$ given by the restriction of the $(-1)^n$-hyperbolic form, as defined above. This is $\GL_g(\ZZ)$-equivariant as it is the restriction of a $\Ugwmin$-equivariant map $\lambda_L$. On the other hand, it is visibly surjective, and the target is infinitely generated, hence so are the coinvariants of the source. By the same argument, this map to $I$ is compatible with stabilisation and hence the colimit is also infinitely generated.
\end{proof}

\begin{rmk}
    The proof of the lemma above can be extended to show that the map of $\QQ$-vector spaces $\smash{\left[\sn(L)\right]_{\Embmcgsfr}}\to I$ defined above is an isomorphism for $g\geq 2$. We sketch this computation now, which hinges on the following classical fact: let $G$ be a group and $M\to N$ be a surjective map of $\ZZ[G]$-modules where the action of $N$ is trivial. If there exists a subgroup $H< G$ such that the map $[M]_H\to N$ is an isomorphism, then so is the map $[M]_G\to N$. Consider the subgroup $H$ of $\Ugwmin$ generated by the matrix $M$ taking $a_i$ to $b_i$, and $b_i$ to $(-1)^na_i$, and by matrices of the form $P\oplus (P^{-1})^\vee$, for $P\in \GL_n(\ZZ)$ either a permutation matrix or a matrix taking the standard basis element $a_i$ to $a_i+a_j$ and fixing the remaining basis elements, for some $i\neq j$. Then one can show that the surjective map $\smash{\left[\sn(L)\right]_H}\to I$ is injective, by a variation of the proof of \cref{reflection inv on watanabe stuff} below. Applying the classical fact above, one concludes the claim.
\end{rmk}

The following lemma concerns the reflection involution on $\firstxg$ and $\firstyg$.

\begin{lemma}\label{reflection inv on watanabe stuff}
    For $g\geq 3$, the $\GL_g(\ZZ[\pi])$-coinvariants of the representations: 
    \begin{enumerate}
        \item\label{coinv of HY vanish} $(H_{Y_g})^{\otimes 2}$ vanish.
        \item\label{coinv of HX are negative} $(H_{X_g})^{\otimes 2}$ are in the $(-1)$-eigenspace of the reflection involution.
    \end{enumerate}
\end{lemma}
\begin{proof}
    We start by proving \ref{coinv of HY vanish}. Let $\{a_i\}_{i=1,\cdots,g}$ be the $\QQ[\pi]$-generating set of $H_{Y_g}$ represented by the inclusions $S^n\times \{*\}\subset V_g\subset Y_g.$ Then, as a $\QQ$-vector space, $(H_{Y_g})^{\otimes 2}$ is generated by $t^ka_i\otimes t^la_j$ for $k,l\in \ZZ$ and $1\leq i,j\leq g.$ We shall show that $t^ka_i\otimes t^la_j$ vanishes in the coinvariants for any choice of indices. We use $[x]$ to represent the class in the coinvariants represented by $x\in (H_{Y_g})^{\otimes 2}$. Let $1\leq i'\leq g$ be an integer different from $i$ and $j$ (possible as $g\geq 3$). Consider the matrix $A\in \GL_g(\ZZ[\pi])$ which takes leaves all generators fixed except $a_{i'}$ which is mapped to $a_{i'}+a_i.$ Then, after taking coinvariants, the following sequence of equalities holds
    \[[t^ka_{i'}\otimes t^la_{j}]=[A(t^ka_{i'})\otimes A(t^la_{j})]=[t^k(a_{i'}+a_i)\otimes t^la_{j}]=[t^ka_{i'}\otimes t^la_j]+[t^ka_i\otimes t^la_{j}],\]
    from where conclude that $[t^ka_i\otimes t^la_{j}]$ vanishes. This proves \ref{coinv of HY vanish}.

    We continue with \ref{coinv of HX are negative}. Let now $\{b_i\}_{i=1,\cdots,g}$ be the classes in $H_{X_g}$ represented by the inclusion $\{*\}\times S^n\subset W_{g,1}\subset X_g$, which together with $\{a_i\}_i$ form a $\QQ[\pi]$-generating set of $H_{X_g}.$ Thus, as a $\QQ$-vector space $(H_{X_g})^{\otimes 2}$ is generated by
    \[t^ka_i\otimes t^la_j \quad\qquad t^ka_i\otimes t^lb_j \quad\qquad t^kb_i\otimes t^la_j \qquad\quad t^k b_i\otimes t^lb_j.\]The same argument as the paragraph above proves that $[t^ka_i\otimes t^la_j]=0=[t^k b_i\otimes t^lb_j].$ As the reflection involution acts by $-1$ on the generators $a_i$ and the $+1$ on the $b_i$, we conclude that it acts by $-1$ on $t^ka_i\otimes t^lb_j$ and $t^kb_i\otimes t^la_j.$ This finishes the proof of \ref{coinv of HX are negative}.
\end{proof}

We are now ready to prove \cref{mainInputWatanabe} using \cref{lambda and ker are infinite,reflection inv on watanabe stuff}.

\begin{proof}[Proof of \cref{mainInputWatanabe}]
    We start by showing \ref{surj of watanabe}. By \cref{HomotopyPlusConstruction} and \cite[1677]{BRW}, we see that there is an isomorphism $\pi_{2n-2}^\QQ(\BEmbsfr(X_\infty)_\ell^+)\cong \firstxg$, so it suffices to show that $\firstyg\to \firstxg$ is surjective. As mentioned below \cref{mainInputWatanabe}, this map is given by the colimit of the maps \eqref{map on universal covers} after taking appropriate coinvariants, which in turn is an extension of the maps defined in \ref{the lambda 2 map} and \ref{the cokernels map}. The former map is induced by the projection $L^Y\to L$ and hence is surjective. The latter is surjective (in fact, the projection to a factor) by \cref{MCGactionYg}. Hence, the induced map after taking coinvariants is a surjection, and so is the map after taking colimits, as colimits commute with cokernels. This establishes \ref{surj of watanabe}.


    We now turn to prove \ref{target watanabe}. To do so, we start by showing that $\firstxg$ and $\ker(\firstyg\to \firstxg)$ are infinitely generated. This claim follows directly from \cref{lambda and ker are infinite}, as the colimits of the systems in loc.cit. receive surjections from $\firstxg$ and $\ker(\firstyg\to \firstxg)$, respectively, coming from taking colimits of the edge morphism $\widetilde{\H}_{2n-2}\left(\BEmbsfr(Y_g;X_g)_\ell^\sim;\QQ\right)\to \sn(L^Y)$ of the spectral sequence used in \cref{HomologyUniversalCOverYg}, and similarly for $X_g$. Now to prove \ref{target watanabe}, observe that $\firstyg$ is a summand of the source of the map we are interested in, by \cref{HomotopyPlusConstructionY}, and that $\firstxg$ is isomorphic to the target of this map, by \cref{HomotopyPlusConstruction}, as $\smash{\pi_{2n-2}^\QQ((\BEmbmcgsfr)^+)}$ vanishes, by the proof of loc.cit. The claim readily follows. 

    We finish this proof by concluding from \cref{reflection inv on watanabe stuff} that $\firstyg$ is in the $(-1)$-eigenspace of the reflection involution, hence establishing \ref{ly is in -1}. First, notice that this property is closed under subquotients, sums, and filtered colimits. Recall that $\firstyg$ is the colimit of the $\smash{\mcgy}$-coinvariants of $\smash{V\coloneq \widetilde{\H}_{2n-2}\left(\BEmbsfr(Y_g;X_g)_\ell^\sim;\QQ\right)}$. As the action of $\smash{\mcgy}$ on the latter vector space factors through $\GL_g(\ZZ[\pi])$, the coinvariants for these two actions agree. By \cref{HomologyUniversalCOverYg}, the vector space $V$ is a sum of a quotient of $(L^Y)^{\otimes 2}$ and a quotient of $N^Y$. By proceeding analogously as in the previous two paragraphs, one deduces that the coinvariants of the former are in the $(-1)$-eigenspace, as $L^Y$ is a sum of $H_{V_g}\otimes H_g$. Recall from \cref{HomotopyBEmbsfrY,MCGactionYg} that $N^Y$ is a sum of a quotient of various quotients of the representations present in \ref{coinv of HY vanish} and \ref{coinv of HX are negative} of \cref{reflection inv on watanabe stuff}. Thus, we see that $V$ lies in the $(-1)$-eigenspace after taking $\GL_g(\ZZ[\pi])$-coinvariants, and hence the same holds for $\firstyg.$ This finishes the proof.
\end{proof}

\section{Proofs of the main theorems}\label{ProofSection}

In this final section, we use the results of the previous sections to establish the main results stated in the introduction, together with some related consequences. Sections \ref{SectionThmA}, \ref{SectionThmB}, and \ref{SectionThmAB} depend on one another, while the remaining sections are essentially independent and may be read in any order.


\subsection{\texorpdfstring{\cref{main even}}{Theorem A}: Even-dimensional solid tori}\label{SectionThmA}

We now combine the work of \cite{BRW} with our computations, in the form of \cref{HomotopyPlusConstruction}, to deduce \cref{main even} for $d=2n$. We can consider the plus-constructed stably framed ``Weiss fibre sequence''
\begin{equation}\label{plusConstructedWeissXg}
    \BDiffsfr(S^1\times D^{2n-1})_L\to \BDiffsfr(X_\infty)_\ell^+\to \BEmbsfr(X_\infty)^+_\ell
\end{equation}defined as in \cite[(15) in p.1666]{BRW}, with the difference that we take the homotopy colimit over $g$ before plus-constructing (see also \eqref{weiss fibre for triads}). Using the work of Galatius and Randal-Williams \cite{GRWII}, the total space of \eqref{plusConstructedWeissXg} is weakly equivalent to $\Omega^\infty_0\Sigma^{\infty-2n}_+(S^1\times \SO/\SO(2n))$ (see also \cite[(19) in p.1668]{BRW}). We shall use the computation of the rational homotopy groups of the latter in \cite[Lemma 7.2.(ii)]{BRW} along with \cref{HomotopyPlusConstruction} to deduce \cref{main even}.

\begin{proof}[Proof of \cref{main even} for $d=2n$]
    We start by computing the rational homotopy groups of $\smash{\BDiffsfr(S^1\times D^{2n-1})_L}$. By \cite[Lemma 7.2.(ii)]{BRW}, we know that $\smash{\pi_k^\QQ(\BDiffsfr(X_\infty)_\ell^+)}$ is one-dimensional for $k=1$ and vanishes for all $1<k\leq 2n+2.$ Moreover, by the discussion in p. 1678 of loc.cit., the induced map $\smash{\pi_1^\QQ(\BDiffsfr(X_\infty)_\ell^+)} \to \pi_1^\QQ(\BEmbsfr(X_\infty)^+_\ell)$ is injective. Hence, we conclude that, for $1\leq *\leq 2n-3$
    \[\pi_*^\QQ(\BDiffsfr(S^1\times D^{2n-1});\ell)\cong \pi_{*-1}^\QQ\left(\BUgmininf^+\right) \oplus\left\{ \begin{array}{cl}
             \firstxg & *=2n-3, \\[4pt]
             0 & \text{otherwise.}
        \end{array}\right.\]by \cref{HomotopyPlusConstruction} and the long exact sequence of the fibre sequence \eqref{plusConstructedWeissXg}. Consider now the fibre sequence
    \[\map_{\partial}(S^1\times D^{2n-1},\SO/\SO(2n))\to\BDiffsfr(S^1\times D^{2n-1})\to \BDiffb(S^{1}\times D^{2n-1}).\]By \cite[1678]{BRW}, the connecting maps in the long exact sequence of this fibre sequence are zero. For $*\leq 2n-4$, we see that the source and target of the injective map
    \[\pi_*^\QQ(\map_{\partial}(S^1\times D^{2n-1},\SO/\SO(2n)))\to\pi_*^\QQ(\BDiffsfr(S^1\times D^{2n-1}))\]are finite-dimensional vector spaces with the same dimension, by the computation of the left-hand side in p.1678 of loc.cit. and the computation of $\pi_{*-1}^\QQ\left(\BUgmininf^+\right)$ in Lemma 7.3 of loc.cit (see also \cref{computation of GW}). Hence, this map is an isomorphism. In degree $2n-3$, we see that the left-hand side above vanishes, by the same computation. Thus, we conclude that
    \begin{equation}\label{pi BDiff even solid torus}
    \pi_*^\QQ(\BDiff_\partial(S^1\times D^{2n-1}))\cong \left\{ \begin{array}{cl}
             \firstxg & *=2n-3, \\[4pt]
             0 & \text{otherwise.}
        \end{array}\right.\end{equation}
        for all $*\leq 2n-3$. This finishes the proof, as $\firstxg$ is infinite-dimensional by \ref{target watanabe} of \cref{mainInputWatanabe}.
\end{proof}

\subsection{\texorpdfstring{\cref{main conc intro}}{Theorem B}: Concordances of even-dimensional solid tori}\label{SectionThmB}

We shall now apply \cref{HomotopyPlusConstructionY,mainInputWatanabe} to compute the rational homotopy type of $\BC(S^1\times D^{2n-1})$ in a range, hence establishing \cref{main conc intro} for $d=2n$. More precisely, we establish the following strengthening of \cref{main conc intro}, which will be crucial to deduce \cref{main range}.

\begin{teo}[\cref{main conc intro}, $d$ even]\label{main even conc}
    For $n\geq 3$ and $1\leq *\leq 2n-3$, then 
    \[\pi_*^\QQ\left(\BC(S^1\times D^{2n-1})\right)\cong \K_{*+1}^{\QQ}(\ZZ[\pi]) \oplus\left\{ \begin{array}{cl}
             \firstyg & *=2n-3, \\[4pt]
             0 & \text{otherwise.}
        \end{array}\right.\]
\end{teo}

We proceed analogously as in the last subsection. Recall the Weiss fibre sequence \eqref{weiss fibre for 4 ads} for the triad $\yg$
\[\BC(S^1\times D^{2n-1})\to \BDiffsfrv(\yg)_\ell\to \BEmbsfr(\yg;\xg)_\ell,\]which deloops once. By stacking in the $D^{2n-1}$-coordinate, the space $\B^2\mathrm{C}(S^1\times D^{2n-1})$ admits the structure of a $\EE_{2n-3}$-algebra. In particular, it has abelian fundamental group if $n\geq 3.$ Thus, the once delooped fibre sequence is plus-constructible in the sense of \cite{plusfibs}. As in the even-dimensional case, we focus on the plus-constructed and stabilised fibre sequence
\begin{equation}\label{PlusConsWeiss}
    \BC(S^1\times D^{2n-1})\to \BDiffsfrv(Y_\infty)^+_\ell\to \BEmbsfr(Y_\infty;X_\infty)^+_\ell.
\end{equation}Work of the first named author, namely by applying \cite[Cor. 7.3.3]{joaoTriads}, identifies the total space as (the identity component of) the infinite loop space $\Omega^\infty_0\Sigma^\infty_+(S^1\times \oo(2n+1)).$ 

\begin{proof}[Proof of \cref{main even conc}]
    We start by computing the rational homotopy groups of $\BDiffsfrv(Y_\infty)^+_\ell\simeq \Omega^\infty_0\Sigma^\infty_+(S^1\times \oo(2n+1))$ in this range. Since the inclusion $\O(2n+1)\subset \O$ is rationally $(4n+2)$-connected, we see that $\pi_k^\QQ(\BDiffsfrv(Y_\infty)^+_\ell)$ is $1$-dimensional for $k=1$ and it vanishes for $1<k\leq 4n+2.$ By combining the long exact sequence of the fibre sequence \eqref{PlusConsWeiss} with \cref{HomotopyPlusConstructionY}, we conclude our initial claim for $2\leq *\leq 2n-4.$ We continue by showing the case $*=1$, for which we claim that the induced map $\pi_1^\QQ(\BC(S^1\times D^{2n-1}))\to \pi_1^\QQ(\BDiffsfrv(Y_\infty)^+_\ell)$ is zero. Since the target of this map is one-dimensional, it suffices to show that the map \begin{equation}\label{the non zero geometric}
        \pi_1^\QQ(\BDiffsfrv(Y_\infty)^+_\ell)\to \pi_1^\QQ(\BEmbsfr(Y_\infty;X_\infty)^+_\ell)\cong \pi_1(\BGL_\infty(\ZZ[\pi])^+)
    \end{equation}is non-zero. Observe that post-composing this map with $\pi_1^\QQ(\BGL_\infty(\ZZ[\pi])^+)\to \pi_1^\QQ(\BUgmininf^+)$ agrees with the composite \[\pi_1^\QQ(\BDiffsfrv(Y_\infty)^+_\ell)\to \pi_1^\QQ(\BDiffsfrv(X_\infty)^+_\ell)\to \pi_1^\QQ(\BUgmininf^+).\]In this composition, the right map is injective by \cite[1678]{BRW}. Thus, it suffices to show that the left map is non-trivial to show that this composition, and hence \eqref{the non zero geometric}, is non-trivial. In loc.cit., the authors construct a non-zero homomorphism
    \[\gamma:\ZZ\to \pi_1(\BDiffsfr(X_g)_\ell)\]for any $g\geq 1$. We shall show that $\gamma$ lifts to $\pi_1(\BDiffsfr(Y_g)_\ell).$ Let $D^{2n}\subset X_g$ be an embedded codimension $0$ disc disjoint from $\vb Y_g$ and let $Y_g'$ denote the triad $(Y_g,X_g\backslash\int(D^{2n}), \vb Y_g\cup D^{2n})$. Let $Y'$ and $Y$ be the homotopy fibres of the maps
    \[\BDiffsfrv(Y_{g-1}')_\ell\to \BDiffsfrv(Y_{g-1})_\ell \qquad  \BDiffsfr(X_{g-1}\backslash \int(D^{2n}))_\ell\to \BDiffsfr(X_{g-1})_\ell\]respectively. Comparing stably framed diffeomorphisms with diffeomorphisms, one can check that the map $Y'\to Y$ is equivalent to the induced map on vertical homotopy fibres of the following commutative square
    \[\begin{tikzcd}[row sep = 15pt]
        \Emb((D^{2n+1},D^{2n}),(Y_g;X_g)) \arrow[d] \arrow[r] & \Emb(D^{2n},X_g) \arrow[d]\\
        \Bun((T^sD^{2n+1},T^sD^{2n}),(T^sY_g,T^sX_g)) \arrow[r] & \Bun(T^sD^{2n}, T^sX_g)
    \end{tikzcd}\]where $T^s(-)$ denotes the stable tangent bundle, and $\Bun(-,-)$ denotes the space of stable vector bundle maps (or of arrows of vector bundles). The vertical maps are given by taking stable derivatives. We leave this check to the reader. Additionally, one can check that the horizontal maps are weak equivalences by the uniqueness of collars. Hence, the map $Y'\to Y$ is also a weak equivalence. In loc.cit., the map $\gamma$ is defined as the induced map on fundamental groups of the map $Y\to \BDiffsfr(X_{g-1}\backslash \int(D^{2n}))_\ell\to \BDiffsfr(X_g)_\ell$. This map lifts to $\BDiffsfrv(Y_g;X_g)_\ell$ by using the map $Y'\to \BDiffsfrv(Y_{g-1}')_\ell\to \BDiffsfrv(Y_g)_\ell$ and the fact that $Y'\to Y$ is an isomorphism on fundamental groups, hence establishing the case $*=1$.
\end{proof}

From the proofs of \cref{main even,main even conc}, we can conclude that the map on the $(2n-3)$-th rational homotopy groups induced by the restriction map $\BC(S^1\times D^{2n-1})\to \BDiffb(S^1\times D^{2n-1})$ is isomorphic to the map in \cref{mainInputWatanabe}. In particular, this proves the following corollary.

\begin{cor}\label{surjectionInWatanabeDegree}
    For $n\geq 3$, the map \[\pi_{2n-3}^\QQ(\BC(S^1\times D^{2n-1}))\to \pi_{2n-3}^\QQ(\BDiffb(S^1\times D^{2n-1}))\]is surjective and has infinitely generated kernel.
\end{cor}

\subsection{\texorpdfstring{\cref{main odd,main conc intro}}{Theorems A and B}: Odd-dimensional solid tori}\label{SectionThmAB}

In this section, we compute the rational homotopy groups of $\BDiffb(S^1\times D^{2n})$ and $\BC(S^1\times D^{2n})$ and establish \cref{main odd} for $d=2n+1$ in a range, from \cref{main even} for $d=2n$ and \cref{mainInputWatanabe,main even conc}. Additionally, we obtain \cref{main conc intro} for $d=2n+1$ as a corollary.

\begin{proof}[Proof of \cref{main odd} for $d=2n+1$]
    Consider the following fibre sequence
    \[\BDiffb(S^1\times D^{2n})\to \BC(S^1\times D^{2n-1})\to \mathrm{BDiff_\partial^{ext}}(S^1\times D^{2n-1})\]where the base space is the classifying space of the collection of components of those diffeomorphisms of $S^1\times D^{2n-1}$ that extend to concordances. By the proof of \cref{main even} for $d=2n$, the rational homotopy groups of the base vanish in degrees $*\leq 2n-4$ and is isomorphic to $\pi_{2n-2}^\QQ(\BEmbsfr(X_\infty)^+_\ell)$ in degree $*=2n-3:$ this follows since the map $\iota:\mathrm{BDiff_\partial^{ext}}(S^1\times D^{2n-1})\to \BDiffb(S^1\times D^{2n-1})$ induces an isomorphism on homotopy groups of degree $*\geq 2$, and an injection on degree $1$. In the latter degree, both groups are abelian, since $\BDiffb(S^1\times D^{2n-1})$ and $\BC(S^1\times D^{2n-1})$ have abelian fundamental groups. We conclude that $\iota$ is an injection on rational homotopy groups in degree $1$, and hence an isomorphism since the target vanishes by loc.cit. By \cref{main even conc}, the rational homotopy group of the total space in degree $2n-3$ is isomorphic to $ \smash{\pi_{2n-2}^\QQ(\BEmbsfr(Y_\infty;X_\infty)^+_\ell)}.$ Moreover, one can check that the right map above in degree $2n-3$ is isomorphic to the map in \cref{mainInputWatanabe}---that is, $\pi_{2n-2}^\QQ(\BEmbsfr(Y_\infty;X_\infty)_\ell^+)\to \pi_{2n-2}^\QQ(\BEmbsfr(X_\infty)_\ell^+)$---which is is a block sum of $\pi_{2n-2}^{\QQ}(\BGL_\infty(\ZZ[\pi])^+)\to \pi_{2n-2}^\QQ\left(\BUgmininf^+\right)\cong 0$ (see \cite[Lem. 7.3]{BRW}) and $\firstyg\to \firstxg$. Hence, there is an extension
    \begin{equation}\label{KerLyLXrmk}
    \begin{tikzcd}
        \pi_{2n-2}^{\QQ}(\BDiffb(S^1\times D^{2n-1}))\rar & \pi_{2n-3}^{\QQ}(\BDiffb(S^1\times D^{2n}))\rar[two heads] & \begin{array}{c}
             \K_{2n-2}^{\QQ}(\ZZ[\pi])\\
             \oplus\\
             \ker(\firstyg\to\firstxg).
        \end{array}
    \end{tikzcd}
    \end{equation}
    The right-hand term in this extension is infinitely generated by \cref{mainInputWatanabe}(ii), and hence so is the middle one. We conclude \cref{main odd} for $d=2n+1$ from the long exact sequence on rational homotopy groups of the fibre sequence above.
\end{proof}

Following the same strategy one dimension higher, we get the odd-dimensional case of \cref{main conc intro}.

\begin{cor}[\cref{main conc intro}, $d$ odd]\label{main odd conc}
    For $n\geq 3$, the map $\BC(S^1\times D^{2n})\to \BDiffb(S^1\times D^{2n})$ is rationally $(2n-1)$-connected. In particular, for $1\leq *\leq 2n-4$, we have that
    \[\pi_*^{\QQ}\left(\BC(S^1\times D^{2n})\right)\cong \K_{*+1}^{\QQ}(\ZZ[\ZZ])\]and that $\pi_{2n-3}^{\QQ}\left(\BC(S^1\times D^{2n})\right)$ is an infinite-dimensional $\QQ$-vector space.
\end{cor}
\begin{proof}
    Once again, we consider the fibre sequence \(\BDiffb(S^1\times D^{2n+1})\to \BC(S^1\times D^{2n})\to \mathrm{BDiff_\partial^{ext}}(S^1\times D^{2n})\). By \cref{main even}, the fibre is rationally $(2n-2)$-connected, so the map to the base is indeed rationally $(2n-1)$-connected (for both spaces are connected). Moreover, $\mathrm{BDiff_\partial^{ext}}(S^1\times D^{2n})\to \BDiffb(S^1\times D^{2n})$ is a rational equivalence (by \cref{main odd} for $d=2n+1$, and the vanishing of $\smash{\K_2^\QQ(\ZZ[\ZZ])}$), showing the claim. The two addendums follow immediately from \cref{main even} for $d=2n+1$.
\end{proof}

\subsection{\texorpdfstring{\cref{main range}}{Theorem C} and \texorpdfstring{\cref{main range embeddings}}{Corollary E}: Rational concordance stable range}

In \cref{main even conc} and \cref{main odd conc}, we showed that, in a certain range, the rational homotopy groups of $\BC(S^1 \times D^{d-1})$ are independent of $d \geq 6$. In this section, we show that not only are these groups abstractly isomorphic, but also that the concordance stabilisation map 
\[
\BC(S^1 \times D^{d-1}) \to \BC(S^1 \times D^{d})
\]
is rationally $(d-4)$-connected but not $(d-3)$-connected, thereby establishing a special case of \cref{main range}. The proof combines results of Krannich--Randal-Williams \cite{krannich2021diffeomorphismsdiscssecondweiss} for $D^d$ with the Burghelea--Lashof--Rothenberg splitting \cite[Section~6]{Burghelea1975}. From this special case, the general statement of \cref{main range} follows from \cite[Thm.~4.4]{KKGoodwillie}, and as a byproduct we also obtain \cref{main range embeddings}.

\subsubsection{Burghelea-Lashof-Rothenberg splittings}

We start by recalling \cite[Section 6]{Burghelea1975} for $M=D^{d-1}$. Consider the fibre sequence
\[
\C(D^{d})\to \C(S^1\times D^{d-1})\to \CEmb^{\cong}(D^{d-1},S^1\times D^{d-1}) 
\]
where the $(-)^{\cong}$ denotes the path components hit by the right map. This sequence splits using the inclusion $S^1\times D^{d-1}\hookrightarrow D^{d} $ for $d\geq 2$. Moreover, this sequence is compatible with concordance stabilisation. We have a map $\lambda_d:\CEmb(D^{d-1},S^1\times D^{d-1})\to  \CEmb(D^{d-1}, \RR\times D^{d-1})$ given by lifting embeddings into the universal cover $\RR\times D^{d-1}$, which is well defined since our embeddings are fixed on $\partial D^{d-1}$ which is non-empty if $d\geq 2.$ We record here the important properties of these spaces established in \cite[Section 6]{Burghelea1975}.


\begin{enumerate}[label=(\roman*)]
    \item The inclusion $\RR\times D^{d-1}\hookrightarrow S^1\times D^{d-1}$ induces a splitting of the sequence
    \[\hofib_\iota(\lambda_d)\to \CEmb(D^{d-1},S^1\times D^{d-1})\to \CEmb(D^{d-1}, \RR\times D^{d-1})\]after taking loop spaces. This is \cite[Lemma 2, p.103]{Burghelea1975}. 
    \item The fibre sequence
    \[
    \begin{tikzcd}[row sep = 15pt]
    \C([-1,0]\times D^{d-1})\rar &\CEmb_{\{-1\}\times D^{d-1}\cup [-1,0]\times \partial D^{d-1}}( [-1,0]\times D^{d-1}, \RR\times D^{d-1})\dar\\ &\CEmb(\{0\}\times D^{d-1} ,\RR\times D^{d-1})_\iota 
    \end{tikzcd}\]where the subscript $\iota$ denotes the component of the inclusion, deloops once and has contractible total space. This is \cite[Lemma 2]{Burghelea1975}. Moreover, this delooping is compatible with the concordance stabilisation maps.  
\end{enumerate}From the last property, we obtain a map 
\begin{equation}\label{SplittingMap}
    \C(S^1\times D^{d-1})\to \C(D^{d})\times \BC(D^{d})
\end{equation}which is compatible with concordance stabilisation and whose fibre identifies with $\hofib_\iota(\lambda_d).$ The next lemma shows that this fibre has vanishing rational homotopy groups.


\begin{prop}\label{VanishingNil}
    If $d\geq 6$, the map \eqref{SplittingMap} induces an isomorphism on rational homotopy groups of degree $1\leq i\leq d-6$. Moreover, if $d$ is even, then the same claim holds additionally for $i=d-5.$
\end{prop}

\begin{proof}
    By \cite[Thm.2, p.104]{Burghelea1975}, this map is surjective on homotopy groups in degrees $i\geq 1$. We shall compute the rational homotopy groups on each side and show they have the same (finite) dimension, hence proving that it induces an isomorphism on rational homotopy groups. By \cref{main even conc,main odd conc}, we have $\pi_i^{\QQ}(\C(S^1\times D^{d-1}))$ is isomorphic to $\K_{i+2}^{\QQ}(\ZZ[\pi]).$  On the other hand, Krannich computed the rational homotopy groups of the right-hand side in this range. By \cite[Thm. A]{krannichpseudo}, the group $\smash{\pi_i^{\QQ}(\BC(D^{d}))}$ is isomorphic to $\smash{\K_{i+1}^{\QQ}(\ZZ))}.$ Thus, the $i$-th rational homotopy group of the right-hand side is 
    \(\K_{i+1}^{\QQ}(\ZZ))\oplus \K_{i+2}^{\QQ}(\ZZ)\), which has the same dimension as $\K_{i+2}(\ZZ[\pi])$ by the Bass--Heller--Swan's theorem. 
\end{proof}

\subsubsection{Proofs of \cref{main range} and \cref{main range embeddings}}

\begin{prop}\label{RAtionalConcordancePSeudoIsotopy}
    For $d\geq 6$, then the concordance stabilisation map
    \[\BC(S^1\times D^{d-1})\to \BC(S^1\times D^{d})\]is rationally $(d-4)$-connected, but not rationally $(d-3)$-connected. Moreover, the $(d-3)$-rd relative rational homotopy group of this map is infinite-dimensional. 
\end{prop}
\begin{proof}
    We first show that this map is rationally $(d-4)$-connected. Since source and target are connected, it suffices to argue that it is rationally $(d-5)$-connected upon looping once. Consider the diagram
    \[\begin{tikzcd}[row sep =15pt]
        \C(S^1\times D^{d-1})\arrow[d] \arrow[r, "\eqref{SplittingMap}"] & \C(D^{d})\times \BC(D^{d}) \arrow[d] \\ 
        \C(S^1\times D^{d}) \arrow[r, "\eqref{SplittingMap}"] & \C(D^{d+1})\times \BC(D^{d+1}). 
    \end{tikzcd}\]The right vertical map is rationally $(d-4)$-connected by \cite[Cor. B]{krannich2021diffeomorphismsdiscssecondweiss}. If $d$ is even, then both horizontal maps induce isomorphisms on rational homotopy groups in degrees $i\leq d-5$ by \cref{VanishingNil}, hence establishing that the left map is $(d-5)$-connected. If $d$ is odd, then the top map induces an isomorphism in degrees $i\leq d-6$ (and a surjection in all positive degrees) and the bottom map induces an isomorphism in degrees $i\leq d-5$, again by loc.cit. Hence, the left vertical map is a surjection in degree $d-5$. 
    
    To show that the map in the claim is not $(d-3)$-connected, we must treat separately the cases $d$ even and $d$ odd. When $d$ is odd, since by the previous argument the map in question is at least $(d-4)$-connected, it suffices to prove that the induced map on $\pi_{d-4}^\QQ$ is not an isomorphism. However, by \cref{main even conc,main odd conc}, we check that the source is infinite-dimensional, while the target is finite-dimensional. For the second claim, we see that the relative rational homotopy group in degree $d-3$ surjects onto the kernel of the induced map on $\pi_{d-4}^\QQ$, which is infinitely generated. 
    
    If $d=2n$ is even, we shall show that the stabilisation map 
    \[
    \Sigma: \pi_{2n-3}^{\QQ}(\BC(S^1\times D^{2n-1}))\longrightarrow \pi_{2n-3}^{\QQ}(\BC(S^1\times D^{2n}))
    \]
    is not surjective. Recall that there are two relevant involutions in the spaces $\BC(S^1\times D^{2n-1})$ and $\BC(S^1\times D^{2n})$; the concordance involution $\iota_{\C}$ and the reflection involution $\rho$. Let us record how $\Sigma$ and the groups involved in this map interact with these two involutions:
    \begin{enumerate}[label = (\roman*)]
        \item For any manifold $M$, the stabilisation map $\pi_*(\BC(M))\to \pi_*(\BC(M\times I))$ is anti-equivariant with respect to the concordance involution $\iota_{\C}$ \cite[Appendix I, Lemma]{HatcherSSeq}. \label{ConcInv1}

        \item The restriction map $\C(M)\to \Diff_\partial(M)$ is equivariant for $\iota_{\C}$ in domain and inversion on the target. \label{ConcInv2}

        \item In a topological group $G$, the inversion map $\mathrm{inv}:G\to G$ acts by $-1$ on homotopy groups. \label{ConcInv3}

        \item The natural map $\BDiff_\partial(M\times [0,1])\to \BC(M)$ is equivariant for the reflection involution $\rho_{[0,1]}$ on the source and the concordance involution $\iota_{\C}$ on the target. If, moreover, $M=N\times I$ for some manifold $N$, then the reflection involutions $\rho_I$ and $\rho_{[0,1]}$ on $N\times I\times [0,1]$ are homotopic, and thus induce the same involution on homotopy groups. \label{ConcInv4}
    \end{enumerate}
    
    Now observe that the restricion map 
    \[
    \pi_{2n-3}^{\QQ}(\BC(S^1\times D^{2n}))\longrightarrow \pi_{2n-3}^{\QQ}(\BDiffb(S^1\times D^{2n}))
    \]
    is an isomorphism since the groups $\pi_*^{\QQ}(\BDiffb(S^1\times D^{2n+1}))$ are zero for $*\leq 2n-2$ by \cref{main even} for $d=2n+2$. Thus, by \ref{ConcInv2} and \ref{ConcInv3}, it follows that $\iota_{\C}$ acts by $-1$ on $\pi_{2n-3}^{\QQ}(\BC(S^1\times D^{2n}))$, the target of $\Sigma$. Since this $\QQ$-vector space is infinite-dimensional by \cref{main odd} for $d=2n+1$, it follows from \ref{ConcInv1} that it suffices to show that the $(+1)$-eigenspace of $\iota_{\C}$ on $\pi_{2n-3}^{\QQ}(\BC(S^1\times D^{2n-1}))$, the source of $\Sigma$, is finite-dimensional---in fact, we will show that it is at most $1$-dimensional, hence also establishing the second claim.

    To this end, consider the extension obtained from the standard fibre sequence $\BDiffb(S^1\times D^{2n})\to \BC(S^1\times D^{2n-1})\to \BDiffb(S^1\times D^{2n-1})$ by applying $\pi_{2n-3}^{\QQ}(-)$:
    \[
    \begin{tikzcd}
        I\rar["\alpha"] & K\oplus\firstyg\rar["p"] & \firstxg,
    \end{tikzcd}
    \]
    where $I\coloneq\pi_{2n-3}^{\QQ}(\BDiffb(S^1\times D^{2n}))$ and $K\coloneq\K_{2n-2}^\QQ(\ZZ[\ZZ])\cong \K^{\QQ}_{2n-3}(\ZZ)$; see \cref{main even conc} and \eqref{pi BDiff even solid torus}. This is an extension of $\QQ[\langle\rho,\iota_{\C}\rangle]$-modules, and we need to understand the effect of $\iota_{\C}$ on the middle term---let us write $\iota_{\C}=\left(\begin{smallmatrix} A & B \\ C & D \end{smallmatrix}\right): K\oplus \firstyg\to K\oplus \firstyg$ in matrix form. We claim that $D=-\Id_{\firstyg}$ and that either $B=0$ or $C=0$. In particular, since $K$ is at most $1$-dimensional, it will follow from this that the $(+1)$-eigenspace of $\iota_{\C}$ on $K\oplus \firstyg$ is at most $1$-dimensional, as we wish to show. 

    So let $(0,x)\in K\oplus \firstyg$ and write $(b,\overline{x})\coloneq\iota_{\C}(0,x)$, i.e. $b= B(x)$ and $\overline{x}=D(x)$. Since $p$ is equivariant for $\iota_{\C}$ on the source and $-1$ on the target, it follows that $(0,x)+\iota_{\C}(0,x)\in \ker p=\im \alpha$, so let $i\in I$ be such that $\alpha(i)=(0,x)+\iota_{\C}(0,x)$. Note that $\alpha(i)$ is in the invariant subspace of $\iota_{\C}$, and thus
    \[
(b,x+\overline{x})=\alpha(i)=\iota_{\C}\alpha(i)=\alpha(\iota_{\C}(i))=\alpha(\rho(i))=\rho(\alpha(i)),
    \]
    where the second to last equality follows from \ref{ConcInv4}. But observe that, as a $\brho$-representation, $K\oplus \firstyg$ really is split and $\rho$ acts by $-1$ on $\firstyg$: indeed, this follows from \cref{HomotopyPlusConstructionY}, \cref{mainInputWatanabe}(iii) and the fact that the map $\pi_{2n-3}^{\QQ}(\BEmbsfr(Y_\infty;X_\infty)^+_\ell)\to \pi_{2n-3}^{\QQ}(\BC(S^1\times D^{2n-1}))$ is a $\brho$-equivariant isomorphism. Hence, we obtain that $(b,x+\overline{x})=(\rho(b),-x-\overline{x})$, which yields $\overline{x}=-x$. This shows that $D=-\Id_{\firstyg}$. But from $\iota_{\C}^2=\left(\begin{smallmatrix} A^2+BC & AB-B \\ CA-C & CB+\Id_{\firstyg} \end{smallmatrix}\right)=\left(\begin{smallmatrix} \Id_K & 0 \\ 0 & \Id_{\firstyg} \end{smallmatrix}\right)$, we get that $CB=0$. Since $\dim K=\dim(\K^{\QQ}_{2n-3}(\ZZ))\leq 1$, it follows that either $B=0$ or $C=0$, as desired. This concludes the proof. 
\end{proof}

\begin{rmk}\label{ZeroMapRmk}
    Note that $\smash{\pi_{2n-2}^{\QQ}(\BC(S^1\times D^{2n}))\to \pi_{2n-2}^{\QQ}(\BC(S^1\times D^{2n+1}))}$ is surjective for $2n\geq 6$: by \cref{main even conc}, we have $\pi_{2n-2}^{\QQ}(\BC(S^1\times D^{2n+1}))\cong \K^{\QQ}_{2n-1}(\ZZ[\pi])\cong \K_{2n-1}^{\QQ}(\ZZ)$, where the second isomorphism follows from the Bass--Heller--Swan decomposition and the well-known fact that $\K_{*}^{\QQ}(\ZZ)=0$ for $*\not\equiv 1\mod 4$ (and a copy of $\QQ$ otherwise). Thus, since $\pi_{2n-2}^{\QQ}(\BC(D^{2n+2}))\cong \K^{\QQ}_{2n-1}(\ZZ)$ by \cite[Cor. B]{krannich2021diffeomorphismsdiscssecondweiss}, the right-hand vertical map in
    \[
    \begin{tikzcd}[row sep = 15pt]
    \pi_{2n-2}^{\QQ}(\BC(S^1\times D^{2n}))\dar[two heads]\rar &\pi_{2n-2}^{\QQ}(\BC(S^1\times D^{2n+1}))\dar[two heads]\\
        \pi_{2n-2}^{\QQ}(\BC(D^{2n+1}))\rar &\pi_{2n-2}^{\QQ}(\BC(D^{2n+2}))
    \end{tikzcd}
    \]
    must be isomorphism by degree reasons since it is a split epimorphism. The bottom horizontal map is also surjective by loc.cit., so it follows that the top map is surjective, as claimed.
\end{rmk}

\begin{proof}[Proof of \cref{main range}]
The result follows from \cref{RAtionalConcordancePSeudoIsotopy} and applying \cite[Thm. 4.4 and Ex. 4.3]{KKGoodwillie} for $k=2$ and $N=S^1\times D^{d-1}$.
\end{proof}

\begin{proof}[Proof of \cref{main range embeddings}] Since the space $\C(N)$ is path-connected as $N$ is simply-connected (by Cerf's \textit{pseudo-isotopy implies isotopy} theorem \cite{CerfPseudoImpliesIso}), we have a map of fibre sequences
    \[\begin{tikzcd}[row sep = 10pt]
        \CEmb(\nu M,N)_\iota\arrow[d]\arrow[r] & \BC(N-\nu M) \arrow[d] \arrow[r] & \BC(N) \arrow[d] \\
        \CEmb(\nu M\times [0,1],N\times [0,1])_\iota\arrow[r] & \BC((N-\nu M)\times[0,1]) \arrow[r] & \BC(N\times [0,1]),
    \end{tikzcd}\]where the vertical maps are the concordance stabilisation maps. By \cite[Cor. C]{KKGoodwillie}, the right-hand vertical map is rationally $(d-3)$-connected and the $(d-3)$-th rational homotopy group of its fibre is finite-dimensional. By our hypothesis, $N-\nu M$ is spin and its fundamental group is isomorphic to $\ZZ$, so by \cref{main range}, the middle vertical map is rationally $(d-4)$-connected and the $(d-4)$-th rational homotopy group of its fibre is infinite-dimensional. Thus, by the long exact sequences of both fibre sequences, together with the $4$-lemma, we obtain that the left vertical map is rationally $(d-4)$-connected and that the $(d-4)$-th rational homotopy group of the fibre is infinite-dimensional. To finish the proof, observe that the restriction map $\CEmb(\nu M,N)_\iota\to \CEmb(M,N)_\iota$ is a weak equivalence; see e.g. \cite[Prop. 3.3]{MElongknots}.
\end{proof}

\subsection{\texorpdfstring{\cref{application to G}}{Corollary D}: Upper bound on the rational concordance stable range}

We will now use \cref{main range} to deduce \cref{application to G}. We will proceed by the following strategy, which inspired by \cite[Corollaries B and C]{KKGoodwillie}. We thank Manuel Krannich for outlining to us this idea. Consider the general setting of a chain of inclusions of smooth compact spin $d$-dimensional manifolds $N\subset M\subset N'$ and suppose that the induced maps $\pi_1(N)\to \pi_1(N')$ is an isomorphism, then we conclude by Theorem 4.4 (ii) of loc.cit that the composition
    \[ \frac{\C(N\times I)}{\C(N)}\to \frac{\C(M\times I)}{\C(M)}\to \frac{\C(N'\times I)}{\C(N')}\]is an equivalence on Postnikov $(d-3)$-truncations, where $\frac{\C(N\times I)}{\C(N)}$ denotes the homotopy fibre of the map $\BC(N)\to \BC(N\times I)$. In particular, we conclude that the map
    \[\pi_*^\QQ(\C(N\times I),\C(N))\to \pi_*^\QQ(\C(M\times I),\C(M))\]is injective for $*\leq d-3.$ We will use this strategy for $N=S^1\times D^{d-1}$.

\begin{proof}[Proof of \cref{application to G}]
    Let $M$ be a $d$-dimensional compact spin manifold such that $\H^1(\pi_1(M),\QQ)\neq 0$. Let $G$ denote $\pi_1(M)$, which is finitely generated since $M$ is compact. Observe that, by \cite[Theorem 4.4 (ii)]{KKGoodwillie}, the group $\pi_{d-4}^\QQ(\C(M\times I),\C(M))$ is of the same isomorphism type as the same group where $M$ is replaced by any other compact spin manifold with fundamental group $G$, so it suffices to prove the claim in the corollary for any such manifold. Choose a handle decomposition of $M$ relative to its boundary, and let $M'$ be the union of all handles of index at least $3$. Observe that the inclusion $\partial M'\hookrightarrow M$ is $2$-connected (since it is equivalent to a relative CW complex of cells of dimensions at least $3$) and that the inclusion $M'\hookrightarrow M$ is $(d-3)$-connected (by reversing the handle structure). Since $d\geq 6$, we see that these inclusions are fundamental groupoid isomorphisms. Thus, by replacing $M$ with $M'$ (which is also spin by restriction), we can assume that $\partial M\hookrightarrow M$ is a fundamental groupoid isomorphism. 
    
    We will now construct the chain of submanifolds indicated above. Since $\H^1(G;\QQ)$ does not vanish and $G=\pi_1(M)$ is finitely generated, there exists a surjective map $p:G\to \ZZ$. Pick an element of $G$ which maps to $1$ under $p$ and represent it by an embedding $S^1\hookrightarrow M^\circ$, which, since $M$ is spin and thus orientable, has trivial normal bundle and can be extended to an embedding of $S^1\times D^{d-1}.$ The image of this embedding will play the role of $N$ above. We now construct $N'$. Let $K$ denote the kernel of $p$. A standard fact in group theory shows that $K$ is normally finitely generated as a subgroup of $G$---that is, it is the normal closure of finitely many elements of $K$ in $G$---since $G$ is finitely generated and $\ZZ$ is finitely presented. Pick finitely many elements $\{s_i\}_{i=1}^k$ generating this module and realise them as embeddings $s_i':S^1\times D^{2n-2}\hookrightarrow \partial M$, which is possible since $n\geq 6$, $\partial M\hookrightarrow M$ is a fundamental groupoid isomorphism, and $\partial M$ is orientable. Let $N'$ be the manifold obtained from $M$ by attaching $2$-handles along the embeddings $s_i'$. Since attaching a cell kills the normal closure of the attaching map on $\pi_1$, we see that $\pi_1(N')\cong G/K\cong \ZZ$, and the map $\pi_1(N)\to \pi_1(N')$ is an isomorphism. This manifold is spin, since $\mathrm{BSpin}(d)$ is $1$-connected. Hence, we can run the argument described above to conclude that the map   
    \[\pi_{d-4}^\QQ(\C(S^1\times D^d),\C(S^1\times D^{d-1}))\to \pi_{d-4}^\QQ(\C(M\times I),\C(M))\]is injective, and hence the target is infinite-dimensional, as the source is by \cref{main range}. 
\end{proof}

\subsection{Homeomorphisms of solid tori} We deduce a computation of the rational homotopy groups of the classifying space $\BTop_\partial(S^1\times D^d)$ of homeomorphisms of solid tori using the already mentioned work of Burghelea--Lashof--Rothenberg.

\begin{teo}\label{Homeomorphisms}
    For $d\geq 6$ and $*< d-3$ if $d$ is even and $*< d-4$ if $d$ is odd, then 
    \[\pi_*^{\QQ}\big(\BTop_\partial(S^1\times D^{d-1})\big)=0.\]Moreover, this vector space is infinite-dimensional if $*=d-3$ and $d$ is even, or if $*=d-4$ and $d$ is odd.
\end{teo}
\begin{proof}
    We start by showing this statement for even $d=2n$. As in \cite[Section 7.7]{BRW}, smoothing theory gives a fibre sequence
    \[\Map_\partial\big(S^1\times D^{2n-1},\tfrac{\Top(2n)}{\O(2n)}\big)\to \BDiffb(S^1\times D^{2n-1})\to \BTop_\partial(S^1\times D^{2n-1})\]and the fibre is rationally $(2n-2)$-connected. The result follows from \cref{main even} for $d=2n$. We move now to $d=2n+1$. We first prove the case $k=1$. Taking the long exact sequence of the analogous smoothing theory fibre sequence for $d=2n+1$ , we have
    \[\cdots \to \pi_1^{\QQ}(\BDiffb(S^1\times D^{2n})) \to \pi_1^{\QQ}(\BTop_\partial(S^1\times D^{2n})\to \pi_0^{\QQ}\big(\Map_\partial\big(S^1\times D^{2n},\tfrac{\Top(2n+1)}{\O(2n+1)}\big)\big)\to \cdots. \]
    Upon rationalising, the leftmost group is isomorphic mto $\K_2(\ZZ[\pi])\cong \K_1(\ZZ)\oplus \K_2(\ZZ)$. This group is finite and thus its rationalisation is trivial. For the rightmost group, first note that there is a natural equivalence $\Map_{\partial}(Y\times D^{d-1},X)\simeq \Omega^{d-1} \Map(Y,X)$, and since the fibre sequence $\Omega X\to \Map(S^1,X)\to X$ admits a section (the constant loops), we get that for any space $X$,
    \[
\pi_k(\Map_\partial(S^1\times D^{2n}, X))\cong \pi_{k+2n}(X)\oplus \pi_{k+2n+1}(X).
    \]
    When $X=\Top(2n+1)/\O(2n+1)$, the homotopy groups $\pi_*(X)$ are finite for $*\leq 2n+2$ by \cite[Essay V, 5.0]{kirbySiebenmann}. Thus, we see that the rationalisation for $k=0$ is trivial, hence proving that $\pi_1^{\QQ}(\BTop_\partial(S^1\times D^{2n}))$ vanishes. We focus now on the case $k>1$. Let $\lambda:\Emb_\partial(D^{2n},S^1\times D^{2n})\to \Emb_\partial(D^{2n},\RR\times D^{2n})$ given by lifting embeddings to the universal cover (this is analogous to $\lambda_d$ from the previous section). Since $\Top_\partial(D^{2n})$ is contractible by the Alexander trick, we deduce from \cite[Thm. 2 for $\Top$,p.104]{Burghelea1975} and \cite[Thm. 3, p.104]{Burghelea1975} that $\pi_i(\BTop_\partial(S^1\times D^{2n}))$ is abstractly isomorphic to $\pi_{i-1}(\hofib(\lambda))$ for $i\geq 2$. On the other hand, \cite[Thm.2, p.104]{Burghelea1975} for the smooth category, we have a splitting
    \[\pi_i(\Diffb(S^1\times D^{2n}))\cong \pi_i(\Diffb(D^{2n+1}))\oplus \pi_{i-1}(\Diffb(D^{2n+1}))\oplus \pi_i(\hofib(\lambda)) \]for all $i\geq 1$. In this range, \cite[Thm. A]{krannich2021diffeomorphismsdiscssecondweiss} implies that $\pi^{\QQ}_i(\Diff_\partial(D^{2n+1}))\cong \K^{\QQ}_{i+2}(\ZZ)$. By applying \cref{main odd} for $d=2n+1$, together with dimension count as in the proof of \cref{VanishingNil}, we see that $\pi^{\QQ}_i(\hofib(\lambda))$ vanishes for $i\leq 2n-5$ and is infinitely generated in $i=2n-4$. We conclude that $\smash{\pi^{\QQ}_i(\BTop_\partial(S^1\times D^{2n}))}$ vanishes for $i\leq 2n-4$ and is infinitely generated for $i=2n-3$.
\end{proof}

\subsection{\texorpdfstring{\cref{BDiffSpinManifoldsMainCor}}{Corollary F}: Diffeomorphisms of \texorpdfstring{$S^1\times N$}{S1xN} for \texorpdfstring{$N$}{N} simply-connected spin}\label{BDiffSpin section}

As an application of \cref{main range}, we now obtain \cref{BDiffSpinManifoldsMainCor}. Let $N^{d-1}$ be as in the statement of this corollary, i.e. smooth compact simply-connected spin of dimension $d-1\geq 5$ (and possibly with boundary). Our approach to study the groups $\pi_*^\QQ(\BDiffb(S^1\times N))$ is via the classical surgery-pseudoisotopy fibre sequence
\begin{equation}\label{SurgPseudoFibSeq}
\bDiffmodDiff(S^1\times N)\longrightarrow \BDiffb(S^1\times N)\longrightarrow \BbDiffb(S^1\times N).
\end{equation}
By \cite[Cor. F(2)]{BurgheleaLashofTransfer}, this fibre sequence splits upon looping twice, truncating in the concordance stable range, and localising at odd primes (in particular, rationalising), so it suffices to study the base and fibre independently. The block term is completely described in terms of $\bDiffb(N)$, by the following Bass--Heller--Swan-like decomposition of Burghelea \cite[Cor. 2.3]{BurgheleaBlockMxS1}.

\begin{lemma}[Burghelea]\label{BlockBHS}
    Let $N$ be a compact connected smooth manifold with $\mathrm{Wh}_1(\pi_1 N)=0$, and let $T$ be $S^1$ if $\partial N=\emptyset$, and a singleton otherwise. Then, there is an equivalence $\bDiffb(S^1\times N)\simeq \bDiffb(N)\times \Omega\bDiffb(N)\times T$.
\end{lemma}

We now focus on the pseudoisotopy term. We will need the following computation.

\begin{lemma}\label{WeirdMonoLem}
    Let $X$ be a based simply-connected space. There is an isomorphism
    \[
    \H_*^{S^1}(L(X\times S^1),LS^1;\QQ)\cong \bigoplus_{r\in \ZZ} \widetilde{\H}_*(LX;\QQ).
    \]
    Under this isomorphism, the split monomorphism (cf. \cref{CEhtpyFLShomologyeq})
    \[
    \begin{tikzcd}
    \H_{*+1}(X\times S^1,S^1;\QQ)\rar[hook, "c"] & \HC^-_{*+1}(\bfS[\Omega(X\times S^1)], \bfS[\Omega S^1];\QQ) &\lar["\cong", "N"'] \H^{S^1}_{*}(L(X\times S^1), LS^1;\QQ)
    \end{tikzcd}
    \]
    only hits the $(r=0)$-summand. Thus, we obtain an isomorphism
    \[
    \frac{\H^{S^1}_{*}(L(X\times S^1), LS^1;\QQ)}{\H_{*+1}(X\times S^1,S^1;\QQ)}\cong \frac{\widetilde{\H}_{*}(LX;\QQ)}{\widetilde{\H}_{*}(X;\QQ)\oplus\widetilde{\H}_{*+1}(X;\QQ)}\oplus\bigoplus_{r\neq 0}\widetilde{\H}_{*}(LX;\QQ).
    \]
    Upon taking coinvariants with respect to the involution $(-1)^{d+1}\tau_{\FLS}$, we obtain an isomorphism
    \[
    \left[\frac{\H^{S^1}_{*}(L(X\times S^1), LS^1;\QQ)}{\H_{*+1}(X\times S^1,S^1;\QQ)}\right]_{(-1)^{d+1}\tau_{\FLS}}\cong \left[\frac{\widetilde{\H}_{*}(LX;\QQ)}{\widetilde{\H}_{*}(X;\QQ)\oplus\widetilde{\H}_{*+1}(X;\QQ)}\right]_{(-1)^{d+1}\tau_{\FLS}}\oplus\bigoplus_{r>0}\widetilde{\H}_{*}(LX;\QQ).
    \]
\end{lemma} 
\begin{proof}
    Observe that there is an $S^1$-equivariant equivalence $L(S^1\times X)=LS^1\times LX\simeq \coprod_{r\in \ZZ} S^1\times LX$ where the action on the right-hand side preserves the coproduct decomposition and is diagonal on each $S^1\times LX$. Since the $S^1$-action on $S^1$ is free, it follows that $(S^1\times LX)_{hS^1}\simeq LX$, and hence the first claim follows from
    \[
    \left(\Sigma^\infty(L(X\times S^1)/LS^1)\right)_{hS^1}\simeq \bigoplus_{r\in \ZZ}(\Sigma^\infty_+(S^1\times LX))_{hS^1}/(\Sigma^\infty_+ S^1)_{hS^1}\simeq \bigoplus_{r\in \ZZ}\Sigma^\infty LX. 
    \]

    To see the second claim, observe that the map $c:\Sigma^\infty_+Y\to (\Sigma^\infty_+ LY)^{hS^1}$ naturally factors through $(\Sigma^\infty_+ L_0 Y)^{hS^1}$, where $L_0Y\subset LY$ is the component of the free loop space corresponding to null-homotopic loops. Also, note that $L_0S^1\simeq S^1$. Thus, we obtain a commutative diagram
    \[
    \begin{tikzcd}[row sep = 15pt]
        \Sigma^\infty(S^1\times X)/S^1\rar["c"]\ar[ddr, dashed, bend right = 20pt] & (\Sigma^\infty L_0(S^1\times X)/L_0S^1)^{hS^1}\rar &(\Sigma^\infty L(S^1\times X)/LS^1)^{hS^1}\\
        &(\Sigma^{\infty+1}L_0(S^1\times X)/L_0S^1)_{hS^1}\rar\uar["N", "\vsim"'] &(\Sigma^{\infty+1}L(S^1\times X)/LS^1)_{hS^1}\uar["N", "\vsim"']\\
        &\Sigma^{\infty+1}LX\rar["r=0"]\uar[phantom, "\vsimeq"] &\bigoplus_{r\in \ZZ}\Sigma^{\infty+1}LX\uar[phantom, "\vsimeq"]
    \end{tikzcd}
    \]
where both of the vertical arrows are $S^1$-norm maps. Both of these are equivalences, for the $S^1$-spectra $\Sigma^\infty_+ L_0(S^1\times X)\simeq (\Sigma^\infty_+S^1)\otimes \Sigma^\infty_+ LX$ and $\Sigma^\infty_+ L(S^1\times X)\simeq (\Sigma^\infty_+S^1)\otimes (\oplus_{r\in \ZZ}\Sigma^\infty_+ LX)$ are equivalent to induced $S^1$-spectra \cite[Thm. D]{KleinDualizingSpectrum} (same applies when replacing $S^1\times X$ with $S^1$, and hence the cofibre thereof appearing in the above diagram). The bottom equivalences implicitly use that $X$ is simply-connected so that $L_0X=LX$. The second isomorphism now follows.

Finally, to obtain the last isomorphism on coinvariants, simply observe that $\tau_{\FLS}$ descends to an involution on the quotient in the right-hand side of the second isomorphism in the statement by \cref{involutionlemma}, and, for $r>0$, identifies the $r$ and $-r$ summands in $L(S^1\times X)=\coprod_{r\in \ZZ} L_r(X\times S^1)=\coprod_{r\in \ZZ} S^1\times LX$. 
\end{proof}

\begin{lemma}\label{AstarLemma}
    Let $N$ be a smooth compact spin manifold of dimension $d-1\geq 5$, and let
    \[
    A_*\coloneq \left[\frac{\H_*(LN;\QQ)}{\H_*(N;\QQ)\oplus \H_{*+1}(N;\QQ)}\right]_{(-1)^{d+1}\tau_{\FLS}}\oplus \bigoplus_{r>0}\H_*(LN;\QQ),
    \]
    where the term inside the coinvariants is the one featuring in \cref{WeirdMonoLem}. Then, there is an isomorphism in degrees $1\leq *\leq d-5$
    \[
    \pi_*^{\QQ}\big(\bDiffmodDiff(S^1\times N)\big)\cong A_*\oplus \pi_*^{\QQ}\big(\bDiffmodDiff(S^1\times D^{d-1})\big),
    \]
    and an epimorphism in degree $*=d-4$
    \[
    \pi_{d-4}^{\QQ}\big(\bDiffmodDiff(S^1\times N)\big)\twoheadrightarrow A_*\oplus \left\{
    \begin{array}{cl}
        0 & \text{if $d$ even,} \\
        \K^{\QQ}_{d-3}(\ZZ[\ZZ]) & \text{if $d$ odd.}
    \end{array}
    \right.
    \]
\end{lemma}

\begin{proof}
Consider the Weiss--Williams map (cf. \eqref{WWMap})
\[
\bDiffmodDiff(S^1\times N)\longrightarrow \Omega^{\infty}(\Hsp(S^1\times N)_{hC_2}).
\]
By \cref{main range}, the rational concordance stable range for $S^1\times N$ is $d-5$, and hence this map is rationally $(d-4)$-connected. We now compute the rational homotopy groups of the target of this map.

Picking an embedded disc $D^{d-1}\subset N$, we obtain a map of $C_2$-spectra $\Hsp(S^1\times D^{d-1})\to \Hsp(S^1\times N)$ (recall the definition of $\Hsp(-)$ from \cref{RecollectionWWtheorems}). This map is split injective on homotopy groups, for both spectra are connective and on infinite loop spaces it is the map $\cH(S^1\times D^{d-1})\to \cH(S^1\times N)$, which is split injective as $\cH(-)\simeq \Omega\Wh(-)$ is a homotopy functor. (A more involved argument using \cite[Lem. 5.15]{MElongknots} in fact shows that the original map of spectra is split, but we will not need this.) Hence, writing $\Hsp(S^1\times N, S^1\times D^{d-1})\coloneq \hocofib(\Hsp(S^1\times D^{d-1})\to \Hsp(S^1\times N))$, we see that there is a commutative diagram
\[
\begin{tikzcd}[row sep = 15pt]
    \pi_*^{\QQ}\big(\bDiffmodDiff(S^1\times D^{d-1})\big)\dar\rar &\pi_*^{\QQ}(\Hsp(S^1\times D^{d-1})_{hC_2})\dar[hook]\\
    \pi_*^{\QQ}\big(\bDiffmodDiff(S^1\times N))\rar & \pi_*^{\QQ}(\Hsp(S^1\times D^{d-1})_{hC_2})\oplus \pi_*^{\QQ}(\Hsp(S^1\times N,S^1\times D^{d-1})_{hC_2})
\end{tikzcd}
\]
where the right-hand vertical map is the inclusion of the first summand, and where the horizontal maps are isomorphisms in degrees in degrees $*\leq d-5$, and surjective in degree $*=d-4$. 

For $*\geq 1$, there is an isomorphism $\pi_*^{\QQ}(\Hsp(S^1\times D^{d-1}))\cong\K^\QQ_{*+1}(\ZZ[\ZZ])$ and, by \cite[Thm. 5.13]{MElongknots}, the left $C_2$-action translates to the involution $(-1)^{d}\tau_{\epsilon}$ on the right, where $\tau_{\epsilon}$ is the canonical involution on algebraic $K$-theory. Under the Bass--Heller--Swan isomorphism $\smash{\K^\QQ_{*+1}(\ZZ[\ZZ])\cong \K^\QQ_{*+1}(\ZZ)\oplus \K^\QQ_*(\ZZ)}$, the canonical involution acts diagonally on the right-hand side (see e.g. \cite{ATheoryBHS}), and $\tau_{\epsilon}$ acts by $-1$ on $\K_*^\QQ(\ZZ)$. We see that for $*\geq 1$, $\pi_*^\QQ(\Hsp(S^1\times D^{d-1})_{hC_2})$ is zero if $d$ is even, and isomorphic to $\K^\QQ_{*+1}(\ZZ[\ZZ])$ if $d$ is odd.

Hence, it remains to identify $\pi_*^\QQ(\Hsp(S^1\times N, S^1\times D^{d-1})_{hC_2})$ as $A_*$. By definition (recall \cref{CEspDefn} from \cref{RecollectionWWtheorems}), there is an equivalence of $C_2$-spectra $\Hsp(S^1\times N, S^1\times D^{d-1})\simeq \Sigma\CEsp(S^1\times (N-D^{d-1}), S^1\times N)$, and by \cref{involutionlemma} and \cref{WeirdMonoLem}, the rationaly homotopy groups $\pi_{*-1}^{\QQ}(\CEsp(S^1\times (N-D^{d-1}), S^1\times N)_{hC_2})$ are indeed isomorphic to $A_*$, as desired.
\end{proof}

\begin{proof}[Proof of \cref{BDiffSpinManifoldsMainCor}]
    By \cite[Cor. F(2)]{BurgheleaLashofTransfer}\footnote{In the statement loc.cit., $\omega(n)$ stands for a number such that $s_M: \C(M)\to \C(M\times I)$ is an isomorphism on Postnikov truncations $\tau_{\leq \omega(n)}(-)$ for all compact $n$-manifolds $M$; see page 20 loc.cit. By working with a fixed manifold $M=N\times S^1$ and over the rationals, we can choose $\omega(n)=\phi^{\QQ}(S^1\times N)-1$.}, the fibre sequence \eqref{SurgPseudoFibSeq} splits in the sense that
    \[
    \tau_{\leq \phi-1}\big(\Omega^2\BDiffb(S^1\times N)\big)\simeq_\QQ \tau_{\leq \phi-1}\big(\Omega^2\BbDiffb(S^1\times N)\times \Omega^2\bDiffmodDiff(S^1\times N)\big),
    \]
    where $\phi\coloneq \phi^{\QQ}(S^1\times N)=d-5$, by \cref{main range}. Thus, the first part of the statement follows by Lemmas \ref{BlockBHS} and \ref{AstarLemma}, and the observation that the map $\bDiffmodDiff(S^1\times D^{d-1})\to \BDiffb(S^1\times D^{d-1})$ is a rational weak equivalence since $\BbDiffb(S^1\times D^{d-1})$ is rationally contractible---this can be seen from \cref{BlockBHS} and the well-known isomorphism $\pi_*(\bDiff_\partial(D^d))\cong \Theta_{d+*+1}$ for $d+*\geq 5$, where $\Theta_n$ is the group of exotic $n$-spheres, which is finite abelian, hence rationally trivial (see \cite[Thms 1.1/2]{KervaireMilnor}).

    The homotopy groups $\pi_*(\bDiffb(N))$ are finitely generated for $N$ simply-connected (cf. \cite[Prop. 5.22(ii)]{kupersfinite}), so the addendum follows from Lemmas \ref{BlockBHS} and \ref{AstarLemma}, and the above splitting. 
\end{proof}

We now discuss two examples and a corollary to showcase the computational strength of \cref{BDiffSpinManifoldsMainCor}.

\begin{exam}[$N=\CP^{n}$]\label{CPnExample} Let $n\geq 3$. We now compute $\pi_*^{\QQ}(\BDiff(S^1\times N))$ in degrees $1\leq *\leq d-5$ using \cref{BDiffSpinManifoldsMainCor}. (By \cref{pi0S1xNRmk}, the case $*=1$ is also covered since $\pi_0(\bDiff(\CP^n))$, and hence $\pi_1(\BDiff(S^1\times N))$, is abelian, so can be rationalised: to see this, note that the surgery exact sequence gives an extension $\smash{\pi_1(\widetilde{\mathcal{S}}(\CP^n))\to \pi_0(\bDiff(\CP^n))\to \pi_0(\Aut(\CP^n))}$, where $\smash{\pi_1(\widetilde{\mathcal{S}}(\CP^n))}$ is abelian, e.g. as a consequence of the $h$-cobordism theorem. Noticing that $[\CP^n,\CP^n]\cong [\CP^n,\CP^\infty]\cong \ZZ$, it easily follows that $\pi_0(\Aut(\CP^n))\cong \ZZ/2$, generated by complex conjugation. Thus, the previous sequence is split.) To do so, it suffices to determine (i) $\pi_*^{\QQ}(\bDiffb(\CP^n))$, (ii) $\H_*(L\CP^n;\QQ)$ and (iii) the $-\tau_{\FLS}$-coinvariants of the cokernel of the homomorphism $\widetilde{\H}_*(\CP^n;\QQ)\oplus \widetilde{\H}_{*+1}(\CP^n;\QQ)\hookrightarrow \widetilde{\H}_*(L\CP^n;\QQ)$, all in degrees $0\leq *\leq 2n-4$.

Item (i) is computed by first understanding $\pi_*^{\QQ}(\Aut(\CP^n))$, either by the Federer spectral sequence or by rational homotopy theory, and later computing the rational homotopy groups of the surgery structure space of $\CP^n$ by means of surgery theory. Both of these can be done by hand (alternatively, see \cite[Cor. 2.E]{BurgheleaCPn}), and one obtains that for $1\leq*\leq 2n-1$,
\[
\pi_*^{\QQ}(\bDiff(\CP^n))\cong \left\{
\begin{array}{cl}
    \QQ & \text{if $*=1,2$ or $3\leq *\leq 2n$ odd,} \\
     0& \text{otherwise.}
\end{array}\right. \oplus \left\{\begin{array}{cl}
    \bigoplus_{\substack{0\leq i\leq n-1,\\ i\equiv k\ \mathrm{mod}\ 2}} \QQ & \text{if $*=2k-1\geq 0$ odd,} \\[4pt]
    0 & \text{otherwise.}
\end{array}
\right.
\]
(The first summand is the rational homotopy of $\Aut(\CP^n)$, and the second the reduced $KO$-theory of $\CP^n$.)

As for (ii), we claim that $\H_*(L\CP^n;\QQ)\cong \QQ$ in each degree $*\leq 2n-3$. Note that the natural inclusion map $\CP^n\xhookrightarrow{} \CP^\infty\simeq \B S^1$ is $(2n+1)$-connected, so $L\CP^n\to L\CP^\infty$ is $2n$-connected. Thus, we show the claim for $L\CP^\infty$. But the free loop space fibration $S^1\simeq \Omega\CP^\infty\to L\CP^\infty\to \CP^\infty$ is trivial since $\CP^\infty\simeq \B S^1$ is a group. Hence, the claim follows from the Künneth isomorphism
\[
\H_*(L\CP^\infty;\QQ)\cong \H_*(S^1;\QQ)\otimes \H_*(\CP^\infty;\QQ)\cong \H_*(\CP^\infty;\QQ)\oplus \H_{*+1}(\CP^\infty;\QQ).
\]

Finally, note that term (iii) vanishes since, by the computation above, the monomorphism $\widetilde{\H}_*(\CP^n;\QQ)\oplus \widetilde{\H}_{*+1}(\CP^n;\QQ)\hookrightarrow \widetilde{\H}_*(L\CP^n;\QQ)$ must be an isomorphism for $*\leq 2n-3$ by dimension reasons.

In particular, for every $1\leq *\leq 2n-3$, the $\QQ$-vector space $\pi_*^{\QQ}(\BDiff(S^1\times \CP^n))$ is infinite-dimensional.
\end{exam}

\begin{exam}[$N=\Sigma^{d-1}$ a $2$-connected rational homology sphere]\label{HomotopySphereExample} Let $d\geq 6$, and let $\Sigma$ be a closed $2$-connected $(d-1)$-manifold with $\H_*(\Sigma;\QQ)\cong \H_*(S^{d-1};\QQ)$. Since $\Sigma$ has the rational homotopy type of $S^{d-1}$, it follows from \cref{BDiffSpinManifoldsMainCor} that the rational homotopy groups $\pi_*^{\QQ}(\BDiff(S^1\times \Sigma))$ are finite-dimensional for $2\leq *\leq d-5$ (and for $*=1$ when $\Sigma$ is an actual homotopy sphere by \cref{pi0S1xNRmk}). 

Even though $\H_{d-4}(L\Sigma;\QQ)=0$, we claim that $\pi_{d-4}^{\QQ}(\BDiff(S^1\times \Sigma))$ is still infinite-dimensional: write $\Sigma^\circ$ for the complement of an open disc in $\Sigma$. By handle trading \cite[Thm. 3]{WallGeometrical}, it follows that the handle dimension of $S^1\times \Sigma^\circ$ is $\leq d-3$. Hence, the homotopy type of the fibre in the fibre sequence
\begin{equation}\label{SigmaCircFibSeq}
\Emb^{\cong}(S^1\times \Sigma^\circ,S^1\times \Sigma)\to\BDiffb(S^1\times D^{d-1})\to \BDiff(S^1\times \Sigma)
\end{equation}
can be analysed in degrees $*\leq d-3$ by means of the embedding surgery-pseudoisotopy approach outlined in \cref{RecollectionWWtheorems}. The block embedding space $\smash{\bEmb}(S^1\times \Sigma^\circ,S^1\times \Sigma)$ is easily seen to have finitely generated homotopy groups in positive degrees by the block analogue of \eqref{SigmaCircFibSeq}, \cref{BlockBHS} and the fact that $\bDiff_\partial(M)$ has finitely generated homotopy groups in positive degrees if $M^{d-1}$ is $1$-connected by \cite[Prop. 5.22]{kupersfinite}. In turn, the pseudoisotopy embedding space $\pEmbNoPartial(S^1\times \Sigma^\circ,S^1\times \Sigma)$ is rationally $(d-4)$-connected, since the same holds for the spectrum $\CEsp(S^1\times \Sigma^\circ,S^1\times \Sigma)$---the latter follows from the fact that $A(S^1)\to A(S^1\times \Sigma)$ is split injective and rationally $(d-2)$-connected by \cite[Prop. 1.1]{waldhausen1978algebraicI}. Thus, we see that the rational homotopy groups of the fibre in \eqref{SigmaCircFibSeq} are finite-dimensional in degrees $1\leq *\leq d-4$, and hence by the long exact sequence and \cref{main even}, it follows that $\pi_{d-4}^{\QQ}(\BDiff(S^1\times \Sigma))$ is indeed infinite-dimensional.
\end{exam}

\begin{cor}\label{InfGenClosedManifold}
    Let $N^{d-1}$ be a compact simply-connected spin manifold of dimension $d-1\geq 5$. Then, for each $1\leq *\leq d-4$, $\pi_*^{\QQ}(\BDiffb(S^1\times N))$ is infinitely generated if $\H_*(N;\QQ)\oplus \H_{*+1}(N;\QQ)\neq 0$. In particular, if $N$ is moreover a closed manifold which is either a $2$-connected rational homology sphere or not a rational homology sphere, then there is some $2\leq *\leq d-4$ for which $\pi_*^{\QQ}(\BDiffb(S^1\times N))$ is infinitely generated.
\end{cor}
\begin{proof}
    The first claim follows immediately from \cref{BDiffSpinManifoldsMainCor} and the observation that $\widetilde{\H}_*(N;\QQ)\oplus \widetilde{\H}_{*+1}(N;\QQ)$ is a split summand of $\widetilde{\H}_*(LN;\QQ)$ for $N$ simply-connected by \cref{WeirdMonoLem}. For the second, the case when $N$ is a $2$-connected rational homology sphere is \cref{HomotopySphereExample}. Otherwise, there is some degree $2\leq *\leq \lceil\tfrac{d-1}{2}\rceil$ for which $\H_*(N;\QQ)\neq 0$. Since for $d\geq 5$ we have that $\tfrac{d-1}{2}\leq d-3$, the result follows from \cref{BDiffSpinManifoldsMainCor}.
\end{proof}

\subsection{Corollaries \ref{main knots} and \ref{RelativeBTopMainCor}: Codimension \texorpdfstring{$2$}{2} long knots and topological Stiefel manifolds}\label{LongKnotsSection}

In this section, we first prove \cref{main knots}, and later reinterpret some of our results concerning long knots in terms of the orthogonal functor $\Vt_2$ of codimension $2$ topological Stiefel manifolds.

\subsubsection{Proof of \cref{main knots}}
Let $d\geq 6$ be an integer. We have a decomposition of $D^2\subset \RR^2$ as the union of $\frac{1}{2}D^2$ and $D^2\backslash \int(\frac{1}{2} D^2)$. The latter is diffeomorphic to $S^1\times D^1.$ By taking products with $D^{d-2}$, we have a decomposition of $D^d\cong D^{d-2}\times D^2$ as the union of $\bar\nu D^{d-2}\coloneq D^{d-2}\times\frac{1}{2}D^2$ and $S^1\times D^{d-1}.$ We will need the following two observations:
\begin{enumerate}[label=\protect\circled{\arabic*}]
    \item The parametrised isotopy extension theorem gives a fibre sequence
    \begin{equation}\label{long knots sequence}
        \Emb_{\partial D^{d-2}\times \frac{1}{2}D^2}(D^{d-2}\times \smash{\tfrac{1}{2}}D^2, D^d)_\iota\to \BDiffb(S^1\times D^{d-1})\to \BDiffb(D^{d}),
    \end{equation}where the subscript $(-)_\iota$ denotes the path component of the inclusion $\iota:D^{d-2}\times \tfrac{1}{2}D^2\hookrightarrow D^d$. This follows from the fact that $\pi_0(\Diffb(S^1\times D^{d-1}))\to \pi_0(\Diffb(D^d))$ is surjective, as it admits a splitting induced by the inclusion of $D^d$ into $S^1\times D^{d-1}.$
    \item\label{restricting is we} Observe that $\iota|_{D^{d-2}\times \{0\}}$ is the unknot, so restricting to this subspace defines a map
    \[r:\Emb_{\partial D^{d-2}\times \frac{1}{2}D^2}(D^{d-2}\times \smash{\tfrac{1}{2}}D^2, D^d)_\iota\to \Emb_\partial(D^{d-2},D^d)_u.\]
    The homotopy fibre of $r$ is equivalent to space of framings of the normal bundle of $D^{d-2}$ in $D^d$ extending the fixed framing on $\partial D^{d-2}$, which is equivalent to the space $\Map_\partial(D^{d-2},\O(2))\simeq \Omega^{d-2}\O(2)$. This space is contractible since $\O(2)\simeq S^1\sqcup S^1$ is $1$-truncated and $d\geq 6$. Thus, $r$ is a weak equivalence.
\end{enumerate}

\begin{proof}[Proof of \cref{main knots}]
    We separate our proof by cases. Let $d=2n$ be even and consider the fibre sequence \eqref{long knots sequence}. By \cref{main even} for $d=2n$, the rational homotopy groups of $\BDiffb(S^1\times D^{2n-1})$ vanish in degrees $* <2n-3$ and are infinitely generated in degree $*=2n-3$. By \cite[Thm. A]{evendisc}, the rational homotopy groups of $\BDiffb(D^{2n})$ vanish in degrees $* <2n-1$. We conclude, by the long exact sequence of \eqref{long knots sequence}, that the rational homotopy groups of the fibre of \eqref{long knots sequence}, and hence of $\Emb_\partial(D^{d-2},D^2)_u$ by \ref{restricting is we}, vanish in degrees $* <2n-3$ and are infinitely generated in degree $*=2n-3$. 

    We move now to the case $d=2n+1$ odd. Consider the following commutative square
    \[\begin{tikzcd}
        \BDiffb(S^1\times D^{2n}) \arrow[d]\arrow[r] & \BDiffb(D^{2n+1}) \arrow[d] \\
        \BC(S^1\times D^{2n-1}) \arrow[r] & \BC(D^{2n}).
    \end{tikzcd}\]The left vertical map induces isomorphisms on rational homotopy groups in degrees $*\leq 2n-4$, by the proof of \cref{main odd} for $d=2n+1$. By the vanishing results of the rational homotopy of $\BDiffb(D^{2n})$, we conclude that the right vertical map induces isomorphisms on rational homotopy groups in degrees $*\leq 2n-3$. The top map is surjective on homotopy groups in every degree $*\geq 2$ by \cite[Thm.2, p.104]{Burghelea1975} (as in the proof of \cref{VanishingNil} but for diffeomorphisms instead of concordances). Thus for $k\geq 1$, the $k$-th rational homotopy groups of $\Emb_\partial(D^{2n-1}, D^{2n+1})$ is the kernel of the map induced by the top map on rational homotopy groups of degree $k$. For $k=2n-3$, the latter map has infinitely generated source and finitely generated target, by \cite[Thm. A]{krannich2021diffeomorphismsdiscssecondweiss}, thus its kernel is infinitely generated. For $k\leq 2n-4$, this kernel is isomorphic to the kernel of the bottom map on rational homotopy groups, which we now calculate. By \cref{VanishingNil}, we know that the bottom map is surjective on rational homotopy groups in degrees $*\geq 2$. Since $\pi^{\QQ}_1(\BC(D^{2n}))\cong \K^{\QQ}_2(\ZZ)=0$, then this map is also surjective on rational homotopy groups in degree $*=1.$ Using the Bass--Heller--Swan splitting $\K^{\QQ}_i(\ZZ[\ZZ])\cong \K_i^\QQ(\ZZ)\oplus \K_{i-1}^\QQ(\ZZ)$ and the surjectivity of the bottom map on rational homotopy groups, we conclude that the kernel of the latter in degree $k$ is abstractly isomorphic to $\K^{\QQ}_{k}(\ZZ).$ This finishes the proof.  
\end{proof}

\subsubsection{Orthogonal calculus of topological Stiefel manifold}\label{VtSection}
In this section, we will prove \cref{RelativeBTopMainCor}, which claims that $\Theta \Vt_2^{(2)}$ is rationally $(-1)$-connective. But before we do, let us explain how the study of $\Vt_2$ is related to that of spaces of codimension $2$ long knots. Note that the space $\Emb_\partial(D^{d-2},D^d)$ is an $\EE_{d-2}$-algebra under stacking in $(d-2)$-many interval directions in $D^{d-2}\cong I^{d-2}$, and hence $\pi_0(\Emb_\partial(D^{d-2},D^d))$ is an abelian monoid if $d\geq 4$. Write $\Emb_\partial(D^{d-2},D^d)^\times$ for the set of invertible path-components, so that $\pi_0(\Emb_\partial(D^{d-2},D^d)^\times)$ is now an abelian group. Consider the orthogonal functor $\Vo_2(V)\coloneq \O(V\oplus \RR^2)/\O(2)$ of codimension $2$ smooth Stiefel manifolds. By smoothing theory, and being extra careful with path-components, there is a homotopy cartesian square of $\EE_{d-2}$-algebras
\[
    \begin{tikzcd}[row sep =15pt]
        \Emb_\partial(D^{d-2},D^{d})^\times\rar\dar &\Emb^t_\partial(D^{d-2},D^d)_u\simeq*\dar\\
        \Omega^{d-2}\Vo_2(\RR^{d-2})\rar & \Omega^{d-2}\Vt_2(\RR^{d-2}),
    \end{tikzcd}
\]
where $\Emb^t$ denotes the space of locally flat topological embeddings (cf. \cite{TurchinSalvatore}). The spaces on the bottom row can be identified as the spaces of smooth and topological immersions $D^{d-2}\looparrowright D^{d}$ (rel boundary), and the vertical arrows as the smooth and topological derivatives, respectively. As explained in the proof of \cite[Prop. 5.11]{BudneyIntegralLongKnots}, for $d\geq 6$, a long knot $D^{d-2}\hookrightarrow D^d$ is invertible if and only if it is a reparametrisation of the unknot---that is, the natural map $\Diff_\partial(D^{d-2})\to \Emb_\partial(D^{d-2},D^d)^\times$ is an isomorphism on $\pi_0$, and hence $\pi_0(\Emb_\partial(D^{d-2},D^d)^\times)\cong \Theta_{d-1}$ for $d\geq 6$. In particular, the map $\Emb_\partial(D^{d-2},D^d)^\times\to \Emb^t_\partial(D^{d-2},D^d)$ only hits the component of the unknot, which is contractible by the Alexander trick. \cref{main knots} gives:
\begin{cor}
    For $d\geq 6$, the $\QQ$-vector space $\pi_{*}^\QQ(\Vt_2(\RR^{d-2}))$ is infinite-dimensional in degree $*=2d-4$ if $d$ is even, and in degree $*=2d-5$ if $d$ is odd.
\end{cor}

We move on to higher derivatives of $\Vt_2$, which concern the spaces $\CEmb(D^{d-2},D^d)$ of codimension $2$ concordance long knots. Note that for $d\geq 4$, $\pi_0(\CEmb(D^{d-2},D^d))$ is again an abelian monoid under stacking. We will need the following analogue of Budney's result, suggested to us by Oscar Randal-Williams.

\begin{lemma}\label{CEmbLongKnotLemma}
    For $d\geq 5$, a concordance embedding in $\CEmb(D^{d-2},D^d)$ is homotopy invertible under stacking if and only if it is smoothly isotopic to the identity. That is, $\CEmb(D^{d-2},D^d)^\times=\CEmb(D^{d-2},D^d)_u$.
\end{lemma}
\begin{proof}
     One direction is clear. For the other, let $\varphi\in \CEmb(\nu D^{d-2},D^d)\simeq \CEmb(D^{d-2},D^d)$ be homotopy invertible, with inverse $\psi$. Writing $C_\varphi$ for the complement of $\varphi$ in $D^{d}\times I$, we claim that the natural map $\ZZ\cong \pi_1(S^1\times D^{d-1})\to \pi_1(C_\varphi)$ is an isomorphism; if so, a Mayer--Vietoris argument shows that $C_\varphi$ is an $h$-cobordism starting at $S^1\times D^{d-1}$. Since $\mathrm{Wh}_1(\ZZ)=0$, it would follow from the $s$-cobordism theorem that $C_\varphi$ is diffeomorphic to the trivial cobordism on $S^1\times D^{d-1}$; a choice of diffeomorphism determines, by isotopy extension, an isotopy from $\varphi$ to the unknot. To establish the claim, we follow the argument in \cite[Prop. 5.11]{BudneyIntegralLongKnots}. By Seifert--Van Kampen, there is a pushout square of groups
    \[
    \begin{tikzcd}[row sep = 10pt]
        \ZZ\dar\rar &\pi_1(C_\varphi)\dar\\
        \pi_1(C_\psi)\rar & \ZZ,
    \end{tikzcd}
    \]
    and both compositions $\ZZ\to \ZZ$ are isomorphisms, so the right vertical map is surjective. It is also injective by the normal form theorem for amalgamated free products, so an isomorphism. The claim follows.
\end{proof}

Write $\CEmb^t$ for the space of (locally flat) topological concordance embeddings. By the previous argument, a concordance embedding $\varphi$ is smoothly isotopic to the identity if and only it is so topologically. Thus, by smoothing theory \cite[Cor. 2(C)]{LashofEmbSpaces}, there is a homotopy cartesian square of $\EE_{d-2}$-algebras
\begin{equation}\label{CEmbKnotsCartSquare}
\begin{tikzcd}[row sep = 15pt]
        \CEmb(D^{d-2},D^{d})_u\rar\dar &\CEmb^t(D^{d-2},D^d)_u\simeq*\dar\\
        \Omega^{d-2}\Vo_2^{(1)}(\RR^{d-2})\rar & \Omega^{d-2}\Vt_2^{(1)}(\RR^{d-2}),
    \end{tikzcd}
\end{equation}
so long as $d\neq 4$. First note that $\Vo_2^{(1)}(\RR^{d-2})\simeq \Omega S^d$. By isotopy extension and the fact that $\BC(D^d)$ has finitely generated homotopy groups, it follows that $\smash{\pi_*^{\QQ}(\CEmb(D^{d-2},D^d))}$ is infinitely generated precisely when $\pi_{*-1}(\BC(S^1\times D^{d-1}))$ is so. These observations together with \cref{main conc intro} and \cref{restrictionRmkIntro}(i) imply claims (iii)--(v) in the following result:

\begin{lemma}\label{Theta1Lemma}
    The stable and unstable first derivatives of the orthogonal functor $\Vt_2$ satisfy the following:
    \begin{enumerate}[label = (\roman*)]
        \item For each $V\in \cJ$, the unstable derivative $\Vt^{(1)}_2(V)$ is $|V|$-connected. Thus, $\Theta \Vt_2^{(1)}$ is $1$-connective.

        \item There is an equivalence $\Omega^{\infty+1}(\Theta \Vt_2^{(1)})\simeq \Omega^\infty(\Ksp(\bfS))$.

        \item For $d\geq 6$ and $d-1\leq *\leq 2d-5$, the $\QQ$-vector space $\pi_{*}^{\QQ}(\Vt_2^{(1)}(\RR^{d-2}))$ is finite-dimensional.

        \item For $d\geq 6$, the $\QQ$-vector space $\pi_{2d-4}^{\QQ}(\Vt_2^{(1)}(\RR^{d-2}))$ is infinite-dimensional.
        
        \item For $d\geq 7$ odd, the $\QQ$-vector space $\pi_{2d-3}^{\QQ}(\Vt_2^{(1)}(\RR^{d-2}))$ is infinite-dimensional.
    \end{enumerate}
\end{lemma}

\begin{proof}
    We deal with (i) first. Consider the orthogonal functor $\Bt_2(V)\coloneq \BTop(V\oplus \RR^2,V)$ which, by definition, fits in a fibre sequence of orthogonal functors $\Vt_2\to \Bt_2\to \Bt(-\oplus \RR^2)$. On unstable derivatives, we obtain fibre sequences $\Vt_2^{(1)}(V)\to \Bt_2^{(1)}(V)\to \Bt^{(1)}(V\oplus \RR^{2})$. Both $\Bt_2^{(1)}(V)$ and $\Bt^{(1)}(V\oplus \RR^{2})$ are $(|V|+1)$-connected: the former since the map $\BTop(2)=\Bt_2(0)\to \Bt_{2}(V)$ is $(|V|+1)$-connected by \cite[Thm. B]{KirbySiebenmannTopmn} (see also \cite[Lem. 8.17(ii)]{KKsdisc}). The latter space since the map $\BTop(V)\to \BTop$ is $|V|$-connected.

    We move on to part (ii). The square \eqref{CEmbKnotsCartSquare} is compatible with stabilisation in the sense that there are maps of homotopy fibre sequences
    \begin{equation}\label{CEmbTangentialMap}
    \begin{tikzcd}
        \CEmb(D^{d-2},D^d)_u\dar["\Sigma_{D^{d-2},D^d}"]\rar &\Omega^{d-2}\Vo_2^{(1)}(\RR^{d-2})\rar\dar["\Omega^{d-2}s^o"] &\Omega^{d-2}\Vt_2^{(1)}(\RR^{d-2})\dar["\Omega^{d-2}s^t"]\\
        \CEmb(D^{d-1},D^{d+1})_u\rar &\Omega^{d-1}\Vo_2^{(1)}(\RR^{d-1})\rar &\Omega^{d-1}\Vt_2^{(1)}(\RR^{d-1}),
    \end{tikzcd}
    \end{equation}
    where $\Sigma_{D^{d-2},D^d}$ is the concordance embedding stabilisation map of \eqref{ConcordanceStabilisationMapsDefn}, and $s^o$ and $s^t$ are the structure maps of the first derivative spectra $\Theta\Vo^{(1)}$ and $\Theta\Vt^{(1)}$. Looping the base, and taking the colimit with respect to these stabilisation maps, we obtain another fibre sequence
    \[
    \begin{tikzcd}
    \Omega^{\infty+1}(\Theta \Vt_2^{(1)})\rar &\sCEnopartial(*,D^2)_u\rar["D^o"] &\Omega^\infty(\Theta \Vo_2^{(1)})\simeq \Omega^{\infty-1}\bfS,
    \end{tikzcd}
    \]
    where $D^o$ is the smooth (stable) derivative map. This map is nullhomotopic (it is so unstably) by the argument in \cite[p. 22]{KKGoodwillie}, and hence it follows that $\Omega^{\infty +1}(\Theta \Vt_2^{(1)})$ is equivalent to the product $\sCEnopartial(*,D^2)_u\times \Omega^\infty \bfS$. Now by \cref{CEWJReq}, we have that 
    \[
    \sCEnopartial(*,D^2)_u\simeq \hofib(\Omega \Wh(S^1)\to \Omega \Wh(*))\simeq \tohofib\Bigg(\begin{tikzcd}[scale= 0.4, column sep = 18pt, row sep = 10pt]
    Q_+S^1\dar\rar["\nu"]&A(S^1)\dar\\
    QS^0\rar["\nu"]&A(*)
\end{tikzcd}\Bigg).
    \]
    Since the unit map $\nu$ is split injective, the fibre of $Q_+S^1\to QS^0$ is $\Omega^{\infty-1}\bfS$, and the fibre of $A(S^1)\to A(*)$ is equivalent to $\Omega^{\infty-1}(\Ksp(\bfS))$ by the Bass--Heller--Swan splitting, claim (ii) follows. This finishes the proof.
\end{proof}

Finally, we derive \cref{RelativeBTopMainCor} from \cref{main range embeddings} for $(M,N)=(D^{d-2},D^d)$.

\begin{proof}[Proof of \cref{RelativeBTopMainCor}]
    Taking vertical fibres in \cref{CEmbTangentialMap} gives, for each $d\neq 4$, a fibre sequence
    \[
    \begin{tikzcd}
            \hofib(\Sigma_{D^{d-2},D^d})\rar & \Omega^{d-2}\Vo_2^{(2)}(\RR^{d-2})\rar &\Omega^{d-2}\Vt_2^{(2)}(\RR^{d-2}).
    \end{tikzcd}
    \]
    Since the map $s^o$ coincides with the loop-suspension map $\O(d+1)/\O(d)\simeq S^{d}\longrightarrow \Omega S^{d+1}\simeq \Omega \O(d+2)/\O(d+1)$ looped once, it follows that $\Vo_2^{(2)}(\RR^{d-2})$ is $(2d-2)$-connected, and hence the middle term is $(d-2)$-connected. When $d\geq 6$, the left term is rationally $(d-5)$-connected (but not rationally $(d-4)$-connected) by \cref{main range embeddings}, and so $\Omega^{d-1}\Vt_2^{(2)}(\RR^{d-2})$ is also rationally $(d-5)$-connected but not rationally $(d-4)$-connected. (In fact, it follows from \cref{Theta1Lemma}(i) and the map of homotopy long exact sequences induced by \eqref{CEmbTangentialMap} that $\Vt_2^{(2)}(\RR^{d-2})$ is itself rationally $(2d-6)$-connected but not $(2d-5)$-connected.) Hence,
    \[
    \pi_*^{\QQ}(\Theta\Vt_2^{(2)})=\colim_{d}\pi_{*+2d-4}^{\QQ}(\Vt_2^{(2)}(\RR^{d-2}))=\colim_{d} \pi_{*+d-3}^{\QQ}(\Omega^{d-1}\Vt_2^{(2)}(\RR^{d-2}))
    \]
    vanishes if $*+d-3\leq d-5$ for $d\gg 0$, i.e. when $*\leq -2$, so $\Theta\Vt_2^{(2)}$ is indeed rationally $(-1)$-connective.
\end{proof}

\appendix

\section{Weiss--Williams technicalities}\label{AppendixWW}
\subsection{The \texorpdfstring{$h$}{h}-cobordism involution on concordance 
embeddings}\label{AppendixCEhCobInv}

The purpose of this appendix is to give an explicit model of the $h$-cobordism involution $\iota_{\H}$ on $\CEmb(P,M)$ presented in \cref{iotaHItem1} of \cref{InvolutionRecollectionSection}. For simplicity, we assume that $\iota: P\hookrightarrow M$ is a codimension-zero embedding of compact manifolds with $d-p\leq 3$ and $d\geq 5$\footnote{A few words on the assumptions: the map $\CEmb(\nu P, M)\to \CEmb(P,M)$ given by restriction to the zero-section $P\subset \nu P$ is an equivalence, so no generality is lost by assuming the geometric codimension of $P$ in $M$ is zero. Regarding the handle codimension condition $d-p\leq 3$, this ensures $\CEmb(P,M)$ is connected by Hudson's theorem. Finally, we assume $d\geq 5$ so that the path-components of the $h$-cobordism spaces $\H(M-\nu P)$ and $\H(M)$ appearing later are understood by the $s$-cobordism theorem; if $d\leq 4$, these should be replaced by their identity components, i.e.\ the classifying spaces $\BC(M-\nu P)$ and $\BC(M)$.}. As a byproduct of the results in this section, we also model the fibre inclusion map $\kappa$ in the fibre sequence (cf. \cref{kappaMapProp})
\[
\begin{tikzcd}
    \CEmb(P,M)\rar["\kappa"] & \H(M-\nu P)\rar & \H(M),
\end{tikzcd}    
\]
and verify that the restriction map $\res_P:\C(M)\to \CEmb(P,M)$ is \emph{anti}-equivariant for the concordance involution $\iota_{\C}$ on the domain and $\iota_{\H}$ on the target (cf. \cref{iotaCvsiotaHCor}). To keep the notation simple, we will write $\nu P$ for $P-\partial_1 P$. Let us first introduce suitable models for $\CEmb(P,M)$, $\iota_{\H}$ and $\kappa$.

\begin{defn}\label{uCEmbDefn}
    (i)\hspace{5pt} Consider the space $\uCEmb(P,M)$ of tuples $(\varphi,\psi, \phi,\eta)$ consisting of 
    \begin{itemize}[leftmargin = 23pt]
        \item a concordance embedding $\varphi: P\times[-1,0]\hookrightarrow M\times [-1,0]$ from $\iota\times \{-1\}$ to $\varphi\mid_{P\times\{0\}}$, a concordance embedding $\psi: P\times [0,1]\hookrightarrow M\times [0,1]$ from $\varphi\mid_{P\times \{0\}}$ to $\iota\times \{1\}$ (both appropriately collared);

        \item a diffeomorphism $\phi: (M-\nu P)\times[-1,1]\cong M\times [-1,1]- \big(\varphi(\nu P\times[-1,0])\cup \psi(\nu P\times [0,1])\big)$ rel boundary; 

        \item an isotopy $\eta$ in $\Diff_\partial(M\times [-1,1])$ between the identity $\Id_{M\times [-1,1]}$ and $\phi\cup (\varphi\cup_{P\times \{0\}}\psi)$.
    \end{itemize}
    
\noindent(ii)\hspace{5pt} The $h$-cobordism involution $\iota_{\H}$ on $\uCEmb(P,M)$ sends $(\varphi,\psi,\phi,\eta)$ to the tuple $(\overline{\psi},\overline{\varphi},\overline{\phi},\overline{\eta})$ given as follows: write $r: [-1,1]\cong [-1,1]$ for $r(t)=-t$, which restricts to a diffeomorphism $r:[-1,0]\cong[0,1]$. Then,
\begin{itemize}[leftmargin = 23pt]
    \item  $\overline{\psi}=(\Id_M\times\ r)\circ\psi\circ(\Id_P\times\ r)$ and $\overline{\varphi}=(\Id_M\times\ r)\circ\varphi\circ(\Id_P\times\ r)$;

    \item $\overline{\phi}$ is the composition of $\Id_{M-\nu P}\times\ r$, $\phi$ and the diffeomorphism
    \[M\times [-1,1]- \big(\varphi(\nu P\times[-1,0])\cup \psi(\nu P\times [0,1])\big)\cong M\times [-1,1]- \big(\overline{\psi}(\nu P\times[-1,0])\cup \overline{\varphi}(\nu P\times [0,1])\big)\]
    induced by the restriction of $\Id_M\times\ r$ to the domain;

    \item $\overline{\eta}(t)=(\Id_M\times\ r)\circ\eta(t)\circ(\Id_M\times\ r)$ for $t\in [0,1]$.
\end{itemize}

\noindent(iii)\hspace{5pt} The map $\kappa: \uCEmb(P,M)\to \H(M-\nu P)$ sends $(\varphi,\psi,\phi,\eta)$ to the partition $(W,F,V)\in \H(M-\nu P)$ with
\[
W=\phi^{-1}\big(M\times [-1,0]- \varphi(\nu P\times [-1,0])\big), \quad F=\phi^{-1}\big(M\times \{0\}-\varphi(\nu P\times \{0\})\big), \quad V=\phi^{-1}\big(M\times [0,1]- \psi(\nu P\times [0,1])\big).
\]
One verifies that $\kappa: \uCEmb(P,M)\to \H(M-\nu P)$ is $\iota_{\H}$-equivariant (see \ref{iotaHItem1} for the definition of $\iota_{\H}$ on $\H(M-\nu P)$).

\noindent(iv)\hspace{5pt} There is a forgetful map $u: \uCEmb(P,M)\to \CEmb(P,M)$ sending $(\varphi,\psi,\phi,\eta)$ to $\varphi$, and a restriction map $\ures_P:\C(M)\to \uCEmb(P,M)$ sending a concordance diffeomorphism $f: M\times I\cong M\times I$ to the tuple $(\varphi,\psi,\phi,\eta)$ given as follows:
\begin{itemize}[leftmargin = 23pt]
    \item Identifying $I$ with $[-1,0]$, $\varphi$ is $f\mid_{P\times [-1,0]}$.
    \item Identifying $I$ with $[0,1]$ and writing $f_1$ for the diffeomorphism $f\mid_{M\times \{1\}}$, $\psi$ is $(f_1\times \Id_{[0,1]})\circ f^{-1}\mid_{P\times [0,1]}$.
    \item Consider the diffeomorphism $\Phi\in \Diff_\partial(M\times [-1,1])$ given by stacking $f$ on $M\times[-1,0]$ with $(f_1\times \Id_{[0,1]})\circ f^{-1}$ on $M\times [0,1]$. Then $\phi=\Phi\mid_{M-\nu P\times [-1,1]}$. 

    \item The diffeomorphism $\Phi$ can alternatively be described as the composition $(f\cup f_1\times \Id_{[0,1]})\circ(\Id_{[-1,0]}\cup f^{-1})$. The isotopy $\eta$ is given, at time $t\in [0,1]$, by the diffeomorphism $\eta(t)=(f_t\cup f_1\times \Id_{[t,1]})\circ(\Id_{[-1,-t]}\cup f^{-1}_t)$, where $f_t$ and $f^{-1}_t$ are the diffeomorphisms of $M\times [-1,t]$ and $M\times [-t,1]$ obtained from $f$ and $f^{-1}$ by linearly identifying $I$ with $[-1,t]$ and $I$ with $[-t,1]$, respectively.
\end{itemize}
By construction, these maps fit in a commutative diagram
\[
\begin{tikzcd}[row sep = 6pt]
    &\uCEmb(P,M)\ar[dd,"u"]\\
    \C(M)\ar[ur, "\ures_P"] \ar[dr, "\res_P"']&\\
    &\CEmb(P,M).
\end{tikzcd}
\]    
\end{defn}

\begin{rmk}
    There is a simplicial model for $\uCEmb(P,M)$, whose $p$-simplices are (roughly) families of tuples $(\varphi,\psi,\phi,\eta)$ parametrised by the geometric simplex $\Delta^p$. This model is implicitly used in the definition of the map $\kappa$ in \cref{uCEmbDefn}(iii), and we will often use both models interchangeably.
\end{rmk}

\begin{lemma}\label{uCEmbLem}
    The map $u: \uCEmb(P,M)\to \CEmb(P,M)$ of \cref{uCEmbDefn}(iv) is an equivalence.
\end{lemma}
\begin{proof}
Consider the based space $X$ consisting of triples $(\varphi,\psi,\phi)$, where $\varphi$, $\psi$ and $\phi$ are as in \cref{uCEmbDefn}(i), and with basepoint the identity tuple $(\iota\times \Id_{[-1,0]}, \iota\times \Id_{[0,1]}, \Id_{(M-\nu P)\times\ [-1,1]})$, which we shall also denote by $\iota$. This space fits in between $\uCEmb(P,M)$ and $\CEmb(P,M)$ as
\[
\begin{tikzcd}
    \uCEmb(P,M)\rar["\beta"] & X\rar["\alpha"] & \CEmb(P,M),
\end{tikzcd}
\]
where $\beta(\varphi,\psi,\phi,\eta)\coloneq(\varphi,\psi,\phi)$ and $\alpha(\varphi,\psi,\phi)\coloneq \varphi$---in the simplicial versions of these three spaces, both $\alpha$ and $\beta$ are Kan fibrations. We obtain a homotopy fibre sequence
\begin{equation}\label{hofibFibreSeq}
\begin{tikzcd}
    \hofib_\iota(\beta)\rar & \hofib_{\iota}(\alpha\circ\beta)=\hofib_\iota(u)\rar & \hofib_{\iota}(\alpha),
\end{tikzcd}
\end{equation}
and, since $\CEmb(P,M)$ is connected under our assumption $d-p\leq 3$, it suffices to show that the middle term in this fibre sequence is contractible. The left term $\hofib_\iota(\beta)$ is visibly the loop space $\Omega\Diff_\partial(M\times [-1,1])$ (i.e. the space of self-isotopies of the identity), whilst the right term $\hofib_\iota(\alpha)$ is the homotopy fibre of the complement map $\kappa: \Emb_{\partial}(P\times [-1,1], M\times [-1,1])\to \BDiffb(M-\nu P\times [-1,1])$; by the isotopy extension theorem, it is equivalent to $\Diffb(M\times [-1,1])$. One verifies that, under these equivalences, the connecting map $\delta:\Omega\hofib_\iota(\alpha)\to \hofib_\iota(\beta)$ is the identity map of $\Omega\Diffb(M\times [-1,1])$. It remains to argue that $\pi_0(\hofib_\iota(u))$ is a singleton. For this, consider the long exact sequence of homotopy groups for \eqref{hofibFibreSeq}: 
\[
\begin{tikzcd}
    \pi_1(\hofib_\iota(\alpha))\rar["\cong"', "\delta"] & \pi_0(\hofib_{\iota}(\beta))\rar["a"] & \pi_0(\hofib_\iota(u))\rar["i"] & \pi_0(\hofib_\iota(\alpha)).
\end{tikzcd}
\]
The map $i$ only hits the trivial component, for it sends (the component of) a tuple $(\iota\times \Id_{[0,1]},\psi,\phi,\eta)$ to $(\iota\times \Id_{[0,1]},\psi,\phi)$, and $\eta$ (and isotopy extension) provides a path from the latter to the basepoint. Now one easily sees that the map $a$ is a pointed group action. But since $\delta$ is an isomorphism, the claim follows.
\end{proof}

\begin{prop}\label{kappaMapProp}
    Let $\H^s(Q)\subset \H(Q)$ denote the component of the trivial $h$-cobordism on $Q$, which is equivalent to the classifying space $\BC(Q)$. If $d-p\leq 3$, there is an $\iota_{\H}$-equivariant homotopy fibre sequence
    \[
    \begin{tikzcd}
        \uCEmb(P,M)\rar["\kappa"] & \H^s(M-\nu P)\rar & \H^s(M).
    \end{tikzcd}
    \]
    (If $d\geq 5$, the above remains a fibre sequence if we replace $\H^s(-)\simeq \BC(-)$ by the full $h$-cobordism space $\H(-)$.) Thus, $\iota_{\H}$ on $\uCEmb(P,M)$ models the $h$-cobordism involution on $\CEmb(P,M)$.
\end{prop}
\begin{proof}
    Recall the fibration $\C(Q)\overset{a}{\to} E(Q)\overset{p}{\to}\H(Q)$ of \eqref{OmegaHFibSeq}, where $E(Q)$ is the space of embeddings $\Emb_{Q\times \{-1\}}(Q\times [-1,0], Q\times [-1,1])$, and hence contractible as it is a space of collars. We construct (up to wrong-way equivalences) a commutative diagram
    \begin{equation}\label{FibSeqHugeDiagram}
    \begin{tikzcd}[row sep = 15pt]    \Omega\uCEmb(P,M)\rar["\underline{\delta}"] \dar & \C(M-\nu P)\dar["a"]\rar & \C(M)\dar["a"]\\
        \Path\uCEmb(P,M)\rar\dar["\mathrm{ev}_1"] & E(M-\nu P)\dar["p"]\rar & E(M)\dar["p"]\\
        \uCEmb(P,M)\rar["\kappa"] & \H^s(M-\nu P)\rar &\H^s(M),
    \end{tikzcd}
    \end{equation}
where the left column is the usual path-space fibration for $\uCEmb(P,M)$, the middle and right columns are the fibre sequences \eqref{OmegaHFibSeq}, and the middle and top rows are homotopy fibre sequences: the middle one only involves contractible spaces, and the top one is obtained from the isotopy extension fibre sequence $\C(M-\nu P)\to \C(M)\xrightarrow{\ures_P}\uCEmb(P,M)$ using \cref{uCEmbLem}. In particular, the bottom row becomes a homotopy fibre sequence upon looping once, and hence is itself a homotopy fibre sequence for the homotopy fibre of $\H^s(M-\nu P)\to \H^s(M)$ is connected by Hudson's theorem, and $\pi_0(\uCEmb(P,M))\cong \pi_0(\CEmb(P,M))=*$ again by Hudson's theorem and \cref{uCEmbLem}.

The right part of \eqref{FibSeqHugeDiagram} is clear, so let us construct the left subsquares. Consider the restriction map $r_P: E(M)\to E(P,M)=\Emb_{P\times {-1}}(P\times [-1,0], M\times [-1,1])$, which is compatible with both $\res_P$ and $\ures_P$. For a compact manifold $Q\subset \RR^\infty$, we model $\BC(Q)$ as the moduli space of manifolds 
\[
\BC(Q)\simeq \Emb_{Q\times \{-1\}}(Q\times [-1,0], \RR^\infty\times [-1,+\infty))/\Diff_{Q\times \{-1\}}(Q\times [-1,0]).
\]
(The map $\C(Q)\to \Diff_{Q\times \{-1\}}(Q\times [-1,0])$ is an equivalence). Then the left subsquares in \eqref{FibSeqHugeDiagram} are
\[\begin{tikzcd}[row sep = 14pt]
	{\Omega\CEmb(P,M)} & {\hofib_\iota(\res_P)} \\
	{\Omega\uCEmb(P,M)} & {\hofib_\iota(\ures_P)} & {\C(M-\nu P)} \\
	{\Path\uCEmb(P,M)} & {\hofib_\iota(E(M)\to E(P,M))} & {E(M-\nu P)} \\
	& {\BC(M-\nu P)} \\
	{\uCEmb(P,M)} && {\H^s(M-\nu P),}
	\arrow["\delta", from=1-1, to=1-2]
	\arrow["{\Omega u}", "\vsim"', from=2-1, to=1-1]
	\arrow["{\underline{\delta}}", from=2-1, to=2-2]
	\arrow[from=2-1, to=3-1]
	\arrow["\simeq"', from=2-2, to=1-2]
	\arrow[from=2-2, to=3-2]
	\arrow["\simeq"', from=2-3, to=1-2]
	\arrow["\simeq", from=2-3, to=2-2]
	\arrow["a", from=2-3, to=3-3]
	\arrow["i", from=3-1, to=3-2]
	\arrow["{\mathrm{ev}_1}"', from=3-1, to=5-1]
	\arrow[from=3-2, to=4-2, "c"]
	\arrow[from=3-3, to=3-2, "\simeq"]
	\arrow["p", from=3-3, to=5-3]
	\arrow["k", from=5-1, to=4-2]
	\arrow["\kappa", from=5-1, to=5-3]
	\arrow["w"', "\simeq", from=5-3, to=4-2]
\end{tikzcd}\]
where notation is as follows: an arrow being unlabelled means that the arrow is the obvious map. The maps $\delta$ and $\underline{\delta}$ are the evident connecting maps for the isotopy extension fibre sequences. The map $i$ sends a path of tuples $\{(\varphi_s,\psi_s,\phi_s,\eta_s)\}_{s\in [0,1]}$ to $(\phi_1\cup \varphi_1, \eta_1\mid_{P\times [-1,0]})$. The map $c$ sends a pair $(\phi, \eta)$ to the complement $M\times [-1,0]-\phi(\nu P\times[-1,0])$. The map $k$ sends a tuple $(\varphi,\psi,\phi,\eta)$ to the complement $M\times [-1,0]-\varphi(\nu P\times[-1,0])$. Finally, the map $w$ sends a partition $(W,F,V)$ to $W\subset \RR^\infty \times [-,+\infty)$ (where we have fixed beforehand an embedding $M-\nu P\subset M\subset \RR^\infty$). All of the subsquares in this diagram are commutative on the nose, except for the bottom subtriangle, which is homotopy commutative (the homotopy being provided by the data of $\phi$ and $\eta$ in a tuple $(\varphi,\psi,\phi,\eta)\in \uCEmb(P,M)$). The map $w$ is clearly an equivalence---both domain and codomain are models for the classifying space of $\C(M-\nu P)$. The middle horizontal wrong-way map is an equivalence as domain and codomain are contractible. The top wrong-way map is an equivalence since the top diagonal map is an equivalence by isotopy extension, and the top-middle vertical map is an equivalence by \cref{uCEmbLem}. Finally, that the fibre sequence in the statement is $\iota_{\H}$-equivariant is clear. This concludes the proof.
\end{proof}

\begin{cor}\label{iotaCvsiotaHCor}
    The restriction map $\ures_P: (\C(M),\iota_{\C})\to (\uCEmb(P,M), \iota_{\H})$ is anti-equivariant. More concretely, the diagram of spaces
    \[
    \begin{tikzcd}[row sep = 15pt]
        \C(M)\dar["\mathrm{inv}\ \circ\ \iota_{\C}"']\rar["\ures_P"] &\uCEmb(P,M)\dar["\iota_{\H}"]\\
        \C(M)\rar["\ures_P"] & \uCEmb(P,M)
    \end{tikzcd}
    \]
    is commutative, where $\mathrm{inv}(\phi)=\phi^{-1}$. Thus, up to equivalence, so is $\res_P: (\C(M),\iota_{\C})\to (\CEmb(P,M),\iota_{\H})$.
\end{cor}
\begin{proof}
    This follows by a straightforward verification. Note that $\mathrm{inv}: G\to G$ on any topological group $G$ acts by $-1$ on homotopy groups based at the identity.
\end{proof}

\subsection{Proof of \texorpdfstring{\cref{BdQProp}}{Proposition 4.3.10}}\label{AppendixBoringAlpha}
The aim of this subsection is to prove \cref{BdQProp}, thus establishing \cref{IntermediateBoringAlphaProp}, and hence \cref{BoringAlphaProp}. We recall the statement:

\begin{prop*}[\ref{BdQProp}]
There is a homotopy commutative square
   \[
\begin{tikzcd}[row sep = 12pt]
    \Omega \C(Q)\dar[hook, "\mathrm{stab}."]\ar[rr,"j\circ \alpha"] && \Omega^{2}\left(\frac{\Diff^b_\partial(Q\times \RR^\infty)}{\Diff_\partial(Q)}\right)\ar[dd,"\Phi^{\Bd_Q}_{\infty}"]\\    
\Omega\mathcal{C}(Q)\dar[dash, "\vsim", "\eqref{cCvscH}"'] &&\\
\Omega^2\cH(Q)\rar[dash,"\eqref{StableCEmbIETSeq}", "\sim"']&\Omega^{\infty+2}(\Hsp(Q))\rar["q"] & \Omega^{\infty+2}\big(\Hsp(Q)_{hC_2}\big),
\end{tikzcd}
   \]
   which is natural in codimension-zero embeddings.
\end{prop*}

This proposition will be a consequence of the following general result in orthogonal calculus---see \cref{OrthCalcRecallSection} for a (rather) brief overview of the theory.

\begin{lemma}
    For $F:\mathcal{J}\to \Spc$ an orthogonal functor, the following commutes up to natural homotopy:
    \begin{equation}\label{PhiFsquare}
    \begin{tikzcd}[column sep = 3pt, row sep = 10pt]
	{F^{(1)}(0)} & {\mathrm{hofib}(F(0)\to F(\RR))} &&&&&&& {F^{(1)}(\RR^\infty)\coloneq\mathrm{hofib}(F(0)\to F(\RR^\infty))} \\
	\\
	{\Omega^\infty(\Theta F^{(1)})} &&&&&&&& {\Omega^\infty\big(\Theta F^{(1)}_{h\O(1)}\big),}
	\arrow["\coloneq", shift right=1.5, draw=none, from=1-1, to=1-2]
	\arrow[hook, "{\mathrm{stab.}}"', from=1-1, to=3-1]
	\arrow[from=1-2, to=1-9]
	\arrow["{\Phi^F_{\infty}}", from=1-9, to=3-9]
	\arrow["q", from=3-1, to=3-9]
\end{tikzcd}
\end{equation}
where the left vertical arrow is the evident inclusion into the colimit $\Omega^\infty(\Theta F^{(1)})\simeq \hocolim_{k} \Omega^k(F^{(1)}(\RR^k))$, $q: \Theta F^{(1)}\to \Theta F^{(1)}_{h\O(1)}$ denotes the natural quotient map, and $\Phi^F_\infty$ is the map from \eqref{PhiFInftyMap}.
\end{lemma}

\begin{proof} In \cite[Prop. 2.2]{MElongknots}, the following was established: for each $k\geq 0$, there are maps
\[
\begin{tikzcd}
    \Phi_k^F: \hofib(F(0)\to F(\RR^k))\rar["\eta_1"] &\hofib(T_1F(0)\to T_1F(\RR^k))\simeq \Omega^\infty(S(\RR^k)_+\wedge_{\O(1)}\Theta F^{(1)}),
\end{tikzcd}
\]
where $S(V)$ denotes the unit sphere of $V$, which give rise to maps of homotopy fibre sequences
\begin{equation}\label{PhikFdiagram}
\begin{tikzcd}[column sep = 15pt]
    \hofib(F(0)\to F(\RR^k))\rar\dar["\Phi_k^F"] &\hofib(F(0)\to F(\RR^{k+1}))\rar\dar["\Phi_{k+1}^F"] &\hofib(F(\RR^k)\to F(\RR^{k+1}))\dar["\mathrm{stab.}"] \\
    \Omega^\infty(S(\RR^k)_+\wedge_{\O(1)}\Theta F^{(1)})\rar & \Omega^\infty(S(\RR^{k+1})_+\wedge_{\O(1)}\Theta F^{(1)})\rar & \Omega^\infty(\Sigma^k\Theta F^{(1)}).
\end{tikzcd}
\end{equation}
The top-right term is the unstable derivative $\Theta F^{(1)}_k$, and the rightmost vertical map ``$\mathrm{stab.}$'' is again the natural inclusion $\smash{\Theta F^{(1)}_k\xhookrightarrow{}\operatorname{hocolim}_m\Omega^m(\Theta F^{(1)}_{k+m})}$. Then, the map $\smash{\Phi^F_{\infty}}$ of \eqref{PhiFInftyMap} is simply $\smash{\hocolim_k \Phi_k^F}$.

The claim follows almost immediately from the diagram \eqref{PhikFdiagram}. For $k=0$, it yields the left subsquare in
    $$
\begin{tikzcd}
    F^{(1)}(0)\dar[hook, "\mathrm{stab.}"] &\lar[equal] \hofib(F(0)\to F(\RR))\dar["\Phi_1^F"]\rar &\hofib(F(0)\to F(\RR^\infty))\dar["\Phi^F_\infty"]\\
    \Omega^\infty\Theta F^{(1)}&\lar["\sim"']\Omega^\infty(S(\RR)_+\wedge_{\O(1)}\Theta F^{(1)})\rar & \Omega^\infty(\Theta F^{(1)}_{h\O(1)}).
\end{tikzcd}
    $$
    Since the action of $\O(1)$ on $S(\RR)=\{\pm1\}$ is free, the map $S(\RR)_+\wedge_{\O(1)}X\to X$ is in fact an isomorphism for every $\O(1)$-spectrum $X$. Composing its inverse with the inclusion map $S(\RR)_+\wedge_{\O(1)}X\to S(\RR^{\infty})_+\wedge_{\O(1)}X=X_{h\O(1)}$ gives the quotient map $q:X\to X_{h\O(1)}$.
\end{proof}

We claim that the diagram of \cref{IntermediateBoringAlphaProp} (looped once) is equivalent to the diagram \eqref{PhiFsquare} (looped twice) for the functor $F=\sfE_{X_g}\coloneq \sfE_{(\iota(X_g), X_g)}$ of \cref{BdAndEorthFunctorsDefn}(ii) which, we recall, is the homotopy fibre $\hofib(\Bd_{S^1\times D^{2n-1}}\to \Bd_{X_g})$ (see \cref{BdAndEorthFunctorsDefn}(i) for the definition of $\Bd_Q$). In order to show this, we need to establish the analogous claim for the functor $\Bd_Q$ for $Q=S^1\times D^{2n-1}$ and $X_g$, which is what we do now.

\begin{proof}[Proof of \cref{BdQProp}]
    We show that the diagram in the statement is naturally equivalent to \eqref{PhiFsquare} looped once for $F=\Bd_Q$---naturality of the square in the statement will then be clear. Consider the diagram
\[
\adjustbox{scale=0.8,center}{
\begin{tikzcd}[row sep = 12pt]
	{\Omega^2\Bd_Q^{(1)}(0)} &&& {\Omega^2\hofib(\Bd_Q(0)\to \Bd_Q(\RR^\infty))} \\
	& {\Omega\C(Q)} & {\Omega^2\left(\frac{\Diff_\partial^b(Q\times \RR^\infty)}{\Diff_\partial(Q)}\right)} \\
	& {\Omega\mathcal{C}(Q)} \\
	& {\Omega^{\infty+2}(\Hsp(Q))} & {\Omega^{\infty+2}\big(\Hsp(Q)_{hC_2}\big)} \\
	{\Omega^{\infty+2}\big(\Theta\Bd_Q^{(1)}\big)} &&& {\Omega^{\infty+2}\big((\Theta\Bd_Q^{(1)})_{h\O(1)}\big).}
	\arrow[from=1-1, to=1-4]
	\arrow["\simeq"', "\mathrm{alex}", from=2-2, to=1-1]
	\arrow["{\mathrm{stab.}}", hook, from=1-1, to=5-1]
	\arrow["{\Phi^{\Bd_Q}_\infty}", from=1-4, to=5-4]
	\arrow["{j\circ\alpha}", from=2-2, to=2-3]
	\arrow[hook, from=2-2, to=3-2, "\mathrm{stab.}"]
	\arrow[equal, from=2-3, to=1-4]
	\arrow["{\Phi^{\Bd_Q}_\infty}", from=2-3, to=4-3]
	\arrow["\vsim","\eqref{cCvscH}+\eqref{StableCEmbIETSeq}"', dash, from=3-2, to=4-2]
	\arrow["q", from=4-2, to=4-3]
	\arrow[equals, from=4-2, to=5-1]
	\arrow[equals, from=4-3, to=5-4]
	\arrow["q", from=5-1, to=5-4]
\end{tikzcd}}\]
The right and bottom subsquares are commutative by definition. The left subsquare is \eqref{ConcStabilisationInclusion}, and is commutative by \cite[Lem. 1.12 \& Cor. 1.13]{WWI} (from which the name of the map ``alex'' is borrowed, and whose construction we will recall later). It remains to show that the top subsquare is also homotopy commutative; that is, we need to construct a homotopy commutative diagram (top subsquare itself commutes by definition)
\begin{equation}\label{Qsquare}
\begin{tikzcd}[row sep = 14pt]
    \Omega \mathsf{Bd}^{(1)}(0)\rar & \Omega \hofib(\mathsf{Bd}(0)\to \mathsf{Bd}(\RR^\infty))\\
    \Omega^2\frac{\Diff^b_\partial(Q\times \RR)}{\Diff_\partial(Q)}\ar[u, equal]\rar&\Omega^2\frac{\Diff^b_\partial(Q\times \RR^\infty)}{\Diff_\partial(Q)}\uar[equal]\dar[hook, "\vsim"]\\
     & \Omega^2\frac{\bDiff^b_\partial(Q\times \RR^\infty)}{\Diff_\partial(Q)}\\
\Omega \C(Q)\ar[uu,"\mathrm{alex}", "\vsim"']\rar["\alpha"] & \Omega^2\left(\frac{\bDiff}{\Diff}\right)_\partial(Q)\uar
\end{tikzcd}
\end{equation}
that is natural in codimension-zero embeddings of compact manifolds $Q\xhookrightarrow{}Q'$. The bottom square in \eqref{Qsquare} is itself obtained as the homotopy fibre of the map from the outer square to the inner square in
\begin{equation}\label{bigQcub}
\adjustbox{scale=0.8,center}{
\begin{tikzcd}[row sep = 15pt, column sep = 18pt]
	{\Omega\Diff_\partial(Q)} &&&& {\Omega\Diff_\partial(Q)} \\
	& {\Omega\Diff^b_\partial(Q\times \RR)} && {\Omega\Diff^b_\partial(Q\times \RR^\infty)} \\
	&&& {\Omega\bDiff^b_\partial(Q\times \RR^\infty)} \\
	& {\Diff_\partial(Q\times I)} & {\bDiff_\partial(Q\times I)} & {\Omega\bDiff_\partial(Q)} \\
	{\Omega\Diff_\partial(Q)} &&&& {\Omega\Diff_\partial(Q).}
	\arrow[equals, from=1-1, to=1-5]
	\arrow["-\times \Id_{\RR}"',from=1-1, to=2-2]
	\arrow[from=1-5, to=2-4]
	\arrow[from=2-2, to=2-4]
	\arrow["\vsim", from=2-4, to=3-4]
	\arrow["{\mathrm{alex}}", "\vsim"', from=4-2, to=2-2]
	\arrow[from=4-2, to=4-3]
	\arrow[from=4-4, to=3-4]
	\arrow["{\widetilde{\Gamma}}"', "\sim", from=4-4, to=4-3]
	\arrow[equals, from=5-1, to=1-1]
	\arrow["\Gamma", from=5-1, to=4-2]
	\arrow[equals, from=5-1, to=5-5]
	\arrow[equals, from=5-5, to=1-5]
	\arrow[from=5-5, to=4-4]
\end{tikzcd}}
\end{equation}
It remains to show that this diagram is homotopy commutative---this is clear for all subsquares except the left and center ones.

We first recall the definition of the Alexander trick-like map ``alex'': given a diffeomorphism $\phi\in \Diff_\partial(Q\times I)$, we obtain a bounded diffeomorphism $\overline\phi$ of $Q\times \RR$ by extending $\phi$ by the identity on $M\times \RR- M\times I$. Then, the based loop $\mathrm{alex}(\phi)$ is given by
$$
\mathrm{alex}(\phi): [-\infty,+\infty]\longrightarrow \Diff^b_\partial(Q\times \RR), \quad s\longmapsto\left\{
\begin{array}{cl}
    \Id_{Q\times \RR} & s=\pm\infty, \\[4pt]
     \tau_s\circ \overline{\phi}\circ \tau_{-s} & s\in (-\infty,+\infty),
\end{array}
\right.
$$
where $\tau_s: Q\times \RR\cong Q\times \RR$ is the translation map $\tau_s(x,t)\coloneq(x,t+s)$.

We now describe an explicit homotopy $\mathrm{alex}\circ\Gamma\sim (-)\times \Id_{\RR}$, showing that the left subsquare in \eqref{bigQcub} is homotopy commutative. Fix once and for all a smooth identification $\psi: [0,+\infty]\cong [0,1]$. Then, such homotopy is given by a map
$$
H: [0,+\infty]\times [-\infty,+\infty]\times \Omega\Diff_\partial(Q)\longrightarrow \Diff^b_\partial(Q\times \RR), \quad (t,s,\{\phi_r\}_{r\in [0,1]})\longmapsto H_t(s,\{\phi_r\})
$$
where $H_t(s,\{\phi_r\})$ is the bounded diffeomorphism of $Q\times \RR$ given by
$$
H_t(s,\{\phi_r\})\equiv\left\{
\begin{array}{cl}
    \tau_{s-t}\circ\overline{\Gamma(\{\phi_r\})}\circ \tau_{t-s}& \text{on $(-\infty, s+\psi(t)-t]$},  \\[4pt]
     \phi_{\psi(t)}\times \Id_{\RR}& \text{on $[s+\psi(t)-t, s+\psi(t)+t]$},  \\[4pt]
     \tau_{s+t}\circ\overline{\Gamma(\{\phi_r\})}\circ \tau_{-s-t}& \text{on $[s+\psi(t)+t,+\infty)$.}
\end{array}
\right.
$$
There are some issues regarding the smoothness of $H_t(s,\{\phi_r\})$ at the points $Q\times \{s+\psi(t)\pm t\}$, but these can easily be overcome by choosing a suitable smoothing of the piecewise linear function
$$
[0+\infty)\times\RR\to\RR, \quad (t,x)\longmapsto\left\{
\begin{array}{cl}
    x+t, & x\leq -t, \\
    0, & |x|\leq t,\\
    x-t, & x\geq t.
\end{array}
\right.
$$

To see that the central subsquare in \eqref{bigQcub} is also homotopy commutative, note that it factors as
\[\begin{tikzcd}[row sep = 20pt]
	{\Omega\Diff^b_\partial(Q\times \RR)} && {\Omega\Diff^b_\partial(Q\times \RR^\infty)} \\
	& {\Omega\bDiff_\partial^b(Q\times \RR)} & {\Omega\bDiff^b_\partial(Q\times \RR^\infty)} \\
	{\Diff_\partial(Q\times I)} & {\bDiff_\partial(Q\times I)} & {\Omega\bDiff_\partial(Q)}
	\arrow[from=1-1, to=1-3]
	\arrow[from=1-1, to=2-2]
	\arrow["\vsim", from=1-3, to=2-3]
	\arrow[from=2-2, to=2-3]
	\arrow["{\mathrm{alex}}", "\vsim"', from=3-1, to=1-1]
	\arrow[from=3-1, to=3-2]
    \arrow["-\times \Id_{\RR}"', from=3-3, to=2-2]
	\arrow["{\mathrm{alex}}","\vsim"', from=3-2, to=2-2]
	\arrow[from=3-3, to=2-3]
	\arrow["{\widetilde{\Gamma}}"', from=3-3, to=3-2]
\end{tikzcd}\]
All sub-parts of this diagram now homotopy commute either by definition, or by setting up the homotopy $H$ above simplicially, which is routine. This finishes the proof of the proposition.
\end{proof}

\begin{rmk}
    As a consequence of this proof and that of \cref{IntermediateBoringAlphaProp}, it follows that the square in that statement looped once is equivalent to that in \eqref{PhiFsquare} for $F=\sfE_{X_g}$ when looped twice.
\end{rmk}

\printbibliography[heading=bibintoc]

\textsc{Department of Mathematics, Karlsruhe Institute of Technology, 76131 Karlsruhe, Germany}

\textit{Email address:} \url{joao.fernandes@kit.edu}

\textsc{Simons Building, Massachusetts Institute of Technology, Cambridge, MA 02139, USA}

\textit{Email address:} \url{smunoze@mit.edu}

\end{document}